%% file: note.tex
\documentclass[10pt,notitlepage]{amsart}
\input{macros}
\usepackage[normalem]{ulem}

\DeclareMathOperator{\codim}{codim}

\DeclareMathOperator{\Hess}{Hess}

\DeclareMathOperator{\dgmod}{dgmod}

\newcommand{\bB}{\mathbb{B}}
\newcommand{\qwer}{\varphi}

\newcommand{\lsub}[2]{\phantom{}_{#1}\!{#2}}
\newcommand{\lsup}[2]{\phantom{}^{#1}\!{#2}}
\newcommand{\lsh}[1]{\lsup{!}{#1}}

\DeclareMathOperator{\Inj}{Inj}

\newcommand{\sXZ}{\oh{\X_{\mkern-2muZ}}}

\newcommand{\DArt}{\mathbf{DArt}}
\newcommand{\DRng}{\mathbf{DRng}}
\newcommand{\DRngfp}{\mathbf{DRng}^{\mkern-1mu\scriptstyle{\mathrm{fp}}}}

\newcommand{\bt}{\beta}
\DeclareMathOperator{\vol}{vol}
\DeclareMathOperator{\coev}{coev}
\DeclareMathOperator{\cotr}{cotr}
\DeclareMathOperator{\coid}{coid}

\renewcommand{\ev}{\operatorname{ev}}

\newcommand{\dgcatidm}{\dgcat^{\mkern-2mu\scriptstyle{\mathrm{idm}}}}
\newcommand{\dgcatbig}{\dgcat^{\mkern-2mu\scriptstyle{\infty}}}

\begin{document}
\title{Thom-Sebastiani \& Duality for Matrix Factorizations}
\author{Anatoly Preygel}
\begin{abstract}
The derived category of a hypersurface has an action by ``cohomology operations'' $k\ps{\bt}$, $\deg \bt = -2$, underlying the $2$-periodic structure on its category of singularities (as matrix factorizations).  We prove a Thom-Sebastiani type Theorem, identifying the $k\ps{\bt}$-linear tensor products of these dg-categories with coherent complexes on the zero locus of the sum potential on the product (with a support condition), and identify the dg-category of colimit-preserving $k\ps{\bt}$-linear functors between Ind-completions with Ind-coherent complexes on the zero locus of the difference potential (with a support condition).  These results imply the analogous statements for the $2$-periodic dg-categories of matrix factorizations.
  
Some applications include: we refine and establish the expected computation of $2$-periodic Hochschild invariants of matrix factorizations; we show that the category of matrix factorizations is smooth, and is proper when the critical locus is proper; we show that Calabi-Yau structures on matrix factorizations arise from volume forms on the total space; we establish a version of Kn\"orrer Periodicity for eliminating metabolic quadratic bundles over a base.
\end{abstract}

\maketitle

\setitemize{label=\textendash}

\setcounter{tocdepth}{1}
\tableofcontents

\section{Introduction}
The goal of the present note is to establish some basic results about tensor products and functor categories between $2$-periodic (=$k\pl{\bt}$-linear, $\deg \bt = -2$) dg-categories of matrix factorizations, beyond the case of isolated singularities.  These results are surely unsurprising, however our approach may be of interest: Rather than working directly in the $2$-periodic or curved context, we deduce the results from more refined statements about the $k\ps{\bt}$-linear dg-category of coherent sheaves on the special fiber so that we are able to remain in the more familiar world of coherent sheaves.  This is done using a convenient (derived) geometric description of the $k\ps{\bt}$-linear structure of cohomology operations.

\subsection{Integral transforms for (Ind) coherent complexes}
Suppose $X$ is a nice scheme (or derived scheme, stack, formal scheme, etc.).  The non-commutative viewpoint tells us to forget $X$ and pass to its ``non-commutative'' shadow: the dg-category $\Perf(X)$ or its Ind-completion $\QC(X)$.  Work of To\"en \cite{Toen-DGCat}, \ldots, Ben Zvi-Francis-Nadler \cite{BFN} provide us with useful tools relating the commutative and non-commutative worlds:
\begin{itemize}
  \item A ``tensor product theorem,'' stating that (derived) fiber products of schemes go to tensor products of dg-categories:
    \[ \Perf(X) \otimes_{\Perf(S)} \Perf(Y) \stackrel{\sim}\longrightarrow \Perf(X \times_S Y) 
    \qquad \QC(X) \ohotimes_{\QC(S)} \QC(Y) \stackrel{\sim}\longrightarrow \QC(X \times_S Y) \]
  \item A description of functor categories: Every quasi-coherent complex on the product gives rise to an ``integral transform'' functor, and this determines an equivalence \[ \QC(X \times_S Y) \stackrel{\sim}\longrightarrow \Fun_{\QC(S)}^L(\QC(X), \QC(Y)) \] identifying $\QC(X \times Y)$ with the dg-category $\Fun_{\QC(S)}^L(\QC(X), \QC(Y))$ of colimit-preserving $\QC(S)^{\otimes}$-linear functors $\QC(X) \to \QC(Y)$ (also known as ``bimodules'').  The identity functor corresponds to $\Delta_* \O_X$ (=the diagonal bimodule), and the trace of endofunctors (=Hochschild homology) corresponds to taking global sections of the pullback along the diagonal.  Thus we have descriptions of the functor category and of the Hochschild invariants in familiar commutative terms.
\end{itemize}

When studying non-smooth schemes $X$, it's convenient to replace vector bundles by coherent sheaves.  Analogously, to replace perfect complexes $\Perf(X)$ by (bounded) coherent complexes $\DCoh(X)$;  and, to replace quasi-coherent complexes $\QC(X)$ by the larger $\QCsh(X) = \Ind \DCoh(X)$ of \emph{Ind-coherent (aka shriek quasi-coherent) complexes}.  Provided we work with finite-type schemes over a perfect base-field, the analogs of the above two theorems remain true: this is essentially the content of Lunts paper \cite{Lunts}.\footnote{The author originally learned that such a result might be true from Jacob Lurie, who attributed it to conversation with Dennis Gaitsgory.  The author wrote up the mild extensions of \autoref{app:coh-fmk} before finding Lunts' paper and realizing that it proved essentially the same thing.} \autoref{app:coh-fmk} develops the mild extensions which we will need (to derived schemes, and with support conditions) in the more geometric language that we will wish to use:  The ``tensor product theorem'' for $\DCoh$ and $\QCsh$ (\autoref{prop:fmk-coh-surj}), and a description of functor categories in terms of ``shriek'' integral transforms (\autoref{thm:coh-fmk}).

Two direct applications may be worth highlighting:
\begin{itemize}
  \item One can write down formulas for the Hochschild invariants of $\DCoh(X)$ for not-necessarily smooth $X$, and with support conditions (\autoref{cor:hh}).  For $\bHH_\bullet$, this makes manifest the ``Poincar\'e duality'' between what might be called Hochschild $K$-theory and Hochschild $G$-theory.  For $\bHH^\bullet$, one obtains the somewhat strange looking fact that $\bHH^\bullet(\DCoh(X)) \isom \bHH^\bullet(\Perf(X))$.
  \item Obviously one re-obtains Lunts' result that $\DCoh(X)$ is smooth, but one also sees that this fails for even very nice formal schemes:\footnote{There are several ways one could wish to define $\DCoh(\XZ)$: Our choice is as the compact objects in $\QCsh(\XZ)$ which is constructed as the $\infty$-categorical inverse limit along shriek-pullback of $\QCsh$ on nilthickenings of $Z$.  See \autoref{thm:cpltn} for a sketch of the comparison and references.} $\DCoh_Z(X) = \DCoh(\XZ)$ is \emph{not} usually smooth (even when both $Z$ and $X$ are smooth), though the failure of smoothness is in a sense mild (e.g., the identity functor is a uniformly $t$-bounded filtered colimit of compacts).  One consequence is that $\bHH_\bullet(\DCoh_Z(X))$ admits a nice description while $\bHH^\bullet(\DCoh_Z(X))$ does not.
\end{itemize}

\subsection{Matrix factorizations}
Suppose $f\colon  M \to \AA^1$ is a map from a smooth scheme to $\AA^1$, and that we are interested in the geometry of $f$ over a formal disc near the origin.  We can attach to it several non-commutative shadows, the two simplest candidates being the dg-categories $\DCoh(M_0)$ and $\Perf(\oh{M_0})=\Perf_{M_0}(M)$.  However, these both lose too much information: they do not depend on the defining function $f$, and the second one doesn't even depend on the scheme structure on $M_0$.  A standard way to remedy this is to consider the $2$-periodic(=$k\pl{\bt}$-linear, $\deg \bt = -2$) dg-category $\MF(M, f) \isom \DSing(M_0)$ of ``matrix factorizations'' or ``LG D-branes'' (at a \emph{single} critical value).

A starting point for out study is the observation\footnote{Due in parts to several people, notably Constantin Teleman for the connection to $S^1$-actions.} that there are three (essentially equivalent, pairwise Koszul dual) refinements of this.   Using $f$, one can put extra structure ($k\ps{\bt}$-linearity) on $\DCoh(M_0)$ and extra structure (an $S^1$-action) on $\Perf(\oh{M_0})$:
\begin{enumerate}
  \item One can regard $\Perf(\oh{M_0}) = \Perf_{M_0}(M)$ as linear over $\Perf(\oh{0})^{\otimes} = \Perf_0(\AA^1)^{\otimes}$.
    \item (See \autoref{sec:mf-gen}.) The $2$-periodicity on $\DSing(M_0)$ comes from a $k\ps{\bt}$-linear structure on $\DCoh(M_0)$, for which we give a (derived) geometric description in \autoref{ssec:Bact}.  We call this $k\ps{\bt}$-linear dg-category $\PreMF(M,f)$ to emphasize the dependence on $f$.  Despite the ``Pre'' in the name, $\PreMF(M,f)$ is a refinement of $\MF(M,f)$. We'll see in \autoref{cor:mf-basechange} and \autoref{cor:premf-supt} that $\PreMF(M,f)$ allows one to recover all the other actors in the story:
    \[ \PreMF(M,f) \otimes_{k\ps{\bt}} \left\{\text{locally $\bt$-torsion $k\ps{\bt}$-modules}\right\} \isom \Perf(M_0) \]
    \[ \PreMF(M,f) \otimes_{k\ps{\bt}} k\pl{\bt}\mod \isom \MF(M,f) \]
    \[ \PreMF(M,f) \otimes_{k\ps{\bt}} k \isom \DCoh_{M_0}(M) \]
    \item (See \autoref{sec:mf-gps}.) There is a (\emph{homotopy}) $S^1$-action on $\DCoh_{M_0}(M) = \DCoh(\oh{M_0})$. This $S^1$-action is fundamental, as it allows one to recover the other actors in the story (\autoref{lem:kt-mf}):\footnote{It \emph{is} important in the following formulas that we passed to compact objects: Taking invariants does not commute with forming $\Ind$ categories.}
    \[ \DCoh(\oh{M_0})_{S^1} \isom \Perf(M_0) \]
    \[ \DCoh(\oh{M_0})^{S^1} \isom \PreMF(M,f) \]
    \[ \DCoh(\oh{M_0})^{\Tate} \isom \MF(M,f) \]
    where these equivalence are $C^*(BS^1)=k\ps{\bt}$-linear.  Furthermore, one can avoid completing along the zero fiber (/imposing support conditions) in two ways: Replacing maps $M \to \AA^1$ by maps $M \to \GG_m$; or, replacing $S^1$-actions with $B\oh{\GG}_a$-actions.
\end{enumerate}

Each viewpoint has its own pros and cons for our purposes:
\begin{itemize}
  \item Viewpoint (iii) is well-suited to formulating comparisons between structures and invariants for  $\Perf(\oh{M_0})$ or $\Perf(M)$ over $k$, and $\PreMF(M,f)$ (resp., $\MF(M,f)$) over $k\ps{\bt}$ (resp., $k\pl{\bt}$).  This viewpoint admits variants which, before passing to invariants, remember global information rather than being completed near the zero fiber.
  \item Viewpoint (ii) is well-suited for reducing questions about $\PreMF(M,f)$ (resp., $\MF(M,f)$) over $k\ps{\bt}$ (resp., $k\pl{\bt}$) to questions about coherent complexes in (derived) algebraic geometry.  Using this we deduce $k\ps{\bt}$- and $k\pl{\bt}$-linear versions of the tensor product theorem (\autoref{thm:TS}) and identifications of functor categories (\autoref{thm:functors}).  It is worth noting that in the $k\ps{\bt}$-linear context, certain \emph{support conditions} appear naturally.\footnote{In the $2$-periodic case, it is largely possible to ignore these by e.g., summing over critical values.}
\end{itemize}

\subsection{Summary of results}
For us an \demph{LG pair} $(M, f)$ consists of a smooth orbifold $M$ and a map $f \colon M \to \AA^1$, not necessarily flat. Then $\PreMF(M, f) = \DCoh(M \times_{\AA^1} 0)$ is coherent complexes on the derived fiber product, equipped with a certain $k\ps{\bt}$-linear structure depending on $f$; $\MF(M,f) = \PreMF(M,f) \otimes_{k\ps{\bt}} k\pl{\bt}$ is its two-periodic version.  Our main results are variants of the ``tensor product theorem'' and description of functor categories in the $k\ps{\bt}$-linear context:
\newtheorem*{thm:ts}{\autoref{thm:TS}}
\begin{thm:ts}[``Thom-Sebastiani''] Suppose $(M,f)$ and $(N, g)$ are two LG pairs.  Set $M_0 =f^{-1}(0)$, $N_0 = g^{-1}(0)$, $(M \times N)_0 = (f \boxplus g)^{-1}(0)$, and let  $\ell: M_0 \times N_0 \to (M \times N)_0$ be the inclusion.
  Then, there is a $k\ps{\bt}$-linear equivalence
  \[ \ell_*(-\boxtimes-)\colon \PreMF(M, f) \otimes_{k\ps{\bt}} \PreMF(N,g) \stackrel{\sim}\longrightarrow \PreMF_{M_0 \times N_0}(M \times N, f \boxplus g)  \]
\end{thm:ts}
\newtheorem*{thm:dual-fun}{\autoref{thm:dual} and \autoref{thm:functors}}
\begin{thm:dual-fun}[``Duality and Functors''] Let $(M,f)$, $(N,g)$, etc. be as before.  Grothendieck duality for $\DCoh(M_0)$ lifts to a $k\ps{\bt}$-linear anti-equivalence $\PreMF(M,f)^\op \isom \PreMF(M,-f)$, and with the above this induces a $k\ps{\bt}$-linear equivalence of dg-categories
  \begin{align*} \Fun^L_{k\ps{\bt}}\left(\PreMF^\kinfty(M,f), \PreMF^\kinfty(N,g)\right) &= \PreMF^\kinfty(M,-f) \ohotimes_{k\ps{\bt}} \PreMF^\kinfty(N, g) \\ 
    &= \PreMF^\kinfty_{M_0 \times N_0}\left(M \times N, -f \boxplus g\right) \end{align*}
  In case $(M,f)=(N,g)$, there are explicit descriptions of the identity functor and ``evaluation''(=Hochschild homology).
\end{thm:dual-fun}
The reader is directed to the actual statements of the Theorems below for variants: support conditions, the $2$-periodic versions, and removing support conditions in the $2$-periodic setting.

As applications of the main results, we establish several expected computations and properties $\MF$.  In the case of computing Hochschild invariants, we also obtain $k\ps{\bt}$-linear refinements.
\newtheorem*{thm:hoch-intro}{\autoref{thm:hoch}}
\begin{thm:hoch-intro} The expected computations of $2$-periodic Hochschild invariants for matrix factorizations hold. The description $\bHH_\bullet^{k\pl{\bt}}(\MF^\tot) = \bHH_\bullet(\Perf M)^{\Tate}$, and its $\bHH^\bullet$ analog, comes from a $B\oh{\GG_a}$-action on the dg-category $\Perf(M)$ inducing a $B\oh{\GG_a}$-action on $\bHH_\bullet(\Perf(M))$ and is consequently compatible with all the functorially attached structures on Hochschild invariants ($\SO(2)$-action on $\bHH_\bullet$, $E_2$-algebra structure on $\bHH^\bullet$, etc.).  Moreover, there are $k\ps{\bt}$-linear refinements, which in the case of $M$ a scheme can be explicitly identified via HKR and local cohomology
    \[ \bHH_\bullet^{k\ps{\bt}}\left(\PreMF(M,f)\right) = \bHH_\bullet^k(\DCoh_{M_0}(M))^{S^1} \left(\isom \RGamma_{M_0}\left(\left[ \Omega^\bullet_M\ps{\bt}, \bt \cdot (-df \wedge -)\right]\right)\right) \]
    (The reader is directed to the body of the text for a more precise statement.)
\end{thm:hoch-intro}
\newtheorem*{thm:smp}{\autoref{thm:sm-proper}}
\begin{thm:smp} Suppose $(M,f)$ is an LG pair. Then, $\MF(M, f)$ is smooth over $k\pl{\bt}$, and is proper over $k\pl{\bt}$ provided that $\crit(f) \cap f^{-1}(0)$ is proper.
\end{thm:smp}
\newtheorem*{thm:cy-intro}{\autoref{thm:cy}} 
\begin{thm:cy-intro}Suppose $(M,f)$ is an LG pair, $m = \dim M$, and that $M$ is equipped with a volume form $\vol_M\colon \O_m \isom \omega_M[-m]\left(=\Omega^m_M\right)$.  Then, $\vol_M$ determines an $m$-Calabi-Yau structure (in the smooth, non-proper sense) on $\MF(M,f)$ over $k\pl{\bt}$.
\end{thm:cy-intro}

Using a mild extension of the ``tensor product theorem'' we prove an extension of Kn\"orrer periodicity allowing one to discard metabolic quadratic bundles; motivated by this, we identify matrix factorizations for quadratic bundles with sheaves over Clifford algebras (and the $k\ps{\bt}$-linear analog, upon imposing a support condition).
\newtheorem*{thm:quad}{\autoref{thm:rel-knorrer} and \autoref{thm:mf-cliff}}
\begin{thm:quad} Suppose $(M,f)$ is an LG pair, and $\tQ$ a non-degenerate quadratic bundle over $M$.  View $(\tQ, q)$ as an LG pair.  Then, the structure sheaf $\O_M$ induces equivalences
  \[ \PreMF^\kinfty_M(\tQ,q) \isom \PreCliff_{\O_M}(\tQ)\mod(\QC(M))\text{ and }\MF^\kinfty(\tQ,q) \isom \Cliff_{\O_M}(\tQ)_{\ZZ/2}\mod_{dg\ZZ/2}(\QC(M)) \]
  Exterior product over $M$ induces an equivalence
  \[ \PreMF(M,f) \otimes_{\Perf(M)\ps{\bt}} \PreMF_M(\tQ,q) \stackrel{\sim}\longrightarrow \PreMF(\tQ, f+q) \]
  and its $2$-periodic analog.  Finally, if $\tQ$ is \emph{metabolic}, in the sense of admitting a Lagrangian sub-bundle $\tL \subset \tQ$, then tensoring by $\O_\tL$ induces an equivalence 
  \[ \O_{\tL} \otimes - \colon \Perf(M)\ps{\bt} = \MF(M, 0) \longrightarrow \MF(\tQ,q). \]
\end{thm:quad}

\subsection{Comments}  A few comments on the main ingredients and methods:
\begin{enumerate}
  \item In addition to the language of derived algebraic geometry, we make use of Grothendieck duality/the upper-shriek functor for $\QCsh$ of derived schemes (and certain nice derived DM stacks).  A reference for this does not yet exist in the literature, and it was not the purpose of this article to provide one; a reference will probably appear in a later volume of Jacob Lurie's DAG series. Partially for this reason, we have chosen to present (in \autoref{sec:main-thm}) a proof of the Main Theorems which we hope is reasonably concrete and minimizes the use of this general machinery.   If one is only interested in matrix factorizations for a flat map $f\colon  M \to \AA^1$ from a smooth scheme, one only needs extra input in one place: determining the kernel of the identity functor in \autoref{thm:TS} uses duality and base-change properties on some very mild derived schemes.
  \item We were heavily inspired by Constantin Teleman's description of $\PreMF$ (resp., $\MF$) arising as $S^1$-invariants (resp., Tate construction) of perfect complexes on the total space: this proves to be a great organizational principle, as well as a useful tool for obtaining natural comparison maps.
  \item The $k\ps{\bt}$-linear structure, and its relation to $\DSing$ of a hypersurface, is well-known in the commutative-algebra literature (as ``cohomology operations'' on $\DCoh$ of ci rings).  Paul Seidel's  2002 ICM talk \cite{Seidel-ICM} shows that a similar structure occurs on the ``mirror side'' (i.e., the bottom row below)
  \begin{center}
   \begin{tabular}{c|c|c|c}
     ``Generic fiber'' & ``Special fiber'' & ``Global sections'' \\ 
     Over $k\pl{\bt}$ & Over $k$ & Over $k\ps{\bt}$ \\\hline
     $\MF(M,f)$ & $\DCoh(M)$ & $\DCoh(M_0)$ \\
     $\Fuk(X)$ &  $\Fuk^\text{wrap}(X \setminus D)$ & $\Fuk^\text{FS}(X, D)$ \\
   \end{tabular}
 \end{center} and it would be interesting to see in what sense, if any, the above table actually matches up.  Seidel's preprint \cite{Seidel-Ainfty}  also explicitly mentions the description of $\DSing$ as arising by inverting $\bt$ on $\PreMF$.
 \item We freely use abstract $\infty$-categorical tools to make life easier: relative tensor products, $\Ind$ completions, limits and colimits (esp. in $\Pr^L$ and $\Pr^R$).
 \end{enumerate}

While this paper was being completed, the author learned that similar results on identifying functor categories and Hochschild-invariants (in the $2$-periodic case) have been independently obtained by Kevin Lin and Daniel Pomerleano.  Their approach is different, using methods of curved dg-modules.

\subsection{Acknowledgements} The author wishes to thank Jacob Lurie for numerous incredibly helpful conversations. 
The author wishes to thank Constantin Teleman for allowing us to include his viewpoint on $\MF$ via $S^1$-actions on categories: this novel viewpoint formed the basis for the author's thinking on the subject.  
The author wishes to thank Paul Seidel for an inspiring conversation.

\section{Notation and background}\label{sec:notation}
\subsection{Grading conventions}
\begin{itemize}
  \item We work throughout over a fixed characteristic zero field $k$.
\item We use \demph{homological grading} conventions (i.e., differentials increase degree) and we write $\pi_{-i}$ for $H^i = H_i$; e.g., $\Ext^i(M,N) = \pi_{-i} \RHom(M,N)$).  For a chain complex $M$, the symbol $M[n]$ denotes the chain complex with $M[n]_k = M_{k-n}$ (i.e., if $M$ is in degree $0$, then $M[n]$ is in homological degree $+n$).
\item $k\ps{\bt}$, $k\pl{\bt}$ will denote the graded-commutative $k$-algebras with (homological) $\deg \bt = -2$. Fix once and for all an equivalence $C^*(BS^1;k) = k\ps{\bt}$ (say by $\bt \mapsto c_1(\O(1))$ in the Chern-Weil model for $c_1$).
  \item We will write \demph{$t$-bounded-below} for what might otherwise be called \emph{homologically bounded-below} = \emph{cohomologically bounded-above} = \emph{almost connective}=\emph{right-bounded}.  Similarly for \demph{$t$-bounded-above} = \emph{homologically bounded} = \emph{truncated} = \emph{left-bounded}; and for \demph{$t$-bounded} = \emph{(co/)homologically bounded}.  For example, if $A, B$ are discrete $R$-modules, then $A \Lotimes B$ is $t$-bounded-below, while $\RHom(A,B)$ is $t$-bounded-above.
\end{itemize}

\subsection{Reminder on dg-categories and $\infty$-categories}
For background on $\infty$-categories and dg-categories, the reader is direct to e.g., \cite{T} and \cite{Toen-DGCat}.

\begin{itemize}
  \item Let $\dgcat_k$ be the $\infty$-category of $k$-linear dg-categories with quasi-equivalences inverted; a Theorem of To\"en identifies this with the (nerve) of the simplicial category whose morphisms are (Kan replacements of the nerve of) a certain full subspace of the $\infty$-groupoid of bimodules.  Let $\dgcatidm_k$ be the $\infty$-categorical ``Morita localization'' of $\dgcat_k$; it may be identified with the $\infty$-category of \demph{small stable idempotent complete $k$-linear $\infty$-categories} (with exact $k$-linear functors).  Let $\dgcatbig_k$ denote the $\infty$-category of \demph{stable cocomplete $k$-linear $\infty$-categories} (with colimit preserving $k$-linear functors). 
  \item We will generally write $\Map$ for simplicial mapping spaces and $\RHom$ (with various decorations) for $k$-linear mapping complexes, so that e.g., $\Map(x,y) \isom \Omega^\infty \RHom(x,y)$. ($\Omega^\infty$ denotes taking the infinite loop space corresponding to a spectrum; if $\RHom(x,y)$ is a $k$-linear chain complex, this may be interpreted as applying the Dold-Kan construction to the connected cover $\tau_{\geq 0} \RHom(x,y)$.)
  \item $\dgcatidm_k$ (resp., $\dgcatbig_k$) is equipped with a symmetric-monoidal \demph{tensor product} $\otimes = \otimes_k$ (resp., $\ohotimes = \ohotimes_k$).  These satisfy the compatibility $\Ind(\C \otimes \C') = \Ind\C \ohotimes \Ind\C'$.  (In particular, $\ohotimes$ preserves the property of being compactly-generated.)
  \item  Many of our dg-categories will be $R=k\ps{\bt}$- or $R = k\pl{\bt}$-linear, in the sense of being module-categories for the symmetric-monoidal $\infty$-categories $\Perf(R) \in \CAlg(\dgcatidm_k)$ (resp., $R\mod \in \CAlg(\dgcatbig_k)$): Heuristically, this is a $\C \in \dgcatidm_k$ equipped with a $k$-linear $\otimes\colon \Perf(R) \times \C \to \C$ suitably compatible with $\otimes_R$ on $\Perf(R)$.  
    This notion gives rise to the same $\infty$-categories $\dgcatidm_R$ (resp., $\dgcatbig_R$) as the more rigid notion of literal $R$-linear dg-category, but is more convenient for our purposes.  If $R$ is a commutative dga, when no confusing arises we will sometimes write $R$ in place of $\Perf R$ or $R\mod$: i.e.,
    \[ \C \otimes_{R} \C' \eqdef \C \otimes_{\Perf R} \C' \qquad \C \otimes_R R' \eqdef \C \otimes_{\Perf R} \Perf R' \]
    \[ \D \ohotimes_{R} \D' \eqdef \D \ohotimes_{R\mod} \D' \qquad \D \ohotimes_R R' \eqdef \D \ohotimes_{R\mod} R'\mod \]
  \item The internal-Hom associated to $\ohotimes_R$ will denoted $\Fun^L_R(-,-)$ (the ``$L$'' standing for left-adjoint, i.e., colimit preserving); $\Fun^L_R(-,-)$ is explicitly a dg-category of bimodules.  Similarly, $\Fun_R^{ex}(-,-)$ will denote the $\infty$-category of exact (i.e., finite limit- and colimit-preserving) functors, etc.
  \item If $\C \in \dgcatidm_R$ or $\dgcatbig_R$, then there is a functor $\RHom^{\otimes_R}_\C(-,-)\colon \C^\op \times \C \to R\mod$ determined up to contractible choices by
    \[ \Map_{R\mod}\left(V, \RHom^{\otimes_R}_\C(\F, \G)\right) = \Map_{\C}\left(V \otimes_R \F, \G\right) \quad \text{for} \quad V \in \Perf(R), \text{and } \F, \G \in \C \]
   Similarly if $R\mod$ is replaced by another rigid cocomplete symmetric-monoidal $\infty$-category (e.g., $\QC(X)$ for $X$ a perfect stack).
  \item $\Pr^L$ (resp., $\Pr^R$) denotes the $\infty$-category of \demph{presentable $\infty$-categories} and left (resp., right) adjoint functors.  They are anti-equivalent, admit small limits and colimits, and forgetting down to $\Cat_\infty$ preserves limits.  Colimits in $\Pr^L$ of a diagram of compactly-generated categories along functors preserving compact objects can be computed by taking $\Ind$ of the colimit of the resulting diagram of categories of compact objects.
\end{itemize}

\subsection{Derived schemes, stacks, etc.}
Mild derived schemes will come up naturally for us.  In order to be able to uniformly discuss the orbifold, and graded, contexts we will also need some mild derived stacks.  The very simplest variants suffice for our desired applications, since for us all the derivedness will be affine over an underived base.  Nevertheless, we find it convenient to use the general language (and in \autoref{sec:more-general} and the Appendices we prove things about derived stacks more general than necessary for our applications).  Our primary references for derived algebraic geometry are \cite{Lurie-thesis}, the DAGs, and To\"en/To\"en-Vezzosi.  Since there does not seem to be a good universal source for notation or terminology, we make clear our choices:

\begin{itemize} 
  \item  Our \demph{derived rings}, $\DRng_k$, will be connective commutative dg-$k$-algebras.  We say that $A \in \DRng_k$ is \demph{coherent} (resp., \demph{Noetherian}) if $\pi_0 A$ is coherent (resp., Noetherian) and each $\pi_i A$ is finitely-presented over $\pi_0 A$.  Meanwhile, $\DRngfp_k$ will be the full subcategory of almost finitely-presented commutative dg-$k$-algebras (=those which are Noetherian with $\pi_0 A$ finitely-presented over $k$). 
   A(n almost finitely-presented) \demph{derived space} is an \'etale sheaf in $\Fun(\DRngfp_k, \Sp)$. A(n almost finitely-presented) \demph{derived $n$-stack} is a derived space which admits a smooth surjection from a disjoint union of affine schemes, such that this map is a relative derived $(n-1)$-stack.  (For $n=0$, take one of affine derived schemes, (Zariski) derived schemes, or derived algebraic spaces.  The first notion gives rise to ``geometric $n$-stack,'' while the last gives the one most easily comparable to usual stacks.)
 \item A \emph{derived scheme} is a Zariski-locally (derived-)ringed space $X = (\X, \O_X)$ which is locally equivalent as such to the Zariski spectrum $\Spec A$ for $A \in \DRng_k$.  A \demph{derived DM stack} (resp., \emph{derived algebraic space}) is an \'etale-locally (derived-)ringed topos $X = (\X, \O_X)$ which is locally equivalent as such to the \'etale spectrum $\Spec A$ for $A \in \DRng_k$.\footnote{Except for algebraic spaces, these definitions are more restrictive than those in \cite{Lurie-thesis}, disallowing any derived-ness in the gluing process.  This is rigged so that e.g., a derived DM stack will have an underlying ($1$-)stack.}.  Having said that, we will \emph{forget it}: We will identify (almost finitely-presented) derived schemes/derived algebraic spaces/derived DM-stacks with their functors-of-points as derived spaces, and will restrict to quasi-compact ones with affine diagonal (so that \emph{affine} derived schemes are the building blocks).
    \item For a derived $n$-stack $X$, there is a universal \demph{discrete (aka ``$0$-truncated'')} derived $n$-stack mapping to it: $\pi_0 X = \Spec_X (\pi_0 \O_X) \to X$ .  Note that this morphism is affine and indeed a closed immersion.  Note also that $\pi_0 X$ is in the essential image of ``ordinary'' Artin $n$-stacks (for $n=1$, it seems justifiable to remove the quotes around ordinary!).\footnote{In particular, for $n > 1$, $\pi_0 X$ need not be equivalent to an ordinary ($1$-)stack.  The issue is most apparent when thinking of derived (DM) stacks in terms of $\infty$-topoi, where the issue is analogous to the difference between an (ordinary) DM stack and a coarse moduli space.  Writing $X = (\X, \O_X)$, we have $\pi_0 X = (\X, \pi_0 \O_X)$.   There \emph{is} an underlying (ordinary) DM stack $\ul{X} = (\tau_{\leq 0} X, \pi_0 \O_X)$, but the natural map $\X \to i_{\leq 0} \tau_{\leq 0} \X$ need not be an equivalence.  The prototypical failure mode is the following: Choose $E_\bullet$ a simplicial diagram of (ordinary) stacks \'etale over $\ul{X}$, and let $\X'$ be the $\infty$-topos of \'etale sheaves of spaces on $\tau_{\leq 0} X$ \emph{over} the geometric realization $|E_\bullet|$; then $(\X', \res{\pi_0 \O_X}{\X'})$ is a perfectly good discrete DM stack, which is not in any reasonable way a nilthickening of an ordinary DM stack.}
    \item For a derived stack $X$: $\QC(X)$ denotes the $k$-linear (stable cocomplete) $\infty$-category of \demph{quasi-coherent complexes} on $X$; it is equipped with a natural $t$-structure, whose heart $\QC(X)^\heart$ is equivalent to the (ordinary) category of quasi-coherent complexes on $\pi_0 X$.  $\Perf(X) \subset \QC(X)$ is the full-subcategory of \demph{perfect complexes}; if $X$ is a quasi-compact and (quasi-)separated derived scheme, or more generally perfect in the sense of \cite{BFN}, then $\QC(X) = \Ind \Perf(X)$.  $\PsCoh(X) \subset \QC(X)$ denotes the full-subcategory of \demph{pseudo-coherent (=``almost perfect'') complexes}, i.e., those $\F \in \QC(X)$ that are (locally) $t$-bounded-below and such that $\tau_{\leq n} \F \in \QC_{\leq n}(X)$ is compact for all $n \in \ZZ$.
  \item We say that a derived stack $X$ is \demph{coherent} (resp., \demph{Noetherian}) if it admits an fppf surjection from $\Spec A$ with $A$ a coherent (resp., Noetherian) derived ring.  If $X$ is coherent, then $\PsCoh(X)$ admits an alternate description: $\F \in \PsCoh(X)$ iff $\F$ is $t$-bounded below and $\pi_i \F$ is a coherent $\pi_0 X$-module for all $i$.  Let $\DCoh(X) \subset \PsCoh(X)$ denote the full subcategory of \demph{coherent complexes}, i.e., complexes with locally bounded, coherent (over $\pi_0 X$), cohomology sheaves.  Let  $\QCsh(X) \eqdef \Ind \DCoh(X)$ denote the $\infty$-category of $\Ind$ objects of $\DCoh(X)$ (``\demph{Ind coherent complexes}'').\footnote{This is not the best \emph{definition} for arbitrary $X$, since it does not manifestly have descent.  Instead, one should make this definition on affines and then extend by gluing as in \autoref{sec:more-general}.  But \autoref{app:descent} implies the two agree on, e.g., reasonable DM stacks.}  We say that $X$ is \demph{regular} if $\Perf(X) = \DCoh(X)$\footnote{For $X = \Spec R$ coherent with $\pi_0 R$ a Noetherian local ring, the inclusion $\DCoh(X) \subset \Perf(X)$ is equivalent to requiring that the residue field $k = R/\mm$ be perfect over $R$.  For $X = \Spec R$ coherent, the inclusion $\Perf(X) \subset \DCoh(X)$ is not automatic since $R$ as it requires that $R$ have only finitely many non-vanishing homotopy groups each of which is finitely-presented over $\pi_0 R$; e.g., it is satisfied for anything of finite Tor-amplitude over an underived stack.  If $X = \Spec R$ coherent with $\pi_0 R$ Noetherian, it seems likely that $X$ is regular iff $R = \pi_0 R$ is a regular ring.}
  \item We say that a derived $n$-stack $X$ is \demph{bounded} if it admits a smooth surjection $U = \Spec A \to X$ which is a bounded relative $(n-1)$-stack.  A $0$-stack (derived scheme or algebraic space) is bounded if it is quasi-compact and quasi-separated.  This is an analog of the technical condition that a scheme is quasi-compact and quasi-separated, but buys one somewhat less: One can try to compute pushforwards via a Cech complex, but this now involves a cosimplicial totalization (rather than a finite limit) and so only commutes with colimits, finite Tor-dimension base-change, etc. on \emph{$t$-bounded-above}. 
  \item With all that out of the way, we now introduce two convenient conditions on a derived stack $X$ (the conditions are somewhat redundant for clarity):
\begin{equation}\label{cond:star}\text{$X$ is Noetherian, has affine diagonal, and is perfect}  \tag{$\star$} \end{equation}
\begin{equation}\label{cond:starf}\text{$X$ is Noetherian, has finite diagonal, is perfect, and is Deligne-Mumford} \tag{$\star_{\mathrm{F}}$} \end{equation}
A \eqref{cond:star} (resp., \eqref{cond:starf}) morphism $f: X \to Y$ of Noetherian derived stacks is one such that $X \times_Y \Spec A$ is an \eqref{cond:star} (resp., \eqref{cond:starf}) derived stack for any $\Spec A \to Y$ almost of finite-presentation.

It will be our \demph{standing assumption} that any derived stack (including plain schemes) for which we consider $\DCoh$ or $\QCsh$ satisfy condition \eqref{cond:star} (and usually they will satisfy \eqref{cond:starf} and be almost finitely-presented over $k$).  Note that \eqref{cond:starf} holds for separated Noetherian schemes, and in char. $0$ for separated Noetherian DM stacks with affine diagonal whose coarse moduli space is a scheme.  Both conditions pass to quotients by finite group schemes (in char. $0$), and $BG$ is \eqref{cond:star} for $G$ reductive (in char. $0$); both pass to things quasi-projective over a base, and are stable under fiber products provided one of the maps is almost of finite-presentation (to preserves the Noetherian condition).
\end{itemize}

\subsection{LG pairs}
\begin{itemize}
  \item An \demph{LG pair} $(M,f)$ is a pair consisting of a smooth \eqref{cond:starf} stack (\demph{``orbifold''}) $M$ over $k$, and a morphism $f \colon M \to \AA^1$. (We do not require $f$ to be non-zero.  However, if $f$ is not flat, various fiber products throughout the paper must be taken in the derived sense.)  If $(M,f)$, $(N, g)$ are two LG pairs, define the \demph{Thom-Sebastiani sum} LG pair to be $(M \times N, f \boxplus g)$ where $(f \boxplus g)(m,n) = f(m) + g(n)$.
  \item For an LG pair $(M,f)$, $\PreMF(M,f)$ denotes the $k\ps{\bt}$-linear $\infty$-category with underlying $k$-linear $\infty$-category $\DCoh(M_0)$ and $\bt$ acting as a ``cohomological operation''.  See \autoref{constr:Bact} below for a geometric description of this structure.  Then, $\MF(X,f)$ denotes the $k\pl{\bt}$-linear (i.e., $2$-periodic) $\infty$-category $\MF(X,f) \eqdef \PreMF(X,f) \otimes_{k\ps{\bt}} k\pl{\bt}$. (The relation of this to actual ``matrix factorizations'' is given by \autoref{prop:comparison} and Orlov's Theorem \cite{Orlov-DBrane}.) We also define Ind-completed versions:
    \[ \PreMF^\kinfty(X,f) \eqdef \Ind\PreMF(X,f) \] \[ \MF^\kinfty(X,f) \eqdef \Ind\MF(X,f) = \PreMF^\kinfty(X,f) \ohotimes_{k\ps{\bt}} k\pl{\bt} \]
  \item $\Omega^\bullet_M \eqdef \oplus_i \Omega^i_M[i]$, i.e., it is placed in homologically positive degrees.  Meanwhile, $\omega_M$ denotes the dualizing complex in its natural degree (not generally zero).  With these conventions if $M$ is smooth of dimension $m$, then $\omega_M \isom \Omega^m_M[m]$ is in homological degree $m$ and there is a sheaf perfect-pairing $\wedge\colon \Omega^\bullet_M \otimes_{\O_M} \Omega^\bullet_M \to \omega_M$.  Similarly $T_M^\bullet \eqdef\oplus_i \bigwedge^i T_M[-i]$, i.e., it is place in homologically negative degrees.
\end{itemize}

\subsection{Primer on $\QCsh = \Ind \DCoh$} 
\begin{na} The usual construction of $\Ind \C$, as the full subcategory of the functor category $\Fun(\C, \Sp)$ generated under filtered colimits by the image of the Yoneda functor, provides a description
  \[ \QCsh(X) = \Fun^{Lex}(\DCoh(X)^\op, \Sp) \qquad \DCoh(X) \ni \K \mapsto \RHom(-, \K) \in \Fun^{Lex}(\DCoh(X)^\op, \Sp) \] 
  where $\Fun^{Lex}$ denotes the full-subcategory of functors preserving finite limits.  In dg-language, this translates to an identification of $\QCsh(X)$ with $\dgmod_k(\DCoh(X)^\op)$: (the derived category of) dg-modules over a dg-category model $\DCoh(X)^\op$.  Our first step will be giving a slightly smaller model:
\end{na}

\begin{lemma}\label{lem:coh-red} Suppose that $X$ is a coherent derived stack.  
  \begin{enumerate}
    \item  Let $i\colon (\pi_0 X) \to X$ be the universal map from a discrete stack (i.e., $\Spec_X (\pi_0 \O_X) \to X$), and $i_*\colon \DCoh( (\pi_0 X) ) \to \DCoh(X)$ the pushforward.  Then, the image  of $i_*$ triangulated-generates $\DCoh(X)$.  In fact, objects of the form $i_* \F$, for $\F \in \DCoh(\pi_0 X)^\heart = \Coh(\pi_0 X)$, triangulated-generate $\DCoh(X)$.
    \item The right-adjoint $i^! \colon \Ind\DCoh X \to \Ind\DCoh(\pi_0 X)$ to $i_*\colon \Ind\DCoh (\pi_0 X) \to \Ind \DCoh X$ is conservative.
    \item Suppose that $\N \subset \pi_0 \O_X$ is a nilpotent ideal sheaf (e.g., the nilradical on a Noetherian derived stack).  Let $i_2\colon X' = \Spec_X \pi_0 \O_X/\N \to X$ be the corresponding map from the discrete derived stack $X' = \Spec_X \pi_0 \O_X/\N$.  Then, the image  of $(i_2)_*$ triangulated-generates $\DCoh(X)$.  In fact, objects of the form $(i_2)_* \F$, for $\F \in \DCoh( X' )^\heart = \Coh(X')$, triangulated-generate $\DCoh(X)$.
   \end{enumerate}
\end{lemma}
\begin{proof} 
  \begin{enumerate}
    \item Suppose $\F \in \DCoh(X)$, and consider the Postnikov stage
   \[ \tau_{\geq (k+1)} \F \longrightarrow \tau_{\geq k} \F \longrightarrow (\pi_k \F)[k] \]  Note that $\pi_k \F$ is a coherent $\pi_0 \O_X$-module since $\F \in \DCoh(X)$, and thus is in the essential image of $i_*$.  Since $X$ is quasi-compact and $\F \in \DCoh(X)$, only finitely many $k$ are non-zero, completing the proof.
 \item Suppose $\F = ``\dlim_\alpha" \F_\alpha \in \Ind \DCoh X$ is such that $i^! \F = 0$.  It suffices to show that $0 = \Map_{\Ind\DCoh X}(\K, \F) = \dlim_\alpha \Map_{\QC(X)}(\K, \F_\alpha)$ for all $\K \in \DCoh(X)$.  By (i), it suffices to note that $0 = \Map_{\Ind\DCoh X}(i_* \K', \F) = \Map_{\Ind\DCoh X}(\K', i^! \F)$ for all $K' \in \DCoh(\pi_0 X)$.
\item 
By the above it suffices to show that the triangulated closure of the image contains $i_* \F$ for $\F \in \Coh(\pi_0 X)$. The filtration of $\F$ by powers of $\N$ \[ \F \supset \N \F \supset \N^2 \F \supset \cdots \] is finite by hypothesis, and each associated graded piece is in the essential image of $(i_2)_*$.\qedhere
  \end{enumerate}
\end{proof}

This yields the following comforting description of $\QCsh(X)$:
\begin{corollary} Suppose $X$ is a Noetherian derived scheme. Let $\C$ denote a dg-category whose objects are ordinary coherent sheaves on $\pi_0 X$ and whose morphisms are $\RHom^{\otimes_k}_X(i_* \F, i_* \G)$.  Then, $\QCsh(X)$ may be identified with the dg-category of dg-modules over $\C^\op$.

Alternatively, let $\C'$ be the dg-category whose objects are coherent sheaves on $(\pi_0 X)^\red$ and whose morphisms are as above.  Then, $\QCsh(X)$ may be identified with the dg-category of dg-modules over $(\C')^\op$.
\end{corollary}

\begin{na}\label{na:qcsh-explicit} In case $X$ is a (discrete) Noetherian scheme, there are more explicit dg-models for $\QCsh(X)$ and $\QCsh(X)^\dual = \Ind(\DCoh(X)^\op)$ in the literature:
  \begin{itemize}
    \item $K(\Inj X)$ the ``homotopy'' dg-category of (unbounded) complexes of injective quasi-coherent sheaves.  This description emphasizes that ``the difference'' between $\QCsh(X)$ and $\QC(X)$ is that the later is complete with respect to the $t$-structure, i.e., acyclic objects are equivalent to $0$.  It models $\QCsh(X)$ by results of Krause.
    \item $K_m(\Proj X)$  Murfet's ``mock homotopy category of projectives.''  In the affine case, one can literally take the dg-category of (unbounded) complexes of projective quasi-coherent modules, while in general one must take a certain localization of the dg-category of (unbounded) complexes of flat quasi-coherent modules.  It models $\QCsh(X)^\dual$ by results of Neeman and Murfet.
  \end{itemize}
\end{na}

\begin{notation} Suppose $S$ is a perfect stack, so that $\QC(S) = \Ind \Perf(S)$.  \mbox{}\begin{itemize}
  \item If $f\colon  X \to S$ is a relative derived stack, then $\QC(X)$ is a $\QC(S)$-module category (via the symmetric monoidal pullback functor).  This gives rise to an inner-Hom functor  $\HHom^{\otimes_S}_{\QC(X)}\colon \QC(X)^\op \times \QC(X) \to \QC(S)$ characterized by
    \[\Map_{\QC(S)}\left(T, \HHom^{\otimes_S}_{\QC(X)}(\F,\G)\right) = \Map_{\QC(X)}\left(f^* T \otimes_{\O_X} \F, \G \right) \]for all $T \in \Perf(S)$, and $\F,\G \in \QC(X)$.      If $X = S$, then we will omit the superscript $\otimes_S$.  If $S = \Spec k$, we will write $\RHom_{\QC(X)}$.
  \item If $f\colon  X \to S$ is an $S$-scheme, then $\QCsh(X)$ is a $\QC(S)$-module category.  This gives rise to $\HHom^{\otimes_S}_{\QCsh(X)}\colon \QCsh(X)^\op \times \QCsh(X) \to \QC(S)$ characterized by
    \[\Map_{\QC(S)}\left(T, \HHom^{\otimes_S}_{\QCsh(X)}(\F,\G)\right) = \Map_{\QCsh(X)}\left(f^* T \otimes_{\O_X} \F, \G \right) \] for all $T \in \Perf(S)$ and $\F,\G \in \QCsh(X)$.
    If $X = S$, then we will omit the superscript $\otimes_S$.  If $S = \Spec k$, we will write $\RHom_{\QCsh(X)}$.  If $\F, \G \in \DCoh(X)$, we may write $\HHom_{\QC(X)}(\F, \G)$ or $\RHom_{\QCsh(X)}(\F,\G)$: Since $\DCoh(X) \to \QC(X)$ is fully-faithful, there is no ambiguity.
\end{itemize}
Note that if $\F \in \Perf(X)$ (or $\F \in \DCoh(X)$) then $f^* T \otimes_{\O_X} \F$ is compact in $\QC(X)$ (or $\Ind\DCoh(X)$) for all $T \in \Perf(S)$, so that $\HHom^{\otimes_S}(\F, -)$ preserves colimits.
\end{notation}

\begin{na}  Associated to a bounded morphism $f\colon  X \to Y$ of derived stacks, one can attach a variety of functors.
  In what follow, $F_R\colon \QC(X) \to \QCsh(X)$ is the right-adjoint to the localization $F_M\colon \QCsh(X) \to \QC(X)$ (so that in particular its restriction to $\DCoh(X)$ is the usual inclusion $\DCoh(X) \to \Ind\DCoh(X)$).
\end{na}
\begin{constr}
  Suppose $\bF\colon \QC_{<\infty}(X) \to \QC_{<\infty}(Y)$ is a colimit-preserving functor (on $t$-bounded above quasi-coherent complexes) which is \demph{$t$-bounded above} in the sense that there exists a constant $N$ such that $\bF(\QC(X)_{\leq k}) \subset \bF(\QC(X)_{\leq k+N})$.  Then, define $\lsh{\bF}\colon \QCsh(X) \to \QCsh(Y)$ as the filtered-colimit extension of  the composite
      \[ \DCoh(X) \stackrel{\bF}\longrightarrow \QC(Y) \stackrel{F_R}\longrightarrow \QCsh(Y) \]
      Since $F_M \circ F_R = \id$, it follows that $F_M \circ \lsh{\bF} \circ F_R = \bF$.  The importance of the $t$-boundedness condition is that $F_R$ commutes with $t$-bounded above (but not arbitrary) colimits, so that the condition guarantees that $F_R \circ \bF = \lsh{\bF} \circ F_R$ on $t$-bounded above objects.
      
      In particular, the $t$-bounded-above condition guarantees that the construction is compatible with composition of functors: If $\bF\colon \QC(X) \to \QC(Y)$ and $\bF'\colon \QC(Y) \to \QC(Z)$ are two functors, one would very much like for the natural map $\lsh{\bF'} \circ \lsh{\bF} \to \lsh{(\bF' \circ \bF)}$ to be an equivalence.  Everything is colimit-preserving, so it suffices to check on the compact objects $\K \in \DCoh(X)$, which are bounded above and remain so after applying $\bF$, so that
      \[ \lsh{\bF'} \circ \lsh{\bF}(\K) =  \lsh{\bF'} \circ F_R \circ \bF (\K) = F_R \circ \bF' \circ \bF(\K) = \lsh{(\bF' \circ \bF)}(\K) \]
      \begin{itemize}
	\item The functor $f_* \colon \QC_{<\infty}(X) \to \QC_{<\infty}(Y)$ is colimit-preserving and $t$-bounded above. Therefore, it gives rise to a functor $f_*\colon \QCsh(X) \to \QCsh(Y)$ by the above procedure.  If  $f$ is a bounded relative proper algebraic space,\footnote{Or e.g., a sufficiently nice bounded relative proper DM stack in characteristic zero.} then $f_*$ preserves compact objects.      
	\item Provided $f$ is of finite Tor-dimension, the functor $f^* \colon \QC(X) \to \QC(Y)$ will be colimit preserving and $t$-bounded above (and below).  In this case, it gives rise to a functor $f^* \colon \QCsh(Y) \to \QCsh(X)$ as above.  Furthermore, there is an adjunction $(f^*, f_*)$.
	\item The functor $f^! = \DD_X \circ f^* \circ \DD_Y \colon \QC_{<\infty}(Y) \to \QC_{<\infty}(X)$ on $t$-bounded above complexes is colimit preserving and $t$-bounded above. (In case $f$ is of finite Tor-dimension, $f^!(-) \isom \omega_f \otimes f^*(-)$ is well-behaved with no boundedness though still preserves boundedness.)  Consequently, it gives rise to a functor $f^! \colon \QCsh(Y) \to \QCsh(X)$. 
	\item To understand $f^!$, it will suffice for our purposes to recall explicit formulae for two special cases: If $f$ is \demph{finite} (i.e., affine, with $f_* \O_X$ pseudo-coherent) then $f^!(-) = \RHom_X(f_* \O_X, -)$ equipped with the evaluation-at-one trace map $\tr_f\colon f_* f^! \to \id$.  If $f$ is \demph{quasi-smooth} (i.e., finite-presentation and $\LL_f$ of Tor-amplitude in $[0,1]$), then $f^!(-) = \det \LL_f \otimes f^*(-)$ equipped with the Berezinian integration trace map $\tr_f\colon f_* f^! \to \id$.
\end{itemize}
\end{constr}
\begin{na}
The natural functors on $\QC$, $f^*$ and $f_*$, are simply adjoint and so determine one another. In contrast, properly spelling out the relations between the two natural functors, $f^!$ and $f_*$, on $\QCsh$ requires some $(\infty,2)$-categorical structures which we won't get into here (e.g., one needs to remember the transformation $\tr_f \colon f_* f^! \to \id$ when it exists, etc.).  Instead, we'll just mention a few facts (that hold in say, the \eqref{cond:starf} case)

The formation of $f_*$ commutes with flat base-change on the target.  The formation of $f^!$ commutes with flat base-change on the target and \'etale base-change on the source.  If $f$ is finite\footnote{or with more difficulty: a bounded relative proper algebraic space, or a sufficiently nice bounded relative proper DM stack in characteristic zero} the natural transformation $\tr_f\colon f_* f^! \to \id$ is the co-unit of an adjunction $(f_*, f^!)$.

Given a commutative square
  \[ \xymatrix{ Y' \ar[d]_{f'} \ar[r]^{g'} & Y \ar[d]^f  \\ X' \ar[r]_g & X } \] there is an equivalence $(g')_* (f')^! \isom f^! g_*$, e.g., in case $f$ proper as the composite
  \[ \xymatrix@1@C=6pc{ (g')_* (f')^! \ar[r]^-{\cotr_f((g')_* (f')^!)} & f^! f_* (g')_* (f')^! \ar[r]^-{\sim} & f^! g_* (f')_* (f')^! \ar[r]^-{f^! g_*(\tr_{f'})} & f^! g_* } \] and in case $f$ smooth a map the other way deduced from the projection formula, base-change, and the map $(g')^* \det \LL_f \to \det \LL_{f'}$.
This natural transformation is an equivalence when the square is Cartesian: Using compatibilities with base-change, the claim is \'etale local on $X$ and $Y$, so we reduce to the case where $f \colon X \to Y$ admits a factorization as a finite morphism followed by a smooth morphism; it then remains to check (using the standard $\QC$ tools, e.g., base-change for star pullback, the projection formula, etc.) that the natural transformation is an equivalence in each of the two cases.
\end{na}

\part{\texorpdfstring{$k\ps{\bt}$}{k[t]}-linear structure via derived group actions}
\section{Generalities on \texorpdfstring{$\PreMF$}{PreMF} and \texorpdfstring{$\MF$}{MF}}\label{sec:mf-gen}
\subsection{Preliminaries}\label{ssec:Bact}
\begin{constr}\label{constr:B} For the duration of this section, set \[ \bB = 0 \times_{\AA^1} 0 = \Spec k[B]/B^2 \qquad \deg B = +1 \]  $\bB$ admits the structure of \demph{derived group scheme} (i.e., its functor of points admits a factorization $\DRng \to \Mon^{gp}(\sSet) \to \Kan$, where $\Mon^{gp}(\sSet)$ denotes group-like Segal-style monoids in $\sSet$) by ``composition of loops''.    Since $\AA^1 = \GG_a$ is a commutative group scheme, $\bB$ admits the structure of \demph{commutative derived group scheme} (i.e., its functor of points admits a factorization $\DRng \to \sAb \to \sSet$) by considering ``pointwise addition of loops.''  These two structures are in fact strictly compatible (i.e., determine a factorization of the functor of points as $\DRng \to \Mon^{gp}(\sAb) \to \sSet$ refining each of the other two in the obvious way).

  \begin{itemize} 
     \item The (``loop composition'') product $\mu\colon \bB \times \bB \to \bB$ and identity $\id\colon \pt \to \bB$ may be explicitly identified as
      \[ \mu\colon \bB \times \bB = (0 \times_{\AA^1} 0) \times (0 \times_{\AA^1} 0) \isom 0 \times_{\AA^1} 0 \times_{\AA^1} 0 \stackrel{p_{13}}\longrightarrow 0 \times_{\AA^1} 0 \]
      \[ \id\colon \pt \stackrel{\Delta}\longrightarrow \pt \times_{\AA^1} \pt = \bB \]
      The rest of the Segal-monoid structure admits a similar description via projections and diagonals.  A homotopy inverse is given by the explicit anti-isomorphism $i\colon \bB \isom \bB^\op$ which on underlying space can be identified with
      \[ i\colon \bB = \pt \times_{\AA^1} \pt \stackrel{\mathrm{switch}}\longrightarrow \pt \times_{\AA^1} \pt = \bB \]
   \item The (``pointwise addition'') product $+\colon \bB \times \bB \to \bB$, identity $0\colon \pt \to \bB$, and inverse $-\colon \bB \isom \bB^\op$ may be explicitly identified as follows: The commutative diagram
      \[\xymatrix{
     0 \times 0 \ar[d] \ar[r]^= & 0 \ar[d]  \\
     \AA^1 \times \AA^1 \ar[r]^-+ &\AA^1 } \]
     gives rise to a map
     \[ +\colon \bB \times \bB \isom (\pt \times \pt) \times_{\AA^1 \times \AA^1} {\pt \times \pt} \longrightarrow \pt \times_\AA^1 \pt = \bB \] 
     Analogously, base changing the identity $0 \to \AA^1$ and the inverse map $-\colon \AA^1 \to \AA^1$ one obtains maps 
     \[ 0\colon \pt = 0 \times_0 0 \longrightarrow 0 \times_{\AA^1} 0 = \bB \]
     \[ -\colon \bB = 0 \times_{\AA^1} 0 \longrightarrow 0 \times_{\AA^1} 0 = \bB \]
     and an (anti-)isomorphism $-\colon \bB \isom \bB^\op$.
   \item For $R \in \DRng$, let $R_\bullet = \Map_{\DRng}(k[x], R) \in \sAb$.  In terms of functor of points on $\DRng$, we have
      \[ \bB(R_\bullet) = \Map_*(S^1, R_\bullet) \] which is equipped with a Segal-monoid structure (``loop composition'') via mapping in wedges of length $n$
      \[ [n] \mapsto \Map_*(\sqcup_n \Delta^1/\sqcup_n \Delta^0, R_\bullet) \]

      For any pointed simplicial set $X$, $\Map_*(X, R_\bullet)$ is naturally a simplicial abelian group via the composite
      \[ \Map_*(X, R_\bullet)^2 = \Map_*(X, R_\bullet^2) \to \Map_*(X, R_\bullet) \]
      providing the lift to $\Mon^{gp}(\sAb)$.

      We will heuristically write these (indicating e.g., maps on ($\pi_0$ of) functor of points) as
      \[ \mu\colon \bB(R_\bullet) \times \bB(R_\bullet) \ni \left( [h_1\colon 0 \to 0], [h_2\colon 0 \to 0]\right) \mapsto \left( [h_1 \cdot h_2\colon 0 \stackrel{h_1}\to 0 \stackrel{h_2}\to 0]\right) \in \bB(R_\bullet) \] and
      \[ +\colon \bB(R_\bullet) \times \bB(R_\bullet) \ni \left( [h_1\colon 0 \to 0], [h_2\colon 0 \to 0]\right) \mapsto \left( [h_1 + h_2\colon  0 \to 0]\right) \in \bB(R_\bullet) \] 
  \end{itemize}
\end{constr}

\begin{constr}\label{constr:QCshB} 
Let $(\QCsh(\bB), \circ)^\otimes$ denote $\QCsh(\bB)$ equipped with its symmetric monoidal convolution product:  $\G \circ \F = +_*(\G \boxtimes \F)$.

More precisely: \autoref{constr:B} provides a lift of $\bB$ to $(\bB,\mu,+) \in \Mon(\CMon(\text{Der. Sch.}))$.  Composing with the lax symmetric monoidal (via exterior product) functor\footnote{In general, the correct way to say this would require considering a suitable coCartesian fibration $(\sq{\QCsh},f_*) \to \text{Der. Sch}$, and then pulling back along the $\Delta^\op \times \Gamma$-shaped diagram encoding $(\bB,\mu,+)$.  In the present case, however, all the maps are finite so that it is not hard to give a strictly functorial diagram of categories.}
\[ X \mapsto \QCsh(X), \qquad f \mapsto f_* \] one obtains $(\QCsh(\bB), \circ_{\mu}, \circ_{+}) \in \Alg(\CAlg(\dgcatbig_k))$.  Finally, $(\QCsh(\bB), \circ)^\otimes$ is the image of this under the forgetful functor to $\CAlg(\dgcat_k)$.
\end{constr}

\begin{remark}\label{remark:explicit-homotopy} One can give explicit dg-algebra models for the two products on $\bB$, for the two actions, and an equivalence between the two (as well as an equivalence with a convenient smaller non-Segal model for loop composition).

Consider the following diagram of cosimplicial commutative dg-algebras:
  \[ \left\{k\overbrace{[B_1,\ldots,B_{\bullet}]}^{\deg B_i = +1}\right\} \longrightarrow \left\{ k\left[x_1, \ldots, x_{\bullet}\right]\overbrace{\left[\begin{gathered} \epsilon^1_1, \ldots, \epsilon^\bullet_1  \\ \epsilon^1_2, \ldots, \epsilon^\bullet_2\end{gathered}\right]}^{\deg \epsilon^i_i = +1}/\left( d\epsilon^j_i = x_i \right) \right\} \longrightarrow \left\{ k[x]\overbrace{\left[\gamma_1,\ldots,\gamma_{\bullet+1}\right]}^{\deg \gamma_i = +1}/(d\gamma_i=x) \right\} \] where the cosimplicial structure maps, and morphisms, are
  \begin{itemize} 
    \item The middle term models the co-commutative ``pointwise-addition'' co-multiplication (and co-identity) on $k[x][\epsilon_1,\epsilon_2]$
      \[ \Delta_+(x) = x \otimes 1 + 1 \otimes x \qquad \Delta_+(\epsilon_1) = \epsilon_1 \otimes 1 + 1 \otimes \epsilon_1 \qquad \coid_+(x) = 0 \]
      under the evident identification of the $n$-th term with the $n$-fold tensor product of $k[x][\epsilon_1,\epsilon_2]$.
    \item The right-hand term models the (Segal-style)  ``loop composition'' co-multiplication (and co-identity) on $k[x][\gamma_1,\gamma_2]$ (note that the $n$-th term is quasi-isomorphic to the $n$-fold tensor product, though isn't strictly isomorphic to it).
    \item The left-hand term gives a compact model for both. It comes from the co-commutative co-multiplication on $k[B]/B^2$ given by $\Delta(B) = B \otimes 1 + 1 \otimes B$, $\coid(B) = 0$.
    \item The right-hand map sends all $x_i$ to $x$, and $\epsilon^i_j \mapsto \gamma_{i+j-1}$.  It is a weak equivalence.
    \item The left-hand maps sends $B_i \mapsto \epsilon^i_1 - \epsilon^i_2$.  It is a weak equivalence.
  \end{itemize}
\end{remark}

The next (standard?) Proposition is the starting point for this Section.  Morally, it is the following Koszul duality computation: One identifies $\O_\bB \isom C_*(S^1; k)$ as $E_\infty$-coalgebra in dg-algebras, so that a cobar construction yields $k\ps{\bt} \isom C^*(BS^1; k)$ as $E_\infty$-algebra.
\begin{prop}\label{prop:mf-prelim}  
  There is a symmetric monoidal equivalence $(\QCsh(\bB), \circ)^{\otimes} \isom (k\ps{\bt}\mod, \otimes_{k\ps{\bt}})^{\otimes}$, given on compact objects by (a suitable enrichment of) \[ \DCoh(\bB) \ni V \mapsto V^{S^1}=\RHom_{\bB}(k, V) \in \Perf k\ps{\bt} \]
\end{prop}
\begin{proof}
It suffices to prove the equivalence on compact objects. We will carry out the computation in the explicit (characteristic zero) dg-model for $\O_\bB$ of \autoref{remark:explicit-homotopy}.  Let $\Cpx(\bB)$ denote the (ordinary) category of dg-$\O_\bB$-modules, and $\Cpx(k\ps{\bt})$ the (ordinary) category of dg-$k\ps{\bt}$-modules.

Identify
\[ \O_\bB = k[x]\underbrace{[\epsilon_1, \epsilon_2]}_{\deg \epsilon_i = +1}/\left(\begin{gathered} \epsilon_i^2 = 0\\d\epsilon_i = x\end{gathered} \right) \]
  as commutative dg-$k$-algebra (recall, related to the smaller model by $B = \epsilon_1 - \epsilon_2$).  The $+$-comultiplication, -coidentity, and -coinverse of \autoref{remark:explicit-homotopy} make $\O_\bB$ into a cocommutative, commutative, dg-Hopf algebra.  Then, $\Delta_+$, $\coid_+$, and $-\otimes_k -$ equip the (ordinary) category $\Cpx(\bB)$ with a symmetric monoidal structure by setting
\[ M \stackrel{\circ}\otimes M' \eqdef (\Delta_+)_*(M \boxtimes_k M') \] with unit $k = \coid_+(k)$ and the evident associativity, unitality, and commutativity constraints coming from those for $\otimes_k$ on complexes of $k$-modules.  

Recall the Koszul-Tate semi-free resolution
\[ k \sim \underbrace{\O_\bB\left[u^m/m!\right]}_{\deg u = +2}/\left\{du = \epsilon_1 - \epsilon_2\right\} \]
on which $k\ps{\bt}$ acts by $\bt = d/du$.  This gives rise to the usual explicit model for $(-)^{S^1}$ (recall $B = \epsilon_1 - \epsilon_2$) as a functor on (ordinary) categories of complexes
\[ \left( (V, d_{int}) \right)^{S^1} = \RHom_{\O_\bB\mod}\left(\O_\bB\left[u^m/m!\right], (V,d_{int})\right) = \left(V\ps{\bt}, d_{int} + \bt B\right) \] 
It will be more convenient for our purposes to instead work with the functor
\[\bF \colon \Cpx(\bB) \to \Cpx(k\ps{\bt}) \qquad  \bF(V) = V \otimes_{\O_\bB} (\O_\bB\ps{\bt}, \bt B) = \left(V[\bt], d_{int} + \bt B\right) \]  Note that for $V$ bounded-above (resp., homologically) the natural map $\bF(V) = V \otimes_{\O_\bB} (\O_\bB\ps{\bt}, \bt B) \to V^{S^1}$ is an isomorphism (resp., equivalence).\footnote{On homologically bounded-above complexes, it follows that $\bF$ preserves quasi-isomorphisms.  On arbitrary complexes it need not: Say that a map $\phi: V \to V'$ is an $\bF$-equivalence  if $\bF(\phi )$ is a quasi-isomorphism, and let $\bF^{\sim}$ denote the collection of $\bF$-equivalences.  One can show that: every $\bF$-equivalence is a quasi-isomorphism, but not vice versa; the localization $\Cpx(\O_\bB)[(\bF^{\sim})^{-1}]$ is naturally identified with $\Ind \DCoh(\bB)$, so that $\bF$ induces a functor $\Ind \DCoh(\bB) \to k\ps{\bt}\mod$ which one can show is an equivalence; since every $\bF$-equivalence is a quasi-isomorphism, we obtain $\Cpx(\O_\bB)[(\bF^{\sim})^{-1}] \to \Cpx(\O_\bB)[\mathrm{qiso}^{-1}] = \QC(\bB)$ which coincides via the above with the usual localization functor.}
The functor $\bF$ is monoidal via the natural isomorphism
\begin{align*} \bF(V,d_{int}) \otimes_{k\ps{\bt}} \bF(V', d'_{int}) &= \left(V[\bt], d_{int} + \bt B\right) \otimes_{k\ps{\bt}} \left(V'[\bt], d'_{int} + \bt B\right)  \\ & \longrightarrow \left( (V \otimes_k V')[\bt], d_{int} \otimes 1 + 1 \otimes d'_{int} + \beta(B \otimes 1 + 1 \otimes B)\right) = \bF\left[(V,d_{int}) \stackrel{\circ}{\otimes} (V',d'_{int})\right] \end{align*} 
and the equality $\bF(k) = k\ps{\bt}$ of tensor units, evident compatibility with associativity, etc.  The symmetry isomorphisms on both sides are given by the usual graded-commutativity rules, and $\bt$ is even, so that $\bF$ is symmetric monoidal.

Note that the symmetric monoidal unit $\O_{\id} = k \in \DCoh(\bB)$ generates $\DCoh(\bB)$ under cones and shifts (c.f., \autoref{lem:coh-red}):  For any $V \in \DCoh(\bB)$ simply consider the finite Postnikov stages $\tau_\geq (m+1) V \to \tau_\geq m V \to (\pi_m V)[m]$, and observe that $\pi_m V[m]$ is a $\pi_0 \bB = k$-module.   It follows that $\bF \isom (-)^{S^1}$ takes $\DCoh(\bB)$ to $\Perf k\ps{\bt}$.We claim that $\bF$ can be used to construct a symmetric-monoidal functor of $\infty$-categories $(\DCoh(\bB), \circ)^\otimes \to (\Perf k\ps{\bt}, \otimes_{k\ps{\bt}})^{\otimes}$ which is equivalent to $(-)^{S^1}$ on underlying categories.  Assuming the claim, we complete the proof. A symmetric-monoidal functor is an equivalence iff it is so on underlying $\infty$-categories so that it suffices to show that $\RHom_{\O_\bB\mod}(k, -) \colon \DCoh(\bB) \to \Perf k\ps{\bt}$ is an equivalence.  The map of complexes $k\ps{\bt} \to \RHom_{\O_\bB\mod}(k,k)$ described above is evidently a quasi-isomorphism.   Since $\DCoh(\bB)$ is stable, idempotent complete, and generated (in the stable, idempotent complete sense) by $k$ it follows by Morita theory that the functor $\RHom_{\O_\bB\mod}(k,-)$ is an equivalence.

The rest of the proof will be devoted to giving the details of obtaining from $\bF$ a symmetric monoidal functor of $\infty$-categories:
\begin{itemize}
  \item Equip $\Cpx(\O_\bB)$ with its \emph{injective} model structure, i.e., the weak-equivalences and cofibrations are maps which are so on underlying complexes.  Together with $\stackrel{\circ}{\otimes}$ above, this makes it into a simplicial symmetric-monoidal  model category in the sense of \cite[Def.4.3.11]{DAG-III}; e.g., the compatibilities of tensor and weak-equivalences/cofibrations follow from the analogous statements for chain complexes over $k$.  It follows that the symmetric-monoidal $\infty$-category $(\DCoh(\bB), \circ)^\otimes \to N(\Gamma)^{\otimes}$ admits a description as the homotopy coherent nerve of a fibrant simplicial category $(\Cpx^{\circ,\DCoh}(\O_\bB))^{\otimes}$ over $\Gamma$ formed as follows: Its objects are tuples $(\langle n \rangle, C_1, \ldots, C_n)$ with $\langle n \rangle \in \Gamma$ and with each $C_i$ a bounded-above injective-fibrant dg-$\O_\bB$-module with bounded coherent cohomology, and its simplicial mapping spaces are
\[ \Map\left( \left(\langle n\rangle, C_1\ldots, C_n\right), \left(\langle m\rangle, C'_1, \ldots, C'_m\right) \right) = \bigsqcup_{\alpha: \langle n \rangle \to \langle m \rangle} \prod_{1 \leq j \leq n} \Map_{\Cpx(\O_\bB)}\left( \stackrel{\circ}\bigotimes_{i \in \alpha^{-1}(j)} \C_i, \C'_j \right) \] with the evident composition law.
\item Equip $\Cpx(k\ps{\bt})$ with its \emph{projective} model structure, i.e., weak-equivalence and fibrations are maps which are so on underlying complexes.  Equipped with $\otimes_{k\ps{\bt}}$, it is also a simplicial symmetric-monoidal model category.  So $(\Perf k\ps{\bt}, \otimes_{k\ps{\bt}})^{\otimes} \to N(\Gamma)^{\otimes}$ admits a description as the homotopy coherent nerve of a fibrant simplicial category $(\Cpx^{\circ,\Perf}(k\ps{\bt}))^{\otimes}$ over $\Gamma$ formed as follows: Its objects are tuples $(\langle n \rangle, C_1, \ldots, C_n)$ with $\langle n \rangle \in \Gamma$ and with each $C_i$ projective-cofibrant perfect dg-$k\ps{\bt}$-modules, and simplicial mapping space are given by the same formula as above (with $\bigotimes$ now being taken over $k\ps{\bt}$).
\item The functor $\bF$ preserves fibrant objects, i.e., $\bF(V)$ is fibrant for every $V$.  If $B$ acts trivially on $V$ this is clear, since then $\bF(V) \isom V \otimes_k k\ps{\bt}$ and $- \otimes_k k\ps{\bt}$ is left-adjoint to the forgetful functor which creates fibrations in the projective model structure; the general case reduces to this by writing $V$ as the cone on $(\im B)[-2] \stackrel{\bt}\to \ker B$ and so $\bF(V)$ as a cone on projective-cofibrant modules.  Furthermore, $\bF$ obviously preserves cofibrant objects.  We have seen that $\bF$ maps complexes with bounded coherent cohomology to perfect complexes.  We conclude that there is a well-defined simplicial functor $\bF^{\otimes}\colon (\Cpx^{\circ,\DCoh}(\O_\bB))^\otimes \to (\Cpx^{\circ,\Perf}(k\ps{\bt}))^{\otimes}$ over $N(\Gamma)$ defined by applying $\bF$ to the objects and using the symmetric monoidal structure on the mapping spaces.
\item Taking homotopy coherent nerves, we obtain a functor $N(\bF^{\otimes}):(\DCoh(\bB), \circ)^{\otimes} \to (\Perf k\ps{\bt}, \otimes_{k\ps{\bt}})^{\otimes}$ of coCartesian fibrations over $N(\Gamma)$.  To prove that it is a symmetric-monoidal functor, it remains to show that it preserves coCartesian morphisms.  The criterion of \cite[Prop.~2.4.1.10]{T} allows us to reduce to showing that if $C_1,\ldots,C_n,D \in \Cpx^{\circ,\DCoh}(\O_\bB)$ and $\bigotimes_i C_i \to D$ is a morphism which induces an equivalence on $\Map(-,E)$ for all $E \in \Cpx^{\circ,\DCoh}(\O_\bB)$, then the same is true for $\bigotimes_i \bF(C_i) \to \bF(D)$.  Since $\bigotimes_i C_i$ is still cofibrant, the first condition is equivalent to the map being a weak-equivalence; since we have seen that $\bF$ preserves cofibrant objects, and restricts to fibrant-cofibrant objects, it suffices to observe that it preserves weak-equivalences.\qedhere
\end{itemize}
\end{proof}

This yields the promised geometric description of the $k\ps{\bt}$-linear structure on $\QCsh$ of a hypersurface:
\begin{constr}\label{constr:Bact} Suppose $f\colon  X \to \AA^1$ and set $X_0 = X \times_{\AA^1} 0$.  Then:
  \begin{itemize}
    \item There is a right action of $\bB$ (with its ``loop composition'') product on $X_0$.  It is easy to give a rigorous Segal-style description via various projections.  Heuristically, it is given as follows on ($\pi_0$ of) functor of points:
      \begin{align*} X_0(R) \times \bB(R) \ni &\left( x \in X(R), [h_f\colon  f(x) \to 0] \in R \right) \times \left( [h\colon 0 \to 0] \in R \right) \\ &\mapsto \left( x \in X(R), [h_f \cdot h\colon f(x) \stackrel{h_f}{\to} 0 \stackrel{h}{\to} 0]\right) \in X_0(R) \end{align*}
    \item There is an action of $\bB$ (with its ``pointwise addition'') product on $X_0$.  It is easy to give a rigorous description of it by base-changing the addition map on $\AA^1$.  Heuristically, it is given as follows on ($\pi_0$ of) functor of points:
      \begin{align*} X_0(R) \times \bB(R) &\ni \left( x \in X(R), [h_f\colon  f(x) \to 0] \in R \right) \times \left( [h\colon 0 \to 0] \in R \right) \\ &\mapsto \left( x \in X(R), [h_f + h\colon f(x) \to 0]\right) \in X_0(R) \end{align*}
      \item As in \autoref{constr:QCshB}, applying $\QCsh$ to the above group actions equips $\QCsh(X_0)$ with the structure of right $\QCsh(\bB)$-module (under convolution along loop composition) and compatibly of $\QCsh(\bB)$-module (under convolution along addition).  These are ``the same up to homotopy'' in the precise sense of the Eckmann-Hilton argument \autoref{lem:e2} (c.f., also \autoref{remark:explicit-homotopy-2}).  Note that the structure maps of these actions are \emph{finite} (i.e., affine and finite on $\pi_0$): So in fact $\DCoh(X_0)$ is a $\DCoh(\bB) = \Perf k\ps{\bt}$-module, and this recovers the above by passing to $\Ind$-objects.
    \end{itemize}
\end{constr}

\begin{lemma}\label{lem:e2}(``Eckmann-Hilton'') Suppose $A \in \Alg\left(\CAlg(\C^\otimes)\right)$.  Let $\ol{A} \in \CAlg\left(\C^\otimes\right)$ and $\sq{A} \in \Alg(\C)$ be its images under the forgetful functors.  Set $\D^\otimes = \ol{A}\mod(\C^\otimes)$, and note that there is a (lax symmetric monoidal) forgetful functor $\D^\otimes \to \C^\otimes$.  Then:
  \begin{enumerate}
  \item The other commuting product on $A$ gives rise to a lift of $\sq{A}$ to an object $A' \in \Alg\left(\D^\otimes\right)$.
  \item There are equivalences of $\infty$-categories
    \[ \xymatrix{ \D &  A'\mod\left(\D\right) \ar[l]_-{\sim} \ar[r]^-{\sim} & \sq{A}\mod\left(\C\right) } \]
\end{enumerate}
\end{lemma}

\begin{remark}\label{remark:explicit-homotopy-2}
  In the spirit of \autoref{remark:explicit-homotopy}, one can give a similar construction of cosimplicial commutative dg-$\O_X$-algebras encoding the actions of $\bB$ on $X_0$, and the map $X_0/\bB \to X$:\footnote{There is a choice of sign appearing here that will probably change at random later.}
  \[ \left\{ \O_X[x_X,x_1,\ldots,x_\bullet]\left[\begin{gathered}\epsilon_1^X, \epsilon_1^1, \ldots, \epsilon_1^\bullet\\ \epsilon_2^X, \epsilon_2^1, \ldots, \epsilon_2^\bullet\end{gathered}\right]/\left(\begin{gathered} d\epsilon_1^j=d\epsilon_2^j=x\\d\epsilon_1^X=x-f, d\epsilon_2^X=x\end{gathered} \right)\right\} \longleftarrow \O_X \]
    \[ \left\{ \O_X[x][\gamma_X, \gamma_1,\ldots,\gamma_{\bullet+1}]/\left(\begin{gathered} d\gamma_i=x\\d\gamma_X=x-f\end{gathered} \right)\right\} \longleftarrow \O_X[x][\gamma_X]/(d\gamma_X=x-f) \sim \O_X\]
      \[ \left\{ \O_X[B_X, B_1,\ldots,B_\bullet]/\left(\begin{gathered} d(B_i)=0\\d(B_X)=-f\end{gathered} \right) \right\}\longleftarrow \O_X\]
In particular, one can avoid explicitly invoking the Eckmann-Hilton argument.
\end{remark}

\begin{remark}
There is an obvious variant of \autoref{constr:Bact} for $\QC(\bB)$ acting on $\QC(X_0)$.  The one notable difference is that this does not pass to compact objects: $\Perf(\bB)$ is not even monoidal, since the putative tensor unit $\O_{\id} = k$ is not perfect. This is however all that goes wrong: $\Perf(\bB)$ is an $\otimes$-ideal of $\DCoh(\bB)$, the inclusion $F_L^\otimes\colon \QC(\bB) \to \QCsh(\bB)$ is symmetric-monoidal, and the inclusion $F_L\colon\QC(X_0) \to \QCsh(X_0)$ is linear over $F_L^\otimes$ (c.f. \autoref{lem:mf-prelim}).  In particular, one may recover the $\QC(\bB)$-action on $\QC(X_0)$ from the $\QCsh(\bB)$-action on $\QCsh(X_0)$:
\[ V \otimes \F = F_R F_L(V \otimes_{\QC(\bB)} \F) = F_R\left( F_L^\otimes(V) \otimes_{\QCsh(\bB)} F_L(\F) \right) \]

The relationship between $\QC(\bB)$ and $\QCsh(\bB)$ is spelled out by the following Lemma:
\end{remark}

\begin{lemma}\label{lem:mf-prelim} Under the identification of \autoref{prop:mf-prelim}, the recollement diagram of $\QC(\bB)$, $\QCsh(\bB)$, $\DSing^\kinfty(\bB)$ associated to the Drinfeld-Verdier sequence\footnote{In dg-category language the functors in the recollement diagram are restriction, induction, and co-induction of right dg-modules over the terms in the Drinfeld-Verdier sequence; i.e., $F_M$ is the restriction along $F$, $F_L$ is induction along $F$, and $F_R$ is coinduction along $F$.}
  \[ \Perf(\bB) = \Perf \O_\bB \stackrel{F}\longrightarrow \DCoh(\bB) = \Perf k\ps{\bt} \stackrel{G}\longrightarrow \DSing(\bB) = \Perf k\pl{\bt} \] may be identified with
  \[ \xymatrix@C=8pc{ k\pl{\bt}\mod\ar[r]|{G_M} & \ar@<-2ex>[l]|{G_L} \ar@<2ex>[l]|{G_R} k\ps{\bt}\mod  \ar[r]|{F_M} & \ar@<-2ex>[l]|{F_L} \ar@<2ex>[l]|{F_R} \O_\bB\mod } \] 
  where  \begin{itemize}
    \item $F_L = (-)_{S^1}[1] = k[1] \otimes_{\O_B} -$;
    \item $F_M = \RHom_{k\ps{\bt}}(k, -) = k \otimes_{k\ps{bt}} -$;
    \item $F_R = (-)^{S^1} = \RHom_{\O_B}(k, -)$;
    \item $G_L = k\pl{\bt} \otimes_{k\ps{\bt}} -$;
    \item $G_M = \RHom_{k\pl{\bt}}(k\pl{\bt}, -) = k\pl{\bt} \otimes_{k\pl{\bt}} -$;
    \item $G_R = \RHom_{k\ps{\bt}}(k\pl{\bt}, -)$.
  \end{itemize}
In particular, these satisfy all the usual relations (e.g., the unit $\id \to F_M \circ F_L$ and the counit $F_M \circ F_R \to \id$ are equivalences, etc.) so that $F_L$ induces an equivalence
\[ F_L \colon \QC(\bB) \stackrel{\sim}\longrightarrow \left\{ \begin{gathered}\text{locally $\bt$-torsion}\\\text{$k\ps{\bt}$-modules} \end{gathered}\right\} \eqdef \Ind \left( \begin{gathered}\text{$\bt$-torsion perfect}\\\text{$k\ps{\bt}$-modules}\end{gathered}\right) \]
\end{lemma}
\begin{proof} We first focus only on the $F$ side: The various functors have the right adjunctions simply by Morita theory, so the identification follows from noting that $F_L$ does the right thing on compact objects; the coincidence of two descriptions for $F_M$ is a straightforward computation.  The $G$ side will follow similarly once we show identify $G\colon \DCoh(\bB) \to \DSing(\bB)$ with $k\pl{\bt} \otimes_{k\ps{\bt}} - \colon \Perf k\ps{\bt} \to \Perf k\pl{\bt}$.  The description of $F_L(\QC(\bB))$ as locally $\bt$-torsion $\bt$-modules then follows from the description of $G_L$.
    
    The cofiber sequence
    \[ k\ps{\bt}[-2] \stackrel{t}\to k\ps{\bt} \to k \] identifies $\DSing^\kinfty(\bB)$ (=the fiber of $F_M$) with the full-subcategory of $k\ps{\bt}\mod$ consisting of objects $\F$ on which $t\colon \F[-2] \to \F$ is an equivalence.  This can be identified with the $\infty$-category of $k\pl{\bt}$-module objects in $k\ps{\bt}\mod$: For any such $\F$, the natural map
    \[ k\pl{\bt} \otimes_{k\ps{\bt}} \F = \dlim_n \frac{1}{\bt^n} \F \longrightarrow \F \] is an equivalence (since the $\lim$ is taken over a diagram of equivalences).  Finally, the adjunction $(G_L, G_R)$ is monadic and identifies $k\pl{\bt}\mod$ with $k\pl{\bt}$-module objects in $k\ps{\bt}$.  Passing to compact objects gives the desired identification.
\end{proof}

\autoref{constr:Bact} tells us that any $\K \in \QCsh(\bB)$ gives rise to an endo-functor of $\QCsh(X_0)$.  We'll spell this out for several distinguished objects of $\QCsh(\bB)$.

\begin{example}\label{ex:O_B} 
    Consider $\O_\bB \in \Perf(\bB)$, i.e., $k[B]/B^2$ as a perfect $k[B]/B^2$-module.  Since it is perfect, $F_L(\O_\bB) = F_R(\O_\bB)$ and both are identified under the equivalence of \autoref{prop:mf-prelim} with
      \[ \RHom_{\O_\bB}(k, \O_\bB) \isom k[1] \in k\ps{\bt}\mod \]
      Base-change in the Cartesian diagram (and the ``loop composition'' description of the action)
      \[ \xymatrix{
      X_0 \times \bB \ar[d]_{p_1} \ar[r]^-{\alpha} & X_0 \ar[d]^i \\ X_0 \ar[r]_i & X } \]
      implies $\O_\bB \otimes_{k\ps{\bt}} - $ may be identified with $i^* i_*\colon \QCsh(X_0) \to \QCsh(X_0)$. (This makes sense on each of $\QC$, $\QCsh$, and $\DCoh$ since $i$ is finite and of finite Tor-dimension.  The functor on $\QC$ restricts to one on $\DCoh$, and the functor on $\QCsh$ is then the $\Ind$-extension of this.)
\end{example}

\begin{example}\label{ex:resolution} 
The object $\O_\bB \in \QC(\bB)$ admits an $S^1$-action (``multiplication by $B$'', i.e., the $S^1$-action on $C_*(S^1;k)$), equipped with equivalences
\[ (\O_\bB)_{S^1} = \O_{\id} \quad \text{(i.e., $C_*(S^1;k)_{S^1} = C_*(S^1/S^1;k)$)} \]
\[ (\O_\bB)^{S^1} = \O_{\id}[1] \quad \text{(i.e., $C_*(S^1;k)^{S^1} = C_*(S^1/S^1;k)[1]$)} \]

The normalized chain complex of the simplicial-bar construction computing the homotopy quotient is the Koszul-Tate resolution used in the proof of \autoref{prop:mf-prelim}:
  \[ \O_{\id} \stackrel{\sim}\longleftarrow \O_B\left[\underbrace{u^k/k!}_{\deg u = +2}\right]/du = B  \qquad \in \QC(\bB) \]

By \autoref{constr:Bact} this gives rise to an explicit equivalence of endo-functors on $\QC(X_0)$
\[ (i^* i_*)_{S^1} \isom i^* i_*[u^k/k!] \sim \id_{\QC(X_0)} \] i.e., a natural equivalence
\[ (i^* i_*)_{S^1} \F = \left| (i^* i_*)^{\bullet+1} \F \right| \isom \left(i^* i_* \F\right)\left[u^k/k!\right] = \hocolim\left\{\cdots i^* i_* \F[2] \stackrel{B}\to i^* i_* \F[1] \stackrel{B}\to i^* i_* \F\right\}\stackrel{\sim}\longrightarrow \F \] for $\F \in \QC(X_0)$.
\end{example}

\begin{example}\label{ex:id}
Consider $\O_{\id} = \Delta_* \O_\pt \in \DCoh(\bB) \subset \QC(\bB)$.  It gives rise to a natural cofiber sequence in $\QCsh(\bB)$
\[ F_L(\O_{\id}) \longrightarrow F_R(\O_{\id}) \longrightarrow \cone\left\{F_L(\O_{\id}) \to F_R(\O_{\id})\right\}  \]  Under the identification of \autoref{prop:mf-prelim}, this sequence may be identified (c.f., \autoref{ex:resolution}) with the \demph{Tate sequence} $k_{S^1}[1] \to k^{S^1} \to k^{\Tate}$:
\[ \underbrace{k\pl{\bt}/k\ps{\bt}[-1]}_{\O_B[u^k/k!]} \longrightarrow \underbrace{k\ps{\bt}}_{\O_{\id}} \longrightarrow k\pl{\bt} \]
The three act as follows: $\O_{\id} = F_R(\O_{\id})$ acts by the identity on $\QCsh(M_0)$ (it had better, it is the tensor unit); $k\pl{\bt}$ acts by inverting $\bt$; $F_L(\O_{\id})$ acts by the colimit/simplicial diagram of \autoref{ex:resolution}, which coincides with the identify functor on $\QC(M_0)$ but not in general.
\end{example}

\begin{example}\label{ex:triangle} The Postnikov filtration of $k[B]/B^2$ yields a fiber sequence
  \[ \underbrace{F_R(k[1])}_{k\ps{\bt}[1]} \longrightarrow \underbrace{F_R(k[B]/B^2)}_{k[1]} \longrightarrow \underbrace{F_R(k)}_{k\ps{\bt}} \]
  which is just (a rotation of) the identification $k[1] = \cone\{t\colon k\ps{\bt}[-1] \to k\ps{\bt}[1] \}$.  By  \autoref{constr:Bact} this gives rise to a fiber sequence of functors $\id[1] \to i^* i_* \to \id$. More explicitly, for any $\F \in \QCsh(M_0)$ there is a triangle
 \[ \F[1] \longrightarrow  i^* i_* \F \longrightarrow \F \to \]
 where the second map is the counit of the adjunction (recall $i$ is of finite Tor-dimension).
\end{example}

\subsection{Circle actions}
\begin{prop}\label{prop:s1-hom} Suppose $\F, \G \in \PreMF(M,f)$.  Let $i\colon M_0 \to M$ be the inclusion.  Then, there is a natural circle action on $i^* i_* \F$ and a natural equivalence $(i^* i_* \F)_{S^1} = \F$.  This gives rise, by adjunction, to natural $S^1$ actions on $\Map(i_* \F, i_* \G)$ and $\RHom_M(i_* \F, i_* \G)$ such that
  \begin{itemize} 
    \item There is a natural equivalence $\Map(\F, \G) = \Map(i_* \F, i_* \G)^{S^1}$;
    \item There is a natural $k\ps{\bt}$-linear equivalence \[ \RHom^{\otimes k\ps{\bt}}_{M_0}(\F,\G) = \RHom_M(i_* \F, i_* \G)^{S^1} \]
    \item Let $\ol{\F},\ol{\G} \in \MF(M,f)$ denote the images of $\F, \G$.  Then, there is a natural $k\pl{\bt}$-linear equivalence
      \[ \RHom^{\otimes k\pl{\bt}}_{\MF(M,f)}(\ol{\F},\ol{\G}) = \RHom_M(i_* \F, i_* \G)^{\Tate} \]
  \end{itemize}
\end{prop}
\begin{proof}
This follows from \autoref{ex:resolution}, which gives an $S^1$ action on $i^* i_* \in \Fun^L_k(\QCsh(M_0), \QCsh(M_0))$ with $(i^* i_*)_{S^1} = \id$.  Since $i^* i_*$ preserves $\DCoh(M_0)$, the indicated equivalence restricts.

Let us spell things out more explicitly: The simplicial bar construction computing $(\O_B)_{S^1}$ identifies under Dold-Kan with the Koszul-Tate resolution of \autoref{ex:resolution}. Thus we have a functorial equivalence in $\QC(M_0)$
  \[ \left[(i^*i_* \F)[u^k/k!]/{du=B}\right] = \Tot^\oplus\left\{ \cdots \stackrel{B}\longrightarrow \Sigma^2 i^* i_* \F \stackrel{B}\longrightarrow \Sigma i^* i_* \F \stackrel{B}\longrightarrow i^* i_* \F\right\}\longrightarrow \F \]
which, since $\PreMF(M,f) = \DCoh(M_0)$ is a full subcategory of $\QC(M_0)$, gives rise to an equivalence
\begin{align*} \Map_{\PreMF(M,f)}(\F,\G) &= \Tot\left\{ \Map_{\DCoh(M_0)}(i^* i_* \F, \G) \right\} \\
  &= \Tot\left\{ \Map_{\DCoh(M)}(i_* \F, i_* \G) \right\} \\
  &= \Map_{\DCoh(M)}(i_* \F, i_* \G)^{S^1} \end{align*}
  where the $S^1$-action is given by $B$. (See below for a yet more explicit form.) 
\end{proof}

\begin{remark}\label{remark:simplic} For the simplicially-inclined reader, we mention the following alternate description: For $\F \in \QC(M_0)$, there is an augmented, $i_*$-split, simplicial object
  \[ \left\{ (i^* i_*)^\bullet \F \right\} = \left\{ \ssetl{i^* i_* \F}{(i^* i_*)^2 \F}\right\} \longrightarrow \F \]
  which realizes $\F$ as the geometric realization of the simplicial diagram (c.f., \autoref{lem:mfplus-monadic}).  Identifying $\O_\bB = C_*(S^1, k)$, the diagram encodes an $S^1$-action on $i^* i_* \F \in \QC(M_0)$, with quotient $(i^* i_* \F)_{S^1} = \F$.  For $\F \in \DCoh(M_0)$ there is also a Grothendieck-dual description (c.f., \autoref{ex:id})
  \[ \F \longrightarrow \left\{ (i^! i_*)_\bullet \F \right\} \isom \left\{ (i^* i_*)_\bullet \F[-1]\right\} \]
\end{remark}

\begin{lemma}\label{lem:s1-bdd}  Suppose $V$ is a complex with $S^1$-action.  Then, the natural map
  \[ V^{S^1} \otimes_{k\ps{\bt}} k \longrightarrow V \] is an equivalence.
 \end{lemma}
 \begin{proof} Identify the $\infty$-category of complexes with $S^1$-action with $k[B]/B^2\mod$.  Identify $k\ps{\bt} = \RHom_{\O_\bB\mod}(k,k)$.  Then, the map in question is identified with the natural evaluation map
   \[ \RHom_{\O_\bB\mod}\left(k, V\right) \otimes_{k\ps{\bt}} k \longrightarrow V \]
   i.e., the counit $F_M \circ F_R \to \id$.  This is an equivalence by \autoref{lem:mf-prelim}.
\end{proof}

\begin{corollary}\label{cor:premf-supt} Suppose $Z \subset M_0$ is a closed subset.  Then, $i_*\colon \DCoh_{Z}(M_0) \to \DCoh_{Z}(M) $ induces an equivalence of $\infty$-categories
  \[ i_*\colon \PreMF_Z(M,f) \otimes_{k\ps{\bt}} k \stackrel{\sim}\longrightarrow \DCoh_{Z} M \]
\end{corollary}
\begin{proof}
First, we will construct the desired lift of $i_*$ via a geometric description of the simplicial bar construction implementing the tensor product: The augmented simplicial diagram
\[ \left\{ M_0 \times \bB^{\bullet-1} \times \pt \right\} = \left\{ \cdots \ssetlar{M_0 \times \bB \times \pt}{\alpha}{p_1}{M_0 \times \pt}\right\} \stackrel{i}\longrightarrow M \] gives rise to an augmented simplicial diagram of $\infty$-categories, which via \autoref{prop:fmk-coh-surj} may be identified with
\[ \left\{ \DCoh(M_0) \otimes \DCoh(\bB)^{\otimes \bullet -1} \otimes \DCoh(\pt) \right\} \stackrel{i_*}\longrightarrow \DCoh(M) \] where the simplicial diagram is the simplicial bar construction for $\PreMF(M,f) \otimes_{\DCoh(\bB)} \DCoh(\pt) = \PreMF(M,f) \otimes_{k\ps{\bt}} k$.  Imposing support conditions everywhere, we obtain the functor of the statement.

Next we verify that this functor is fully faithful: It suffices to check that for any $\F, \G \in \PreMF(M,f)$ the natural map
  \[ \RHom^{\otimes k\ps{\bt}}_{\PreMF(M,f)}(\F,\G) \otimes_{k\ps{\bt}} k \longrightarrow \RHom_{M}(i_* \F, i_* \G) \] is an equivalence.   This follows immediately from the triangle of \autoref{ex:triangle} and adjunction.

Alternatively, by \autoref{prop:s1-hom} we may identify this with the map
  \[ \RHom^{\otimes k\ps{\bt}}_{\PreMF(M,f)}(\F, \G) \otimes_{k\ps{\bt}} k = \RHom_{M}(i_* \F, i_* \G)^{S^1} \otimes_{k\ps{\bt}} k \longrightarrow \RHom_{M}(i_* \F, i_* \G) \]
  which is an equivalence by \autoref{lem:s1-bdd}. 

We now prove that the functor is essentially surjective: Since it is fully faithful, and $\DCoh_Z M$ is a sheaf on $M$, the question is local and we may suppose $M$ is a quasi-compact coherent scheme.  Since both sides are stable and idempotent complete, it suffices to show that it has dense image.  We conclude by noting that $i_*\colon \DCoh_Z(M_0) \to \DCoh_{Z} M$ has dense image by the usual $t$-structure and filtration argument, since $Z \subset M_0$ (c.f., \autoref{lem:coh-red}).
\end{proof}

\subsection{Computational tools}
The following Lemma is rather categorical, and may be safely skipped on first reading: we will try to extract and emphasize its more concrete consequences.
\begin{lemma}\label{lem:mfplus-monadic} Suppose $X', X$ are coherent derived stacks, that $i\colon X' \to X$ is finite and of finite Tor-dimension.  Then, 
\begin{enumerate}
  \item The adjoint pair $\adjunct{i^*\colon \QC(X)}{\QC(X') \colon i_*}$ is monadic, and induces an equivalence
    \[ \QC(X') \isom (i_* \O_{X'})\mod\left(\QC(X)\right) \]
  \item The adjunction of (i) restricts to an adjoint pair $\adjunct{i^*\colon \DCoh(X)}{\DCoh(X') \colon i_*}$.  It is also monadic, and induces an equivalence
  \[ \DCoh(X') \isom i_*\O_{X'}\mod\left(\DCoh(X)\right) \] 
  \end{enumerate}
\end{lemma}
\begin{proof} \mbox{}
  \begin{enumerate}
    \item Since $i$ is affine, $i_*$ admits a Cech description and hence preserves all colimits.  Furthermore, $i_*$ is conservative: The question is local, and the result is true for $X$ affine because $i^*$ will pull back a generator to a generator.  Since $\QC(X')$ has all colimits, Lurie's Barr-Beck Theorem applies to give $\QC(X') \isom i_* i^*\mod\QC(X)$ and this monad visibly identifies with that given by the algebra $i_* i^* \O_{X} = i_* \O_{X'}$.
    \item
    Since $i$ is assumed finite, $i_*$ preserves $\DCoh$.  Since $i$ is assumed of finite Tor-dimension, $i^*$ preserves $\DCoh$.  So, the adjunction restricts to one on the full subcategories $\adjunct{i^*\colon \DCoh(X)}{\DCoh(X')\colon i_*}$.  It thus suffices to show that for $\G \in \QC(X')$, we have $\G \in \DCoh(X')$ iff $i_* \G \in \DCoh(X)$.  Since $i$ is affine, $i_*$ is $t$-exact; that is,
    \[ i^\heart_*\left(\pi_m \G\right) = \pi_m\left(i_* \G\right) \in \QC(X)^\heart = \pi_0(\O_X)\mod \]  
    It remains to show that
    \[ i^\heart_*\colon \QC(X')^\heart \to \QC(X)^\heart \] is conservative, and that $i^\heart_*(M)$ is finitely-presented over $\pi_0(\O_{X})$ iff $M$ is finitely-presented over $\pi_0(\O_{X'})$.  

    The question is local on $X$, so we will assume $X = \Spec A$ and $X' = \Spec B$. Note that $i_0 \colon \Spec \pi_0 B \to \Spec \pi_0 A$ is a finite map of discrete coherent rings, and that we must show that the pushforward $(i_0)^\heart_*$ of discrete modules is conservative, and that it preserves and detects the property of a module being finitely-presented.  That it is conservative is obvious.  Since $\pi_0 B$ is finite over $\pi_0 A$, it preserves the property of being finitely-presented.  To see that it detects coherence, suppose $M$ is a $\pi_0 B$-module such that the corresponding $\pi_0 A$-module, denoted $M_A$, is coherent.  Considering the surjection $M_A \otimes_{\pi_0 A} \pi_0 B \to M$, we see that $M$ is finitely-generated, so that if $\pi_0 B$ is Noetherian we are done.  To handle the general coherent case, we reduce to the case of $\pi_0 A$ (and so $\pi_0 B$) Noetherian: It suffices to note that we may find a Noetherian subring of $\pi_0 A$ over which the coherent $\pi_0 A$-algebra $\pi_0 B$, the coherent $\pi_0 A$-module $M_A$, and the $\pi_0 B$-action on $M_A$, are all defined.\qedhere
  \end{enumerate}
\end{proof}

\begin{corollary}\label{cor:mf-monadic} Suppose $(M,f)$ is an LG pair. Set
  \[ \A = \left(\underbrace{\O_M\left[B_M\right]}_{\deg B_M=+1}/\begin{gathered}B_M^2=0\\d B_M=f\end{gathered}\right) \] so that $\A$ is an algebra.  Then, $(i_*, i^*)$ induces equivalences
  \[ \QC(M_0) = \A\mod\left(\QC(M)\right)  \qquad \DCoh(M_0) = \A\mod\left(\Perf(M)\right) \] 
\end{corollary}

\begin{na} Under the identification of \autoref{cor:mf-monadic}, we can make explicit the $k\ps{\bt}$-linear structure of \autoref{prop:mf-prelim}(iii) as follows:
  Suppose $\F \in \A\mod(\QC(M))$.  The resolution of \autoref{ex:resolution} is
  \[ \F \sim \underbrace{\F[B][u^k/k!]}_{\deg B = +1, \deg u = +2}/\left\{\begin{gathered}B^2 = 0 \\ dB=0 \\ du = B \end{gathered}\right\} \] where $B_M$ acts on the RHS by $(B_M)_{\F} + B$ (in particular, $\F$ is not a submodule). The $k\ps{\bt}$ action on the right hand-side is given by $\bt = d/du$.  For $\F\in\A\mod(\Perf M)$ there is also the Grothendieck dual resolution (c.f., \autoref{remark:simplic})
    \[ \F \sim \left(\F[B]\ps{\bt}, d_\F + B \bt\right) \] on which $B_M$ acts as above and the $\bt$ action is evident.
\end{na}

\begin{na}\label{na:s1-act} Under the identification of \autoref{cor:mf-monadic}, we can make explicit the $S^1$-action of \autoref{prop:s1-hom} as follows: 
  Adjunction provides an equivalence
\[ \RHom_{\O_M}(\F,\G) = \RHom_{\O_M[B_M]}(\F[B], \G) \qquad \phi\left(f\right) \mapsto \sq{\phi}\left(f+ B f'\right) = \phi\left(f - B_M \cdot f'\right) + B_M \phi\left(f'\right) \]
On the right hand side, the $B$-operator of the circle action is the (graded) dual of multiplication by $B$.  A straightforward computation then shows that the induced operation on the left-hand side is (at least up to signs)
\[ B \phi = B_M \circ \phi + \phi \circ B_M \]
\end{na}

\subsection{Comparison of definitions}
\begin{prop}\label{prop:comparison} Suppose $(M,f)$ is an LG pair.  Then, the natural functor 
  \[ \PreMF(M,f) \longrightarrow \MF(M,f) \eqdef \PreMF(M,f) \otimes_{k\ps{\bt}} k\pl{\bt} \]
  factors through the quotient functor $\DCoh(M_0) \to \DSing(M_0) = \DCoh(M_0)/\Perf(M_0)$.  The induced functor
  \[ \DSing(M_0) \longrightarrow \MF(M,f) \] is an idempotent completion.
\end{prop}
\begin{proof} This can be found in the literature, but as this is important for our approach we sketch an argument.
  
{\noindent}{\bf Claim:} Suppose $\F \in \DCoh(M_0)$. Then, TFAE
\begin{enumerate}
  \item $\F$ is perfect;
  \item $\RHom^{\otimes k\ps{\bt}}_{\DCoh(M_0)}(\F, \F)$ is $\bt$-torsion (i.e., there is an $N>0$ such that $\bt^N$ is null-homotopic);
  \item $\RHom^{\otimes k\ps{\bt}}_{\DCoh(M_0)}(\F, \F)$ is locally $\bt$-torsion (i.e., it is a filtered colimit of perfect $t$-torsion $k\ps{\bt}$-modules);
  \item $1 \in \pi_* \RHom^{\otimes k\ps{\bt}}_{\DCoh(M_0)}(\F, \F)$ is $\bt$-torsion;
  \item $\F \mapsto 0 \in \MF(M, f)$.
\end{enumerate}

Assuming the claim, we complete the proof.  The existence of a factorization through a functor $\DSing(M_0) \to \MF(M,f)$ follows from (iv) by the universal property of a Drinfeld-Verdier quotient (as cofiber in small $k$-linear $\infty$-categories).  Since the image of $\DCoh(M_0)$ is dense (i.e., its thick closure is the whole) in both, it suffices to show that this functor is fully-faithful.  More precisely, it suffices to show that for $\F, \G \in \DCoh(M_0)$ the natural map
\[ \pi_0 \RHom^{\otimes k\pl{\bt}}_{\DSing(M_0)}(\ol{\F}, \ol{\G}) \to  \pi_0 \left[ \RHom^{\otimes k\ps{\bt}}_{\DCoh(M_0)}(\F, \G) \otimes_{k\ps{\bt}} k\pl{\bt}\right] \] is an equivalence.

At this point, we may conclude in several ways:
\begin{itemize}
  \item (Lazy) We may identify
    \begin{align*}
      \pi_0 \left[ \RHom^{\otimes k\ps{\bt}}_{\DCoh(M_0)}(\F, \G) \otimes_{k\ps{\bt}} k\pl{\bt}\right] &= \pi_0 \dlim_n\left[\frac{1}{\bt^n} \RHom^{\otimes k\ps{\bt}}_{\DCoh(M_0}(\F,\G)\right]\\
      &= \dlim_n \pi_0 \left[\frac{1}{\bt^n} \RHom^{\otimes k\ps{\bt}}_{\DCoh(M_0}(\F,\G)\right]\\
      &= \dlim_n \pi_{-2n} \RHom^{\otimes k\ps{\bt}}_{\DCoh(M_0}(\F,\G)
    \end{align*}
    However, the following formula for maps in $\pi_0 \DSing(M_0)$ appears in the literature
    \[ \pi_0 \RHom^{\otimes k\pl{\bt}}_{\DSing(M_0)}(\ol{\F}, \ol{\G}) = \dlim_n \pi_{-2n} \RHom^{\otimes k\ps{\bt}}_{\DCoh(M_0)}(\F, \G)  \]
    and one may check that it is induced by the map we wrote down above.
  \item (Less lazy)
    Consider applying $\DCoh(M_0) \otimes_{\DCoh(\bB)} -$ to the (idempotent completed) Drinfeld-Verdier sequence
    \[ \Perf(\bB) \to \DCoh(\bB) \to \DSing(\bB) \]
    which by \autoref{lem:mf-prelim} may be identified with
    \[ \left\{ \begin{gathered} \text{$\bt$-torsion}\\\text{perfect $k\ps{\bt}$-mod}\end{gathered}\right\} \longrightarrow \Perf k\ps{\bt} \stackrel{-\otimes_{k\ps{\bt}}}\longrightarrow \Perf k\pl{\bt} \] 
      The result will again be an (idempotent completed) Drinfeld-Verdier sequence (\autoref{lem:small-tensor}).  The Claim implies that
      \[ \DCoh(M_0) \otimes_{\DCoh(\bB)} \Perf(\bB) = \Perf(M_0) \] Indeed, the LHS identifies with the full subcategory of $\DCoh(M_0)$ consisting of objects with locally $\bt$-torsion endomorphisms.  Since the first two terms of a Drinfeld-Verdier sequence determine the third up to idempotent completion, this completes the proof.
 \end{itemize}

 {\noindent}{\it Proof of Claim:} Recall the resolution in $\QC(M_0)$ (\autoref{ex:resolution})
  \[ \hocolim\left\{\cdots \to i^* i_* \F[2] \to i^*i_*\F[1] \to i^*i_* \F \right\} \stackrel{\sim}\longrightarrow \F \]
  Computing $\RHom^{\otimes k\ps{\bt}}_{\DCoh(M_0)}(\F, \F)$ using this resolution, we see that for $N > 0$ a null-homotopy of $\bt^N$ on $\RHom^{\otimes k\ps{\bt}}(\F, \F)$ realizes $\F$ as a homotopy retract of
  \[ \hocolim\left\{ i^* i_* \F[N] \to \cdots \to i^* i_* \F \right\} \] and conversely if $\F$ is a homotopy retract of this then $\bt^N$ is null-homotopic on $\RHom^{\otimes k\ps{\bt}}(\F, \F)$.

  If $\F \in \Perf(M_0)$, then it is compact in $\QC(M_0)$ and the identity factors through a finite piece as above.  Thus (i) implies (ii).
  
Conversely: $i_* \F$ is coherent since $i$ is finite, and thus perfect since $M$ is regular.  Thus $i^* i_* \F$ is perfect, and so is anything built from it by shifts, finite colimits, and retracts.  This proves that (ii) implies (i).

The implication (ii) implies (iii) is immediate.

Note that $\A = \RHom^{\otimes k\ps{\bt}}_{\DCoh(M_0)}(\F, \F)$ is a $k\ps{\bt}$-algebra, and consider the unit $1\colon k\ps{\bt} \to \A$.  If $\A$ is locally $t$-torsion, then we may write $\A = \dlim P_\alpha$ with $P_\alpha$ perfect and $t$-torsion; since $k\ps{\bt}$ is perfect, $1$ factors through $1\colon k\ps{\bt} \to P_\alpha$ for some $\alpha$. Consequently, there exists $N>0$ so that $\bt^N \cdot 1\colon k\ps{\bt}[2N] \to P_\alpha$, and so also $\bt^N \cdot 1\colon k\ps{\bt}[2N] \to \A$, is null-homotopic.  This implies that $\bt^N \cdot 1 = 0 \in \pi_{-2n} \A$, proving that (iii) implies (iv).

Conversely, if $\bt^N \cdot 1 = 0 \in \pi_{-2N} \A$ then $\bt^N \cdot 1\colon k\ps{\bt}[2N] \to \A$ is null-homotopic.  Since $\A$ is an algebra, we conclude that $\bt^N\colon \A[2N] \to A$ is null-homotopic.  This proves that (iv) implies (ii).

Finally, note that $\F \mapsto 0 \in \MF(M,f)$ if and only if $1 = 0 \in \RHom^{\otimes k\pl{\bt}}_{\MF(M,f)}(\F, \F)$.  Since
\begin{align*}
  \pi_0 \RHom^{\otimes k\pl{\bt}}_{\MF(M,f)}(\F, \F) &= \pi_0 \dlim_N \frac{1}{\bt^N} \RHom^{\otimes k\ps{\bt}}_{\DCoh(M_0)}(\F, \F)\\
  &= \dlim_N \pi_{-2N} \RHom^{\otimes k\ps{\bt}}_{\DCoh(M_0)}(\F, \F)
\end{align*} where the filtered limit is formed under multiplication by $\bt$.  This proves that (iv) $\Leftrightarrow$ (v).\end{proof}

\begin{lemma}\mbox{}\label{lem:small-tensor} 
  \begin{enumerate}
    \item Suppose $\A^\otimes \in \CAlg(\dgcatidm_k)$ is a rigid symmetric-monoidal dg-category, $\D \in \A\mod(\dgcatidm_k)$ an $\A$-module category, and
      \[ \C' \stackrel{F}\longrightarrow \C \stackrel{G}\longrightarrow \C'' \]
      a diagram in $\A\mod(\dgcatidm_k)$ which is a Drinfeld-Verdier sequence.  Then,
      \[ \C' \otimes_\A \D \longrightarrow \C \otimes_\A \D \longrightarrow \C'' \otimes_\A \D \] is again a Drinfeld-Verdier sequence.
    \item Suppose $R \in \CAlg(k\mod)$ is a commutative dg-$k$-algebra, $\D \in \dgcatidm_R$ an $R$-linear dg-category, and
\[ \C' \longrightarrow \C \longrightarrow \C'' \]
      an $R$-linear Drinfeld-Verdier sequence.  Then,
      \[ \C' \otimes_R \D \longrightarrow \C \otimes_R \D \longrightarrow \C'' \otimes_R \D \] is again a Drinfeld-Verdier sequence.
  \end{enumerate}
\end{lemma}
\begin{proof} For (ii) in the literal dg-framework see \cite[Prop.~1.6.3]{Drinfeld-Quotient}. We include a proof in the framework in which we work: Note that (ii) follows from (i) by taking $\A = \Perf(R)$ and taking into account the equivalence $\dgcatidm_R = (\Perf(R))\mod(\dgcatidm_k)$. To prove (i), note that it suffices to pass to the following $\Ind$-completed version: Observe that
\[ \Ind \C'' \longleftarrow \Ind \C \longleftarrow \Ind \C' \] is a diagram in $(\Ind\A)\mod(\dgcatbig_k)$ whose underlying diagram in $\Pr^L$ is a cofiber sequence along colimit and compact preserving maps (i.e., a recollement sequence).  It suffices to show that applying $- \ohotimes_{\Ind \A} \Ind \D$ sends this to another cofiber sequence in $\Pr^L$: All three terms are compactly generated, and the arrows are compact and colimit preserving, so that the diagram of compact objects will be a cofiber sequence of idempotent complete $\infty$-categories (i.e., a Drinfeld-Verdier sequence).  

It thus suffices to show that $- \ohotimes_{\Ind\A} \Ind \D$ preserves the property of being a colimit diagram in $\Pr^L$, or passing to right adjoints that it preserve the property of being a limit diagram in $\Pr^R$.  Note that the forgetful functors $(\Ind\A)\mod(\dgcatbig_k) \to \Pr^R \to \Cat_\infty$ create limits, so that it is enough to show that $\Ind \D$ is dualizable over $\Ind \A$, since $\A^\otimes$ is rigid, so that tensoring by it preserves limits.
  
The hypothesis that $\A^\otimes$ be \emph{symmetric}-monoidal is un-necessary.  Suppose that $\D$ is a left $\A^\otimes$-module category.  Since $\A$ was assumed rigid, one can show that $\D^\op$ may be equipped with the structure of right $\A^\otimes$-module category, heuristically given by $d \stackrel{\D^\op}\otimes V \eqdef V^\dual \stackrel{\D}\otimes d$.  Then, $\Ind(\D^\op)$ is the $\Ind(\A)$-linear dual of $\Ind(\D)$, i.e., there is a natural equivalence
\[ T \otimes_{\Ind A} \Ind(\D) \stackrel{\sim}\longrightarrow  \Fun^{ex}_{\rmod\A}(\D^\op, T)  =\Fun^L_{\rmod\Ind\A}(\Ind \D^\op, T) \]
 for $T \in \rmod\Ind(\A)$.
\end{proof}

The above proof also showed:
\begin{corollary}\label{cor:mf-basechange} Suppose $(M,f)$ is an LG pair.  Then, the Drinfeld-Verdier quotient sequence \[ \Perf(M_0) \to \DCoh(M_0) \to \DSing(M_0) \] is obtained, by tensoring $\DCoh(M_0) \otimes_{\DCoh(\bB)} -$, from the universal example of $M = \pt$ (so that $M_0 = \bB$).
\end{corollary}

\begin{na} Combining the previous Proposition with Orlov's Theorem relating \emph{actual} matrix factorizations to $\DSing$ \cites{Orlov-DBrane, Orlov-nonaffine}, one finally sees that the notation $\MF(M,f)$ is justified.  Strictly speaking, Orlov's Theorem is only stated in the case where $f$ is flat (i.e., not zero on any component). However, it is possible to show that the above definition in fact coincides with (any reasonable definition of) matrix factorizations in general: 
  
The assignment $U \mapsto \MF^\kinfty(U, f)$ is an \'etale sheaf of $k\pl{\bt}$-linear $\infty$-categories on $M$ (\autoref{prop:mf-sheaf}). The same should be true in any other reasonable definition of infinite rank matrix factorizations, so that we are reduced to the affine case.  Passing to a connected component of $M$, we may suppose that $f$ is either flat (covered by Orlov's Theorem), or $f = 0$ (covered by direct inspection: both categories simply give $2$-periodic $\O_M$-modules).
\end{na}

\begin{remark} The above comparison to literal matrix factorizations has two defects:
  \begin{itemize}
    \item Requiring a separate argument in the case $f = 0$ is unsatisfying and doesn't actually show the existence of a natural functor.
    \item $\MF(M,f)$ is $2$-periodic with no assumptions on $M$, unlike $\DSing(M_0)$.  It is thus reasonable to ask for an explicit description of it in terms of matrix factorizations, without passing through the intermediary $\DSing(M_0)$ and perhaps without regularity assumptions on $M$.
  \end{itemize}
It seems likely that one can remedy these, but the author has not tried to check the details: There is a natural candidate for the replacement of (the inverse to) the usual $\mathrm{cok}$ functor.  Let $\Cpx^\flat(M_0)$ be the (``homotopy'') $\infty$-category of dg-$\O_M[B_M]$-modules whose underlying graded module are quasi-coherent and $\O_M$-flat.  Explicitly, an object is a triple $(V_\bullet, B_M, d)$ of a (flat, quasi-coherent) graded $\O_M$-module $V_\bullet$, a degree $+1$ square-zero endomorphism $B_M$, and a degree $-1$ square-zero endomorphism $d$ satisfying the commutation relation $(d + B_M)^2 = f$.  To it, one can associate the $2$-periodic curved complex of $\O_M$-modules
\[ \mathrm{cok}^{-1}\left( V_\bullet, B_M, d \right) = (V_\bullet,d) \otimes_{\O_M[B_M]} (\O_M[B]\pl{\bt}, \bt B) \subset \left(V_\bullet\pl{\bt}, d + \bt B_M\right) \in \Cpx^{\ZZ/2,curv}(M)\] 
as well as the $k\ps{\bt}$-linear analog
\[ \left( V_\bullet, B_M, d \right) \longmapsto (V_\bullet,d) \otimes_{\O_M[B_M]} (\O_M[B]\ps{\bt}, \bt B) \subset \left(V_\bullet\ps{\bt}, d + \bt B_M\right) \in \Cpx^{k\ps{\bt},curv}(M)\] 
There is presumably a reasonable notion of equivalence in $\Cpx^{\ZZ/2,curv}(M)$ and $\Cpx^{k\ps{\bt},curv}(M)$ such that these become identified with the quotient functors $\Cpx^\flat(M_0) \to \QCsh(M_0)$ and $\Cpx^\flat(M_0) \to \QCsh(M_0) \otimes_{k\ps{\bt}} k\pl{\bt}$.  (e.g., In the case $M = \pt$, where $M$ is affine and there is no curvature, it suffices to take quasi-isomorphisms of $k\ps{\bt}$-modules as equivalences.)
\end{remark}

\section{Thom-Sebastiani \& duality Theorems for (pre-) matrix factorizations}\label{sec:main-thm}
\subsection{Thom-Sebastiani}
\begin{na} For the duration of this section, suppose $(M,f)$ and $(N,g)$ are two LG pairs.  Set
  \[ (M \times N)_0 = (f \boxplus g)^{-1}(0) \qquad \left[(M \times N)_0\right]_0 = (f \boxminus g)^{-1}(0) \] 
  Define
  \[ \PreMF(M \times N, f, g) \eqdef \DCoh\left(M_0 \times N_0\right) \quad \text{equipped with its $k\ps{\bt_M, \bt_N}$-linear structure}\]  
  \[ \PreMF(M \times N, f \boxplus g, f\boxminus g) \eqdef \DCoh\left(\left[M \times N\right]_0\right) \quad \text{equipped with its $k\ps{\bt_+, \bt_-}$-linear structure}\]
  By \autoref{thm:coh-fmk}, exterior product induces a $k\ps{\bt_M, \bt_N}$-linear equivalence $\PreMF(M \times N, f, g) \isom \PreMF(M,f) \otimes_k \PreMF(N, g)$.  These two constructions are related as follows:
\end{na}

\begin{lemma}\label{lem:plus} Suppose $(M,f)$, $(N,g)$ are two LG pairs and $Z_M \subset M$, $Z_N \subset N$ closed subsets. Then, the $k$-linear equivalence $\boxtimes: \DCoh_{Z_M}(M_0) \otimes \DCoh_{Z_N}(N_0) \to \DCoh_{Z_M \times Z_N}( M_0 \times N_0)$ of \autoref{thm:coh-fmk} naturally lifts to a 
  \[ k\ps{\bt_+, \bt_-} \stackrel{\sim}\longrightarrow k\ps{\bt_M, \bt_N} \qquad \bt_+ \mapsto  \bt_M + \bt_N, \quad \bt_- \mapsto \bt_M - \bt_N \in k\ps{\bt_M, \bt_N} \]
  linear equivalence  \[ \boxtimes\colon \PreMF_{Z_M}(M, f) \otimes \PreMF_{Z_N}(N, g) \longrightarrow \PreMF_{Z_M \times Z_N}\left(M \times N, f \boxplus g, f\boxminus g\right)  \]
\end{lemma}
\begin{proof}  Set $\left[ (M \times N)_0 \right]_0 = (f \boxminus g)^{-1}\left( (M \times N)_0\right)$.  The equivalence
  \[ M_0 \times N_0 \stackrel{\sim}\longrightarrow \left[ (M \times N)_0 \right]_0 \] \[ \left(m,n,[h_f\colon  f(m) \to 0], [h_g\colon  g(n) \to 0]\right) \longmapsto \left(m,n, [h_f+h_g\colon f(m)+g(n) \to 0], [h_f-h_g\colon  f(m)-g(n) \to 0]\right) \]  is equivariant with respect to the following group automorphism of $\bB^2$:
  \[ \left(\Delta, \ol{\Delta}\right)\colon \bB^2 \longrightarrow \bB^2 \qquad \left(\left[h_1\colon 0 \to 0\right], \left[h_2\colon  0 \to 0\right]\right) \longmapsto \left( \left[h_1+h_2\colon  0 \to 0\right], \left[h_1-h_2\colon  0 \to 0\right]\right) \]

  The $k\ps{\bt_M, \bt_N}$-action on $\DCoh(M_0) \otimes \DCoh(N_0) \isom \DCoh(M_0 \times N_0)$  is obtained from the above action of $\bB^2$ on $M_0 \times N_0$.  The $k\ps{\bt_+, \bt_-}$-action on $\DCoh( (M \times N)_0)$ is obtained from the above action of $\bB^2$ on $\left[(M \times N)_0\right]_0$.

  It thus suffices to verify that pushforward along the group automorphism $(\Delta, \ol{\Delta})_*$ induces the indicated automorphism on the endomorphisms of the symmetric monoidal unit:  In terms of the explicit model of \autoref{prop:mf-prelim}, $(\Delta, \ol{\Delta})$ corresponds to the algebra homomorphism
  \[ \phi\colon \O_\bB^{\otimes 2} \isom k[x_+, x_{-}][\epsilon_i^+, \epsilon_i^{-}] \longrightarrow \O_\bB^{\otimes 2} \isom k[x_M, x_N][\epsilon_i^M, \epsilon_i^N] \]
  \[ \phi(x_+) = x_M + x_N \qquad \phi(\epsilon^+_i ) = \epsilon^M_i + \epsilon^N_i \]
  \[ \phi(x_{-}) = x_M - x_N\qquad  \phi(\epsilon^{-}_i) = \epsilon^M_i - \epsilon^N_i \]
  and is strictly compatible with the Hopf-algebra structure maps. Consider the Koszul-Tate resolutions that give the identifications with $k\ps{\bt_M, \bt_N}$ and $k\ps{\bt_+,\bt_-}$:
  \[ k_{+,-}  \sim k[x_+,x_-][\epsilon_i^+,\epsilon_i^{-}][u_+^m/m!, u_{-}^m/m!] \qquad \bt_+ = \frac{\partial}{\partial u_+},  \bt_- = \frac{\partial}{\partial u_-}\]
  \[ k_{M,N}  \sim k[x_M,x_N][\epsilon_i^M,\epsilon_i^{N}][u_M^m/m!, u_{N}^m/m!]  \qquad \bt_M = \frac{\partial}{\partial u_M},  \bt_N = \frac{\partial}{\partial u_N}\]
There is an isomorphism $\phi'\colon k_{+,-} \stackrel{\sim}\to \phi_* k_{M,N}$ of $\O_{\bB}^{\otimes 2}$-modules: Explicitly it is the algebra map given by $\phi$ on the $x$ and $\epsilon$ variables and
  \[ \phi'(u_+) = u_M + u_N \qquad \phi'(u_-) = u_M - u_N \]  It is now a simple check that this identifies the actions $\bt_+ = \bt_M + \bt_N$ and $\bt_- = \bt_M - \bt_N$.
\end{proof}

The first Main Theorem is the following result reminiscent of a Thom-Sebastiani theorem:
\begin{theorem}[Thom-Sebastiani for Matrix Factorizations]\label{thm:TS} Suppose $(M,f)$, $(N,g)$ are LG pairs, $(M \times N, f \boxplus g)$ their Thom-Sebastiani sum.  Suppose $Z_M \subset f^{-1}(0)$ and $Z_N \subset g^{-1}(0)$ are closed subsets. (The special case $Z_M = f^{-1}(0)$, $Z_N = g^{-1}(0)$ is the main one of interest.  Note that the support conditions will still matter on the product since $(f \boxplus g)^{-1}(0)$ will generally properly contain $Z_M \times Z_N$!)  Then,
  \begin{enumerate}
    \item The external tensor product determines an \emph{equivalence} of $k\ps{\bt}$-linear $\infty$-categories
      \[ \ell_*(- \boxtimes -) \colon \PreMF_{Z_M}(M,f) \otimes_{k\ps{\bt}} \PreMF_{Z_N}(N,g) \stackrel{\sim}{\longrightarrow} \PreMF_{Z_M \times Z_N}(M \times N, f \boxplus g). \]  
      Passing to $\Ind$-completions, it induces an \emph{equivalence} of cocomplete $k\ps{\bt}$-linear $\infty$-categories
      \[  \ell_*(- \boxtimes -) \colon \PreMF^\kinfty_{Z_M}(M,f) \ohotimes_{k\ps{\bt}} \PreMF^\kinfty_{Z_N}(N,g) \stackrel{\boxtimes}{\longrightarrow} \PreMF^\kinfty_{Z_M \times Z_N}(M \times N, f \boxplus g). \] 
    \item The external tensor product determines an \emph{equivalence} of $k\pl{\bt}$-linear $\infty$-categories
      \[  \ell_*(- \boxtimes -) \colon \MF_{Z_M}(M,f) \otimes_{k\pl{\bt}} \MF_{Z_N}(N,g) \stackrel{\sim}{\longrightarrow} \MF_{Z_M \times Z_N}(X \times Y, f \boxplus g). \]  
      Passing to $\Ind$-completions, it induces an \emph{equivalence} of cocomplete $k\pl{\bt}$-linear $\infty$-categories
      \[  \ell_*(- \boxtimes -) \colon \MF^\kinfty_{Z_M}(X,f) \ohotimes_{k\pl{\bt}} \MF^\kinfty_{Z_N}(N,g) \stackrel{\sim}{\longrightarrow} \MF^\kinfty_{Z_M \times Z_N}(M \times N, f \boxplus g). \] 
    \item The functors of (ii) induce a $k\pl{\bt}$-linear \emph{equivalence}
      \[ \bigoplus_{\lambda \in -\cval(f) \cap \cval(g)} \MF(M,f+\lambda) \otimes_{k\pl{\bt}} \MF(N, g-\lambda) \stackrel{\sim}{\longrightarrow} \MF(M \times N, f \boxplus g). \]
      Passing to $\Ind$-completions, it induces an equivalence of cocomplete $k\pl{\bt}$-linear $\infty$-categories
      \[ \bigoplus_{\lambda \in -\cval(f) \cap \cval(g)} \MF^\kinfty(M,f+\lambda) \otimes_{k\pl{\bt}} \MF^\kinfty(N, g-\lambda) \stackrel{\sim}{\longrightarrow} \MF^\kinfty(M \times N, f \boxplus g). \]
  \end{enumerate}
\end{theorem}
\begin{proof} Certainly the $\Ind$-complete versions follow from the small versions, so we will show those.
  \begin{enumerate}
    \item 
      Let $i\colon M_0 \to M$, $j\colon N_0 \to N$, $k\colon (M \times N)_0 \to M \times N$, and $\ell\colon M_0 \times N_0 \to (M \times N)_0$ be the various inclusions.   The functor in question will be a refinement of the $k$-linear functor \[ \DCoh(M_0) \otimes \DCoh(N_0) \longrightarrow \DCoh_{M_0 \times N_0}( (M \times N)_0) \qquad  \F \otimes \G \mapsto \ell_*\left( \F \boxtimes \G \right) \]  

Before writing down a $k\ps{\bt}$-linear functor, we show how to conclude from this: Once a $k\ps{\bt}$-linear functor is written down, it suffices to check that the underlying $k$-linear functor is an equivalence.  Write
	  \begin{align*} \PreMF_{Z_M}(M,f) \otimes_{k\ps{\bt}} \PreMF_{Z_N}(N,g)  &\stackrel{\boxtimes}\isom \PreMF_{Z_M \times Z_N}\left(M \times N, f, g\right) \otimes_{k\ps{\bt_M,\bt_N}} k\ps{\bt}  \\ 
	    \intertext{Applying \autoref{lem:plus} identity this with}
  &\isom \left[\PreMF_{Z_M \times Z_N}\left(M \times N, f \boxplus g, f \boxminus g\right)\right] \otimes_{k\ps{\bt_+, \bt_{-}}} k\ps{\bt_+} \\
  &\isom \PreMF_{Z_M \times Z_N}\left( (M \times N)_0, f\boxminus g\right) \otimes_{k\ps{\bt_{-}}} k \\
  \intertext{Finally, applying \autoref{cor:premf-supt} we see that $\ell_*$ induces an equivalence}
  &\stackrel{\ell_*}\isom \DCoh_{Z_M \times Z_N}\left( (M \times N)_0 \right)
\end{align*}

We now complete the proof by constructing the desired $(\DCoh(\bB),\circ)$-linear functor \[ \DCoh(M_0) \otimes_{\DCoh(\bB)} \DCoh(N_0) \to \DCoh( (M \times N)_0 ) \] using a suitable augmented simplicial diagram of derived schemes with $\bB$-action and $\bB$-equivariant maps
\[ X_\bullet = \left\{ M_0 \times \bB^{\bullet-1} \times N_0 \right\} \longrightarrow (M \times N)_0 = X_{-1}\] constructed as follows:
\begin{itemize} \item Informally (i.e., at the level of $\pi_0$ of functor of points)
 \[ X_{-1}(R) = \left( m \in M(R), n \in N(R), [h_{f+g}\colon f(m) + g(n) \to 0] \right) \] and for $\ell \geq 0$
 \[ X_\ell(R) = \left( m \in M(R), n \in N(R), [h_f\colon  f(m) \to 0], [h_g\colon  g(m) \to 0], [h_1\colon 0 \to 0], \cdots, [h_\ell\colon 0 \to 0] \right) \]  with the simplicial degeneracies given by the inserting the identity loop $[\id\colon 0 \to 0]$, and the simplicial face maps given by ``pointwise addition'' of the appropriate loops.  The augmentation is given by adding together all the loops.  It is clear that this is a simplicial diagram, and that it is $\bB$-equivariant.
 \item The same formulas can be made precise, either at the level of actual functors of points on $\DRng$, or using an explicit cosimplicial diagram of sheaves of dg-algebras on $M \times N$ (representing the pushforwards of the structure sheaves to $M \times N$).  For the second approach, one can use the explicit models
   \[ X_{-1}(R) = \Spec_{M \times N} \O_{M \times N}\left[\begin{gathered} x_{+} \end{gathered}\right]\left[\begin{gathered}\epsilon^+_1, \epsilon^+_2\end{gathered}\right]/\left(\begin{gathered} d\epsilon^+_1 = x_+ - (f \otimes 1 + 1 \otimes f), d\epsilon^+_2 = x_+ \end{gathered}\right) \]
     \[ X_\ell(R) = \Spec_{M \times N} \O_{M \times N}\left[\begin{gathered} x_M\\x_N\\x_1,\ldots,x_\ell\end{gathered}\right]\left[\begin{gathered}\epsilon^M_1, \epsilon^M_2\\ \epsilon^N_1, \epsilon^N_2 \\ \epsilon^i_1, \epsilon^i_2, i=1,\ldots,\ell\end{gathered}\right]/\left(\begin{gathered} d\epsilon^M_1 = x_M - (f \otimes 1), d\epsilon^M_2 = x_M \\ d\epsilon^N_1 = x_N - (1 \otimes f), d\epsilon^N_2 = x_N \\ d\epsilon^i_1 = d\epsilon^i_2 = x\end{gathered}\right) \]
       where the simplicial structure maps use the co-identity and co-multiplication/co-action (c.f., proof of \autoref{prop:mf-prelim}) and the augmentation uses those together with $x_+ \mapsto x_M + x_N$ and $\epsilon_i^+ \mapsto \epsilon_i^M + \epsilon_i^N$.
\end{itemize}
This completes the construction as follows:
\begin{itemize}
  \item Applying ($\DCoh(-)$, $f_*$) to the augmented simplicial diagram $X_\bullet$ gives an augmented simplicial diagram $\DCoh(X_\bullet)$ in $\dgcat_k$.  Applying \autoref{prop:fmk-coh-surj}, we identify the simplicial diagram with the simplicial bar construction computing $\DCoh(M_0) \otimes_{\DCoh(\bB)} \DCoh(N_0)$.  Thus, the augmented diagram precisely encodes a functor
    \[ \DCoh(M_0) \otimes_{\DCoh(\bB)} \DCoh(N_0) \longrightarrow \DCoh( (M \times N)_0 ) \] which is in fact an enrichment of $\ell(- \boxtimes -)$ since the map $X_0 \to X_{-1}$ is precisely $\ell$.
  \item Since $X_\bullet$ was a $\bB$-equivariant diagram, the previous functor is naturally $\DCoh(\bB)$-linear.
\end{itemize}
    \item Follows from (i) by the definition $\MF(M,f) = \PreMF(M,f) \otimes_{k\ps{\bt}} k\pl{\bt}$.
    \item Let $Z = \crit(f \boxplus g) \cap (M \times N)_0$ be the components of the critical locus of $f \boxplus g$ lying over zero.  There is a disjoint union decomposition
      \[ Z = \bigsqcup_{\lambda \in -\cval(f) \cap \cval(g)} Z_\lambda \qquad \text{where} \qquad Z_\lambda = Z \cap \left(f^{-1}(-\lambda) \times g^{-1}(\lambda)\right). \]
      By \autoref{prop:orlov-cpltn},  the inclusion induces an equivalence $\MF_Z(M \times N, f \boxplus g) = \MF(M \times N, f \boxplus g)$ and the above disjoint union decomposition gives
      \[ \MF_Z(M \times M, f \boxplus g) = \bigoplus_{\lambda \in -\cval(f) \cap \cval(g)} \MF_{Z_\lambda}(M \times M, f \boxplus g)  \]
      Combining with (ii) completes the proof.\qedhere
  \end{enumerate}
\end{proof}

\begin{remark} Item (iii) above admits the following re-interpretation.  Define $\MF^\tot(M, f)$ to be the sheaf of $\infty$-categories on $\AA^1$, supported on $\cval(f)$, given heuristically by $\lambda \mapsto \MF(N, f-\lambda)$.  Then,  \[ \MF^\tot(M \times N, f \boxplus g) = \MF^\tot(M, f) \ast \MF^\tot(N, g) \] where $\ast$ denotes convolution.
\end{remark}

\begin{remark} In case $\cval g = \{0\}$, item (iii) has an especially simple formulation: $\MF(M,f) \otimes_{k\pl{\bt}} \MF(N,g) \isom \MF(M \times N, f \boxplus g)$.  Taking $N = \AA^2$, $g = x^2+y^2$, one can show that there is a $k\pl{\bt}$-linear equivalence $\MF(N, g) \isom \Perf k\pl{\bt}$.  So, (iii) recovers \emph{Kn\"orrer Periodicity} as a special case:
  \[ \MF\left(M \times \AA^2, f \boxplus \{ x^2+y^2 \}\right) = \MF(M, f) \otimes_{k\pl{\bt}} \MF(\AA^2, \{x^2+y^2\}) \isom \MF(M, f). \]
\end{remark}

The following is of course well-known, e.g., from its role in \cite{Orlov-completion}. We sketch a proof for the reader's convenience:
\begin{prop}\label{prop:orlov-cpltn} (Recall that we work in the world of idempotent complete categories, and so have implicitly passed to idempotent completions!) Suppose $(M,f)$ is an LG pair, and $Z \subset f^{-1}(0)$ a closed set containing all components of the critical locus lying over $0$.  Then, the $k\ps{\bt}$-linear inclusion $\PreMF_Z(M,f) \hookrightarrow \PreMF(M,f)$ induces a $k\pl{\bt}$-linear \emph{equivalence} $\MF_Z(M,f) \stackrel{\sim}{\longrightarrow} \MF(M,f)$.
\end{prop}
\begin{proof}[Sketch]  
  Set $X = f^{-1}(0)$, $U = X \setminus Z$.  A $k\pl{\bt}$-linear functor is an equivalence iff it is an equivalence when regarded as a $k$-linear functor, so it is enough to prove this.  Since $U$ is regular, $\Perf(U) = \DCoh(U)$ and this follows from the following diagram of $k$-linear idempotent complete $\infty$-categories
  \[ \xymatrix{
  \Perf_Z(X) \ar@{^{(}->}[d] \ar@{^{(}->}[r] & \DCoh_Z(X) \ar@{^{(}->}[d] \ar@{->>}[r] & \DSing_Z(X) \isom \MF_Z(X, f) \ar[d] \\
  \Perf(X) \ar@{->>}[d] \ar@{^{(}->}[r] & \DCoh(X) \ar@{->>}[d] \ar@{->>}[r] & \DSing(X) \isom \MF(X,f) \ar[d] \\
  \Perf(U)  \ar[r]^{\sim} & \DCoh(U) \ar[r] & \DSing(U) = 0 }
  \] where each row and column is a Drinfeld-Verdier quotient.
\end{proof}

\subsection{Duality and functor categories}
Recall the following standard Lemma
\begin{lemma}\label{lem:dual-mod} \mbox{}
  \begin{itemize}
    \item Suppose $\A^\otimes \in \CAlg(\dgcatidm_k)$ is a rigid symmetric-monoidal dg-category, and $\C \in \A\mod(\dgcatidm_k)$ is an $\A$-module category.  Then, $\Ind \C \in \Ind\A\mod(\dgcatbig_k)$ is dualizable, with dual $\Ind \C^\op$.
    \item Suppose $R$ is an $E_\infty$-algebra, and $\C \in \dgcatidm_R$.  Then, $\Ind \C = \dgmod_R(\C^\op)$ is dualizable, with dual $\Ind \C^\op = \dgmod_R(\C)$.
  \end{itemize}
\end{lemma}
\begin{proof} C.f., \autoref{lem:small-tensor}.
\end{proof}

Our second Main Theorem analyzes the interplay of passage to matrix factorizations with the usual (coherent) Grothendieck duality:
\begin{theorem}[Duality for Matrix Factorizations]\label{thm:dual} Suppose $(M,f)$ is an LG pair, $Z \subset f^{-1}(M)$ a closed subset.  Then,
  \begin{enumerate}
    \item The usual Grothendieck duality lifts to a $k\ps{\bt}$-linear anti-equivalence $\DD(-)\colon \PreMF_Z(M,f)^\op \isom \PreMF_Z(M, -f)$.  So, $\PreMF^\kinfty_Z(M,f)$ is dualizable as cocomplete $k\ps{\bt}$-linear category and (the above lift of) Grothendieck duality induces a $k\ps{\bt}$-linear equivalence
      \[  \PreMF^\kinfty_Z(M,f)^\dual \stackrel{\sim}{\longrightarrow} \PreMF_Z^\kinfty(M, -f). \]
    \item The usual Grothendieck duality lifts to a $k\pl{\bt}$-linear anti-equivalence $\DD(-)\colon \MF_Z(M,f)^\op \isom \MF_Z(M, -f)$.  So, $\MF^\kinfty_Z(M,f)$ is dualizable as cocomplete $k\pl{\bt}$-linear category, and the usual Grothendieck duality functor induces an equivalence
      \[ \MF^\kinfty_Z(M,f)^\dual \stackrel{\sim}{\longrightarrow} \MF^\kinfty_Z(M, -f). \]
  \end{enumerate}
\end{theorem}
\begin{proof} Note that (i) implies (ii). By \autoref{lem:dual-mod}, it suffices to prove either the first or the second sentence of (i): passing to compact objects, or to Ind-completions, goes between the two versions.  Note that Grothendieck duality preserves support conditions: $\DCoh_Z(M)$ is generated by pushforwards from $Z$, and Grothendieck duality commutes with proper pushforward.  So, it suffices to prove the version without support conditions. 
  
In a rigid-enough dg-model for the $k\ps{\bt}$-linear $\infty$-category $\PreMF(M, f)$, it should be possible to prove the first sentence of (i) by direct computation. Since we are not in such a framework, we adopt a more indirect approach.    It suffices to write down a colimit preserving, $k\ps{\bt}$-linear functor
  \[ \langle \cdot \rangle \colon \PreMF^\kinfty(M, -f) \ohotimes_{\QCsh(\bB)} \PreMF^\kinfty(M,f) \longrightarrow \QCsh(\bB) \] and then show that it is a ``perfect pairing'' in the sense that the induced functor \[ \PreMF^\kinfty(M,-f) \to \Fun^L_{\QCsh(\bB)}\left(\PreMF^\kinfty(M,f), \QCsh(\bB)\right) \isom \Ind \PreMF(M,f)^\op \]   is an equivalence.

    Let $(M^2)_0 = (-f \boxplus f)^{-1}(0)$; $\ell\colon (M_0)^2 \to (M^2)_0$ and $k\colon (M^2)_0 \to M^2$ the inclusions; $\Delta\colon M \to M^2$ the diagonal, and $\ol{\Delta}\colon M \to (M^2)_0$ its factorization through $k$.  
To define $\langle \cdot \rangle$, we apply \autoref{thm:TS} to identify
  \[ \ell_*\left(- \boxtimes -\right)\colon \PreMF^\kinfty(M, -f) \ohotimes_{\QCsh(\bB)} \PreMF^\kinfty(M,f) \stackrel{\sim}\longrightarrow \PreMF^\kinfty(M^2, -f \boxplus f) \]
  and then the functor
  \[ \RHom^{\otimes k\ps{\bt}}_{\PreMF^\kinfty(M^2,-f\boxplus f)}\left(\ol{\Delta}_* \O_M, -\right)\colon\PreMF^\kinfty(M^2, -f \boxplus f) \to \QCsh(\bB) \]
  is colimit preserving (since $\ol{\Delta}_* \O_M \in \DCoh((M^2)_0)$ is compact) and naturally admits a $\QCsh(\bB)$-linear structure (since $\QCsh(\bB)$ is symmetric monoidal).  Define $\langle\cdot\rangle$ as the composite
  \[ \langle - \otimes - \rangle = \RHom^{\otimes k\ps{\bt}}_{\PreMF^\kinfty(M^2, -f \boxplus f)}\left(\ol{\Delta}_* \O_M, \ell_*(-\boxtimes-)\right) \]

It remains to show that this induces a perfect pairing, i.e., that the adjoint map is an equivalence.  For this it suffices to check that the underlying $k$-linear functor of the adjoint is an equivalence.  But this underlying $k$-linear functor is simply $\Ind$ of regular Grothendieck duality for $\DCoh(M_0)$: Note that there is a Cartesian diagram
\[ \xymatrix{ M_0 \ar[r]^i \ar[d]_{\Delta_{M_0}} & M \ar[d]^{\ol{\Delta}} \\ (M_0)^2 \ar[r]_\ell & (M^2)_0  } \] with $i$ and $\ell$ of finite-Tor dimension.  So, for $\F, \G \in \DCoh(M_0)$ there are natural equivalences 
\begin{align*} \RHom_{M_0}(\DD \F, \G) &= \RGamma(M_0, \F \shotimes \G) \\
  &= \RHom_{M_0}( (\Delta_{M_0})_* \O_{M_0}, \F \boxtimes \G) \\
  &= \RHom_{M_0}( (\Delta_{M_0})_* i^* \O_M, \F \boxtimes \G) \\
  &= \RHom_{M_0}( \ell^* \ol{\Delta}_* \O_M, \F \boxtimes \G) \\
  &= \RHom_{M_0}( \ol{\Delta}_* \O_M, \ell_*(\F \boxtimes \G)) \qedhere
\end{align*}
\end{proof}

Formally combining the above two theorems, we obtain the following descriptions of functor categories:
\begin{theorem}[Functors between Matrix Factorizations]\label{thm:functors} Suppose $(M,f)$, $(N, g)$ are LG pairs.  Then,
  \begin{enumerate}
    \item There is a $k\ps{\bt}$-linear equivalence
  \[ \Fun_{k\ps{\bt}}^L\left(\PreMF^\kinfty_{Z_M}(M,f), \PreMF^\kinfty_{Z_N}(N, g)\right) \stackrel{\sim}{\longrightarrow} \PreMF^\kinfty_{Z_M \times Z_N}(M \times N, -f \boxplus g) \] 
\item There is a $k\pl{\bt}$-linear equivalence
      \[ \Fun_{k\pl{\bt}}^L\left(\MF^\kinfty_{Z_M}(M,f), \MF^\kinfty_{Z_N}(N, g)\right) \stackrel{\sim}{\longrightarrow} \MF^\kinfty_{Z_M \times Z_N}(M \times N, -f \boxplus g) \] 
    \item Summing (ii) over support conditions giving different components of $(-f \boxplus g)^{-1}(0)$ yields an equivalence
      \[ \bigoplus_{\lambda \in \cval(f) \cap \cval(g)} \Fun^L_{k\pl{\bt}}\left(\MF^\kinfty(M,f-\lambda), \MF^\kinfty(N, g-\lambda)\right) \stackrel{\sim}\longrightarrow \MF^\kinfty(M \times N, -f \boxplus g)  \]
    \item Specializing (i) and (ii) to the case $M=N$, $f=g$, we obtain equivalences
      \[ \Fun^L_{k\ps{\bt}}\left(\PreMF^\kinfty_{Z}(M,f), \PreMF^\kinfty_{Z}(M,f)\right) \stackrel{\sim}\longrightarrow \PreMF^\kinfty_{Z \times Z}\left(M \times M, -f \boxplus f\right) \]
      \[ \Fun^L_{k\pl{\bt}}\left(\MF^\kinfty_{Z}(M,f), \MF^\kinfty_{Z}(M,f)\right) \stackrel{\sim}\longrightarrow \MF^\kinfty_{Z \times Z}\left(M \times M, -f \boxplus f\right) \]
      Let $(M^2)_0 = (-f \boxplus f)^{-1}(0)$.  The diagonal $\Delta\colon M \to M^2$ factors through $\ol{\Delta}\colon M \to (M^2)_0$.  Set \[ \ol{\O_\Delta} = \ol{\Delta}_* \O_M \in \DCoh( (M^2)_0, \qquad \ol{\omega_{\Delta,Z}} = \ol{\Delta}_* \RGamma_Z \omega_M \in \Ind\DCoh_{Z^2}( (M^2)_0).\] Under the equivalence above,
      \[\xymatrix@1{ \id_{\PreMF^\kinfty_{Z}(M,f)} \ar@{|->}[r] & \ol{\omega_{\Delta,Z}}} \qquad \xymatrix@1{ \ev_{\PreMF^\kinfty_Z(M,f)}(-) \ar@{|->}[r] & \RHom^{\otimes k\ps{\bt}}_{\PreMF^\kinfty(M^2,-f\boxplus f)}(\ol{\O_\Delta}, -) }\]
      \[ \xymatrix@1{ \id_{\MF^\kinfty_{Z}(M,f)} \ar@{|->}[r] & \ol{\omega_{\Delta,Z}}} \qquad \xymatrix@1{\ev_{\MF^\kinfty_Z(M,f)}(-) \ar@{|->}[r] & \RHom^{\otimes k\pl{\bt}}_{\MF^\kinfty(M^2, -f\boxplus f)}(\ol{\O_\Delta}, -)}\]
    \item Specializing (iii) to the case $M = N$, $f=g$, we obtain an equivalence
      \[ \bigoplus_{\lambda \in \cval(f)} \Fun^L_{k\pl{\bt}}\left(\MF^\kinfty(M,f-\lambda), \MF^\kinfty(M,f-\lambda)\right) \stackrel{\sim}\longrightarrow \MF^\kinfty(M^2, -f \boxplus f) \]
      under which
      \[ \xymatrix@1{ \oplus_{\lambda \in \cval(f)} \id_{\MF^\kinfty(M,f-\lambda)} \ar@{|->}[r] & \ol{\omega_\Delta} = \ol{\Delta}_* \omega_M} \] \[ \xymatrix@1{\oplus \ev_{\MF^\kinfty(M,f-\lambda)}(-) \ar@{|->}[r] & \RHom^{\otimes k\pl{\bt}}_{\MF^\kinfty(M^2, -f\boxplus f)}(\ol{\O_\Delta}, -) }\]
      \end{enumerate}
\end{theorem}
\begin{proof}\mbox{}
  \begin{enumerate}
    \item The first equality follows from the adjunction of $\Fun^L_{R}$ and $\ohotimes_R$ on $\dgcatbig_R$, together with \autoref{thm:dual}.  The second from \autoref{thm:TS}.
    \item Base change of (i).
    \item Combine \autoref{thm:dual} with the adjunction and \autoref{thm:TS}(iii).
    \item The only new statement is the identification of the functor represented by $\ol{\omega_{\Delta,Z}}$ and the trace.  The identification of $\ev$ trace follows from the proof of \autoref{thm:dual}.  To identify $\ol{\omega_{\Delta, Z}}$ with the identify functor, we must trace through the equivalence of the theorem.

Let $i\colon M_0 \to M$, $j\colon N_0 \to N$, and $k\colon (M \times N)_0 \to M \times N$ be the inclusions.  For a compact object $\K \in \PreMF_{M_0 \times N_0}(M \times N, -f \boxplus g)$, let $\Phi'_{\K}$ denote the corresponding functor.  We claim that $\Phi'$ is determined by the following refinement of the statement that $j_* \circ \Phi'_{\K} = \Phi_{k_* \K} \circ i_*$ compact objects:\\
{\noindent}{\bf Claim:} There is a $k\ps{\bt}$-linear equivalence
\begin{align*} \RHom^{\otimes k\ps{\bt}}_{\PreMF_{Z_N}^\kinfty(N, g)}(T, \Phi'_\K(T')) &= 
  \RHom^{\otimes k\ps{\bt}}_{\PreMF^\kinfty(M^2 \times N,f\boxplus -f \boxplus g)}\left(\ol{\Delta}_* \O_M \boxtimes T, T' \boxtimes \K\right)  \\
&= \RHom_{\QCsh(M^2 \times N)}\left(\Delta_* \O_M \boxtimes j_* T, i_* T' \boxtimes k_* \K\right)^{S^1}  \\
&=  \RHom_{\QCsh(N)}(j_* T, \Phish_{k_* \K}(i_* T'))^{S^1} 
\end{align*}
naturally in $T \in \PreMF_{Z_N}(N,g)$ and $T' \in \PreMF_{Z_M}(M,f)$, where $\Phish_{k_* \K}$ denotes the shriek integral transform of \autoref{thm:coh-fmk}.  \\
{\noindent}{\it Proof of Claim:} Tracing through the proof and using the previous Theorems repeatedly, we see that
 \begin{align*} \RHom_{\PreMF^\kinfty(N,g)}&\left(T, \Phi'_{\F \otimes \G}(T')\right) \\ 
   &= \RHom_{\PreMF^\kinfty(M^2 \times N,-f\boxplus f \boxplus g)}(\ol{\Delta}_* \O_M, \F \boxtimes T') \otimes_{k\ps{\bt}} \RHom_{\PreMF^\kinfty(N,g)}(T, \G)\\
   &= \RHom_{\PreMF^\kinfty(M^2 \times N,-f\boxplus f \boxplus g)}(\ol{\Delta}_* \O_M \boxtimes T, \F \boxtimes T' \boxtimes \G) \\
   &= \RHom_{\PreMF^\kinfty(M^2 \times N,f\boxplus -f \boxplus g)}(\ol{\Delta}_* \O_M \boxtimes T, T' \boxtimes \F \boxtimes \G) 
 \end{align*}
 so that, extending by colimits, we obtain
 \begin{align*} \RHom_{\PreMF^\kinfty(N,g)}\left(T, \Phi'_{\K}(T')\right) &= \RHom_{\PreMF^\kinfty(M^2 \times N,f\boxplus -f \boxplus g)}\left(\ol{\Delta}_* \O_M \boxtimes T, T' \boxtimes \K\right) \\
   &= \RHom_{\QCsh(M^2 \times N)}\left(\Delta_* \O_M \boxtimes j_* T, i_* T' \boxtimes k_* \K\right)^{S^1} \\
   \intertext{Running the analogous argument for $\QCsh$ and $\Phish$, we identify the last line with}
   &= \RHom_{\QCsh(N)}\left(j_* T, \Phish_{k_* \K}(i_* T')\right)^{S^1}
 \end{align*} as claimed.\\

We now complete the proof:  Note that the property in the claim characterizes $\Phi'_{\K}$ up to natural equivalence: Since $\Phi'_{\K}$ is colimit-preserving, it is determined by its restriction to compact objects $T' \in \PreMF_{Z_M}(M,f)$, and since $\PreMF_{Z_N}^\kinfty(N,g)$ is compactly-generated it is determined by the above mapping functors.

First, a feasibility check: Note that \[ k_* \ol{\omega_{\Delta, Z}} = \Delta_* \RGamma_Z \omega_M \]  so that \autoref{thm:coh-fmk} implies that $\Phish_{k_* \K} = \RGamma_Z$ which is naturally the identity on the essential image of $i_*\colon \DCoh_Z(M_0) \to \DCoh(M)$; thus, we're done up to identifying the $S^1$-action.  This seems inconvenient in this viewpoint, so instead we take a different approach.  

Consider the (simplicial diagram of) Cartesian diagrams
\[ \xymatrix{ M_0 \times \bB^\bullet \ar[d]_{D_1} \ar[r]^{D_2} & M \times M_0 \times \bB^\bullet \ar[d]^{\ol{\Delta}_1} \\ M_0 \times M \times \bB^\bullet \ar[r]_{\ol{\Delta}_2} & (M^3)_0 \times \bB^\bullet } \] where the $D_i, \ol{\Delta}_i$ are the evident diagonal maps.  All the arrows finite and of finite Tor-dimension.  Considering $\bullet = 0$, we obtain a $k$-linear identification
\begin{align*}
  \RHom_{\PreMF^\kinfty(M,f)}\left(T, \Phi'_{\ol{\omega_{\Delta,Z}}}(T')\right)
  &=\RHom_{\PreMF^\kinfty(M^3,f\boxplus-f \boxplus f)}\left(\ol{\Delta}_* \O_M \boxtimes T, T' \boxtimes \ol{\Delta}_* \RGamma_Z \omega_M\right)  \\
  &=\RHom_{\PreMF^\kinfty(M^3,f\boxplus-f \boxplus f)}\left(\ol{\Delta_1}_*(\O_M \boxtimes T),\ol{\Delta_2}_*(T' \boxtimes \ol{\Delta}_* \RGamma_Z \omega_M)\right)  \\
  &=\RHom_{\QCsh(M \times M_0)}\left(\O_M \boxtimes T,\ol{\Delta_1}^!\ol{\Delta_2}_*(T' \boxtimes \ol{\Delta}_* \RGamma_Z \omega_M)\right)  \\
  &=\RHom_{\QCsh(M \times M_0)}\left(\O_M \boxtimes T, (D_2)_* (D_1)^! (T' \boxtimes \ol{\Delta}_* \RGamma_Z \omega_M)\right)  \\
  &=\RHom_{\QCsh(M_0)}\left((D_2)^*(\O_M \boxtimes T), (D_1)^! (T' \boxtimes \ol{\Delta}_* \RGamma_Z \omega_M)\right)  \\
  &=\RHom_{\QCsh(M_0)}\left(\O_{M_0} \otimes T, T' \shotimes \RGamma_Z \omega_{M_0}\right)  \\
  &=\RHom_{\QCsh(M_0)}\left(T, T'\right) = \RHom_{\PreMF^\kinfty(M,f)}(T, T')
\end{align*} 
To obtain a $\DCoh(\bB)$-linear identification we apply an analogous argument for all $\bullet$, obtaining natural $k$-linear identifications
\[ \RHom_{\PreMF^\kinfty(M,f)}\left( V_1 \otimes \cdots \otimes V_\bullet \otimes T, \Phi'_{\ol{\omega_{\Delta,Z}}}(T')\right)  = \RHom_{\PreMF^\kinfty(M,f)}\left(V_1 \otimes \cdots \otimes V_\bullet \otimes T, T'\right) \] for $V_1, \ldots, V_\bullet \in \DCoh(\bB)$ and $T, T' \in \PreMF(M,f)$.
\item Follows from (iv). \qedhere
  \end{enumerate}
\end{proof}

\begin{remark} Note that in (iv), the $\Hom$ in the formula for $\ev$ is taking place in a category \emph{without} support conditions.  Note also that, owing to the application of $\RGamma_{Z}$ in obtaining the identity functor, the identity functor on $\PreMF$ will \emph{almost never} be compact---that is, $\PreMF$ is almost never smooth over $k\ps{\bt}$.
\end{remark}

\begin{remark} The above proof of (iv) is somewhat opaque, due to the attempt to isolate and minimize the use of operations on $\infty$-categories.  The discussion of \autoref{sec:mf-gps} allows for an argument via identifying the $S^1$-action, based on viewing $\ol{\omega_{\Delta,Z}}$ as a lift of $\Delta_* \RGamma_Z \omega_{M}$ to $S^1$-invariants for a certain action on the category of endofunctors on $\QCsh_Z(M)$.  Meanwhile, \autoref{sec:more-general} contains an alternate argument based on a description of the above equivalence via shriek integral transform functors on the simplicial diagram $M_0 \times \bB^\bullet \times M_0$.
\end{remark}

\begin{remark} In the previous Theorem we have written down \emph{an} equivalence: roughly, the one corresponding to Grothendieck duality $\DD(-) = \RHom(-, \omega_{M_0})$ using the dualizing complex $\omega_{M_0}$ on $M_0$.  Working from the viewpoint of literal matrix factorizations, it seems more natural to write down a \emph{different} equivalence: roughly the one corresponding to Grothendieck duality $\DD'(-) = \RHom(-, \omega_{M_0/M})$ using the (trivialized by $f$, in degree $-1$) relative dualizing complex $\omega_{M_0/M}$ on $M_0$.  For instance, it is this other equivalence that is written down by Lin and Pomerleano.  
\end{remark}
\begin{na} \emph{Warning:} The two equivalences give rise to \emph{different} explicit identification of the trace and identity functors.
\end{na}

\part{Circle actions, etc.}
\section{Completion via derived Cech nerve and derived groups}\label{sec:more-general}
In this section, we put the constructions of \autoref{sec:mf-gen} and \autoref{sec:main-thm} into a more general context and use this to give what we feel are better statements and proofs.  Unfortunately, making precise some parts of this requires more $(\infty,2)$-categorical preliminaries on the relations of $f^!$ and $f_*$ on $\QCsh$ than we wish to get into here.   Consequently, we will only sketch these proofs (being cavalier about these compatibilities) and will further defer these sketches to their own subsection.

\subsection{Motivation}
The starting point for this section is the following re-interpretation of \autoref{cor:premf-supt}, using the identification $\DCoh_{M_0}(M) = \DCoh(\oh{M_0})$ (c.f., \autoref{thm:cpltn}):
\begin{na} Associated to the natural inclusion $i\colon M_0 \to \oh{M_0} = \oh{M_{M_0}}$, is a map $\ol{i}$ from its Cech nerve. Since $\oh{M_0} \to M$ is a monomorphism, this identifies with
  \[ \oh{M_0} \stackrel{\ol{i}}\longleftarrow \left\{ M_0^{\times_M \bullet} \right\} \isom \left\{ M_0 \times \bB^{\times \bullet-1}\right\} \] 
  The realization (say in \'etale sheaves) of the last simplicial diagram is the definition of $M_0/\bB$, and we have constructed a map $\ol{i}\colon M_0/\bB \to \oh{M_{M_0}}$.  At the level of $R$-points
  \begin{itemize}
    \item $M_0(R)$ consists of an $m \in M(R)$ together with a factorization through $f=0$, while $\bB(R)$ acts transitively on these factorizations. So, $(M_0/\bB)(R)$ consists of those $R$-points in $M(R)$ which, \'etale locally, admit a factorization through $f=0$.
    \item Meanwhile, $\oh{M_0}(R)$ consists of the $R$-points in $M(R)$ which, \'etale locally, admit a factorization through $f^n=0$ for some $n$.
  \end{itemize}
\end{na}

\begin{na} Using \autoref{thm:cpltn}, identify $\QCsh_{M_0}(M) = \QCsh(M_0)$.  Applying \autoref{prop:fmk-coh-surj}
  \begin{align*}
    \QCsh(M_0/\bB) &= \holim^{\Pr^R} \left\{ \QCsh(M_0 \times \bB^{\times\bullet-1}), f^! \right\} \\
    &= \hocolim^{\Pr^L} \left\{\QCsh(M_0 \times \bB^{\times\bullet-1}), f_* \right\} \\
    &= \hocolim^{\Pr^L} \left\{ \QCsh(M_0) \otimes \QCsh(\bB)^{\otimes\bullet-1} \otimes \QCsh(\pt), f_*\right\} \\
    &= \QCsh(M_0) \ohotimes_{\QCsh(\bB)} \QCsh(\pt) 
  \end{align*} 
So, \autoref{cor:premf-supt} may be re-interpreted as saying that the inclusion $\ol{i}$ induces an equivalence on $\QCsh$.  The approach of this section will be to give a direct proof of this sort of statement.
 \end{na}

\subsection{Support conditions, completion, and (derived) Cech nerves}
\begin{na}Recall our notation $\QCsh(X) \eqdef \Ind\DCoh(X)$; the notation is suggested by the fact that for a morphism $f\colon  X \to Y$ the natural notion of pullback $\Ind\DCoh(X) \to \Ind\DCoh(Y)$ does not come from the pullback $f^*$ of quasi-coherent complexes, but from the \emph{shriek pullback} $f^!$ of Grothendieck duality theory.  This tells us that $\QCsh(\XZ)$ must be defined as, roughly, a sequence of sheaves on nilthickenings of $Z$ related by \emph{shriek} pullback:
\end{na}

\begin{defn} Suppose $\X \in \Fun(\DRngfp, \Sp)$ is a derived space over $k$.  Define
  \[ \QCsh(\X) \eqdef \holim_{\Spec A \to \X} (\Ind \DCoh(\Spec A), f^!) \]
  A natural transformation $f\colon  \X \to \Y$ of functors gives rise to a colimit-preserving functor $f^!\colon \QCsh(\Y) \to \QCsh(\X)$ by restricting the test diagram along $f$.  We will see later (\autoref{lem:mono}) how to define a colimit-preserving $f_*\colon \QCsh(\X) \to \QCsh(\Y)$ using base-change, and that $(f_*, f^!)$ is an adjoint pair if $f$ is (representable and) finite or close to it.  \autoref{app:descent} shows that this definition is sheaf for the smooth topology, and that it coincides with $\Ind \DCoh(\X)$ for $\X$ a \eqref{cond:starf} derived DM-stack.
\end{defn}

\begin{defn} Suppose $\X$ is a derived stack and $Z \subset \X$ is the complement of an open substack.  Define $\sXZ$ to be the sub-functor of $\X$ given by
  \[ \sXZ(R) = \left\{ t\in \X(R) \colon \text{$t$ factors set-theoretically through $Z$, i.e., $\Spec(\pi_0 R) = t^{-1}(Z)$} \right\} \]
\end{defn}

\begin{na} Suppose $\X$ is a locally Noetherian discrete stack, and $\I_\Z$ a defining ideal for $Z$.  Then, the above definition agrees with the usual one when restricted to discrete $R$: Note that $\Spec_\X \O_\X/\I_\Z^n \to \X$ is a monomorphism on discrete rings, with
  \[ (\Spec_\X \O_\X/\I_\Z^n)(R) = \left\{ t \in \X(R) \colon \I_\Z^n \cdot R = 0 \right\}, \text{and} \]
  \[ \left(\dlim\Spec_\X \O_\X/\I_Z^n \right)(R) = \left\{ t \in \X(R) \colon \text{$\I_\Z \cdot R$ is nilpotent on $\Spec R$} \right\} \] Since $\X$ was locally Noetherian, $\I_\Z$ is coherent so that $\I_\Z \cdot R$ is nilpotent iff it is contained in the nilradical of $R$ (i.e., $t$ factors set-theoretically through $Z$).
\end{na}

\begin{na} In the derived setting, a similar directed colimit description is possible \emph{locally} (e.g., when there is an ample family of line bundles) using suitable Koszul complexes (c.f., the proofs of \autoref{lem:dcohz-supt} and \autoref{lem:BGa-act-mf}).  But now there is also a global way to understand completions via a Cech-nerve construction, in the style of the Adams spectral sequence:
  \end{na}
\begin{constr}\label{constr:adams} Suppose $\X$ is a derived space, $Z \subset \X$ a closed subset (i.e., compatible family of closed subsets of $\Spec A$ for all $\Spec A \to \X$), and $p: \Z \to \X$ a finite map having support $Z$ (e.g., if $\X$ is a Noetherian derived stack one can take the discrete stack $\Z = Z_{\red}$).  Form the Cech nerve of $p$
  \[ p_\bullet \colon \left\{ \Z_\bullet = \Z^{\times_{\X} \bullet+1} \right\} \longrightarrow X \] Note that $|p_\bullet|$ factors through the monomorphism $i: \sXZ \to \X$, and let $\ol{p}: \{\Z_\bullet\} \to \sXZ$ be the factorization.

  Note that the structure maps in this augmented simplicial diagram are all finite.  Note that even if $\Z$ and $\X$ were discrete, the other terms in the Cech nerve will generally not be.  As defined above, $\QCsh$ takes colimits of derived spaces to limits of categories so that \[ \QCsh\left(|\Z_\bullet|\right) \isom \Tot\left\{\QCsh(\Z_\bullet), i^!\right\} = \Tot\left\{\ssetlar{\QCsh(\Z)}{(p_1)^!}{(p_2)^!}{\QCsh(\Z \times_X \Z)}\cdots\right\} \] \end{constr}

\begin{theorem}\label{thm:adams} With notation as in \autoref{constr:adams}, there are adjoint pairs
  \[ \adjunct{\ol{p}_* \colon \QCsh\left(|\Z_\bullet|\right)}{\QCsh(\sXZ) \colon \ol{p}^!} \]
  \[ \adjunct{|p_\bullet|_* \colon \QCsh\left(|\Z_\bullet|\right)}{\QCsh(\X) \colon |p_\bullet|^!} \]
such that
  \begin{enumerate}
    \item The adjoint pair $\left(\ol{p}_*, \ol{p}^!\right)$ consists of mutually inverse equivalences $\QCsh(|\Z_\bullet|) \isom \QCsh(\sXZ)$.
    \item The adjoint pair $\left(|p_\bullet|_*, |p_\bullet|^!\right)$ identifies $\QCsh(|\Z_\bullet|)$ with $\QCsh_Z(\X)$.  More precisely, $(p_\bullet)_*$ is fully faithful with essential image $\QCsh_Z(\X)$ and $(p_\bullet)_* (p_\bullet)^! \isom \ul{\RGamma}_Z$.
  \end{enumerate}
\end{theorem}

As a consequence, we obtain the following result which is morally important for us:
\begin{theorem}\label{thm:cpltn} 
Suppose that $\X$ is a coherent derived stack, that $Z$ is the complement of a quasi-compact open substack, and let $i\colon \sXZ \to \X$ be the inclusion.  With the above definition, $i^!\colon \QCsh(X) \to \QCsh(\sXZ)$ restricts to an \emph{equivalence} $i^!\colon \Ind\DCoh_Z(\X) = \QCsh_Z(\X) \to \QCsh(\sXZ)$ with inverse $i_*$.  This equivalence identifies the adjoint pairs $((i_Z)_*, \RGamma_Z)$ and $(i_*, i^!)$:
  \[ \xymatrix@C=6pc{\QCsh_Z(\X)  \ar@<.2pc>[r]^{(i_Z)_*} \ar@{<->}[d]_{\sim} & \QCsh(\X) \ar@<.2pc>[l]^{\RGamma_Z} \ar@{=}[d] \\
        \QCsh(\sXZ)\ar@<.2pc>[r]^{i_*}  &\ar@<.2pc>[l]^{i^!}   \QCsh(\X)}\] 
(In case $\X$ is smooth, it turns out that $i^*\colon \QC(\sXZ) \to \QC_Z(\X)$ is also an equivalence.  However, its inverse is substantially more complicated than $i_*$.)
\end{theorem}
\begin{proof}[Sketch] The only thing new beyond \autoref{thm:adams} is that $\QCsh_Z(\X) = \Ind\DCoh_Z(\X)$, which is the content of \autoref{lem:dcohz-supt}.  For moral comfort, we sketch a less derived-looking argument (independent of derived Cech nerve) in case $X$ is a locally Noetherian discrete stack:\footnote{Locally on $\X$, a similar argument can be made in the derived setting by replacing powers of $\I_Z$ by a suitable filtered diagram of Koszul-type complexes.}
  Let $\I_Z$ be an ideal of definition for $Z$, and $X_n = \Spec_X \O_X/\I_Z^n$ for all $n \geq 1$.  Consider the diagram 
  \[ X_1 \stackrel{t_1}\longrightarrow X_2 \stackrel{t_2}\longrightarrow X_3 \stackrel{t_3} \longrightarrow \cdots \longrightarrow \XZ \stackrel{i}\longrightarrow X \] where each of the $t_i$ is proper, so that $( (t_i)_*, (t_i)^!)$ is an adjoint pair.  Observe that by definition, together with the previous adjunction,
  \begin{align*}
    \QCsh(\XZ) &= \ilim^{\Cat_\infty} \left\{ \QCsh(X_1) \stackrel{(t_1)^!}{\longleftarrow} \QCsh(X_2) \stackrel{(t_2)^!}{\longleftarrow} \cdots \stackrel{(t_{n-1})^!}{\longleftarrow} \QCsh(X_n) \stackrel{(t_n)^!}{\longleftarrow} \cdots \right\} \\
&= \ilim^{\Pr^R} \left\{ \QCsh(X_1) \stackrel{(t_1)^!}{\longleftarrow} \QCsh(X_2) \stackrel{(t_2)^!}{\longleftarrow} \cdots \stackrel{(t_{n-1})^!}{\longleftarrow} \QCsh(X_n) \stackrel{(t_n)^!}{\longleftarrow} \cdots \right\} \\
 &= \dlim^{\Pr^L} \left\{ \QCsh(X_1) \stackrel{(t_1)_*}{\longrightarrow} \QCsh(X_2) \stackrel{(t_2)_*}{\longrightarrow} \cdots \stackrel{(t_{n-1})_*}{\longrightarrow} \QCsh(X_n) \stackrel{(t_n)_*}{\longrightarrow} \cdots \right\} \\
 \intertext{Since $t_n$ is proper, the functor $(t_n)_*\colon \QCsh(X_n) \to \QCsh(X_{n+1})$ will preserve compact objects. 
 We have the following recipe for forming a colimit in $\Pr^L$ of compactly generated categories along left-adjoints preserving compact objects: Take $\Ind$ of the colimit of categories of compact objects. In this case, this identifies the previous displayed line as}
&= \dlim^{\Pr^L} \left\{ \QCsh(X_1) \stackrel{(t_1)_*}{\longrightarrow} \QCsh(X_2) \stackrel{(t_2)_*}{\longrightarrow} \cdots \stackrel{(t_{n-1})_*}{\longrightarrow} \QCsh(X_n) \stackrel{(t_n)_*}{\longrightarrow} \cdots \right\}  \\
&= \Ind\left( \dlim^{\Cat_\infty} \left\{ \DCoh(X_1) \stackrel{(t_1)_*}{\longrightarrow} \DCoh(X_2) \stackrel{(t_2)_*}{\longrightarrow} \cdots \stackrel{(t_{n-1})_*}{\longrightarrow} \DCoh(X_n) \stackrel{(t_n)_*}{\longrightarrow} \cdots \right\}\right)\\
\intertext{One can show, essentially by computing local cohomology, that the natural functor $\dlim \DCoh(X_n) \to \DCoh_Z(X)$ is fully-faithful; it is essentially surjective by \autoref{lem:coh-red}.  Combining with \autoref{lem:dcohz-supt}, we identify the previous displayed line with}
  &= \Ind(\DCoh_Z X) = \QCsh_Z(X).\qedhere
\end{align*}
\end{proof}

\begin{remark} Passing to compact objects, one obtains $\DCoh_Z(X) = \DCoh(\XZ)$ where now $\DCoh(\XZ) \eqdef \QCsh(\XZ)^c$ are what one might normally call the \emph{torsion} coherent complexes.
\end{remark}

\subsection{Geometric Koszul duality for $\QCsh$}
\begin{na}\label{na:very-good} For the duration of this section:
  \begin{itemize} 
    \item We work over a base $S$, which is assumed to be a smooth stack over $k$ that is \demph{very good} in the sense that the conclusion of \autoref{prop:fmk-coh-surj} and \autoref{thm:coh-fmk} holds over $S$ (where in interpreting \autoref{thm:coh-fmk} we must work relative to $S$, i.e., the ``dualizing complex'' of $f\colon X \to S$ is $f^!(\O_S)$).  In particular, when we write $\pt$ we mean $S$.
    \item $\Y$ be (the functor of points of) a smooth formal $S$-scheme, $\pt \in \Y$. Let $\oh{\pt}$ denote the formal completion of $\pt \in \Y$, i.e., $\Spf \oh{\O_{\Y,y}}$.  Note that $\oh{\pt}$ is also a formal scheme, and the map $\oh{\pt} \to \Y$ is an inclusion of connected components on functors of points.
    \item $\GG \eqdef \pt \times_{\Y} \pt$ viewed as a derived (formal) group scheme by ``composition of loops'' as in \autoref{ssec:Bact}.  If $\Y$ is itself a (commutative) group formal scheme, then $\GG$ may be equipped with a compatible (commutative) group structure by ``pointwise addition'' of loops (also as in \autoref{ssec:Bact}).  We will mostly ignore ``pointwise addition'' in this section, but one can put it back in to obtain $\QCsh(\GG)^{\otimes}$-linear statements as follows: After one writes down the relevant functors using pointwise addition (instead of loop composition), using the added $\GG$-equivariance coming from using the commutative product, it suffices to check that the underlying $k$-linear functors are equivalence;  an Eckmann-Hilton argument show that this $k$-linear functor is homotopic to that gotten by using composition of loops, thereby reducing to the case considered in this subsection.
    \item $\X, \X' \in \Fun(\DRngfp, \Sp)$ be derived spaces, equipped with natural transformations $f\colon  \X \to Y$ and $g\colon \X' \to Y$ which are relative \eqref{cond:starf} derived DM stacks.
  \end{itemize}
\end{na}

\begin{constr} Imitating \autoref{ssec:Bact} we observe
  \begin{itemize}
    \item $\X_\pt = \X \times_\Y \pt$ and $\X'_\pt = \X' \times_\Y \pt$ are right $\GG$-schemes, via ``composition of loops.''
    \item $\ol{\X}_\pt, \ol{\X'}_\pt$ are the left $\GG$-stacks obtained from $\X_\pt, \X'_\pt$ using the inverse (``read loop backwards'') $i\colon \GG^\op \isom \GG$.
    \item $\lsub{\pt}{\X} = \pt \times_\Y \X$, and $\lsub{\pt}{\X'} = \pt \times_\Y \X'$ are left $\GG$-stacks.
    \item $\lsub{\pt}{\ol{\X}}, \lsub{\pt}{\ol{\X'}}$ are the right $\GG$-stacks obtained from $\lsub{\pt}{\X}, \lsub{\pt}{\X'}$ using the inverse $i\colon \GG^\op \isom \GG$.
    \item There are obvious ($\GG$-equivariant) equivalences $\X_\pt \isom \lsub{\pt}{\ol{\X}}$, $\lsub{\pt}{\X'} \isom \ol{\X'}_\pt$, etc.
  \end{itemize}
\end{constr}

We now isolate a key part of the proofs of \autoref{thm:TS} and \autoref{thm:functors}:
\begin{constr}  Consider the ``Koszul duality'' map of derived spaces over $\Y^2$:\footnote{On $R$-points, $\oh{\pt}(R)$ is the union of connected components of $\Y(R)$ consisting of maps such that, \'etale locally, the reduced pair $(\Spec \pi_0 R)_{\red}$ admits a factorization through $\pt \to Y$; $B\Omega_{\pt} \Y(R)$ is the union of connected components of $\Y(R)$ consisting of maps which \'etale locally themselves admit a factorization through $\pt \to \Y$.}
  \[ B\Omega_{\pt} \Y = \pt\quot\GG = \left| \ssetll{\pt}{\GG}{\GG^2} \right| \to \oh{\pt} \] 
Base-changing, we obtain an augmented simplicial diagram
  \[ \left\{\X \times_\Y \GG^{\bullet} \times_\Y \X'\right\}  \longrightarrow \X \times_\Y \oh{\pt} \times_\Y \X' \left( \hookrightarrow  \X \times_\Y \Y \times_\Y \X' \right) \] given heuristically on functor-of-points by
  \begin{align*}& \left(x \in \X(R), x' \in \X'(R), [h_f\colon  f(x) \to \pt], [h_1\colon \pt \to \pt], \ldots, [h_\bullet\colon \pt \to \pt], [h_g\colon  \pt \to g(x')]\right) \\ &\longmapsto \left(x \in \X(R), y \in \Y(R), [h_f \cdot h_1 \cdot \cdots \cdot h_\bullet \cdot h_g\colon  f(x) \to g(x')]\right) 
  \end{align*}
Taking geometric-realization, this gives a map
\[ i\colon \X_\pt \stackrel{\GG}{\times} \ol{\X'}_\pt \longrightarrow  \X \times_\Y \oh{\pt} \times_\Y \X'  = \oh{\X_\pt \times \lsub{\pt}{\X}} \]
\end{constr}

There is the following ``tensor product theorem,'' which roughly asserts that although $\pt\quot\GG \to \oh{\pt}$ is not an equivalence, it is universally an equivalence on $\QCsh$:
\begin{theorem}\label{thm:more-gen-1}
  There is a commuting diagram of equivalences
  \[ \xymatrix{ \QCsh\left(\X_\pt \stackrel{\GG}{\times} \ar@<1ex>[r]^-{i_*} \lsub{\pt}{\X'}\right) & \QCsh\left(\X \times_\Y \oh{\pt} \times_\Y \X'\right) \ar@{}[l]|-{\sim} \ar@<1ex>[l]^-{i^!} \ar@<1ex>[d]^-{i_*} \\
  \ar[u]_-{\sim}^-{\boxtimes} \QCsh(X_\pt) \ohotimes_{\QCsh(\GG)} \QCsh(\lsub{\pt}{\X'}) \ar[r]^-{\boxtimes}_-\sim & \QCsh_{\X_\pt \times \lsub{\pt}{\X'}}\left( \X \times_\Y \X'\right) \ar@<1ex>[u]^-{i^!} \ar@{}[u]|-{\sim} }
      \]
 \end{theorem}
 \begin{proof} 
   It suffices to prove that the functors in the left column, right column, and top row are equivalences.  The right column follows by \autoref{thm:cpltn}.  The top row follows from \autoref{thm:adams}, since the simplicial object for $\X_\pt \stackrel{\GG}\times \lsub{\pt}{\X'}$ is nothing but the Cech nerve for $\X_\pt \times \lsub{\pt}{\X'} \to \X \times_\Y \X'$.  It remains to handle the left-column: Applying $\QCsh$ to the simplicial diagram defining $\X_\pt \stackrel{\GG}\times \lsub{\pt}{\X'}$, and using that the structure maps are finite so that $(f_*, f^!)$ is an adjoint pair, we find
\begin{align*} \QCsh\left( \X_\pt \stackrel{\GG}\times \ol{\X'}_\pt \right) &= \Tot \left( \QCsh(\X_\pt \times \GG^\bullet \times \lsub{\pt}{\X'}), f^! \right) \\
  &= \left| \QCsh(\X_\pt \times \GG^\bullet \times \lsub{\pt}{\X'}), f_*\right|^{\Pr^L} \\
  &= \QCsh(\X_\pt) \ohotimes_{\QCsh(\GG)} \QCsh(\lsub{\pt}{\X'}) 
\end{align*}
where the last equality is computing the relative tensor product by a bar-construction.
\end{proof}

\begin{remark}
Suitably interpreted, a version of the previous Theorem is true more generally (e.g., replacing $\pt \to \Y$ by an lci map $i\colon Z \to M$):
\[ \QCsh\left(Z \quot^{\tiny Z} \oh{(Z \times_M Z)_Z}\right)\isom \QCsh\left( \oh{M_Z} \right) \]
The case of $i$ a regular closed immersion can be deduced from the above.  The case of $i$ smooth is the equivalence of $D$-modules via the de Rham stack and $D$-modules as crystals.  We will return to this in \cite{note-completion}
\end{remark}

Once this is done, we can deduce an identification of functor categories (which is perhaps more clear than \autoref{thm:functors}; e.g., it makes identifying the identity functor more straightforward):
\begin{theorem}\label{thm:more-gen-2} The categories of \autoref{thm:more-gen-1} are all equivalent to  \[ \Fun^L_{\QCsh(\GG)}\left(\QCsh(\X'_{\pt}), \QCsh(\X_\pt)\right)  \] via a cosimplicial diagram of shriek integral transform functors
  \[ \lsup{\bullet}\Phish\colon \QCsh\left(\X_\pt \times \GG^\bullet \times \lsub{\pt}{\X'}\right)  \longrightarrow \Fun_k^L\left(\QCsh(\X'_{\pt} \times \GG^\bullet), \QCsh(\X_\pt)\right) \]
\end{theorem}
\begin{proof}
To see this, we use \autoref{thm:coh-fmk} and the explicit cobar resolution of the functor category:
      \begin{align*} 
	\QCsh\left( \X_\pt \stackrel{\GG}\times \lsub{\pt}{\X'}\right) &\isom \Tot\left\{\QCsh\left(\X_\pt \times \GG^{\bullet} \times \lsub{\pt}{\X'}\right)\right\} \\
	&\stackrel{\lsup{\bullet}{\Phish}}\longrightarrow \Tot\left\{ \Fun^L\left(\QCsh(\X'_\pt \times \GG^\bullet ), \QCsh(\X_\pt)\right) \right\}  \\
	&\stackrel{\Fun^L(\boxtimes,-)}{\longrightarrow} \Tot\left\{ \Fun^L\left(\QCsh(\X'_\pt) \ohotimes \QCsh(\GG)^{\ohotimes \bullet}, \QCsh(\X_\pt)\right) \right\}  \\
	&= \Fun^L_{\QCsh(\GG)}(\QCsh(\X'_\pt), \QCsh(\X_\pt))
      \end{align*}
      We must verify that the various 
      \[ \lsup{\bullet}\Phish\colon \QCsh\left(\GG^{\bullet} \times X_\pt \times \lsub{\pt}{X'}\right)  \longrightarrow \Fun^L\left(\QCsh(\GG^\bullet \times \lsub{\pt}{X}), \QCsh(\lsub{\pt}{X'})\right) \] 
      commute with the cosimplicial structure maps.  The first instance of this verification is the following: Let 
      \[ \sq{\alpha}, \sq{\alpha}'\colon \GG \times X_\pt \times \lsub{\pt}{X'} \longrightarrow X_\pt \times \lsub{\pt}{X'} \] be given by
      \[ \sq{\alpha}(g,x,x') = (x g, x') = (g^{-1} x,x') \qquad \text{and} \qquad \sq{\alpha}'(g,x,x') = (x, g x') = (x, x' g^{-1}) \] Then, there are natural equivalences
      \[ \Phish_{(\sq{\alpha}')^! \K}(V \boxtimes \F) \isom \Phish_{\K}(V \otimes \F) \qquad \text{and}\qquad \Phish_{\sq{\alpha}^! \K}(V \boxtimes \F) = V \otimes \Phish_{\K}(\F) \]
      The verification is routine using projection, base-change, etc.
\end{proof}

\subsection{Graded speculation}
\begin{defn} For the remainder of this section, $S$ will be a smooth stack having the \emph{very good} property (\autoref{na:very-good}), and $\L$ will be a fixed line bundle on $S$.  A \demph{relative LG pair over $S$} is a pair $(M/S,  f)$ consisting of a smooth $S$-scheme $M$, and a map of $S$-schemes $f\colon  M \to \L$.
\end{defn}

\begin{remark}  Note that the line bundle $\L$ is defined on the base $S$, not on $M$.  This restriction can be lifted by working in a yet more general context: instead of obtaining a formal group $\GG$ below, one obtains a formal groupoid over $M_0$.  We will return to descriptions of this type in \cite{note-completion}.  The difficulty seems to be in finding an analogue of ``inverting $\bt$'' in order to e.g., produce a smooth dg-category.
\end{remark}

\begin{na} The motivating example ``is'' $S = B\GG_m$ and $\L = \O_{B\GG_m}(d)$ for some $d$.  However, it is \emph{not clear} that $B\GG_m$ is very good with our definition of $\QCsh$.  The tensor product property holds over $B\GG_m$ by \autoref{lem:very-good}, however \autoref{thm:coh-fmk} (and thus very-goodness) also requires a good theory of Grothendieck duality over $B\GG_m$.  We do not pursue this further here, thus ``speculations'' in the title.
\end{na}

\begin{example}
  A graded LG model may be regarded as a relative LG model over $B\GG_m$, as follows:  Suppose $(M, f)$ is an LG pair, that $\GG_m$ acts on $M$, and that $f$ is $\GG_m$-equivariant when we equip $\AA^1$ with its weight $d$-action of $\GG_m$.  Then, 
  \[ \xymatrix{ M\quot\GG_m \ar[r]^-f \ar[d] & \L = \O_{B\GG_m}(d) \ar[dl]\\ S = B\GG_m } \] is a relative LG model over $B\GG_m$.
\end{example}

\begin{lemma}\label{lem:very-good} Suppose $S$ is a very good stack, and that $S' \to S$ is a smooth $S$-stack with smooth affine diagonal.  Then, the conclusions of the conclusion of \autoref{prop:fmk-coh-surj} hold over $S'$.
\end{lemma}
\begin{proof} 
In what follows, unadorned products (resp., tensor products) are over $S$ (resp., $\QC(S)$).  Suppose $X, Y$ are almost finitely-presented over $S'$.  Write $X \times_{S'} Y = (X \times Y) \times_{(S')^2} S'$.  It suffices to show that
  \[ \QCsh(X \times Y) \otimes_{\QC((S')^2)} \QC(S') \longrightarrow \QCsh(X \times_{S'} Y) \] is an equivalence, as we may identify the left hand side (using that $S$ is very good) with
  \[ \left(\QCsh(X) \otimes \QCsh(Y)\right) \otimes_{\QC(S')^{\otimes 2}} \QC(S') = \QCsh(X) \otimes_{\QC(S')} \QCsh(Y). \]

Note that $S'$ is smooth over $k$, so that it is regular and excellent.  We now conclude by applying \autoref{prop:ext-smth-fmk}(i).
\end{proof}

\begin{na} In this situation, the analog of $k\ps{\bt} = \QCsh(\bB)$ acting on $\QCsh(M_0)$ is the following:
  \begin{itemize}
    \item Let $\GG = \oh{(S \times_\L S)_S}$; it is a formal derived group scheme over $S$ by ``composition of loops.''  As in the case $S = \pt$, it is also compatibly a commutative derived group scheme over $S$ by ``pointwise addition'' (using the additive structure on the fibers of $\L$).  
    \item There is a symmetric monoidal equivalence
  \[ (\QCsh(\GG), \circ) = (\QCsh_S(S \times_L S), \circ) \isom \left(\Sym_{\O_S}\left(\L^\dual[-2]\right)\mod(\QC(S)), \otimes\right) \] 
\item Let $M_S = M \times_\L S$ be the intersection of $f$ and the zero-section.  There is a right action of $\GG$ on $M_S$ by loop composition, and a compatible action by pointwise addition in the fibers.  Let $\PreMF^\kinfty(M/S, \L, f)$ be $\QCsh(M_S)$ equipped with this action of $(\QCsh(\GG), \circ)^\otimes$.
\item We can ask for an analog of $k\pl{\bt}$ that would e.g., allow us to recover graded matrix factorizations.  The author has not checked the details, but a reasonable guess seems to be $\DSing^\kinfty(\GG)$.
  \end{itemize}
\end{na}

\begin{na} Using the results of the present section, and the very goodness of $S$, one obtains the analogs of \autoref{thm:TS}, \autoref{thm:dual}, \autoref{thm:functors}:
  \begin{itemize} 
    \item {\bf Thom-Sebastiani:} Suppose $(M/S, \L, f)$, $(N/S, \L, g)$ are two relative LG models over $S$ with the same line bundle.  Then, exterior product induces a $\QCsh(\GG)$-linear equivalence \[ \PreMF^\kinfty(M/S, \L, f) \otimes_{\QCsh(\GG)} \PreMF^\kinfty(N/S, \L, g) \stackrel{\sim}\longrightarrow \PreMF_{M_S \times_S N_S}^\kinfty(M \times_S N/S, \L, f \boxplus g). \]
    \item {\bf Duality:} Suppose $(M/S, \L, f)$ is a relative LG models over $S$.  Then, Grothendieck duality induces a $\QCsh(\GG)$-linear duality \[ \PreMF^\kinfty(M/S, \L, f)^\dual = \PreMF^\kinfty(M/S, \L, -f).\]
    \item {\bf Functors:} Suppose $(M/S, \L, f)$, $(N/S, \L, g)$ are two relative LG models over $S$ with the same line bundle.  Then, there is a natural $\QCsh(\GG)$-linear equivalence \[ \Fun^L_{\QCsh(\GG)}\left(\PreMF^\kinfty(M/S, \L, f), \PreMF^\kinfty(N/S, \L, g) \right) \stackrel{\sim}\longrightarrow \PreMF_{M_S \times_S N_S}^\kinfty(M \times_S N/S, \L, -f \boxplus f).\]
  \end{itemize}
  In the case of $S = B\GG_m$, the first of these does not depend on duality and so is unconditional.  The second and third are conditional on a suitable Grothendieck duality over $B\GG_m$.
\end{na}

\subsection{Sketch proofs}
\begin{lemma}\label{lem:mono} Suppose $i\colon \X' \to \X$ is a map of derived spaces.  Then,  
\begin{enumerate}
\item There is a well-defined functor
      \[ i_*\colon \QCsh(\X') \to \QCsh(\X) \qquad (i_* \F)(\alpha\colon \Spec R \to \X) \eqdef \hocolim_{\vcenter{\xymatrix@1{\Spec R'\ar[r]^{i'}\ar[d] & \Spec R \ar[d] \\ \X'\ar[r]_{i}& \X }}} (i')_* \F(R' \to \X') \] 
together with a natural map $i_* i^! \to \id$.  
\item If $i$ is a monomorphism, then there is a natural equivalence $\id \stackrel{\sim}\to i^! i_*$;
\item Suppose that $i\colon \X' \to \X$ can be written as a colimit of $i_\alpha\colon \Z_\alpha \to \X$ with each $i_\alpha$ and each transition map finite.  Then, the map $i_* i^! \to \id$ of (i) is the counit of an adjunction $(i_*, i^!)$.
\item Suppose the hypotheses of (iii) are satisfied, that $i$ is a monomorphism, and that $i^!$ is conservative.  Then, $i^!$ and $i_*$ are mutually inverse equivalences.
  \end{enumerate}
\end{lemma}
\begin{proof}\mbox{} \begin{enumerate}
  \item The structure maps in the hocolim defining $i_*$ are defined as follows: Given an arrow $j\colon \Spec R'' \to \Spec R'$ over $\Spec R \times_{\X} \X'$.  There is a natural equivalence $\F(R'' \to \X') \to j^! \F(R' \to \X')$ giving rise to a composite
  \[ j_* \F(R'' \to \X') \to j_* j^! \F(R' \to \X') \to \F(R' \to \X') \] where the second arrow is the candidate co-unit that always exists (though is not always a counit of an adjunction).

  To prove existence of $i^* i^! \to \id$, consider
  \[(i_* i^! \F)(\Spec R \to \X) = \hocolim_{\Spec R' \to \Spec R \times_{\X} \X'} (i')_* \F(\Spec R' \to \X' \to \X) \]
  and compose with the natural arrow 
  \[ (i')_* \F(\Spec R' \to \X' \to \X) \to (i')_* (i')^! \F(\Spec R \to \X) \to  \F(\Spec R \to \X) \]
\item If $i$ is a monomorphism, then $\X' \times_\X \X' = \X'$.  So, for any $\Spec R' \to \X'$ the following diagram is Cartesian
  \[ \xymatrix{ \Spec R' \ar@{=}[r] \ar[d] & \Spec R' \ar[d] \\ \X'  \ar@{=}[r] \ar[d]& \X'\ar[d] \\ \X' \ar[r]& \X } \]
  Writing
  \[ (i^! i_* \F)(\Spec R' \to \X') = i_*(\Spec R' \to \X' \to \X) = \hocolim_{\Spec R_2' \to \Spec R' \times_{\X} \X'} (i')_* \F(R_2' \to \X')  \]
  we see that the diagram over which the hocolim is taken has a terminal object, given by $\Spec R'$ itself.  The inclusion of this terminal object induces a natural equivalence $\F(\Spec R' \to \X') \stackrel{\sim}\to (i^! i_* \F)(\Spec R' \to \X')$.
\item We first handle the case of $i$ itself finite.  Since $i$ is affine, $i_*$ takes on an especially nice form
  \[ (i_* \F)(\Spec R \to \X) = (i')_* \F(\Spec R \times_{\X} \X') \]
  since $\Spec R \times_\X \X'$ is again affine.   Since fiber products commute with colimits in a (pre-)sheaf category, we have
  \[ \X' = \X \times_\X \X' = \hocolim_{\Spec R \to \X}(\Spec R \times_\X \X') \] and so $\QCsh(\X') = \lim_{\Spec R \to \X} \QCsh(\Spec R \times_\X \X')$ and we may identify $i^!$ and $i_*$ with the limits \[ i_*\colon \QCsh(X') = \lim_{\Spec R \to \X} \QCsh(\Spec R \times_\X \X') \Leftrightarrow \lim_{\Spec R \to \X} \QCsh(\Spec R) = \QCsh(X)  \colon i^!\]
 Since $i$ is finite, so is each $i'\colon \Spec R \times_\X \X' \to \Spec R$.  Since they are adjoint at each stage of the limit via the indicated counit, the same is true of the limit.
 
 Now the general case: Note that we have
  $ \QCsh(\X') = \holim \QCsh(\Z_\alpha)$ and $i^! = \holim i_{\alpha}^!$.  Since the transition maps are finite, the above implies that we're taking a holimit of a diagram in $\Pr^R$; also by the above, the $(i_\alpha)^!$ are morphisms in $\Pr^R$, and so general non-sense tells us that the same is true of $i^!$.  So, $i^!$ admits a left adjoint which general non-sense tells us is the colimit of the left adjoints $(i_\alpha)_*$; this colimit coincides with $i_*$ by inspection.
\item If follows by (iii) that there is an adjunction $(i_*, i^!)$.  It follows from (ii) that $i_*$ is fully-faithful, and it suffices to show that it is essentially surjective.  Considering the factorization of the identity $\id_{i_* \F} = \mathrm{counit}_{i_* \F} \circ i^!(\mathrm{unit}_\F)$, we see that $i^!(\mathrm{unit}_\F)$ is an equivalence for each $\F$.  Since $i^!$ is conservative we see that the unit for the adjunction is an equivalence, completing the proof.\qedhere
\end{enumerate}
\end{proof}

\begin{lemma}\label{lem:dcohz-supt} Suppose $\X$ is a \eqref{cond:starf} derived stack, that $j\colon U\to \X$ is a quasi-compact open substack, and $Z$ its closed complement.  Set $\QCsh_Z(\X) = \ker j^*\colon \Ind\DCoh X \to \Ind \DCoh U$ and $\DCoh_Z(\X) = \QCsh_Z(\X) \cap \DCoh(\X)$.  Then, there is a natural equivalence
  \[ \Ind \DCoh_Z(\X) = \QCsh_Z(\X) \]
\end{lemma}
\begin{proof}[Sketch] If there are line bundles $\L_i$, $i=1,\ldots,k$, and sections $s_i \in \Gamma(\X, \L_i)$ such that $U = \bigcup_i D(s_i)$ then one has an explicit model for $j^!$ by inverting the sections, and one can show that $\QCsh_Z(\X)$ is generated by Koszul-type objects
  \[ \K \otimes \bigotimes_{i=1}^{k} \cone\left\{ s_i^{n_i}\colon \L_i^{\otimes -n_i} \to \O  \right\} \qquad \K \in \DCoh(\X), n_i \in \ZZ_{> 0} \]  Indeed, suppose $\F \in \QCsh_Z(\X)$, $\K \in \DCoh(\X)$ and that $\phi\colon \K \to \F$ is a map in $\Ind\DCoh(\X)$. The formula for $j^*$ as a filtered colimit under multiplication by the $s_i$, together with compactness of $\K$, implies that  there exist $n_i > 0$ such that $s_i^{n_i} \circ \phi$ is null-homotopic.  A choice of null-homotopies then gives rise to a factorization of $\phi$ through the appropriate Koszul-type object of $\DCoh_Z(\X)$.

When $\X$ is affine, or more generally has an ample family of line bundles, this is automatically satisfied. In general, $\Ind \DCoh_Z(\X) \to \QCsh_Z(\X)$ is fully-faithful, since objects in $\DCoh_Z(\X)$ are compact in $\Ind\DCoh(\X)$ and $\QCsh_Z(\X)$ is closed under colimits in $\Ind\DCoh(\X)$.  We have just proved that it is an equivalence locally, so that it suffices to verify that $\QCsh_Z(\X)$ has smooth descent.  Since the formation of $j^*$ commutes with smooth base change, it suffices to note that $\QCsh(\X)$ is a sheaf on $\X_{sm}$ (\autoref{thm:qcsh-sheaf}).
\end{proof}

\begin{proof}[Sketch proof of \autoref{thm:adams}] \mbox{}
  \begin{enumerate} \item
      First observe that $\ol{p}$ is a monomorphism: It suffices to check this on $R$-points before \'etale sheafification, where it is just the claim that $|\cosk f\colon  X \to S| \to S$ is a monomorphism for any map $f\colon  X \to S$ of spaces (c.f., \cite[Prop.~6.2.3.4]{T}). Furthermore the hypotheses of \autoref{lem:mono}(iii) are visibly satisfied (consider the simplicial diagram).  Applying \autoref{lem:mono}(iv) it suffices to to show that $\ol{p}^!$ is conservative.  Letting $\ol{p}_0: \Z \to \sXZ$, it suffices to show that $\ol{p}_0^!$ is conservative.  
      
  Since $\sXZ \to \X$ is a monomorphism, $\ol{p}_0$ is affine (indeed finite) since $\Z \to \X$ is.  Suppose $\F \in \QCsh(\sXZ)$ is non-zero; then by definition, there is some $t\colon T \to \sXZ \subset \X$ such that $t^! \F \neq 0$.  Let $t'\colon T' \to \Z$ be the base change of $t$ along the affine morphism $\ol{p}_0$, and $p'\colon T' \to T$ the corresponding base-change of $\ol{p}_0$.  Note that $(t')^! (\ol{p}_0)^! \F = (p')^! t^! \F$, so it suffices to prove that $(t')^!$ is conservative.  This follows from \autoref{lem:coh-red}, upon noting that $T' \to T$ is an equivalence on reduced parts since it is base-changed from $\Z \to \sXZ$.

\item The composite $|p_\bullet| = i \circ \ol{p}$ is again a monomorphism.  So \autoref{lem:mono}(iii) shows that $(|p_\bullet|)_*$ is fully-faithful and left-adjoint to $(|p_\bullet|)^!$.  Let $j\colon U \to \X$ be the inclusion of the open complement to $Z$. Base-change implies $j^* (|p_\bullet|)_* = 0$ so that the essential image of $(|p_\bullet|)_*$ is contained in $\ker j^! = \QCsh_Z(\X)$.  It suffices to show that the restriction of $(|p_\bullet|)^!$ to $\QCsh_Z(\X)$ is conservative.  In light of (i), it suffices to show that the restriction $\res{i^!}{\QCsh_Z(\X)} \colon \QCsh_Z(\X) \to \QCsh(\sXZ)$ is conservative.

Suppose $\F \in \QCsh_Z(\X)$ is non-zero, so that by definition there is some $t\colon T \to \X$ such that $t^! \F \neq 0$.  Let $p'\colon T' = (t^{-1}(Z))_\red \to T$ be the reduced-induced (discrete) scheme structure on $t^{-1}(Z) \subset \pi_0 T$, and note that $t \circ p'\colon T'\to \X$ factors through $\sXZ$.  Thus, it suffices to show that $(p')^!\colon \QCsh_{t^{-1}(Z)}(T) \to \QCsh(T')$ is conservative.  Since $\QCsh_{t^{-1}(Z)}(T) = \Ind \DCoh_{t^{-1}(Z)}(T)$ by \autoref{lem:dcohz-supt} in the affine case, this follows by \autoref{lem:coh-red}.\qedhere
  \end{enumerate}
\end{proof}

\section{Matrix factorizations via groups acting on categories}\label{sec:mf-gps}
In this section, we sketch Constantin Teleman's description of matrix factorizations (for a map to $\GG_m$ instead of $\AA^1$) as arising from $S^1$-actions on complexes on the total space.   We also give a variant of this replacing $S^1$ by $B\oh{\GG_a}$ which actually corresponds to map to $\AA^1$.  For simplicity we will first focus on the case where $M$ is a scheme rather than an orbifold, and later will indicate the necessary modifications.

\subsection{Hypersurfaces and \texorpdfstring{$S^1$}{S^1}-actions on coherent sheaves}\label{ssec:s1-hypersurf}
\begin{lemma}\label{lem:s1-act-mf} Suppose $M$ is a (discrete)  $k$-scheme, and $Z \subset M$ a closed subset.
  \begin{enumerate} 
    \item Suppose that $M$ is finite-type over $k$.  Then, there is an equivalence of $\infty$-groupoids 
      \[ H^0(M, \O_M)^\times = \GG_m(M)  \stackrel{\sim}{\longrightarrow} \left\{\begin{gathered} \text{$S^1$-actions on $\DCoh(M)$}\\\text{as $k$-linear $\infty$-category}\end{gathered}\right\} \]
    \item Suppose that $M$ is finite-type over $k$.  Then, there is an equivalence of $\infty$-groupoids 
      \[ H^0(\oh{Z}, \O_{\oh{Z}})^\times = \GG_m(\oh{Z}) \stackrel{\sim}\longrightarrow \left\{\begin{gathered} \text{$S^1$-actions on $\DCoh_Z(M)$}\\\text{as $k$-linear $\infty$-category}\end{gathered}\right\} \]
      \item There is an equivalence of $\infty$-groupoids
	\[ H^0(M, \O_M)^\times = \GG_m(M) \stackrel{\sim}\longrightarrow\left\{\begin{gathered} \text{$S^1$-actions on $\Perf(M)$}\\\text{as $k$-linear $\infty$-category}\end{gathered}\right\}  \]
  \end{enumerate}
\end{lemma}
\begin{proof}\mbox{}\begin{enumerate}
  \item Since $S^1$ is connected, any $S^1$-action on $\Ind \C$ must be by colimit and compact object preserving functors (since these properties are preserved under equivalence of functors).  So restriction to compact object identifies the space of $S^1$-actions on $\C$ and the space of $S^1$-actions on $\Ind \C$.
     
      We first compute the space of $S^1$-actions on $\Ind \DCoh(M)$, without support conditions:
      An $S^1$-action can be identified with the data of a loop map
  \[ S^1 \stackrel{\otimes}\longrightarrow \Aut_{k}(\QCsh(M)) \]
  and since $S^1 = B\ZZ$, this is the same as giving a double loop map
  \[ \ZZ \stackrel{\otimes^2}\longrightarrow \Aut_{\Fun^L_{\QCsh(M)}(\QCsh(M))}\left(\id_{\QCsh(M)}\right) \subset  \End_{\Fun^L_{k}(\QCsh(M))}\left(\id_{\QCsh(M)}\right) = \Omega^\infty\bHH^\bullet(\QCsh(M)) \] We know that $\bHH^\bullet(\QCsh(M))$ has no positive homotopy groups, so that this final space identifies with the (discrete) monoid \[ \pi_0 \bHH^\bullet(\QCsh(M)) = H^0(M, \O_M) \]  The monoid composition is by multiplication, so the sub-monoid of automorphisms identifies with the discrete monoid $H^0(M, \O_M)^\times = \GG_m(M)$.  Since both $\ZZ$ and $\GG_m(M)$ are discrete, $2$-fold loop maps are just ordinary (abelian) group homomorphisms.  
  We conclude that that the space of $S^1$-actions on $\DCoh(M)$ is precisely $\GG_m(M)$.
  \item As in (i), noting that the relevant right-hand side here is (the units in)
 \begin{align*} \Omega^\infty \bHH^\bullet\left(\QCsh_Z(M)\right) &= \Omega^\infty \RHom_{\QCsh_{Z^2}(M^2)}\left(\Delta_* \RGamma_Z \omega_M, \Delta_* \RGamma_Z \omega_M\right) \\
& = \Omega^\infty \RHom_{\QCsh(M^2)}\left(\Delta_* \RGamma_Z \omega_M, \Delta_* \omega_M\right) \\
&= \Omega^\infty \RHom_{\QCsh(M^2)}\left(``\dlim_n" \Delta_* \HHom(\O/\I_Z^n, \omega_M), \Delta_* \omega_M\right) \\
&=  \Omega^\infty \ilim_n \RHom_{\DCoh(M^2)}\left(\Delta_* \HHom(\O/\I_Z^n, \omega_M), \Delta_* \omega_M\right) \\
&=  \Omega^\infty \ilim_n \RHom_{\DCoh(M^2)}\left(\Delta_* \O_M, \Delta_* \O/\I_Z^n\right) \\
&=  \pi_0 \ilim_n \RHom_{\DCoh(M^2)}\left(\Delta_* \O_M, \Delta_* \O/\I_Z^n\right) \\
&=  \ilim_n \pi_0 \RHom_{\DCoh(M^2)}\left(\Delta_* \O_M, \Delta_* \O/\I_Z^n\right) \\
&=  H^0\left(\oh{Z}, \O_{\oh{Z}}\right)
\end{align*}
(While the full $\bHH^\bullet(\QCsh_Z(M))$ may be unwieldy, it still has no positive homotopy groups so that $\Omega^\infty = \pi_0$ admits a nice description.)
\item As in (i), but using the description
  \[ \Omega^\infty \bHH^\bullet(\QC(X)) = \Omega^\infty \RHom_{\QC(X)}(\Delta_* \O_X, \Delta_* \O_X) = \pi_0 \RHom_{\QC(X)}(\Delta_* \O_X, \Delta_* \O_X) = H^0(X, \O_X). \qedhere \]
\end{enumerate}
\end{proof}

\begin{defn} Suppose $M$ is a discrete finite-type $k$-scheme, $Z \subset M$ a closed subset, and $\oh{Z}$ the formal completion.  For $f \in \GG_m(M)$ (resp., $f \in\GG_m(\oh{Z})$), define $\CircMF(M,f)$ (resp., $\CircMF(\oh{Z},f)$) to be $\DCoh(M)$ (resp., $\DCoh_Z(M) = \DCoh(\oh{Z})$) equipped with the $S^1$-action of \autoref{lem:s1-act-mf}.
\end{defn}

\begin{na} The previous Lemma encoded the intuitive statement that an $S^1 = B\ZZ$ action on $\C$ is just a compatible family of automorphisms of the Hom-spaces, here given by multiplication by $f \in \GG_m(M)$.  We now move on to showing the intuitive claim that $\C^{S^1}$ consists of objects equipped with trivializations of this automorphism, with maps given by the the fixed points for the induced $S^1$-action on mapping spaces.  Since a trivialization of multiplication by $f$ is precisely a null-homotopy of $f-1$, we will show that $(\DCoh M)^{S^1} = \DCoh (M_1)$ where $M_1$ is the derived fiber of $f$ over $1 \in \GG_m$.  This can be viewed as explaining \autoref{prop:s1-hom}.
\end{na}

\begin{lemma}\label{lem:invts} Suppose a simplicial group $G$ acts on a small, idempotent complete, $\infty$-category $\C$.  Then:
  \begin{enumerate}
    \item $(\Ind \C)^G \isom \Ind(\C_{G^{op}})$, and in particular the former is compactly generated.
    \item The natural functor $i\colon (\Ind\C)^G \to \Ind(\C)$ admits a compact-object preserving left adjoint $i^L$, and a colimit-preserving right adjoint $i^R$.  Furthermore, $i$ is conservative and so induces equivalences
      \[ (\Ind\C)^G = (i \circ i^L)\mod\left(\Ind\C\right) \qquad (\Ind\C)^G = (i \circ i^R)\comod\left(\Ind\C\right) \]
    \item The natural functor $\C^G \to (\Ind\C)^G$ is fully-faithful, with essential image consisting of those objects $x \in (\Ind\C)^G$ for which $i(x) \in \C \subset \Ind(\C)$.   In particular, $\C_{G^{op}} \subset \im \C^G$.
    \item There is a natural equivalence
      \[ \C^G = (i \circ i^L)\mod(\C). \]
  \end{enumerate}
\end{lemma}
\begin{proof} \mbox{}
  \begin{enumerate}
    \item Note that $G$ acts on the presentable $\infty$-category $\Ind(\C)$ by right-adjoint maps; their left adjoints may be taken to be the inverses of the action, i.e., the action of $G^op$ on $\C$.  So $(\Ind\C)^G$ may be computed in $\Pr^R$, or equivalently as the colimit of the opposite diagram in $\Pr^L$.  This opposite diagram is just the action of $G^{op}$ on $\Ind(\C)$, and it is by colimit and compact object preserving maps; so the colimit in $\Pr^L$ may be computed by taking the colimit of the (small, idempotent complete) $\infty$-categories of compact objects and then forming $\Ind$.  Putting this together, we obtain:  
      \[ \underbrace{(\Ind\C)^G}_{\Pr^R} \isom \underbrace{\Ind(\C)_{G^op}}_{\Pr^L} \isom \Ind(\C_{G^\op}) \]
    \item Since $G$ acts by equivalences, the limit $(\Ind\C)^G$ may be computed in either $\Pr^L$ or $\Pr^R$, the natural functor $i$ is both colimit and limit preserving; since the diagram of left-adjoints is consists of compact-preserving functors, $i$ is also compact-preserving (as in the argument for (i)).  Since $\Ind(\C)$ is compactly-generated, this implies the existence of left- and right-adjoints with the indicated properties.

      The fact that $i$ is conservative follows from observing that $(\Ind\C)^G$ is the homotopy limit over a connected diagram ($BG$).  Now, Lurie's Barr-Beck Theorem implies the desired equivalences.
    \item Since $\C \to \Ind(\C)$ is fully-faithful, the same is true for any limit.  Realize the $G$-actions, and the natural functor, by a diagram \[ \xymatrix{ \sq{\C} \ar[dr] \ar[r] & \sq{\Ind(\C)} \ar[d] \\ & BG } \] where the vertical maps are Cartesian fibrations and the horizontal map preserves Cartesian simplices (and is fully-faithful since $\C \to \Ind\C$ is).  Then, $\C^G$ is explicitly given by Cartesian sections of $\sq{\C} \to BG$, while $(\Ind \C)^G$ is explicitly given by Cartesian sections of $\sq{\Ind\C} \to BG$.  Since $BG$ is connected, to check if a Cartesian section of $\sq{\Ind\C} \to BG$ lands in $\sq{\C}$ it suffices to check at the base-point.
    \item Note that the monad $i \circ i^L$ on $\Ind \C$ preserves compact objects, and so gives rise to a monad on the compact objects.  Then, this follows by combining (ii) and (iii).\qedhere
  \end{enumerate}
\end{proof}

\begin{lemma}\label{lem:kt-mf} Suppose that $M$ is a (discrete) $k$-scheme, $Z \subset M$ closed.  Set $\C = \Perf(M)$ (resp., if $M$ is finite-type, $\C = \DCoh_Z(M) = \DCoh(\oh{Z})$).  Via \autoref{lem:s1-act-mf}, a morphism $f\colon  M \to \GG_m$ (resp., $f\colon  \oh{Z} \to \GG_m$) gives rise to a natural $S^1$-action on $\C$.  The monad $i \circ i^L$ of \autoref{lem:invts} identifies with $\O_{M_1} = \O_M \otimes_{\O_{\GG_m}} k \in \Alg(\QC(M))$ (resp., $\O_{\oh{\Z}_1} = \O_{\oh{\Z}} \otimes_{\O_{\GG_m}} k \in \Alg(\QC(\oh{Z}))$) giving rise to natural equivalences
  \[ \C^{S^1} = \O_{M_1}\mod(\C) \qquad \text{and} \qquad \C_{S^1} = \left(\O_{M_1}\mod(\Ind \C)\right)^c \] (resp. $\O_{\oh{\Z}_1}$ versions).  Consequently,
\begin{itemize}
  \item If $\C = \DCoh(M)$ (resp. $\DCoh_Z(M)$), this gives $\C^{S^1} \isom \DCoh(M_1)$ (resp., $\DCoh(\oh{\Z}_1)$).  In particular, $\DCoh(M_1)$ (resp., $\DCoh(\oh{\Z}_1)$) is naturally $C^*(BS^1,k)$-linear.\footnote{This corresponds to the $k\ps{\bt}$-linear structure one gets by applying the construction of \autoref{sec:more-general} with base $\Y = \GG_m$ in place of $\GG_a$.  The formal exponential induces an equivalence \[ \exp\colon 0 \times_{\GG_a} 0 \stackrel{\sim}\to 1 \times_{\GG_m} 1 \] and so a symmetric monoidal equivalence $\DCoh(1 \times_{\GG_m} 1) \isom \DCoh(0 \times_{\GG_a} 0) \isom \Perf k\ps{\bt}$.}
  \item If $\C = \Perf(M)$, this $\C_{S^1} = \Perf(M_1)$.
\end{itemize}
\end{lemma}
\begin{proof} 
We will treat the case of $\C = \DCoh(U)$, the rest being analogous.
 
For notational convenience we reduce to the affine case: For any $U \subset M$, one may restrict $f$ to $U$; in this way, one can make $U \mapsto \DCoh(U)$ into a sheaf of $\infty$-categories with $S^1$-action.  A homotopy limit of sheaves is a still a sheaf, so the assignment $U \mapsto \DCoh(U)^{S^1}$ forms a sheaf of $C^*(BS^1, k)$-linear categories on $M$.  Since $U \mapsto \DCoh(U_1)$ is also a sheaf  on $M$, the claim that one can naturally identify the two is local.  So assume $M = \Spec R$.

The map $\ZZ \to R$ gives rise to a multiplicative $R$-line $R_f^\otimes$ on $S^1 = B\ZZ$: In more down to earth terms, we consider consider $R$ as an algebra over the group ring $k[\ZZ] = \O_{\GG_m}$.  It is then clear that \[ C_*(S^1, R_f^\otimes) = R \otimes_{k[\ZZ]} k = R_1 \]  Meanwhile, $\DCoh(R)$ gives rise to a locally constant sheaf of categories on $BS^1$ (with fiber $\C$) whose global sections are $\C^{S^1}$: To $U \subset BS^1$ it assigns a twisted form of $R$-local systems on $U$, so that in particular $\C^{S^1}$ is a twisted form of $R$-local systems on $BS^1$.   

Tracing through the constructions, the monadic machinery is doing the following: To $\M$ a twisted $R$-local system, it assigns the fiber over the base-point $\M_\pt$ acted on by the algebra $C_*(S^1, R_f^\otimes)$.  This identifies the monad $i \circ i^L$ with $C_*(S^1, R_f^\otimes) \otimes -$ equipped with the algebra  structure coming from the multiplicative structure on $R_f^\otimes$.

The ``In particular'' claims then follow from \autoref{lem:mfplus-monadic} (and a suitable variant for $\oh{\Z}$).
\end{proof}

\begin{remark} The previous two Lemmas imply a very slight refinement of the statement that $\MF(M,f)$ depends only on a formal completion of the (critical locus intersect the) zero fiber in $M$: It depends only on the $\infty$-category of coherent complexes on the completion, together with an $S^1$-action encoding the function.
\end{remark}

\begin{na} If $M = U \quot G$ is a global quotient orbifold, then $G$ acts on $\Perf(U)$ with $\Perf(U)^G = \Perf(M)$ (by faithfully flat descent for $\Perf(-)$).  It follows from \autoref{lem:s1-act-mf}, applied to $M$, that a $G$-invariant invertible function on $M$ gives rise to an action of $S^1$ on $\Perf(U)$ compatible with this $G$-action.  Thus we obtain an $S^1$-action on $\Perf(U)^G = \Perf(M)$, and applying \autoref{lem:s1-act-mf} on $G$ we see $\Perf(M)^{S^1} = \DCoh(U_1)^G = \DCoh(U_1 \quot G)$.  
  
However, even in the global quotient case there could be \emph{other} $S^1$-actions not coming from a function on the quotient.  The point is that $\bHH^0(\Perf(M))$ involves functors and so is naturally local on $M^2$, rather than $M$.  In the scheme case this went way, but in the orbifold case the inertia stack $I_M = \pi_0 LM = \pi_0(M \times_{M^2} M)$ will intervene
\end{na}

\begin{lemma}\label{lem:s1-act-orbifold} Suppose $M$ is an orbifold.  Then, there is an equivalence of $\infty$-groupoids
  \[  \Ext^0_M(\O_{I_M}, \O_M)^\times \isom \bHH^0_k(\Perf(M))^\times  \stackrel{\sim}\longrightarrow \left\{\begin{gathered}\text{$S^1$-actions on $\Perf(M)$}\\\text{as $k$-linear $\infty$-category}\end{gathered}\right\}  \]
    e.g., if $M = U\quot G$ with $U$ a smooth scheme and $G$ a finite group, then the RHS is the units (for a certain product) in
    \begin{align*} \Ext^0_M(\O_{I_M}, \O_M) &\isom \left(\bigoplus_{g \in G\atop{\codim(U^g) = 0}} g \# H^0(U^g, \pi_0 \O_{U^g}) \right)^G \\
      &\isom \bigoplus_{\text{$[g]$ conj. class in $G$}\atop{\codim(M^g) = 0}} H^0(U^g, \O_{U^g})^{Z_G([g])}
    \end{align*}
    where the (right) $G$-action on the direct sum is by $(g \# a) \cdot h = h^{-1} g h \# a h$.  (In case $M$ is disconnected, we must sum over components of $M^g$ of codimension zero.)
\end{lemma}
\begin{proof} The equivalence follows with the same proof as above, noting that
  \begin{align*} \bHH_k^0(\Perf(M)) &= \Ext^0_{\QC(M^2)}(\Delta_* \O_M, \Delta_* \O_M) \\
    &= \Ext^0_{\QC(LM)}(\Delta^* \Delta_* \O_M, \O_M) \\
    &= \Ext^0_{\QC(LM)}(\tau_{\leq 0} \Delta^* \Delta_* \O_M, \O_M) \\
    &= \Ext^0_{\QC(LM)}(\O_{I_M}, \O_M) 
  \end{align*}

The final computation follows from a suitable computation of Hochschild cohomology of an orbifold, which we sketch:  Consider the commutative (not Cartesian) diagram
  \[ \xymatrix{ U \ar[r]^-q\ar[d]_\Delta & U\quot G \ar[r]^-\pi \ar[d]^\Delta & BG \ar[d]^\Delta \\ U^2 \ar[r]_-q & U^2/G^2 \ar[r]_-\pi  &  BG^2 } \] 
A straightforward computation shows that 
\[ q^* \Delta_* \O_{U\quot G} = \bigoplus_{g \in G} g \# (\Gamma_g)_* \O_U \in \QC(U^2) \]  The fact that $\Delta_* \O_{U \quot G}$ is an algebra (indeed, the monoidal unit) for the convolution product on $\QC( (U \quot G)^2)$ manifests itself in the usual crossed product associative algebra structure $(g \# a) (g' \# a') = g g' \# a^{g'} a'$.  It is $\O_{U^2}$-linear by $(a_1 \otimes a_2) \cdot (g \# a) = g \# a_1 a (a_2)^g$. There is a right $G^2$-action on this giving descent data to $\QC(U^2/G^2)$: Locally on a $G$-invariant affine piece it is $( g \# a )^{(g_1, g_2)} = g_1^{-1} g g_2 \# a^{g_2}$ for $g, g_1, g_2 \in G$ and $a \in \O_U$ (regarding $\O_U$ as having a right $G$-action in the natural way).
Pulling back,  \[ \Delta^* q^* \Delta_* \O_{U \quot G} = \bigoplus_{g \in G} g \# \Delta^* (\Gamma_g)_* \O_U \in \QC(U) \] equipped with the diagonal of the above $G$-action as descent data to $U \quot G$.

In particular, by descent
\begin{align*} \bHH^0_k\left(\Perf(U \quot G)\right) &=\left[ \Ext^0_{U^2}\left(q^* \Delta_* \O_{U \quot G}, q^* \Delta_* \O_{U \quot G}  \right) \right]^{G^2} \\
  &= \left[ \Ext^0_{U^2}\left(\bigoplus_{g \in G} g\# (\Gamma_g)_* \O_U, \bigoplus_{g' \in G} g'\#(\Gamma_{g'})_* \O_U\right)\right]^{G^2}
\end{align*} a form which makes the product structure evident.  Any such $G^2$-equivariant self-map is determined by where it sends $\id_{G} \# 1$, and is in fact just right-multiplication by the image of $1$.  Writing this image as $\sum_{g'} g'\#\phi_{g'}$ for $\phi_{g'} \in \Gamma(\O_U)$, we see that it must satisfy various conditions such as $\phi_{g'}(a-a^{g'})=0$ for all $g' \in G$ and $a \in \Gamma(U)$.  From these one can deduce the indicated description in terms of connected components and supports of fixed sets. However, we prefer to give a more geometric description, via essentially describing all of $\bHH^\bullet$:
\begin{align*}
\bHH_k^\bullet\left(\Perf(U \quot G)\right) &=  \RHom_{U\quot G}\left(\Delta^* \Delta_* \O_{U\quot G}, \O_{U \quot G}\right)\\ 
&=\left[\RHom_{U}\left(q^* \Delta^* \Delta_* \O_{U \quot G}, q^* \O_{U \quot G}   \right)\right]^G \\
&=\left[\RHom_{U}\left(\bigoplus_{g \in G} g\# \Delta^* (\Gamma_g)_* \O_U, \O_U \right)\right]^G \\
\intertext{In the following lines, $L_g U$ denotes the \emph{derived} fixed points $\Spec_{X^2} \O_{\Delta} \Lotimes_{\O_{X^2}} \O_{\Gamma_g}$, while $U^g = \pi_0 L_g U$ denotes the ordinary closed subscheme of fixed points.}
&=\left[\bigoplus_{g \in G} \RHom_{U}\left(\O_{L_g U}, \O_U \right)\right]^G \\
\end{align*} 
Passing to $\pi_0$:
\begin{align*} \bHH^0_k\left(\Perf(U \quot G)\right) &= \left[\bigoplus_{g \in G} \Ext^0_{U}\left(\O_{L_g U}, \O_U \right)\right]^G\\
  &= \left[\bigoplus_{g \in G} \Ext^0_U(\O_{U^g}, \O_U) \right]^G\\
  &= \left[\bigoplus_{g \in G\atop{\codim(U^g)=0}} H^0(U^g, \O_{U^g}) \right]^G
\end{align*} where the final equality results from noting that for a connected (discrete) closed subscheme $Z \subset U$, $\Ext^0_U(\O_Z, \O_U) = 0$ unless $Z$ is a connected component of $U$ in which case $\Ext^0_U(\O_Z, \O_U) = H^0(Z, \O_Z)$.
\end{proof}

One can also describe the invariants for these ``exotic'' $S^1$-actions, but the description is less geometric: In the case of a global quotient, it is like a ``non-commutative'' fiber over $1 \in \GG_m$ for the crossed product algebra.
\begin{lemma}\label{lem:kt-mf-orbifold} Suppose $M$ is an orbifold and set $\C = \Perf(M)$. Via \autoref{lem:s1-act-orbifold}, an element $\alpha \in \bHH^0(\Perf(M))^\times$ gives rise to an $S^1$-action on $\C$.  The monad $i \circ i^L$ of \autoref{lem:invts} identifies with $\O_{\Delta} \otimes_{\O_{\GG_m}} k \in \Alg(\QC(M^2), \circ)$, where $\O_\Delta$ is a $k[\ZZ]$-algebra by $n \mapsto \alpha^n$ and where $\QC(M^2)$ is equipped with its convolution product and its ``star integral transforms'' action on $\C$.  So, $\C^{S^1} = (\O_{\Delta} \otimes_{\O_{\GG_m}} k)\mod(\Perf(M))$.

In case $M = U \quot G$, and $\alpha = \sum_g f_g \in (\oplus_g H^0(U^g, \O_{U^g}))^G$ (with $f_g \neq 0$ only on codimension zero components),  this admits a ``crossed product'' description
\[ \Perf(U \quot G)^{S^1} = \left[\left(\bigoplus_{g \in G} g \# (\Gamma_g)_* \O_U \right) \otimes_{\O_{\GG_m}} k\right]\mod(\Perf(U)). \]
\end{lemma}
\begin{proof} The proof is analogous to \autoref{lem:kt-mf}, noting that we did not actually need to work locally.  Instead of a multiplicative $R$-line on $S^1$, we obtain a multiplicative local system of endo-functors of $\C$ (under composition product).  Identifying $S^1 = B\ZZ$, the data of this local system is that of making $\id_{\C} \in \Fun(\C, \C)$ into a $k[\ZZ]=\O_{\GG_m}$-algebra, and the monad identifies with $C_*(B\ZZ, \id_{\C}) = \id_{\C} \otimes_{k[\ZZ]} k$ as an algebra in endofunctors.  Identifying $\Fun(\C,\C)$ with a full-subcategory of $\Fun^L(\Ind \C, \Ind \C) = \QC(M^2)$ using star integral transforms, recall that $\id_{\C}$ corresponds to the diagonal $\Delta_* \O_{M}$.  This implies the indicated description.
 
Suppose now that $M = U \quot G$.  Note that $\Perf(M) = \Perf(U)^G$ combined with \autoref{lem:invts}(iv) give $\Perf(M) = (q^! q_*)\mod\Perf(U)$.  We may identify the monad $q^! q_*$ with the ``crossed product'' algebra in endofunctors: $q^* \Delta_* \O_{U \quot G} \in \QC(U^2)$ under star convolution.  The $k[\ZZ]$-action on $\Delta_* \O_{U \quot G}$ corresponds to a $G^2$-equivariant $k[\ZZ]$-action on $q^* \Delta_* \O_{U \quot G} \isom q^! q_*$, where $n \in \ZZ$ acts by right multiplication by $(\sum_g g \# f_g)^n$.  It follows that $\Perf(M)^{S^1} = (q^! q_* \otimes_{k[\ZZ]} k)\mod(\Perf(U))$, whence the desired formula.
\end{proof}

\subsection{Hypersurfaces and \texorpdfstring{$B\oh{\GG_a}$}{BG_a}-actions on coherent sheaves}\label{ssec:BGa-hypersurf}
\begin{na} The use of $S^1$-actions gives rise to natural comparison maps $\bHH^{k\ps{\bt}}_\bullet( (\Perf M)^{S^1}) \to \bHH^k_\bullet(\Perf M)^{S^1}$, etc.  However, it imposes the constraint that we work with an invertible function $f\colon M \to \GG_m$ instead of the usual superpotential $f \colon M \to \AA^1$.  If we are willing to complete near the zero fiber, it is always possible to replace $f$ by $e^f$.  However, completing is inconvenient in cases where we wish to retain nice global properties of $M$ (e.g., smoothness of $\Perf M$) and incompatible with the graded context, so we give here an alternate approach.  It is a bit more complicated, requiring replacing the (constant) simplicial group $S^1$ by the formal group stack $B\oh{\GG_a}$ (which we think of, via Hinich, as $\Spf C^*(S^1;k)$).
\end{na}

\begin{defn} \mbox{}\begin{itemize} \item A \demph{derived Artin $k$-algebra} is an $A \in \DRng_k$ such that $\pi_* A$ is a finite-dimensional $k$-vector space, and $\pi_0 A$ is a local Artin $k$-algebra with residue field $k$.  Let $\DArt_k$ be the full-subcategory of $\DRng_k$ spanned by the derived Artin rings.  For any $A \in \DArt_k$, there is a natural $A \to k$ whose fiber will be denote $\mm_A$.
  \item A \demph{formal moduli problem} over $k$ is a functor in $\X \in \Fun(\DArt_k, \Sp)$ such that $\X(\pt) \isom \pt$ and the natural map $\X(B \times_{k \oplus k[\ell]} k) \to \fib\{\X(\Spf B) \to \X(k \oplus k[\ell])\}$ is an equivalence for all $\ell>0$ and all $B \to k \oplus k[\ell] \in \DArt_k$ (c.f., \cite[Remark~6.18]{Lurie-fmp}).  A \demph{derived formal group} $G$ is a group object in formal moduli problems, i.e., a formal moduli problem $G$ together with a factorization of its functor of points through $\sGp$.  If $G$ is a derived formal group, let $BG$ denote the universal formal moduli problem receiving a map from $A \mapsto B(G(A))$; if $G(k \oplus k[\ell])$ is connected for $\ell > 0$, then this is already a formal moduli problem.  Our motivating examples are: $\oh{\GG_a}$, whose functor of points is $\oh{\GG_a}(A) = \mm_A$ (viewed as a simplicial abelian group via Dold-Kan); and $B\oh{\GG_a}$, whose functor of points is $B\oh{\GG_a}(A) = B\mm_A$ (since it satisfies the connectivity assumption above).
  \item For a derived formal group $G$, let $\pi\colon BG \to \Aff(BG)$ be the ``fake affinization'' : i.e., it is a notational fiction for the hypothetical space with $\QC(\Aff(BG)) = \O_{\Aff(BG)}\mod$ where $\O_\Aff(BG) \eqdef \Gamma(BG, \O_{BG})$.\footnote{There will, in general, be no functor on $\DRng$ having this as its $\infty$-category of quasi-coherent complexes.}  Let $i\colon \pt \to BG$ be the inclusion, $p\colon BG \to \pt$ and $q\colon \Aff(BG) \to \pt$ the projections.  There are adjoint functors 
    \[ \adjunct{\pi^*\colon \O_{BG}\mod}{\QC(BG)\colon \pi_*} \qquad \adjunct{\pi^*\colon \dgcatidm_{\O_{BG}}}{\dgcatidm_{BG} \pi_*}\] refining $p^*$ (=trivial $G$-action) and $p_*$ (=$G$-invariants).  They are defined by noting that for any $A \in \DArt_k$ and map $t: \Spf A \to BG$, $A$ is equipped with the structure of object in $\CAlg(\O_{BG}\mod)$; so one may e.g., $t^* \pi^* \C = \C \otimes_{\O_{BG}} A$. 
  \item Suppose $G$ is a derived formal group and $\C \in \dgcatidm_k$. Then, a \demph{$G$-action on $\C$} is a small, stable, idempotent-complete category $\sq{\C}$ over $BG$ equipped with an equivalence $i^* \sq{\C}{k} \isom \C$.  Explicitly, it the datum of (suitably compatible) $A$-linear actions of $G(A) \in \sGp$ on $\C \otimes_k A$ for each $A \in \DArt_k$.  In this case, we will write $\C^G$ for $\pi_* \sq{\C} \in \dgcatidm_{\O_{BG}}$.
\end{itemize}
\end{defn}

We have the following analog of \autoref{lem:s1-act-mf}:
\begin{lemma}\label{lem:BGa-act-mf} Suppose $M$ is a (discrete) finite-type separated $k$-scheme.  There is an equivalence of (discrete) $\infty$-groupoids
\[ \Gamma(M, \O_M)\stackrel\sim\longrightarrow\left\{\begin{gathered} \text{$B\oh{\GG_a}$-actions on $\DCoh(M)$}\end{gathered}\right\}   \]
\end{lemma}
\begin{proof} Note that since $\DCoh(M)$ is smooth $\bHH_{A}^\bullet(\DCoh(M) \otimes_k A) = \bHH_k^\bullet(\DCoh(M)) \otimes_k A$.  Since $B\oh{\GG_a}(A)$ is connected for every $A$, the data of $B\oh{\GG_a}$-action on $\C = \DCoh(M)$ identifies with a compatible family of $E_2$ maps
  \[ \oh{\GG_a}(A) \stackrel{\otimes^2}\longrightarrow \Omega^\infty\fib\left\{\bHH_k^\bullet(\DCoh(M, \O_{M})) \otimes_k A \longrightarrow \bHH_k^\bullet(\DCoh(M, \O_M)) \otimes_k A\right\}  \]

  We claim that the formal moduli problem $\oh{\GG_a}$ is Kan-extended from discrete Artin $k$-algebras,  and in fact that the natural map \[ \oh{\GG_a} \longleftarrow \hocolim_n \Spf k[x]/x^n \isom \hocolim_{n} \Spf \Kos(x^n\colon k[x]\to k[x]) \]  of functors on $\DArt_k$ is an equivalence.  On discrete local Artin test rings, this is the standard fact that the augmentation ideal consists precisely of the nilpotent elements.  It thus suffices to show that the RHS is a formal moduli problem, and that the map induces an equivalence on tangent complexes.  Certainly each $\Spf \Kos(x^n)$ is a formal moduli problem, so it suffices to note that filtered colimits commute with finite limits: It follows that any filtered colimit of formal moduli problems is again a formal moduli problem.  Now it is enough to verify that the map is an equivalence on the test-rings $k \oplus k[q]$, $q \geq 0$.  For $q = 0$ this follows from the discrete case, and it suffices to check that it is an equivalence on $\pi_0$ for all $q > 0$ (i.e., that both sides are smooth).  To carry out this computation we use the description of $\Kos(x^n)$ as a coequalizer $\mathrm{coeq}\{x^n, 0\colon k[x] \to k[x]\}$ in $\DRng_k$ to show that $\pi_0 \Map(\Spf k \oplus k[q], \Spf \Kos(x^n))$ is $\pi_{-q}$ of the complex $k \stackrel{n x^{n-1}}\to k$ in degrees $0, -1$. For $q > 1$ this already vanishes for each $n$.  For $q = 1$ one checks that the map from the $n^\text{th}$ to $(n+1)^\text{st}$ term in the colimit is zero on $\pi_0$, either by explicating the map directly or by comparison to the familiar (dual) map on global cotangent complexes (which up to constant is multiplication by $x$, so restricts to zero on the base-point).

Consequently, in computing the relevant mapping space we may restrict to discrete Artin $k$-algebras.  If $A = \pi_0 A$ is a discrete Artin $k$-algebra, then \[ \fib\left\{\Omega^\infty(\bHH_k^\bullet(\DCoh(M, \O_M)) \otimes_k A) \to \Omega^\infty(\bHH_k^\bullet(\DCoh(M, \O_M)) \otimes_k k)\right\} \isom \pi_0 \bHH_k^\bullet(\DCoh(M, \O_M))) \otimes_k \mm_A \] It is equipped with the $E_2$-structure (=commutative group structure, since everything is discrete) induced from \emph{multiplication} in $\bHH^0(\DCoh(M, \O_M)) = \Gamma(M, \O_M)$, so that it may be identified (as commutative discrete formal group) with the completion of $\Gamma(M,\O_M)$ at $1$ under multiplication.  It remains to compute the space of (abelian) discrete formal group homomorphisms $\oh{\GG_a} \to \oh{\Gamma(M, \O_M)_1}$: formal Lie theory tells us that this space is discrete and in bijection with $\Gamma(M, \O_{M})$; the bijection takes $f \in \Gamma(M, \O_M)$ to the homomorphism  $t \mapsto e^{tf}$.
\end{proof}

\begin{lemma}\label{lem:BGa-kt-mf} Suppose that $M$ is a smooth $k$-scheme.
  \begin{enumerate}
    \item Both group homomorphisms in the diagram \[ B\ZZ \longrightarrow B\GG_a \longleftarrow B\oh{\GG_a} \] induce isomorphisms on $\O_{B(-)}$.  Using this, we may identify $k\ps{\bt} \isom \O_{B^2\ZZ} \isom \O_{B^2\oh{\GG_a}}$.
    \item Via \autoref{lem:BGa-act-mf}, a morphism $f\colon  M \to \AA^1$ gives rise to a $B\oh{\GG_a}$-action on $\C = \DCoh(M)$.  Under the identification of (i), there is a natural $k\ps{\bt}$-linear equivalence
      \[ \C^{B\oh{\GG_a}} = \O_{M_0}\mod(\DCoh(M)) = \DCoh(M_0) \]
      where the $k\ps{\bt}$-linear structure on the left is as $\O_{B^2\oh{\GG_a}}$ and on the right is as in \autoref{ssec:Bact}. 
\end{enumerate}
\end{lemma}
\begin{proof} \mbox{}
  \begin{enumerate}
    \item It is a standard fact that both induces isomorphisms on $\QC(-)$, since all three can be identified with $C^*(S^1, k)\comod \isom \O_\bB\mod$.
    \item Analogous to \autoref{lem:kt-mf}: The functor $\C^{B\oh{\GG_a}} \to (\Ind \C)^{B\oh{\GG_a}}$ is again fully-faithful with essential image detected on the test object $\Spf k$ (i.e., underlying category). Now, the monad $i \circ i^L$ associated to $i \colon (\Ind \C)^{B\oh{\GG_a}} \to \Ind \C$ identifies (in the affine case for notational convenience) with the algebra $R \otimes_{k[x]} k$ where $R$ is made into a $k[x]$-algebra by the map to $\AA^1$. \qedhere
  \end{enumerate}
\end{proof}

And the orbifold variants, which are exactly analogous to what we have done above.
\begin{lemma}\label{lem:BGa-act-orbifold} Suppose $M$ is a an orbifold.  There is an equivalence of (discrete) $\infty$-groupoids
\[  \bHH^0_k(\Perf(M)) \isom \Ext^0_M(\O_{I_M}, \O_M) \stackrel{\sim}\longrightarrow \left\{\begin{gathered}\text{$B\oh{\GG_a}$-actions on $\Perf(M)$}\\\text{as $k$-linear $\infty$-category}\end{gathered}\right\} \]
e.g., if $M = U\quot G$ with $U$ a smooth scheme and $G$ a finite group, then the RHS is 
    \begin{align*} \Ext^0_M(\O_{I_M}, \O_M) &\isom \left(\oplus_{g \in G\atop{\codim(U^g) = 0}} g \# H^0(U^g, \pi_0 \O_{U^g}) \right)^G \\
      &\isom \bigoplus_{\text{$[g]$ conj. class in $G$}\atop{\codim(M^g) = 0}} H^0(U^g, \O_{U^g})^{Z_G([g])}
    \end{align*}
    where the (right) $G$-action on the direct sum is by $(g \# a) \cdot h = h^{-1} g h \# a h$.  (In case $M$ is disconnected, we must sum over components of $M^g$ of codimension zero.)
\end{lemma}
\begin{lemma}\label{lem:BGa-kt-mf-orbifold} Suppose $M$ is an orbifold and set $\C = \Perf(M)$. Via \autoref{lem:BGa-act-orbifold}, an element $\alpha \in \bHH^0(\Perf(M))$ gives rise to a $B\oh{\GG_a}$-action on $\C$.  The monad $i \circ i^L$ of \autoref{lem:invts} identifies with $\O_{\Delta} \otimes_{k[x]} k \in \Alg(\QC(M^2), \circ)$, where $\O_\Delta$ is a $k[x]$-algebra via $x \mapsto \alpha$ and where $\QC(M^2)$ is equipped with its convolution product and its ``star integral transforms'' action on $\C$.  So, $\C^{S^1} = (\O_{\Delta} \otimes_{k[x]} k)\mod(\Perf(M))$.

In case $M = U \quot G$, and $\alpha = \sum_g f_g \in (\oplus_g H^0(U^g, \O_{U^g}))^G$ (with $f_g \neq 0$ only on codimension zero components),  this admits a ``crossed product'' description
\[ \Perf(U \quot G)^{S^1} = \left[\left(\bigoplus_{g \in G} g \# (\Gamma_g)_* \O_U \right) \otimes_{k[x]} k\right]\mod(\Perf(U)) \]
\end{lemma}

%
%
%

\section{Comparison of three viewpoints}\label{sec:viewpoints}
\subsection{Back and forth} 
We begin with the following variant of \autoref{lem:s1-act-mf}, which is motivated by the idea that $\PreMF(\AA^1, x)$ over $k\ps{\bt}$ is ``like'' $\oh{\GG_a}$ over $\oh{\GG_a}$.
\begin{lemma}\label{lem:s1-kt-act} There is an equivalence of $\infty$-groupoids
  \[ \left\{ \begin{gathered} \text{$S^1$-actions on $k\mod$}\\\text{as $k\ps{\bt}$-linear category}\end{gathered} \right\} \longleftrightarrow k\ps{x}^\times \]
\end{lemma}
\begin{proof} This is a variant of \autoref{lem:s1-act-mf}, using \autoref{thm:functors}: An $S^1$-action on $k\mod$ as $\QCsh(\bB)$-linear category, is the same as the data of a loop map
  \[ S^1 \stackrel{\otimes}\longrightarrow \Aut_{\QCsh(\bB)}(\QCsh(k)) \]
  Since $S^1 = B \ZZ$, this is the same as giving a double loop map
  \[ \ZZ \stackrel{\otimes^2}\longrightarrow \Aut_{\Fun^L_{\QCsh(\bB)}(\QCsh(k))}\left( \id_{\QCsh(k)}\right)\subset \End_{\Fun^L_{\QCsh(\bB)}(\QCsh(k))}\left( \id_{\QCsh(k)}\right) \]
  Identify $\QCsh(k) = \PreMF^\infty(\AA^1, -x)$ as $\QCsh(\bB)$-linear category.  By \autoref{thm:functors} (and \autoref{thm:cpltn}) there is an equivalence of $\infty$-categories
  \[ \Fun^L_{\QCsh(\bB)}(\QCsh(k), \QCsh(k)) = \PreMF_{0 \times 0}^\infty(\AA^2, -x+y) = \QCsh_{0 \times 0}\left( \left\{x=y\right\} \right) =  \QCsh_0 \AA^1 = \QCsh \oh{0} \]
  under which the identify functor corresponds to 
  \[ \id_{\QCsh(k)} \mapsto \ol{\Delta}_*\left( \RGamma_0 \omega_{\AA^1}\right) \mapsto \RGamma_0 \omega_{\AA^1} \mapsto \omega_{\oh{0}} \] so that
  \begin{align*} \End_{\Fun^L}(\id_{\QCsh(k)}) &= \End_{\QCsh \oh{0}}\left( \omega_{\oh{0}} \right) 
    = \Omega^\infty \RHom_{\QCsh \oh{0}}\left( \omega_{\oh{0}}, \omega_{\oh{0}} \right) \\
    &= \Omega^\infty \RHom_{\QC \oh{0}}\left( \O_{\oh{0}}, \O_{\oh{0}}\right) 
    = \Omega^\infty k\ps{x} = k\ps{x}.\end{align*}
  In particular, we see that (as a $2$-fold loop space) $\Aut(\id_{\QCsh(k)})$ identifies with the (discrete) $2$-fold loop space $k\ps{x}^\times$.  Since both $\ZZ$ and $k\ps{x}^\times$ are discrete, $2$-fold loop maps are just the same as ordinary (abelian) group homomorphisms:
  \[ \Map_{\otimes^2}\left(\ZZ, \Aut_{\id_{\QCsh(k)}}\right) = \Map_{\otimes^2}\left(\ZZ, k\ps{x}^\times\right) = \Map_{\mathrm{AbGp}}(\ZZ, k\ps{x}^\times) = k\ps{x}^\times.\qedhere \]
\end{proof}

\begin{defn} For $\qwer \in k\ps{x}^\times$, let $k\mod_\qwer$ and $\Perf k_\qwer$ (or just $k_\qwer$ for short) denote $k\mod$ and $\Perf k$ equipped with the $S^1$-action of the previous Lemma.  Although we don't introduce notation for it, it should be regarded as a mixture of $\PreMF$ and $\CircMF$ with the two functions $-x$ (to $\GG_a$) and $\qwer$ (to $\GG_m$):
  \[ k_\qwer = \mathop{Pre/CircMF}(\oh{\GG_a}, -x, \qwer) \qquad \xymatrix@1{\oh{\GG_a} \ar[r]^-{-x \times \qwer} & \oh{\GG_a} \times \oh{\GG_m}} \] The $k\ps{\bt}$-action is from taking the fiber over $0$ in the first variable, the $S^1$-action is from the second.
\end{defn}


This allows us to incorporate $S^1$ actions into \autoref{cor:premf-supt}:
\begin{prop}\label{prop:s1-trans-1} Suppose $(M,f)$ is a formal LG pair, and $\qwer \in k\ps{x}^\times$.  Set $M_0 = M \times_{\AA^1} 0$, and $\oh{M_0}$ the formal completion of $M$ along $M_0$. Then, there is an $S^1$-equivariant equivalence
  \[ \PreMF(M,f) \otimes_{k\ps{\bt}} k_{\qwer} = \CircMF(\oh{M_0}, \qwer(f)) \]
\end{prop}
\begin{proof} 
 At the level of underlying dg-categories,
 \[ \PreMF(M,f) \otimes_{k\ps{\bt}} k_\qwer = \PreMF_{M_0 \times 0}(M \times \AA^1, f \boxplus -x) = \DCoh_{M_0}(\Gamma_{f}(M)) = \DCoh(\oh{M_0}) \]
  Consider the diagram
  \[ \xymatrix{\oh{M_0} \ar[r]^-{\Gamma_f} & \oh{M_0} \times \oh{\GG_a} \ar[r]^-{p_2} & \oh{\GG_a} \ar[r]^-{\qwer }& \oh{\GG_m}} \] 
 The $S^1$-action comes from the second projection $\qwer\colon \oh{\GG_a} \to \oh{\GG_m}$, i.e., $\qwer(f)$.
\end{proof}

Finally, we sketch a few of the compatibilities between the various constructions we have seen:
\begin{prop}\label{prop:s1-trans-2} Suppose $M$ is a smooth formal $k$-scheme and $f \in \GG_m(M)$.  Then, there is a $k\ps{\bt}$-linear equivalence
  \[ \CircMF(M, f)^{S^1} \isom \PreMF(\oh{M_1}, \log(f)) \]
  where $\log(f) = \log(1 + (f-1)) = \sum (-1)^m (f-1)^m/m$.
\end{prop}
\begin{proof} Consider the diagram
  \[ \xymatrix{ \oh{M_1} \ar[r]^f & \oh{\GG_m} \ar[r]^-{\log}_{\sim} & \oh{\GG_a}} \] and note that the second map is an equivalence of abelian formal groups.  Now combine \autoref{sec:more-general} and \autoref{lem:kt-mf} (with its footnote).
\end{proof}

\begin{prop}\label{prop:s1-trans-3} Suppose $(M, f)$ is an LG pair, and $\qwer \in k[x] \cap k\ps{x}^\times$ with $\qwer(0) = 1$. Shrinking $M$ if necessary, suppose that $\qwer(f) \in \GG_m(M)$ so that both $\CircMF(M,\qwer(f))$ and $\PreMF(M,f)$ make sense.  Then,
  \begin{enumerate}
    \item There is a $k\ps{\bt}$-linear equivalence
  \[ \CircMF(M, \qwer(f))^{S^1} = \PreMF(M, f) \otimes_{k\ps{\bt}} (k_{\qwer})^{S^1} \]
\item If $\qwer'(x) \neq 0$, then $(k_{\qwer})^{S^1}$ is an invertible $k\ps{\bt}$-module category (in fact, equivalent to $\Perf k\ps{\bt}$).
  \end{enumerate}
\end{prop}
\begin{proof} \mbox{}
  \begin{enumerate}
    \item The inclusion $\CircMF(\oh{M_0}, \qwer(f)) \to \CircMF(M, \qwer(f))$ induces an equivalence on $S^1$-fixed points. By \autoref{prop:s1-trans-1} it remains to check that the natural map \[ \PreMF(M,f) \otimes_{k\ps{\bt}} (k_\qwer)^{S^1} \to (\PreMF(M, f) \otimes_{k\ps{\bt}} k_\qwer)^{S^1} \] is an equivalence.  This will follow from (ii) upon noting that the underlying $k$-linear category on both sides identifies with $\DCoh(M_0)$ by \autoref{lem:kt-mf}, and the $k$-linear functor with the identity functor.
    \item Consider the diagram
      \[ \xymatrix{ \oh{\GG_a} \ar[r]^-{\Delta} & \oh{\GG_a} \times \oh{\GG_a} \ar[r]^-{\id \times \qwer} & \oh{\GG_a} \times \oh{\GG_m} } \] 
      Using \autoref{prop:s1-trans-2}, there is a $k\ps{\bt} \otimes k\ps{\bt}$-linear identification of $(k_\qwer)^{S^1}$ with $\DCoh$ on the fiber over $0 \times 1$.  By hypothesis, $\id \times \qwer$ is an isomorphism of formal groups, so this identifies with $\PreMF(\oh{\GG_a}, x, x)$. \qedhere
  \end{enumerate}
\end{proof}

\begin{remark}\label{rem:other-points} If $(M,f)$ is an LG pair, then $\Perf(M)$ with the $B\oh{\GG_a}$-action corresponding to $f$ (or  $\CircMF(M,f)$ in the case of $f \in \GG_m(M)$) remembers information about all the fibers of $f$.   An easier version of the construction in this section tells us how: The space of $B\oh{\GG_a}$-actions  (resp., $S^1$-actions) on $\Perf k$ identifies with $\Gamma(\pt, \O_\pt) = k$ (resp., $k^\times$).  For $t \in k$ (resp., $\lambda \in k^\times$) let $k_t$ (resp., $k_\lambda$) denote this.  Then, we can twist the formation of invariants by $k_t$ (resp., $k_\lambda$)
  \[ \Fun_{B\oh{\GG_a}}(k_t, \Perf(M)) = \left(k_{-t} \otimes \Perf(M)\right)^{B\oh{\GG_a}} = \PreMF(M, f-t) \]
One could therefore hope for a refined version of \autoref{thm:hoch}, not factoring through taking invariants on the category level, which retains information about finer global invariants (e.g., non-commutative Hodge structures). 
\end{remark}

\subsection{Summary: Two Koszul dualities}
\begin{na}\label{na:heur1} Suppose $\C \in \dgcatidm_k$ is smooth, and $\pi_i \bHH^\bullet_k(\C) = 0$ for $i > 0$.  Then, there are equivalences of (discrete) spaces
  \[ \left\{\begin{gathered}\text{$B\oh{\GG_a}$-action}\\\text{on $\C$}\end{gathered}\right\} \Leftrightarrow \left\{ \bHH^0(\C)\right\}  \Leftrightarrow \left\{\begin{gathered}\text{$\Perf(\GG_a)^\otimes$-linear}\\\text{structure on $\C$}\end{gathered}\right\} \] The right is the standard description of a $k[x]$-linear structure as classified by $E_2$-maps $k[x] \to \bHH^\bullet(\C)$: By our connectivity assumptions, this identifies with $E_2$-maps between the discrete algebras $k[x] \to \bHH^0(\C)$.  The left is a summary of the proof of \autoref{lem:s1-act-mf} (among others), and can also be summarized as ``take the derivative'': Maps $B^2 \oh{\GG_a} \to B\Aut(\C)$ (=the completion of $\dgcat$ at $\C$) are classified by Lie algebra maps $k[-1] \to \bHH^\bullet(\C)[-1]$.   In other words, this can be thought of as a simple case of commutative Koszul duality.  
      
An alternate formulation, skipping the middle-man, would be to say that the symmetric monoidal category $\Perf(B\oh{\GG_a})^\circ$ (with its convolution product) is, at least almost, the same as $\Perf(\GG_a)^{\otimes}$ (with its usual tensor) compatibly with forgetting to $\Perf k$.
\end{na}

\begin{na}\label{na:heur2} There is also an $E_2$ Koszul duality between $k\ps{x}$ and $k\ps{\bt}$, and in fact this is what underlies \autoref{ssec:Bact}: The bar construction $k \otimes_{k[x]} k$ is $\O_B$, and then the cobar construction $\RHom_{\O_\bB}(k,k)$ is $k\ps{\bt}$.  More topologically (and for the $S^1$-action viewpoint), start with $\O_{\GG_m} = C_*(\ZZ, k)$ so that Eilenberg-Moore identifies bar and cobar constructions geometrically and without characteristic assumptions $k \otimes_{C_*(\ZZ, k)} k = C_*(S^1, k)$, $\RHom_{C_*(S^1,k)}(k,k) = C^*(BS^1, k)$.  This provides a relationship between $k[x]$-linear (i.e., $\Perf(\GG_a)^{\otimes}$-linear) categories and $k\ps{\bt}$-linear categories, though as we have seen it is not strictly invertible or strictly monoidal.  The point of Sections~3--5 was to provide a geometric way to study this relationship (hiding one cobar construction via the trick of taking $\DCoh(\bB)$), and more precisely understand how close to invertible, monoidal, etc. it is.
\end{na}

\part{Applications}
\section{Applications}\label{sec:applications}
\subsection{Smoothness (and properness) of \texorpdfstring{$\MF$}{MF}}
Using \autoref{thm:functors}, we are able to obtain the show that $\MF$ is smooth, and that it is proper when the critical locus is proper:
\begin{theorem}[Smoothness and Properness]\label{thm:sm-proper} Suppose $(M,f)$ is an LG pair, $Z \subset f^{-1}(0)$ closed.  Then,
  \begin{enumerate}
    \item Suppose $Z_{\red}$ is proper.  Then, $\PreMF_Z(M, f)$ is proper over $k\ps{\bt}$ and $\MF_Z(M, f)$ is proper over $k\pl{\bt}$.
    \item Suppose $Z$ contains each connected component of $\crit(f)$ which it intersects.  Then, $\MF_Z(M, f)$ is smooth over $k\pl{\bt}$.
    \item Suppose $\crit(f) \cap f^{-1}(0)$ is proper.  Then, $\MF(M, f)$ is smooth and proper over $k\pl{\bt}$.
  \end{enumerate}
\end{theorem}
\begin{proof}\mbox{} \begin{enumerate}
  \item It suffices to show that $\ev$, restricted to compact objects, factors through $\Perf k\ps{\bt}$.  Unraveling, it suffices to verify that 
    \[ \ev(\F \otimes \G) = \RHom^{\otimes k\ps{\bt}}_{\PreMF(M,f)}(\F, \G) \in k\ps{\bt}\mod \] is perfect for all $\F, \G \in \PreMF_Z(M,f)$.  By \autoref{prop:s1-hom}
    \[ \RHom^{\otimes k\ps{\bt}}_{\PreMF(M,f)}(\F, \G) = \RHom_{\DCoh(M)}(i_* \F, i_* \G)^{S^1} \] where $i\colon M_0 \to M$. Regarding $\Hom_{\DCoh(M)}(i_* \F, i_* \G)$ with its $S^1$-action as a $k[B]/B^2$-module, it suffices by \autoref{prop:mf-prelim} to show that it is $t$-bounded and coherent over $k[B]/B^2$, or equivalently perfect over $k$.  
    
    Thus, it is enough to show that $\RHom_{\DCoh(M)}(\F', \G') \in \Perf k$ for any $\F', \G' \in \DCoh_Z(M)$.  Let $k\colon Z_\red \to M$ be the inclusion.  By \autoref{lem:coh-red} we are reduced to the case where $\G' = k_* \ol{\G}$ for some $\ol{\G} \in \DCoh(Z_{\red})$.  Since $M$ is regular, $\F'$ is perfect and hence the pullback $k^* \F'$ is also perfect so that $\HHom(k^* \F', \ol{\G}) \in \DCoh(Z_{\red})$.    But now, $Z_\red$ is proper so that
    \[ \RHom_{\DCoh(M)}(\F', k_* \ol{\G}) = \RHom_{\DCoh(Z_\red)}(k^* \F', \ol{\G}) = \RGamma\left(Z_{\red}, \HHom_{\DCoh(Z_{\red})}(k^* \F', \ol{\G})\right) \in \Perf k \] as desired.
  \item We must prove that $\id_{\MF_Z(M,f)}$ is a compact object in the functor category.  An object in a $k\ps{\bt}$-linear $\infty$-category is compact iff it is compact in the category, viewed as a plain $\infty$-category.  Using \autoref{thm:functors}, we are reduced to showing that 
    \[ \ol{\omega_{\Delta,Z}} = \ol{\Delta}_*\left( \ul{\RGamma}_Z \omega_M\right) \in \MF^\kinfty_{Z^2}(M^2, -f \oplus f) = \Ind \DCoh_{Z^2}( (M^2)_0) \ohotimes_{k\ps{\bt}} k\pl{\bt} \] is compact.
    Since $\omega_M$ is coherent and $\Delta$ proper, $\ol{\omega_\Delta} = \ol{\Delta}_* \omega_M$ is coherent and so compact in $\PreMF^\kinfty(M^2, -f \boxplus f) = \Ind\DCoh( (M^2)_0)$.   Note that $\ol{\Delta}_* \ul{\RGamma}_Z \omega_M = \ul{\RGamma}_{Z^2} \ol{\Delta}_* \omega_M$, since $Z \subset M_0$, so that it is only the $\ul{\RGamma}_{Z^2}$ that can cause problems.

  Let $W = \crit(-f \oplus f) \cap (M^2)_0$ be the components of the critical locus of $-f \oplus f$ lying in the zero fiber.  By \autoref{prop:orlov-cpltn}, the natural inclusions 
    \[ \xymatrix{ \DCoh_{Z^2 \cap W}( (M^2)_0 ) \ar@{^{(}->}[d] \ar@{^{(}->}[r] & \DCoh_{Z^2}( (M^2)_0 ) \ar@{^{(}->}[d] \\ \DCoh_{W}( (M^2)_0) \ar@{^{(}->}[r] & \DCoh(M^2) } \]
    induce, upon applying $\Ind(-) \ohotimes_{k\ps{\bt}} k\pl{\bt}$
    \[ \xymatrix{ \ul{\RGamma}_{Z^2 \cap W}(\ol{\omega_\Delta}) \in \MF^\kinfty_{Z^2 \cap W}(M^2, -f \boxplus f) \ar@{^{(}->}[d] \ar[r]^{\sim} & \MF^\kinfty_{Z^2}(M^2, -f \boxplus f ) \ni \ul{\RGamma}_{Z^2}(\ol{\omega_\Delta}) \ar@{^{(}->}[d] \\ \ul{\RGamma}_W(\ol{\omega_\Delta}) \in \MF^\kinfty_{W}( M^2, -f \boxplus f) \ar[r]^{\sim} & \MF^\kinfty(M^2, -f \boxplus f) \ni \ol{\omega_\Delta} } \]
    The functors in this diagram are left adjoints, whose right adjoints are the appropriate $\ul{\RGamma}_{-}$ functors.  Using the top row, we see that it suffices to show that $\ul{\RGamma}_{Z^2 \cap W}(\ol{\omega_\Delta})$ is compact in $\MF^\kinfty_{Z^2 \cap W}(M^2, -f \boxplus f)$.  Using the bottom row, we see that $\ul{\RGamma}_W(\ol{\omega_\Delta})$ is compact in $\MF^\kinfty_W(M^2, -f \boxplus f)$.  It thus suffices to show that $\ul{\RGamma}_{Z^2 \cap W}\colon \MF^\kinfty_{W}(M^2, -f \boxplus f) \to \MF^\kinfty_{Z^2 \cap W}(M^2, -f \boxplus f)$ preserves compact objects; the property of preserving compact objects is preserved under $-\ohotimes_{k\ps{\bt}} k\pl{\bt}$, so it suffices to show that $\ul{\RGamma}_{Z^2 \cap W}\colon \Ind\DCoh_{W}( (M^2)_0) \to \Ind\DCoh_{Z^2 \cap W}( (M^2)_0)$ preserves compact objects.   But, our assumptions on $Z$ imply that $Z^2 \cap W$ is a union of connected components of $W$: so, $\RGamma_{Z^2 \cap W}$ may be identified with the restriction to those connected components, and in particular preserves compact objects.
  \item Set $Z = \crit(f) \cap f^{-1}(0)$.  By (i) and (ii), $\MF_Z(M, f)$ is smooth and proper over $k\ps{\bt}$.  By \autoref{prop:orlov-cpltn}, the inclusion induces an equivalence $\MF_Z(M,f) \isom \MF(M,f)$.\qedhere
\end{enumerate}
\end{proof}

\begin{remark} It seems likely that the Theorem remains true if $(M,f)$ is replaced by a \emph{formal LG pair}: i.e., a relative DM stack $f\colon \X \to \oh{\pt} = \Spf \oh{\O_{\AA^1}} \subset \AA^1$ over $\oh{\pt}$ with $\X$ formally smooth.  However, the methods of this paper seem to be insufficient for this beyond the algebrizable case.
\end{remark}

\subsection{Hochschild-type Invariants}

\begin{na}\label{na:hoch-bga}
Suppose $(M,f)$ is an LG pair.  The $B\oh{\GG_a}$-action on $\DCoh(M)$ corresponding to $f$ (\autoref{lem:BGa-act-mf}) provides a $B\oh{\GG_a}$-action on ${\bHH}_\bullet(\DCoh(M))$ and ${\bHH}^\bullet(\DCoh(M))$, and by naturality maps
\[ {\bHH}_\bullet^{k\ps{\bt}}\left(\DCoh(\DCoh(M))^{B\oh{\GG_a}}\right) \longrightarrow {\bHH}^k_\bullet(\DCoh(M))^{B\oh{\GG_a}} \]\[ {\bHH}_k^\bullet(\DCoh(M))^{B\oh{\GG_a}}  \longrightarrow {\bHH}^\bullet_{k\ps{\bt}}\left(\DCoh(M)^{B\oh{\GG_a}}\right) \]  preserving all the structures naturally present on Hochschild invariants ($\SO(2)$-action on $\bHH_\bullet$, $E_2$-algebra structure on $\bHH^\bullet$, and $\bHH^\bullet$-module structure on $\bHH_\bullet$). The goal of this section will, roughly, be to study the degree to which these are equivalences.  Our main tool will be an alternate description of these actions: Under the identification of $S^1$- and $B\oh{\GG_a}$-actions on complexes, these two actions can be identified (under \autoref{cor:hh}) with the $S^1$-actions on
\[ {\bHH}_\bullet(\DCoh(M)) = \RHom_{\DCoh(M)}\left(\ell_* \ol{\Delta}_* \O_M, \ell_* \ol{\Delta}_* \omega_M\right) \] \[ {\bHH}^\bullet(\DCoh(M)) = \RHom_{\DCoh(M)}\left(\ell_* \ol{\Delta}_* \O_M, \ell_* \ol{\Delta}_* \O_M\right) \] from \autoref{prop:s1-hom} (where $\ell: (M_0)^2 \to (M^2)_0$ is the natural map).
\end{na}

\begin{na} More generally suppose $\C$ is a smooth $k$-linear dg-category for which $\bHH^\bullet_k(\C)$ vanishes in homologically positive degrees.  To $\alpha \in \bHH^0(\C)$, the proof of \autoref{lem:BGa-act-mf} associates a $B\oh{\GG_a}$-action on $\C$, which we think of as e.g., multiplication by $e^{t\alpha}$ on the mapping spaces.  This gives rise to $B\oh{\GG_a}$-actions on $\bHH^\bullet(\C)$ and $\bHH_\bullet(\C)$, which have a simple explicit description in terms of $\alpha$: 
\end{na}
\begin{lemma}\label{lem:BGa-act-hoch} The $B$-operator for the $B\oh{\GG_a}$-action on $\bHH^\bullet(\C)$ (resp., $\bHH_\bullet(\C)$) is given by Lie multiplication (i.e., using the Browder operation on the $E_2$-algebra $\bHH^\bullet(\C)$) by $\alpha \in \bHH^0(\C)$ in $\bHH^\bullet(\C)$ (resp., its $E_2$-module $\bHH_\bullet(\C)$).
\end{lemma}
\begin{proof} 
  We handle the case of $\bHH^\bullet$, with $\bHH_\bullet$ being analogous.  Consider the three rows, each of which is a delooping of the one above it
  \[ B^2\oh{\GG_a} \stackrel{B^2(e^{tf})}\longrightarrow (\dgcat_k)\oh{\phantom{}_{\C}} \longrightarrow B\Aut(\bHH^\bullet \C)\oh{\phantom{}_{\id}} \] 
  \[ B\oh{\GG_a} \stackrel{B(e^{tf})}\longrightarrow \Aut_{\dgcat_k}(\C)\oh{\phantom{}_{\id}} \longrightarrow \Aut(\bHH^\bullet \C)\oh{\phantom{}_{\id}} \]
  \[ \oh{\GG_a} \stackrel{e^{tf}}\longrightarrow \Aut(\id_\C)\oh{\phantom{}_{1}}  \longrightarrow \Omega_{\id} \Aut(\bHH^\bullet \C)\oh{\phantom{}_{\id}}  \]
  The left-most arrows are encoding the $B\oh{\GG_a}$-action on $\C$, the right-most are obtained by functoriality of $\bHH^\bullet \C$ under automorphisms, and the composites are the ones we want to understand.
  Note that the $B$-operator of a $B\oh{\GG_a}$-action is just the shift of its derivative (i.e., map on tangent complexes), and passing to tangent complexes in the middle row
  \[ k[1]  \stackrel{f}\longrightarrow \bHH^\bullet(\C)[1] \longrightarrow \End\left(\bHH^\bullet \C)\right)\]
  we see that it suffices to identify the second map with ``adjoint action'' of the dg-Lie algebra $\bHH^\bullet(\C)[1] \isom T_{\Aut(\C)\oh{\phantom{}_{\id}}}$.  Note that the question is now completely independent of our original $B\oh{\GG_a}$-action, and is instead a universal question of relating two possible actions of the formal group $\Aut(\C)\oh{\phantom{}_{\id}}$ on the complex $\bHH^\bullet(\C)$: There is the action obtained by functoriality of $\bHH^\bullet(\C)$ under automorphisms, and there is the adjoint action on its Lie algebra under the identification of its tangent space with $\bHH^\bullet(\C)[1]$.  Note that the ``functoriality'' action $\Aut_{\dgcat_k}(\C) \longrightarrow \Aut(\bHH^\bullet \C)$ is just ``conjugation'': for $F \in \Aut(\C)$ and $e \in \End(\id_\C)$ one has $F \circ e \circ F^{-1} \in \End(F \circ \id_\C \circ F^{-1}) = \End(\id_\C)$.  It remains to carefully check that this is compatible with the identification of complexes $\bHH^\bullet(\C) = \End(\id_\C)$ and $T_{\Aut(\C)\oh{\phantom{}_{\id}}}[-1] \isom T_{\Aut(\id_\C)\oh{\phantom{}_{1}}}$.
\end{proof}

It is well-known that for $M$ a smooth scheme a modified HKR map identifies the $E_2$-algebra $\bHH^\bullet(M)$ and its module $\bHH_\bullet(M)$ with the Schouten-Nijenhuis bracket, Lie derivative, etc. structure on $T^\bullet_M$ and $\Omega_M^\bullet$.  For the Lie action of these classes in $\bHH^0(M)$  one does not even have to modify it, and we can verify the computation explicitly:
\begin{prop}\label{prop:hkr-df} Suppose $(M,f)$ is an LG pair with $M$ a scheme. Then, the HKR identifications $\ul{\bHH}_\bullet(\Perf(M)) = \Omega_M^\bullet$ and $\ul{\bHH}^\bullet(\Perf(M)) = T^\bullet_M = \oplus \bigwedge^i T_M[-i]$ lift to $S^1$-equivariant identifications, where the $S^1$-actions on the right are given by $-df \wedge -$ and $-i_{df}(-)$.
\end{prop}
\begin{proof}
We prove the result in the affine case $M = \Spec R$; the proof globalizes in the same way as HKR itself, by completing the cyclic bar complexes, etc.  Recall the cyclic bar-type resolution of $\Delta_* \O_M = R$ as $\O_{M^2} = R \otimes R$-bimodule
\[ R \longrightarrow \left| R^{\otimes (\bullet+2)}\right| \qquad \text{degeneracies given by inserting $1$, face maps given by multiplying adjacent elements}. \]
It remains to give this a structure of $\O_{M^2}[B_{M^2}]$-module quasi-isomorphic to $\ol{\Delta}_* \O_M$.  We claim that this can be explicitly done by setting
\[ B_{M^2}(a_1 \otimes \cdots a_n) = \sum_{i=1}^{n-1} (-1)^i a_1 \otimes \cdots \otimes a_{i} \otimes f \otimes a_{i+1} \otimes \cdots \otimes a_n \] Indeed, a straightforward computation verifies that $B_{M^2}^2 = 0$ and $d(B_{M^2} \cdot x) = (-f \boxplus f) \cdot x + B_{M^2} \cdot dx$, where $d$ is the internal differential on the cyclic complex.  Regarding $\Delta_* T$ as an $\O_{M^2}[B_{M^2}]$-module via the augmentation $\O_{M^2}[B_{M^2}] \to\O_{M^2}/(-f \boxplus f)$, \autoref{na:s1-act} tells us that the $S^1$-action on \[ \HHom(\Delta_* \O_M, \Delta_* T) = \HHom_A(\Delta^* \Delta_* \O_M, T) = \Tot\left\{\HHom_A(A^{\otimes(\bullet+1)}, T)\right\} \] is simply dual to $B = B_{M^2} \otimes_{\O_{M^2}} \O_M$ in the first variable.  Finally, it suffices to observe that the usual HKR map intertwines $-df \wedge -$ and $B$: We compute
\begin{align*}
  B = B_{M^2} \otimes_{\O_{M^2}} \O_M\left( a_1 \otimes \cdots \otimes a_{m} \right) &= \sum_{i=1}^{m-1} (-1)^i a_1 \otimes \cdots \otimes a_i \otimes f \otimes a_{i+1}\otimes \cdots \otimes a_m \\
  &+ (-1)^m a_1 \otimes \cdots \otimes a_m \otimes f
\end{align*}
so that
\begin{align*} \mathop{HKR}\left( B \left( a_1 \otimes \cdots \otimes a_m \right) \right) &= \frac{1}{m!}\left[\sum_{i=1}^m (-1)^i a_1 \wedge \cdots da_i \wedge df \wedge da_{i+1} \right] \\
  &= -\frac{df}{m!} \wedge \left(m a_1 da_2 \wedge \cdots da_m \right) \\
  &= -df \wedge \mathop{HKR}\left( a_1 \otimes \cdots \otimes a_m \right)
\end{align*}

The analogous operator on $\ul{\bHH}^\bullet$ is dual, which is $-i_{df}(-)$. (Warning: I don't know the sign rules well enough to be sure of the sign.)
\end{proof}

\begin{remark}
Expanding on the proof of of \autoref{lem:s1-act-orbifold}, and noting that the natural map $L(U^g) \to L_g U$ is an equivalence,\footnote{It is evidently an equivalence on $\pi_0$, both terms being identified with $U^g$.  So it suffices to verify that we have an equivalence on cotangent complexes.  Applying Luna's Slice Theorem, one sees that $U^g$ is smooth and that its cotangent bundle is the $\langle g \rangle$-invariant piece in the $\res{\Omega_U}{U^g}$; in particular, the conormal bundle (say at each point) contains only non-trivial $\langle g\rangle$ representations so that $g$ acts invertibly on the conormal bundle. One can identify the cotangent complex of $L_g U$ with the cone of the action of $g$ on $\LL_U$; the cotangent complex of $L(U^g)$ with the cone of the zero map on $\LL_{U^g}$; and the map of cotangent complexes with the pullback $\res{\LL_{U}}{U^g} \to \LL_{U^g}$.  Since both $U$ and $U^g$ are smooth, this restriction map is surjective and its kernel is the conormal bundle of $U^g$ in $U$.  It thus suffices to recall that $g$ acts invertibly on the conormal bundle.} one obtains a natural equivalence
\[ L(U \quot G) = \left( \oplus_{g \in G} L_g U \right)\quot G \isom \left( \oplus_{g \in G} L(U^g)\right) \quot G \]
So that
\[ \bHH_\bullet(U \quot G) = \RGamma\left(L(U \quot G), \O_{L(U \quot G)}\right) = \left[\oplus_{g \in G} \RGamma(\O_{L_g U})\right]^G = \left[\oplus_{g \in G} \RGamma(\O_{L(U^g)})\right]^G \]
\[ \bHH^\bullet(U \quot G) = \RGamma\left(L(U \quot G), \omega_{L(U \quot G)/U\quot G}\right) = \left[\oplus_{g \in G} \RHom_U(\O_{L_g U}, \O_U)\right]^G = \left[ \oplus_{g \in G} \RHom_U(\O_{L(U^g)}, \O_U)\right]^G \]  Identifying $\O_{L(U^g)} = \ul{\bHH}_\bullet(\Perf U^g)$ and using its HKR description, one obtains \emph{an} HKR description of these orbifold Hochschild invariants.\footnote{Presumably there is a different HKR-type description where one stops at $L_g U$, so that the normal bundles of $U^g$ appears explicitly.}  Presumably a similar explicit computation is possible, though we have not tried to carry it out.
\end{remark}

Finally, using \autoref{thm:functors} we are able to complete the computation of Hochschild-type invariants:
\begin{theorem}[Hochschild-type Invariants]\label{thm:hoch} Suppose $(M,f)$ is an LG pair with, and $Z \subset f^{-1}(0)$ a closed set.  Then,
  \begin{enumerate}
    \item There are natural $k\ps{\bt}$-linear equivalences
      \[ \bHH_\bullet^{k\ps{\bt}}\left(\PreMF_Z(M, f)\right) = \bHH_\bullet^k(\DCoh_Z(M))^{S^1} \]
    and
    \[ \bHH^\bullet_{k\ps{\bt}}\left(\PreMF_Z(M, f)\right) = \bHH^\bullet_k(\DCoh_Z(M))^{S^1} \]
      (where the circle action is given, under the HKR isomorphism, by $-df \wedge -$ and \emph{not} the usual $B$-operator)
      The descriptions as invariants are compatible with the $B$-operator on $\bHH_\bullet$ (=de Rham differential), and the $E_2$-algebra structure on $\bHH^\bullet$, and the $\bHH^\bullet$-module structure on $\bHH_\bullet$.
    \item There is a natural $k\pl{\bt}$-linear equivalence
      \[ \bHH_\bullet^{k\pl{\bt}}\left(\MF_Z(M, f)\right) = \bHH^k_\bullet(\DCoh_Z(M))^{\Tate} \]
    \item Either assume $0$ is the only critical value of $f$, or set
      \[ \MF^\tot = \bigoplus_{\lambda \in \cval(f)} \MF^\kinfty(X, f-\lambda). \]  Then,  there are natural $k\pl{\bt}$-linear equivalences
      \[ \bHH_\bullet^{k\pl{\bt}}\left(\MF^\tot\right) = \left(\bHH_\bullet^{k}(\DCoh M)\right)^{\Tate} \]
      \[ \bHH^\bullet_{k\pl{\bt}}\left(\MF^\tot\right) = \left(\bHH^\bullet_{k}(\DCoh M)\right)^{\Tate} \]
      The description in terms of Tate-cohomology of an $S^1$-action on the Hochschild complex of $\DCoh(M)$ is compatible with: the $B$-operator on $\bHH_\bullet$, the $E_2$-algebra structure on $\bHH^\bullet$, the $\bHH^\bullet$-module structure on $\bHH_\bullet$.  Given a volume form on $M$ inducing a CY structure on $\MF(M,f)$ (see \autoref{thm:cy} below) the description is compatible with the resulting BV-algebra structure on $\bHH^\bullet$. 
    \item Suppose furthermore that $M$ is a scheme.  Then, HKR induces equivalences
 \[ \bHH_\bullet^{k\ps{\bt}}\left(\PreMF_Z(M, f)\right)  \isom   \RGamma_Z\left(\left[\Omega^\bullet_M\ps{\bt}, \bt \cdot (-df \wedge -)\right] \right)   \]
 \[ \bHC_\bullet^{k\ps{\bt}}\left(\PreMF_Z(M, f)\right)  \isom   \RGamma_Z\left(\left[\Omega^\bullet_M\ps{\bt, u}, \bt \cdot (-df \wedge -) + u \cdot d\right] \right)  \]
      \[ \bHH_\bullet^{k\pl{\bt}}\left(\MF_Z(M, f)\right) \isom\RGamma_Z\left(\left[\Omega^\bullet_M\pl{\bt}, \bt \cdot (-df \wedge -)\right] \right) \]
      \[ \bHC_\bullet^{k\pl{\bt}}\left(\MF_Z(M, f)\right) \isom\RGamma_Z\left(\left[\Omega^\bullet_M\pl{\bt}\ps{u}, \bt \cdot (-df \wedge -) + u \cdot d\right] \right) \]
      \[ \bHH_\bullet^{k\pl{\bt}}\left(\MF^\tot\right) \isom   \RGamma\left(\left[\Omega^\bullet_M\pl{\bt}, \bt \cdot (-df \wedge -)\right]\right) \]
      \[ \bHC_\bullet^{k\pl{\bt}}\left(\MF^\tot\right) \isom   \RGamma\left(\left[\Omega^\bullet_M\pl{\bt}\ps{u}, \bt \cdot (-df \wedge -) + u \cdot d\right]\right) \]
      \[ \bHH^\bullet_{k\pl{\bt}}\left(\MF^\tot\right) \isom   \RGamma\left(\left[{\bigwedge}^\bullet T_M[1]\pl{\bt}, \bt \cdot i_{df}(-)\right]\right] \]
  \end{enumerate}
\end{theorem}
\begin{proof}\mbox{}
  \begin{enumerate}
    \item Let $k\colon (M^2)_0 \to M^2$ be the inclusion, $\Delta\colon M \to M^2$ the diagonal, and $\ol{\Delta}\colon M \to (M^2)_0$ the reduced diagonal. By \autoref{thm:functors},
      \[ \bHH_\bullet^{k\ps{\bt}}\left(\PreMF_Z(M, f)\right) = \ev(\id_{\PreMF^\kinfty_Z(M,f)}) = \RHom^{\otimes k\ps{\bt}}_{\PreMF^\kinfty(M^2, -f \oplus f)}\left(\ol{\Delta}_* \O_M, \ol{\Delta}_* \ul{\RGamma}_Z \omega_M\right) \]
      Since $\omega_M$ is coherent, the standard formula for local cohomology shows that we may write 
      \[ \ol{\Delta}_* \ul{\RGamma}_Z \omega_M = \dlim_\alpha \K_\alpha \] as a uniformly $t$-bounded filtered colimit of compacts.  Then, applying \autoref{prop:s1-hom}:
      \begin{align*}
	\RHom^{\otimes k\ps{\bt}}_{\PreMF^\kinfty(M^2, -f \oplus f)}\left(\ol{\Delta}_* \O_M, \ol{\Delta}_* \ul{\RGamma}_Z \omega_M\right) &= \dlim_\alpha\left[ \RHom_{\DCoh(M^2)}(\Delta_* \O_M, k_* \K_\alpha)^{S^1} \right] 
	\intertext{By $t$-boundedness of the $\K_\alpha$, and regularity of $M^2$, we see that $\left\{\RHom_{\DCoh(M^2)}(\Delta_* \O_M, k_* \K_\alpha)\right\}_\alpha$ will be uniformly $t$-bounded.  Since taking $S^1$-invariants commutes with uniformly $t$-bounded colimits, we obtain}
	&= \left[ \dlim_\alpha \RHom_{\DCoh(M^2)}(\Delta_* \O_M, k_* \K_\alpha) \right]^{S^1} \\
	&= \left[ \RHom_{\QCsh(M^2)}(\Delta_* \O_M, \Delta_* \ul{\RGamma}_Z \omega_M) \right]^{S^1}
	\intertext{which by \autoref{cor:hh} we may identify with}
	&= \left[ \bHH_\bullet^k(\DCoh_Z(M)) \right]^{S^1}
      \end{align*}

      Analogously,
      \begin{align*}
	\bHH^\bullet_{k\ps{\bt}}\left(\PreMF_Z(M,f)\right) &= \RHom^{\otimes k\ps{\bt}}_{\PreMF^\kinfty(M^2,-f\boxplus f)}\left( \ol{\Delta}_* \ul{\RGamma}_Z \omega_M, \ol{\Delta}_* \ul{\RGamma}_Z \omega_M\right) \\
	&= \RHom^{\otimes k\ps{\bt}}_{\PreMF^\kinfty(M^2,-f\boxplus f)}\left( ``\dlim_\alpha" \K_\alpha, ``\dlim_{\alpha'}" \K_{\alpha'}\right) \\
	&= \ilim_\alpha \dlim_{\alpha'} \RHom^{\otimes k\ps{\bt}}_{\PreMF^\kinfty(M^2, -f \boxplus f)}\left(\K_\alpha, \K_{\alpha'}\right) \\
	&= \ilim_\alpha \dlim_{\alpha'} \left[\RHom_{\QCsh(M^2)}\left(\K_\alpha, \K_{\alpha'}\right)\right]^{S^1} \\
	\intertext{As before, the $t$-boundedness of $\K_{\alpha'}$ and the regularity of $M^2$ imply that we may commute $\dlim_{\alpha'}$ past the invariants.  Finally, we commute the $\ilim_\alpha$ past the invariants, to obtain}
	&= \left[\RHom_{\QCsh(M^2)}\left(\Delta_* \ul{\RGamma}_Z \omega_M, \Delta_* \ul{\RGamma}_Z \omega_M\right)\right]^{S^1} \\
	\intertext{which by \autoref{cor:hh} we may identify with}
	&= \left[\bHH^\bullet_k\left(\DCoh_Z(M)\right)\right]^{S^1}
      \end{align*}
    
      Recall that the compatibility with the various structures follows from the argument of \autoref{na:hoch-bga}.
   \item Follows from (i) since $\bHH_\bullet$ is compatible with the symmetric monoidal functor $- \otimes_{k\ps{\bt}} k\pl{\bt}$.
   \item The computation follows in a manner analogous to (i) from \autoref{thm:functors}(v). 
   \item We first prove the first equality: From (i), we must identify $\bHH_\bullet^k(\DCoh_Z(M))$, compute the $S^1$-action on it, and then conclude.
	By \autoref{cor:hh}, $\bHH_\bullet^k(\DCoh_Z(M)) = \RGamma_Z \ul{\bHH}_\bullet(\DCoh(M))$.  Since $M$ is regular, $\DCoh(M) \isom \Perf(M)$ and HKR identifies this inner term (de Rham complex) and its $B$ operator (de Rham differential).  Then, \autoref{prop:hkr-df} identifies the circle action with $-df \wedge -$.  Finally, the desired computation follows by noting that $\RGamma_Z$ is a right adjoint and so commutes with homotopy limits, e.g., taking $S^1$-invariants:
	\begin{align*}
	  \left[ \bHH_\bullet^k(\DCoh_Z(M)) \right]^{S^1} &=  \left[ \RGamma_Z\left(\left[\oplus_i \Omega^\bullet_M, 0\right] \right) \right]^{S^1}\\
	  &= \left[ \RGamma_Z\left(\left[\Omega^\bullet_M, 0\right]^{S^1} \right) \right]\\
	  &=  \RGamma_Z\left(\left[\Omega^\bullet_M\ps{\bt}, \bt (-df \wedge -) \right]\right) 
	\end{align*} 
	
	The second equality follows from the first, since $\otimes_{k\ps{\bt}} k\pl{\bt}$ is monoidal, upon noting that $R\Gamma_Z$ commutes with the filtered colimit of inverting $\bt$.  The third and fourth equality follow analogously from (iii) and \autoref{prop:hkr-df}. \qedhere
  \end{enumerate}
\end{proof}

\begin{remark} The presence of support conditions, and the existence of a comparison map, has a down-to-earth description in terms of \autoref{prop:s1-hom} and \autoref{lem:s1-tensor}:  Use an explicit cyclic bar construction to write (leaving the differentials implicit)
   \[ \bHH_\bullet(\DCoh(M))= \bigoplus_{n \geq 1} \bigoplus_{c_1, \ldots, c_n \in \DCoh(M)} \RHom(c_1, c_2) \otimes_k \cdots \otimes_k \RHom(c_n, c_1) \]
   Then, \autoref{lem:coh-red} and Morita-invariance of $\bHH_\bullet$ give a quasi-isomorphisms
 \begin{align*} \bHH_\bullet(\DCoh_Z(M)) &= \bigoplus_{n \geq 1} \bigoplus_{c_1, \ldots, c_n \in \DCoh_Z(M)} \RHom(c_1, c_2) \otimes_k \cdots \otimes_k \RHom(c_n, c_1) \\
&= \bigoplus_{n \geq 1} \bigoplus_{c'_1, \ldots, c'_n \in \Coh_Z(M_0)} \RHom(i_*(c'_1), i_*(c'_2)) \otimes_k \cdots \otimes_k \RHom(i_*(c'_n), i_*(c'_1))
 \end{align*}
We thus obtain a natural map 
 \begin{align*} \bHH_\bullet(\DCoh_Z(M))^{S^1} &\longleftarrow \bigoplus_{n \geq 1} \left(\bigoplus_{c'_1, \ldots, c'_n \in \Coh_Z(M_0)} \RHom(i_*(c'_1), i_*(c'_2)) \otimes_k \cdots \otimes_k \RHom(i_*(c'_n), i_*(c'_1))\right)^{S^1} \\
   &= \bigoplus_{n \geq 1} \bigoplus_{c'_1, \ldots, c'_n \in \Coh_Z(M_0)} \RHom(i_*(c'_1), i_*(c'_2))^{S^1} \otimes_{k\ps{\bt}} \cdots \otimes_{k\ps{\bt}} \RHom(i_*(c'_n), i_*(c'_1))^{S^1} \\
   &= \bHH_\bullet^{k\ps{\bt}}(\PreMF_Z(M,f))
 \end{align*}
 at least upon verifying that the above identifications are compatible with the differentials (i.e., that \autoref{prop:s1-hom} plays well with composition).  As already implicit in the above, the inner direct sum is uniformly $t$-bounded and so commutes with $(-)^{S^1}$.  From this perspective, it is not clear why the outer direct should also commutes with $(-)^{S^1}$; this is some sort of ``convergence'' statement about the cyclic bar complex.
\end{remark}

\begin{lemma}\label{lem:s1-tensor} Suppose $V, V'$ are $t$-bounded complexes with $S^1$-action, and $V \otimes_k V'$ their tensor product as complex with $S^1$-action.  Then, the natural map
  \[ V^{S^1} \otimes_{k\ps{\bt}} (V')^{S^1} \longrightarrow (V \otimes_k V')^{S^1} \]
  is an equivalence.
\end{lemma}
\begin{proof} C.f., the proof of \autoref{prop:mf-prelim}.
\end{proof}

\subsection{Calabi-Yau structures on \texorpdfstring{$\MF$}{MF}}
We first recall the notion of a Calabi-Yau structure on a smooth, not necessarily proper, dg-category (e.g., \cite[Def.~4.2.6 \& Remark~4.2.17]{Lurie-Cob}):
\begin{defn} Suppose $\C \in \dgcatidm_R$ is smooth.  An \demph{$m$-Calabi-Yau structure} on $\C$ is an $\SO(2)$-invariant cotrace
  \[ \cotr\colon R \to \ev \circ \coev[-m](R) = \bHH_\bullet(\C)[-m] \]
satisfying the following \demph{non-degeneracy condition}:
  \begin{itemize}
    \item Note that $\cotr$ gives rise to a $1$-morphism in $\Fun^L_R(R\mod,R\mod) \isom R\mod$:
      \[ \cotr(V) = \id_V \otimes_R \cotr\colon \id(V) \isom V \otimes_R R \longrightarrow V \otimes_R \bHH_\bullet(\C)[-m] \isom \ev \circ \coev[-m](V) \] 
    \item The non-degeneracy condition is that this be the co-unit of an adjunction $(\coev[-m], \ev)$, i.e., that the composite
      \[\Map_{\Fun^L_R(\Ind \C, \Ind \C)}(\coev(V)[-m], \F) \stackrel{\ev}\longrightarrow \Map_{R\mod}\left(\ev \circ \coev(V)[-m], \F\right) \stackrel{\cotr}\longrightarrow \Map_{R\mod}\left(V, \coev \F \right) \]
      be an equivalence for all $V \in R\mod$ and $\F \in \Fun^L_R(\Ind \C, \Ind \C)$.  Since $\C$ is smooth, it suffices to check that this condition is verified for $V = R[n]$, $n \in \ZZ$, and $\F \in \left(\Fun^L_R(\Ind \C, \Ind \C)\right)^c$ compact.
  \end{itemize}
\end{defn}

Of course the motivating example is:
\begin{lemma}\label{lem:cy} Suppose that $M$ is an $m$-dimensional Calabi-Yau variety (in the weak sense that $M$ is Gorenstein and $\omega_M[-m]$ is trivializable), and that $\vol_M\colon \O_M \isom \omega_M[-m]$ is a holomorphic volume form.  Then, $\vol_M$ gives rise to an $m$-Calabi-Yau structure on $\DCoh(M)$ as follows:
  \begin{itemize} \item There is a $\cotr_{[\vol_M]}\colon k \to \bHH^\bullet(\DCoh(M))[-m]$ determined by
      \[ [\vol_M] = \Delta_* \vol_M \in \Map_{M^2}(\Delta_* \O_M, \Delta_* \omega_M[-m]) = \Map_{k\mod}(k, \bHH_\bullet(\DCoh(M))[-m]) \]
  \item There is a natural $\SO(2)$-invariant lift of $[\vol_M]$, which determines $\SO(2)$-equivariance data for $\cotr_{\vol_M}$, and $\cotr_{\vol_M}$ is non-degenerate in the above sense.
\end{itemize}
  Furthermore, this determines a bijection between equivalence classes of the $m$-Calabi-Yau structures and the set of holomorphic volume forms.
\end{lemma}
\begin{proof} By assumption, $\Delta_* \O_M$ and $\Delta_* \omega_M[-m]$ are both coherent sheaves, i.e., the heart of the $t$-structure.  It follows that
  \begin{align*} \Map_{M^2}(\Delta_* \O_M, \Delta_* \omega_M[-m]) &= \Omega^\infty \RHom_{M^2}(\Delta_* \O_M, \Delta_* \omega_M[-m]) \\
    &= \Ext^0_{M^2}(\Delta_* \O_M, \omega_M[-m]) \\
    &= \Ext^0_{M}(\O_M, \omega_M[-m]) \\
    &= \Map_M(\Delta_* \O_M, \omega_M[-m])
\end{align*}
and that both spaces are discrete.  Any homotopy $\SO(2)$-action on a discrete space is trivial, so that
\[ \pi_0\left(\Map_{M^2}(\Delta_* \O_M, \Delta_* \omega_M[-m])^{\SO(2)}\right) \isom \Map_{M^2}(\Delta_* \O_M, \Delta_* \omega_M[-m]) \]
This shows that $[\vol_M]$ lifts to $\SO(2)$-invariants.  The same argument shows that $\vol_M$ is an isomorphism iff $\Delta_* \vol_M$ is so, proving the ``Furthermore.'' 

{\noindent}{\bf Claim:} $[\vol_M]$ has a \emph{natural} lift to $\SO(2)$-invariants provided by its description as a pushforward along $\Delta$.  More precisely, we will show below that $\Delta_*$ lifts to a map \[ \RHom_M(\O_M, \omega_M[-m]) \longrightarrow \RHom_{\QCsh(M^2)}(\Delta_* \O_M, \Delta_* \omega_M[-m])^{\SO(2)} \] 

Assuming the claim, we now complete the proof: We must show that $\cotr_{\vol_M}$ is non-degenerate, i.e., that the composite
\[\Map_{\Fun^L_k(\QCsh(M), \QCsh(M))}(\coev(V)[-m], \F) \stackrel{\ev}\longrightarrow \Map_{k\mod}\left(\ev \circ \coev(V)[-m], \F\right) \stackrel{\cotr}\longrightarrow \Map_{k\mod}\left(V, \ev \F \right) \] is an equivalence for all $V \in k\mod$ and $\F \in \Fun^L_k(\QCsh(M),\QCsh(M))^c \isom \DCoh(M^2)$.  By \autoref{thm:coh-fmk}, we know that $\DCoh(M)$ is smooth; so, it suffices to verify the condition for $V = k[n]$, $n \in \ZZ$, and $\F$ compact.  Using the identification of \autoref{thm:coh-fmk}, we may identify the relevant map with (global sections of shifts of)
\[ \HHom_{\DCoh(M^2)}\left(\Delta_* \omega_M[-m], \F\right) \longrightarrow \HHom_{\DCoh(M^2)}\left(\Delta_* \O_M, \F\right) \] given by pre-composing with the equivalence $\Delta_* \vol_M$.

\smallskip
Finally, we include two proofs of the claim: the first by general nonsense for which we do not give all the details, and the second much more concrete in case $M$ is smooth:\\
{\noindent}{\it First Proof: }\
      \item Base-change for the diagram
    \[ \xymatrix{ LM \ar[d]_{p_1} \ar[r]^{p_2} & M \ar[d]^\Delta \\ M \ar[r]_\Delta & M^2 } \] 
identifies
\[ \RHom_{\QCsh(M^2)}(\Delta_* \F, \Delta_* \G) = \RHom_{LM}\left((p_2)^* \F, (p_1)^! \G\right) \]
so that in particular
\[ \RHom_{\QCsh(M^2)}(\Delta_* \O_M, \Delta_* \omega_M[-m]) = \RHom_{LM}(\O_{LM}, \omega_{LM}[-m]) \]
Let $s: M \to LM$ be the inclusion of constant loops, which is naturally $\SO(2)$-equivariant.  Under the above, $\Delta_*$ is identified with
\[ \xymatrix@C=4pc{\RHom_M(\O_M,\omega_M[-m]) \ar[r]^-{s_*} & \RHom_{LM}(s_* s^* \O_{LM}, s_* s^! \omega_{LM}[-m])\ar[r]^-{\tr_s \circ (-) \circ \mathrm{unit}_s}&\RHom_{LM}(\O_{LM}, \omega_{LM}[-m])  } \]
The lift to $\SO(2)$-invariants is provided by naturality from the $\SO(2)$-equivariance of $s$.

\smallskip

{\noindent}{\it Second Proof: }\\
In case $M$ is smooth we can be completely explicit: By HKR, we may identify $\ul{\bHH}_\bullet(\Perf(M)) = \bigoplus_i \Omega^i_M[i]$ and the $\SO(2)$-action with the de Rham differential $d_{DR}$.  Then, the lift of $\Delta$ is
\[
\RGamma(\omega_M[-m])= R\Gamma(\boxed{\Omega_M^m}) \longrightarrow
 \RGamma\left(\vcenter{\xymatrix@R=.6pc{
\boxed{\Omega_M^m}      & \\
\Omega_M^{m-1} \ar[r]^d & \bt \Omega_M^m &  \\
\Omega_M^{m-2} \ar[r]^d & \bt \Omega_M^{m-1} \ar[r]^d& \bt^2 \Omega_M^{m} & \\
\cdots    &\cdots & \cdots & \ddots \\
\O_M \ar[r]^{d} & \bt \Omega^1_M \ar[r]^d &  \bt^2 \Omega^2_M \ar[r]^d & \ddots\\
& \bt \O_M  \ar[r]^d & \bt^2 \Omega^1_M \ar[r]^d &\ddots  \\
}}\right) = (\bHH_\bullet(\Perf M)[-m])^{S^1} \]
where the boxed entries are in degree $0$.  The induced map on mapping \emph{spaces} is $\Omega^\infty$ of this, which is just the identity on $H^0(\Omega_M^m)$.
\end{proof}

\begin{remark} The cotrace can also be made very explicit in the Dolbeault model (over $\CC$) for Hochschild homology: Represent $\vol_M$ by a holomorphic $(n,0)$-form, $[\vol_M] \in \Gamma(A^{n,0})$.  Then, $[\vol_M]$ is visibly a cycle in 
  \[ \left(\bHH_\bullet(\Perf M)[-m]\right)^{\SO(2)} = \left[\left(\oplus_{p,q} \Gamma(\A^{p,q})[p-q-m]\right)\ps{u}, \ol{\partial} + u \cdot \partial\right] \]
  Indeed $\ol{\partial}$ vanishes since $[\vol_M]$ is holomorphic, and $\partial$ vanishes since it is is an $(n,0)$-form.
\end{remark}

\begin{theorem}[Calabi-Yau structures]\label{thm:cy} Suppose $(M, f)$ is an LG pair, $m = \dim M$, and  that $M$ is equipped with a volume form $\vol_M\colon \O_M \isom \omega_M[-m]$.  Then, $\vol_M$ determines an $m$-Calabi-Yau structure on $\MF(M, f)$.
\end{theorem}
\begin{proof} Replacing $M$ by an open subset, we may suppose for simplicity that $0$ is the only critical point of $M$.  For the remainder of the proof, let $\MF = \MF(M, f)$, and $\Fun^L = \Fun^L_{k\pl{\bt}}(\MF, \MF)$ which we will identify with $\MF(M^2, -f \boxplus f)$ via \autoref{thm:TS} (with the support condition dropped by the reasoning of \autoref{thm:sm-proper}).  Let $k\colon (M^2)_0 \to M^2$ be the inclusion.

  The $m$-Calabi-Yau structure on $\DCoh(M)$ corresponding to $\vol_M$
  \[ [\vol_M] = \Delta_* \vol_M \in \bHH_\bullet(\DCoh(M))[-m] = \RHom_{\DCoh(M^2)}(\Delta_* \O_M, \Delta_* \omega_M[-m]) \] 
  admits the refinement
  \[\ol{[\vol_M]}^{S^1} = \ol{\Delta}_* [\vol_M] \in \RHom_{\PreMF(M^2,-f \boxplus f)}\left(\ol{\Delta}_* \O_M, \ol{\Delta}_* \omega_M[-m]\right) = \RHom_{M^2}\left( \Delta_* \O_M, \Delta_* \omega_M[-m]\right)^{S^1}  \]
  which upon inverting $\bt$ gives an element
  \[ \ol{[\vol_M]}^\Tate \eqdef \ol{\Delta}_* [\vol_M] \in \bHH_\bullet^{k\pl{\bt}}(\MF)[-m] = \RHom_{\MF(M^2,-f \boxplus f)}\left(\ol{\Delta}_* \O_M, \ol{\Delta}_* \omega_M[-m]\right) \]

  {\noindent}{\bf Claim 1:} There is an $\SO(2)$-action on $\RHom_{\PreMF(M^2,-f \boxplus f)}\left(\ol{\Delta}_* \O_M, \ol{\Delta}_* \omega_M[-m]\right)$ refining the natural $\SO(2)$-action on $\bHH_\bullet^{k\pl{\bt}}(\MF)$.\\
  {\noindent}{\bf Claim 2:} The description as a pushforward via $\ol{\Delta}$ equips $\ol{[\vol_M]}^{S^1}$ (and so $\ol{[\vol_M]}^\Tate$) with a natural lift to $\SO(2)$-invariants.\\

  Assuming the claims for now we complete the proof:   From the claim, it follows that $\ol{[\vol_M]}^\Tate$ determines an $\SO(2)$-invariant cotrace
  \[ \cotr_{\vol_M}\colon k\pl{\bt}[m] \longrightarrow \bHH^{k\pl{\bt}}_\bullet(\MF) \]
  We will be done if we can prove that $\cotr_{\vol_M}$ is \emph{non-degenerate} in the sense that it is the unit for an adjunction $(\coev[-m], \ev)$, i.e., that the composite
\[ \Map_{\Fun^L}\left( \coev[-m](V),  \F\right) \longrightarrow \Map_{k\pl{\bt}\mod}\left( \ev \circ \coev[-m](V), \ev \F\right) \stackrel{\id_V \otimes \cotr_{\vol_M}^\dual}\longrightarrow \Map_{k\pl{\bt}\mod}\left(V, \ev \F\right) \]
is an equivalence for all $V \in k\pl{\bt}\mod$ and $\F \in \Fun^L$.  Since $\MF$ is smooth (\autoref{thm:sm-proper}) it suffices to verify this condition when $V = k\pl{\bt}[n]$, $n \in \ZZ$ and $\F \in \Fun^L$ is compact.  Using the identifications of \autoref{thm:functors}, we see that it suffices to check that
\begin{align*} \RHom_{\MF(M^2, -f \boxplus f)}\left(\ol{\Delta}_* \omega_{M}[-m], \F\right) &\longrightarrow \RHom_{k\pl{\bt}\mod}\left( \RHom_{\Fun^L}\left(\ol{\Delta}_* \O_M, \ol{\Delta}_* \omega_M[-m] \right), \RHom_{\Fun^L}\left(\ol{\Delta}_* \O_M, \F\right)\right) \\
  &\longrightarrow \RHom_{k\pl{\bt}\mod}\left(k\pl{\bt}, \RHom^{k\pl{\bt}}_{\MF(M^2, -f \boxplus f)}\left(\ol{\Delta}_* \O_M, \F \right)\right) \\
  &= \RHom^{k\pl{\bt}}_{\MF(M^2, -f \boxplus f)}\left( \ol{\Delta}_* \O_M, \F \right)
  \end{align*}
  is an equivalence for all $\F \in \MF(M^2, -f \boxplus f)$.  The composite is just given by pre-composition with $\ol{[\vol_M]}^{\Tate}$, so it suffices to observe that $[\ol{\vol_M}]^\Tate$ is an equivalence:  It is the image by a functor of $\ol{[\vol_M]}^{S^1} = \ol{\Delta}_* \vol_M$, and $\vol_M$ is an equivalence.

  {\noindent}{\it Proof of Claims:} Let $B\oh{\GG_a}$ act on $\Perf(M)$ corresponding to $f$ (\autoref{lem:BGa-act-mf}), and on $\Perf(M^2)$ corresponding to $-f \boxplus f$.  By functoriality we know that $B\oh{\GG_a}$-acts on $\bHH_\bullet^k(\Perf(M))$ compatibly with the $\SO(2)$-action, and this induces an $\SO(2)$ action on \[ \bHH_\bullet^k(\Perf(M))^{B\oh{\GG_a}} = \RHom_{M^2}(\Delta_* \O_M, \Delta_* \omega_M)^{B\oh{\GG_a}} = \RHom^{\otimes k\ps{\bt}}_{\PreMF(M^2, -f \boxplus f)}(\ol{\Delta}_* \O_M, \ol{\Delta}_* \omega_M) \] which evidently refines that on $\bHH^{k\pl{\bt}}_\bullet(\MF(M,f)) = \bHH_\bullet^k(\Perf(M))^{\Tate}$.  We must produce a lift of $[\vol_M] \in \Map_{M^2}(\Delta_* \O_M, \Delta_* \omega_M[-m])$ to $\SO(2)\times B\oh{\GG_a}$ invariants.

The mere existence of a lift is actually automatic: Since
\[ \Map = \Map_{M^2}(\Delta_* \O_M, \Delta_* \omega_M[-m]) = \Omega^\infty \RHom_{M^2}(\Delta_* \O_M, \Delta_* \omega_M[-m]) \] with both $\Delta_* \O_M$ and $\Delta_* \omega_M[-m]$ in the heart of the $t$-structure, this space is \emph{discrete}.  Regarding it as a complex in degree $0$, it has an action of $\SO(2) \times B\oh{\GG_a}$ such that the map to the whole Hochschild complex is equivariant.  But, since it is in degree $0$ the action cannot help but be trivializable.  So, $[\vol_M]$  admits a lift to fixed points which is unique up to equivalence (but not necessarily up to contractible choices); and similarly, $\pi_0 \Map^{\SO(2) \times B\oh{\GG_a}} = \pi_0 \Map = \Map$.

As before, it is possible to make the choice naturally (i.e., dependent only on some universal choice).  We describe how to do this in the case where $M$ is a smooth variety, so that we can use HKR descriptions: Will will produce a lift
\[ \RHom_M(\O_M, \omega_M[-m]) \longrightarrow \RHom_{\PreMF(M^2, -f \boxplus f)}(\ol{\Delta}_* \O_M, \ol{\Delta}_* \omega_M[-m])^{\SO(2)}\] 
of $\ol{\Delta}_*$  in the HKR model of \autoref{thm:hoch}.  There's an obvious map
\[ \RGamma(\Omega_M^m) \to \left(\bHH^k_\bullet(\Perf M)[-m]\right)^{B\oh{\GG_a} \times \SO(2)} = \RGamma\left(\left[\Omega^\bullet_M[-m]\ps{\bt, u}, \bt \cdot (-df \wedge -) + u \cdot d \right] \right) \]
since $\Omega_M^m$ is the degree $0$ piece of the complex of sheaves on the right, and this piece has no differentials into it (there's nothing in positive degree) or out of it (both $d$ and $-df \wedge -$ vanish for degree reasons).  Again, the map on connective covers can be identified with the identity on $H^0(M, \Omega_M^m)$.
\end{proof}

\begin{remark} The Claim is also apparent in a Dolbeault model (over $\CC$): If $\vol_M \in \Gamma(\A^{m,0})$ is a holomorphic volume form, it evidently gives rise to a cycle in
  \[ \left(\bHH^k_\bullet(\Perf M)[-m]\right)^{B\oh{\GG_a} \times \SO(2)} = \left[\left(\oplus_{p,q} \Gamma(\A^{p,q})[p-q-m]\right)\ps{\bt, u}, \ol{\partial} + \bt \cdot (-df \wedge -) + u \cdot \partial \right] \]
  Indeed, $\ol{\partial} \vol_M$ vanishes since $\vol_M$ is holomorphic, while $\partial \vol_M$ and $-df \wedge \vol_M$ vanish since they would have to be $(m+1,0)$-forms.
\end{remark}

\section{Quadratic bundles}\label{sec:quadratic}
The goal of this section is two-fold:
\begin{itemize}
  \item We carry out a first class of computation of $\PreMF$, in the spirit of Kapustin-Li: For non-degenerate quadratic bundles over a space, $\PreMF$ (with supports along the zero section) admits a description in terms of a $k\ps{\bt}$-linear variant of sheaves of Clifford algebras.  (Upon inverting $\bt$, this recovers a relative form of the computations of Kapustin-Li for matrix factorizations.)
  \item We use a variant of \autoref{thm:TS} to prove a a relative form of Kn\"orrer periodicity for metabolic quadratic bundles, and so re-construct using $\MF$ a $2$-periodic version of the Clifford invariant on the Witt group.  Note that Kn\"orrer-type result is valid only after inverting $\bt$.
\end{itemize}

\subsection{Metabolic quadratic bundles and relative Kn\"orrer periodicity}
\begin{defn} A \demph{quadratic bundle} $(\sQ,Q)$ over a scheme $X$ is a pair consisting of:  a locally free sheaf $\sQ$, and a \emph{non-degenerate} symmetric bilinear pairing $Q\colon \sQ \otimes_{\O_X} \sQ \to \O_X$.\footnote{Non-degenerate means that the induced sheaf map $\sQ \to \sQ^\dual$ is an isomorphism.}
\end{defn}

\begin{na} We associate to a quadratic bundle $(\sQ,Q)$ over $X$:
  \begin{itemize}
    \item The total space $\tQ \to X$, a scheme smooth over $X$: $\tQ = \AA(\sQ^\dual) \eqdef \Spec_X \Sym_{\O_X} \sQ^\dual$.
    \item The quadratic form $q\colon \tQ \to \AA^1$: defined on points by $q(v) = \frac{1}{2} Q(v \otimes v)$ (or on sheaves, by $\O_X \to \sQ^\dual \otimes \sQ^\dual \to \Sym^2(\sQ^\dual)$).
  \end{itemize}
  We will regard $(\tQ, q)$ as an LG pair.
\end{na}

\begin{lemma}\label{lem:quad-crit} Let $(\sQ,Q)$ be a quadratic bundle over a smooth scheme $X$, and $(\tQ, q)$ the resulting LG pair.  Then:
  \begin{enumerate}
    \item $\crit(q) = X$, in particular $0$ is the only critical value.
    \item There is a natural identification $\N_{X/\tQ} = \sQ$, under which $Q$ corresponds to the Hessian.  In particular, $q$ is Morse-Bott.
  \end{enumerate}
\end{lemma}
\begin{proof}\mbox{}
  \begin{enumerate}
    \item Follows from the condition that $Q$ is non-degenerate:  Working locally, we may suppose $\sQ = \oplus_{i=1}^{n} \O_X v_i$, so that $\O_{\tQ} = \O_X[x_1,\ldots,x_n]$ ($x_i$ dual to $v_i$) and $q = \frac{1}{2} \sum_{i,j} Q(v_i,v_j) x_i x_j$.  Then, $\crit(q)$ is cut out by the equations
      \[ 0 = \frac{dq}{dx_i} = \sum_j Q(v_i, v_j) x_j \qquad \text{for $i=1,\ldots,n$}\]
      \[ 0 = \xi \cdot dq = \frac{1}{2} \sum_{i,j} \left(\xi \cdot dQ(v_i,v_j)\right) x_i x_j \qquad \text{for $\xi \in TX$} \]
      The first set of equations may be reformulated as the vanishing of the vector $Q (x_1,\ldots,x_n)^T$.  Since $Q$ was assumed non-degenerate, this cuts out precisely the locus $x_1=\ldots,x_n=0$, i.e., $X$.  The second set of equations are contained in the ideal generated by the first, i.e., they vanish along $X$ as well.
    \item The (dual) identification is routine: $\N^\dual_{X/\tQ} = \res{\I_X}{X} = \res{\sQ^\dual \otimes \O_{\tQ}}{X} = \sQ^\dual$.  The previous computation in local coordinates shows that
      \[ \frac{dq}{dx_i dx_j} = Q(v_i, v_j) \] which, tracing through the identification in this case, proves the claim about the Hessian.\qedhere
  \end{enumerate}
\end{proof}

\begin{remark} The formal Morse Lemma tells us that the LG pairs $(\tQ,q)$ are (formally locally) representative of LG pairs with Morse-Bott singularities.
\end{remark}

\begin{na} Suppose $X$ is a smooth variety. Regard $\Perf(X)^{\otimes}$ as a symmetric monoidal $\infty$-category, and let $\Perf(X)\ps{\bt}^{\otimes} = \Perf(X) \otimes_k k\ps{\bt}$ (resp., $\Perf(X)\pl{\bt}^{\otimes} = \Perf(X) \otimes_k k\pl{\bt}$) be the associated $k\ps{\bt}$- (resp., $k\pl{\bt}$-)linear symmetric monoidal $\infty$-category.  If $\C, \D$ are $\Perf(X)\ps{\bt}$- (resp., $\Perf(X)\pl{\bt}$-)module dg-categories, let us denote
  \[ \C \otimes_{X\ps{\bt}} \D \eqdef \C \otimes_{\Perf(X)\ps{\bt}} \D \qquad \left(\text{resp., }\C \otimes_{X\pl{\bt}} \D \eqdef \C \otimes_{\Perf(X)\pl{\bt}} \D \right) \]
\end{na}

\begin{na} Earlier in the paper, we noted that Kn\"orrer periodicity could be deduced from our Thom-Sebastiani Theorem together with an explicit computation of matrix factorizations for a rank $2$ quadratic form.  Part (ii) of the following Theorem provides a globalized version of Kn\"orrer periodicity.  Part (i) of the following Theorem provides a relative form of the Thom-Sebastiani Theorem, under an additional hypothesis:
\end{na}

\begin{theorem}[Relative Kn\"orrer Periodicity]\label{thm:rel-knorrer} Suppose $X$ is a smooth variety, $(\sQ,Q)$ is a quadratic bundle over $X$, and $(\tQ, q)$ is the associated LG pair.
  \begin{enumerate}
    \item Suppose $(Y,f)$ is a relative LG pair over $X$: That is $Y$ is a smooth $X$-scheme equipped with a map $f$ to $\AA^1$.  For any closed subset $Z \subset f^{-1}(0)$, exterior tensor product induces $\Perf(X)\ps{\bt}$- (resp., $\Perf(X)\pl{\bt}$-)linear equivalences
      \[ \PreMF_Z(Y,f) \otimes_{X\ps{\bt}} \PreMF_X(\tQ,q) \stackrel{\sim}\longrightarrow \PreMF_{Z \times_X X}(Y \times_X \tQ, f \boxplus q) \]
      \[ \MF_Z(Y,f) \otimes_{X\ps{\bt}} \MF(\tQ,q) \stackrel{\sim}\longrightarrow \MF_{Z \times_X X}(Y \times_X \tQ, f \boxplus q) \]
    \item Suppose $(\sQ, Q)$ is a \demph{metabolic} quadratic bundle, i.e., there is a locally free subsheaf $\L \subset \sQ$ such that $\L = \L^\perp$ and $\L$ is locally a direct summand (i.e., a subbundle).  Let $\tL = \Spec \Sym_{\O_X} \L^\dual$ be the total space of $\L$.  Regard $\tL$ as a closed subscheme of $\tQ_0$, so that $\O_{\tL}$ is an object of $\DCoh(\tQ_0)$ and thus of $\MF(\tQ,q)$.  Then, tensoring with $\O_{\tL}$ induces an equivalence
      \[ \Perf(X)\pl{\bt} = \MF(X, 0) \longrightarrow \MF(\tQ, q). \]
  \end{enumerate}
\end{theorem}
\begin{proof} \mbox{}
  \begin{enumerate}
    \item The first equivalence follows from the proof of \autoref{thm:TS}, replacing the reference to \autoref{prop:fmk-coh-surj} with \autoref{prop:ext-smth-fmk} (applied to the second factor, since $X$ is certainly always smooth over $X$).  The second equivalence follows from the first.
    \item It suffices to prove the Ind-completed version, i.e., that $\QC(X)\ps{\bt} \to \MF^\kinfty(\tQ,q)$ is an equivalence.  Both sides are \'etale sheaves (\autoref{prop:mf-sheaf}) and the functor is evidently local, so that the claim is local.  We are thus free to assume that $X = \Spec R$.  It now suffices to verify the following two claims:\\
     {\noindent}{\bf Claim 1:} $\O_{\tL}$ generates $\MF(\tQ,q)$ (recall, $X$ is affine).\\
     Note that $\tL \supset X = \crit(\tQ_0)$, so the inclusion $\MF_{\tL}(\tQ,q) \to \MF(\tQ,q)$ is an equivalence by \autoref{prop:orlov-cpltn}.  It thus suffices to show that $\O_{\tL}$ generates $\DCoh_{\tL}(\tQ_0)$.  By \autoref{lem:coh-red}, $\DCoh_{\tL}(\tQ_0)$ is generated by the image of $i_*\colon \DCoh(\tL) \to \DCoh_{\tL}(\tQ_0)$ so that it suffices to show that $\O_{\tL}$ generates $\DCoh(\tL) = \Perf(\tL)$.  Since $X$ was assumed affine, so is $\tL$ and the claim follows by the Hopkins-Neeman-Thomason Theorem.

    \smallskip 
     
    {\noindent}{\bf Claim 2:} The natural map $\O_X\ps{\bt} \to \RHom^{\otimes}_{\tQ_0}\left(\O_{\tL},\O_{\tL}\right)$ becomes an equivalence after $- \otimes_{k\ps{\bt}} k\pl{\bt}$.\\
     The claim is local on $X$, so that we may assume that
     \begin{itemize}
       \item There are trivializations $\L \isom \oplus_{i=1}^{r} \O_X \cdot y_i$ and $\L^\dual \isom \oplus_{i=1}^{r} \O_X \cdot x_i$.
       \item $(\sQ,Q)$ is not just metabolic, but hyperbolic (see e.g., Bass' work on quadratic forms over rings): i.e., there exists an isomorphism $(\sQ,Q) \isom H(\L) \eqdef (\L \oplus \L^{\dual}, Q_H)$ where $Q_H$ is just the natural duality pairing pairing. 
     \end{itemize}
     In terms of the above local identifications:
     \[ \O_{\tL} = \O_X[x_1,\ldots,x_r], \quad \O_{\tQ_0} = \O_X\left[\begin{gathered}x_1,\ldots,x_r,\\y_1,\ldots,y_r\end{gathered}\right]/q, \text{ and} \quad q = \sum_{i=1}^{r} x_i y_i \]
     Writing $\O_{\tQ_0} \sim \left(\Sym_{\O_X} \sQ^\dual\right)[\epsilon]/d\epsilon = q$, we are led to the following Koszul-Tate resolution of $\O_{\tL}$ over $\O_{\tQ_0}$:
     \begin{align*} \O_{\tL} &\sim \Kos_{\O_{\tQ_0}}\left(\L^\dual \otimes \O_{\tQ_0} \stackrel{y_i}\longrightarrow \O_{\tQ_0}\right)\left[\frac{u^k}{k!}\right]/\left\{du=-e\right\} \\ &\sim \underbrace{\O_{\tQ_0}[\delta_1,\ldots,\delta_r,u^k/k!]}_{\deg \delta_i = +1, \deg u = +2}/\left\{\begin{gathered} \delta_i^2 = 0\\ d\delta_i = y_i \\ du = -\sum x_i \delta_i \end{gathered}\right\} \end{align*}
       where $e \in \Kos_{\O_{\tQ}}(\L^\dual \otimes \O_{\tQ} \to \O_{\tQ})$ satisfies $d(e) = q$ (in local coordinates, it may be given by the formula above).  The natural $\O_X\ps{\bt}$ action on this resolution admits the following description: $\bt$ acts by $\bt = d/du$, while $\O_X$ acts by multiplication.  It remains to use the resolution to compute $\RHom^{\otimes}_{\tQ_0}(\O_{\tL},\O_{\tL})$ as $\O_X\ps{\bt}$-module, and show that after inverting $\bt$ it is a free module on the identity morphism.  Dualizing the differentials, we readily compute:
       \[ \RHom^{\otimes}_{\tQ_0}(\O_{\tL},\O_{\tL}) = \underbrace{\O_{\tL}\ps{\bt}\left[\gamma_1,\ldots,\gamma_r\right]}_{\deg \gamma_i = -1, \deg \bt = -2}/\left\{\begin{gathered} \gamma_i^2 = 0 \\ d\gamma_i = - x_i \bt\end{gathered}\right\} \] 
	 where $1$ is the identity map.  Note that the ``$y_i$''~Koszul differentials have gone to zero, and we are left only with the differentials in the ``$u$-direction.'' Moreover these remaining differentials are all part of potentially truncated ``$-x_i$''~Koszul complexes.  Upon inverting $\bt$, the truncation disappears and we obtain a splitting into shifts of a Koszul complexes resolving $\O_X$:
	 \begin{align*} \RHom^{\otimes}_{\tQ_0}(\O_{\tL},\O_{\tL}) \otimes_{k\ps{s}} k\pl{\bt} &= \O_{\tL}\pl{\bt}\left[\gamma_1,\ldots,\gamma_r\right]/\left\{\begin{gathered} \gamma_i^2 = 0 \\ d\gamma_i = - x_i t\end{gathered}\right\}\\
	   &= \bigoplus_{i \in \ZZ} t^i \Kos_{\O_{\tL}}\left( t^{-1} \L \otimes \O_{\tL} \stackrel{-x_i}\longrightarrow \O_{\tL}  \right) \\
	   &\isom \bigoplus_{i \in \ZZ} t^i \O_X = \O_X\pl{\bt}\qedhere
	 \end{align*}
  \end{enumerate}
\end{proof}

\subsection{Witt group and ``derived Azumaya algebras''}
\begin{na} Suppose $(\sQ_1, Q_1)$, $(\sQ_2, Q_2)$ are quadratic bundles over $X$, with associated LG pairs $(\tQ_1, q_1)$, $(\tQ_2, q_2)$.  Form the ``orthogonal sum'' $(\sQ_1 \oplus \sQ_2, Q_1 \perp Q_2)$; its associated LG pair will be $(\tQ_1 \times_X \tQ_2, q_1 \boxplus q_2)$.

Define the \demph{Witt semigroup} $W^s(X)$ of $X$ to be the semi-group of (isomorphism classes of) quadratic bundles over $X$, equipped with orthogonal sum.  Define the \demph{Grothendieck-Witt group} $GW(X)$ to be the Grothendieck group of the Witt semigroup.  Define the \demph{Witt group} $W(X)$ to be the quotient of $GW(X)$ by the subgroup generated by metabolic quadratic bundles.  

Any element of $GW(X)$ may be written in the form $\sQ_1 - \sQ_2$.  Letting $\ol{\sQ}_2$ denote $\sQ_2$ equipped with the negative quadratic form, we may rewrite 
\[ \sQ_1 - \sQ_2 = \left(\sQ_1 \perp \ol{\sQ}_2\right) - \left(\sQ_2 \perp \ol{\sQ}_2\right) \] where now $\sQ_2 \perp \ol{\sQ}_2$ is \emph{metabolic} (with Lagrangian subspace $\L = \Delta_{\sQ_2}$ the diagonal).  In particular, $W(X)$ is the quotient semigroup of $W^s(X)$ by the metabolic elements.
\end{na}

Thus \autoref{thm:rel-knorrer} implies
\begin{corollary} The assignment 
  \[ (\sQ, Q) \longrightarrow \MF(\tQ, q) \]
  takes orthogonal sum of quadratic bundles to tensor product of $\infty$-categories over $\Perf(X)\pl{\bt}$.  It takes isomorphisms to equivalences.  It takes metabolic bundles to the tensor unit (i.e., $\MF(X,0) = \Perf(X)\pl{\bt}$).  Therefore, it descends to a group homomorphism
  \[ (W(X), \perp) \longrightarrow \left\{\begin{gathered} \text{Equivalence classes of invertible}\\\text{$\Perf(X)\pl{\bt}$-linear $\infty$-categories}\end{gathered}, -\otimes_{X\pl{\bt}}-\right\} \]
\end{corollary}

\begin{remark} In the statement of the previous Corollary, ``invertible'' means in the sense of invertible object for the tensor product $-\otimes_{X\pl{\bt}}-$.  The right hand side is thus a $2$-periodic version a ``derived Azumaya algebra'' of To\"en \cite{Toen-Azumaya}.
\end{remark}

\subsection{Relation to Clifford bundles}
\begin{na}
  At this point (if not earlier), the conscientious reader should object: There's a more down-to-Earth construction of (usual) Azumaya algebras out of elements in the Witt group, by taking the bundle of Clifford algebras associated to $\sQ$.  The following Theorem explains this.  Since it's proof is independent of the above, we could presumably have proven part (ii) of \autoref{thm:rel-knorrer} in the world of Clifford algebras.\footnote{Though I'm not aware of the desired $\ZZ/2$-graded Morita equivalence appearing in the literature in the metabolic case.  If $X$ is affine, then any metabolic bundle is hyperbolic and the Morita equivalence is well-known and visibly $\ZZ/2$-graded.  \cite{Knus-Ojanguren} shows that the Brauer class does vanish in the metabolic case, but that it is is not necessarily the case that $\Cliff_{\O_X}(\sQ) \isom \End_{\O_X}(\bigwedge^* \L)$ as one might naively guess.}
\end{na}

\begin{na} Suppose $(\sQ,Q)$ is a quadratic bundle on a scheme $X$.  Then, $\Cliff_{\O_X}(\sQ)$ is the following sheaf of $\ZZ/2$-graded algebras
  \[ \Cliff_{\O_X}(\sQ)_{\ZZ/2} \eqdef \O_X \langle \sQ \rangle / \left\{v^2 = - Q(v,v)\right\} \] where $\sQ$ is in odd degree, and $v$ denotes a section of $\sQ$.
\end{na}
\begin{na} In this paper, it has been our convention to replace $\ZZ/2$-graded objects with $\ZZ$-graded objects over $k\pl{\bt}$, $\deg \bt = -2$.  Under this equivalence, the above sheaf of algebras goes to
  \[ \Cliff_{\O_X}(\sQ) \eqdef \O_X\pl{\bt} \langle \sQ \rangle / \left\{v^2 = - Q(v,v) \bt\right\} \] where $\sQ$ is in degree $-1$, and $\bt$ is in degree $-2$.  There is also a $k\ps{\bt}$-linear version:
  \[ \PreCliff_{\O_X}(\sQ) \eqdef \O_X\ps{\bt} \langle \sQ \rangle /\left\{v^2 = - Q(v,v) \bt\right\} \]
\end{na}

\begin{theorem}[Relative Kapustin-Li]\label{thm:mf-cliff} Suppose $X$ is a smooth scheme, $(\sQ,Q)$ a quadratic bundle on $X$, and $(\tQ, q)$ the associated LG pair.  Then, the structure sheaf $\O_X$ induces a natural equivalence of $k\ps{\bt}$-linear dg-categories
  \[ \PreMF^\kinfty_X(\tQ, q) \isom \PreCliff_{\O_X}\!(\sQ)\mod(\QC(X)) \] and of $k\pl{\bt}$-linear dg-categories
  \[ \MF^\kinfty(\tQ, q) \isom \Cliff_{\O_X}\!(\sQ)\mod(\QC(X)) = \Cliff_{\O_X}\!(\sQ)_{\ZZ/2}\mod_{dg\ZZ/2}(\QC(X)) \]
\end{theorem}
\begin{remark} Before giving a complete proof of the Theorem, we should point that it is in essence a straightforward computation of a relative Koszul dual over $\O_X$ (with a little extra book-keeping for the $k\ps{\bt}$-action). In local coordinates, it is saying that the Koszul dual of the dg-algebra
  \[ \O_{\tQ_0} \sim \O_X[x_1,\ldots,x_n][\epsilon]/d\epsilon=q \]
  is $\PreCliff_{\O_X}(\sQ)$ viewed as a filtered algebra (depending on the differential above), having associated graded
  \[ \O_X[\delta_1,\ldots,\delta_n][t] \]
  \end{remark}
  \begin{proof} We first outline our plan of proof:  We first use descent to reduce to the affine case, and note that  (locally on $X$) $\PreMF_X(\tQ, q)$ (and $\MF(\tQ, q)$) are generated by $\O_X$.  It thus suffices to identify $\RHom^{\otimes}_{\tQ_0/X}(\O_X,\O_X)$ with $\PreCliff_{\O_X}(\sQ)$ and analogously for $\MF$.  In Step~2, we will explicitly construct a resolution on which we can see the $k\ps{\bt}$-action and use this to compute the underlying complex of the endomorphisms in Step~3.  Finally, to identify the algebra structure we explicitly describe the Clifford action on this resolution in Step~4.
\smallskip

  {\noindent}{\bf Step 1: Identifying the generator.}\\
  Note that both sides are \'etale sheaves on $X$ by \autoref{app:descent}.  In the following steps, we will identify $\RHom^{\otimes X}_{\PreMF(\tQ,Q)}(\O_X \O_X) \isom \PreCliff_{\O_X}(\sQ)$ as sheaves of algebras, i.e., objects in $\Alg(\QC(X))$; note that this implies the analogous identifies on $- \otimes_{k\ps{\bt}} k\pl{\bt}$.  Then, the functor in question will be $\RHom^{\otimes X}_{\PreMF^\kinfty(\tQ,Q)}(\O_X, -)$ (resp., $\RHom^{\otimes X}_{\MF^\kinfty}$), factored through $\RHom^{\otimes X}_{\PreMF^\kinfty(\tQ,Q)}(\O_X, -)$-modules (resp., $\MF^\kinfty$); in particular, the functor is local on $X$ by \autoref{lem:rhom}.  To complete the proof it then suffices to show that $\RHom^{\otimes X}_{\PreMF^\kinfty(\tQ,Q)}(\O_X, -)$ (resp. $\MF$) is an equivalence locally on $X$.  To do this, it is enough by Morita theory to note that locally on $X$ $\O_X$ generates $\PreMF_X(\tQ,Q)$ by \autoref{lem:coh-red}; and $\O_X$ generates $\MF(\tQ,Q)$ by the preceding, \autoref{lem:quad-crit}, and \autoref{prop:orlov-cpltn}.

\medskip

{\noindent}{\bf Step 2: Constructing the resolution.}\\  Let $j\colon \tQ_0 \to \tQ$ be the inclusion.  For $\F \in \DCoh(\tQ_0)$, recall the functorial resolution of \autoref{ex:resolution}
  \[ \F \stackrel{\sim}{\leftarrow} \Tot^{\oplus} \left\{ j^* j_* \F \stackrel{B}{\leftarrow} \Sigma j^* j_* \F \stackrel{B}{\leftarrow} \Sigma^2 j^* j_* \F \leftarrow \cdots \right\} \]
  
  We begin with the Koszul resolution of $j_*(i_* \O_X)$ over $\O_{\tQ}$, which we will think of in two ways:
  \[ j_*(i_* \O_X)\stackrel{\sim}{\leftarrow} \Kos_{\O_{\tQ}}\left(m\colon \sQ^\dual \otimes_{\O_X} \O_{\tQ} \to \O_{\tQ} \right) = \left(\Omega^\bullet_{\tQ/X}, i_E = \sum_i x_i \partial_{x_i} \right)\]
  where $m$ is the ``multiplication'' map (recall, $\O_{\tQ} = \Sym_{\O_X} \sQ^\dual$).  Here we have identified $\sQ^\dual \otimes_{\O_X} \O_{\tQ} = \Omega^1_{\tQ/X}$, and so have identified the Koszul complex with the relative differential forms.  The multiplication map gives rise to a differential on this, which can be described as contraction $i_E$ with an ``Euler vector field'' $E = \sum_i x_i \frac{\partial}{\partial x_i}$.  Pulling back to $\tQ_0$ we obtain
  \[ j^* j_*(i_* \O_X)\stackrel{\sim}{\leftarrow} \Kos_{\O_{\tQ_0}}\left(m\colon \sQ^\dual \otimes_{\O_X} \O_{\tQ_0} \to \O_{\tQ_0} \right) =\left(j^* \Omega^\bullet_{\tQ/X}, i_E \right) \]

It remains to compute the map $B\colon \Sigma j^* j_* \F \to j^* j_* \F$ in these terms:  It is (the restriction to $\O_{\tQ_0}$ of) left-multiplication by the section
\[ \frac{dq}{2}\in \Gamma(\tQ, \Omega^1_{\tQ/X})  \qquad \text{or} \qquad \frac{1}{2} \sum_i x_i \otimes \frac{dq}{dx_i} \in \Gamma(\tQ, \sQ^\dual \otimes_{\O_X} \O_{\tQ}) \] This satisfies $i_E(B x) = q \cdot x - B i_E(x)$ since $q$ is homogeneous of degree $2$ so that $i_E(dq/2)=q$.
  
Putting this together, we obtain the following resolution for $i_* \O_X$ as $\O_{\tQ_0}$-module.  For convenience, we follow the Tate convention of writing Koszul-type complexes in terms of graded-commutative (divided-power) algebras:
\[ i_* \O_X \stackrel{\sim}\leftarrow (j^*) \left(\O_{\tQ_0}\left[\underbrace{ \Sigma(\sQ^\dual \otimes_{\O_X} \O_{\tQ_0}) }_{\deg=+1}\right]\left[\underbrace{u^k/k!}_{\deg u=+2}\right], \begin{gathered} d(x \otimes a) = x a \\ d(u) = dq/2 \end{gathered}\right) = \left[\left(j^* \Omega^\bullet_{\tQ/X}\right)[u^k/k!], i_E + \frac{dq}{2} \frac{\partial}{\partial u} \right]\]

  \medskip

  {\noindent}{\bf Step 3: Identifying the underlying complex.}\\ Using the previous complex, we readily compute that (as object in $\QC(X)$, ignoring the algebra structure)
  \begin{align*}
    \HHom^{\otimes_X}_{\tQ_0}(i_* \O_X, i_* \O_X) &= \HHom^{\otimes_X}_{\tQ_0}\left(\left[j^* \Omega^\bullet_{\tQ/X}[u^k/k!], m + (dq/2) \cdot \partial/\partial u\right], i_* \O_X\right) \\
&= \HHom_{X}\left(\left[i^* j^* \Omega^\bullet_{\tQ/X}[u^m/m!], 0\right], \O_X\right) \\
&= \O_X\ps{\bt}\left[\underbrace{\sQ^\dual[-1]}_{\deg = -1}\right] \end{align*}
    Indeed, $\RHom^{\otimes}$ takes the hocolimit (i.e., $\Tot^{\oplus}$) in the first variable to a holim (i.e., $\Tot^{\Pi}$), and all the differentials vanish.

\medskip

  {\noindent}{\bf Step 4: Producing the algebra map.}\\ We wish to produce a map of (sheaves of dg) algebras
  \[\phi\colon \PreCliff_{\O_X}(\sQ) \longrightarrow \RHom^{\otimes}_{\tQ_0/X}(i_* \O_X, i_* \O_X) \] 
  
  To construct the map, we make $\PreCliff_{\O_X}(\sQ)$ act on our explicit resolution: $\O_X$ acts via $\O_X \to \O_{\tQ_0}$; $\bt$ acts by $d/du$ (i.e., shifting the resolution in the $u$).  It remains to describe the action for $v \in \sQ = T_{\tQ/X}$, and show that it satisfies the Clifford relations and (anti-)commutes with the differentials.  We define the action of $v \in \Gamma(T_{\tQ/X})$ by contracting $i_v$ where we imagine $u$ as standing in for the Hessian tensor; explicitly:
  \begin{itemize}
    \item On the Koszul piece, $\Omega^{\bullet}_{\tQ/X}$, $v$ acts by contraction $i_v$.
    \item On $u$, $v$ acts by taking it to the contraction of the Hessian of $q$ (i.e., $Q$) by $v$:
      \[ v \cdot u = i_{v} \Hess(q) \left(= i_v \left( \sum_{i,j} dx_i dx_j Q(v_i, v_j) \right) = \sum_{i} dx_i \frac{dq}{dx_i dv}\right) \]
    \item We extend by requiring the action to be by derivations
      \[ v \cdot \left(\omega \frac{u^k}{k!}\right) = i_v(\omega) \cdot \frac{u^k}{k!} + (-1)^{|\omega|} \left(\omega \wedge (i_v \Hess(q))\right) \frac{u^{k-1}}{(k-1)!} \] and linearity.
  \end{itemize}

By explicit computation, this satisfies the Clifford relations
\begin{align*}
  v\cdot\left(v \cdot (\omega \frac{u^k}{k!})\right) &= v\cdot\left( i_v(\omega) \cdot \frac{u^k}{k!} + (-1)^{|\omega|} \omega \wedge (i_v \Hess(q)) \frac{u^{k-1}}{(k-1)!}\right) \\
  &= (-1)^{|\omega|-1} i_v(\omega) \wedge (i_v \Hess(q)) \frac{u^{k-1}}{(k-1)!} + (-1)^{|\omega|} i_v(\omega) \wedge (i_v \Hess(q)) \frac{u^{k-1}}{(k-1)!} \\
  & + (-1)^{|\omega| + |\omega|-1} \omega \wedge i_v(i_v \Hess(q)) \frac{u^{k-1}}{(k-1)!}\\
  &= -\omega \wedge Q(v,v) \frac{u^{k-1}}{(k-1)!} = (-Q(v,v) \bt) \cdot \left(\omega \frac{u^k}{k!}\right)
\end{align*}
and (anti-)commutes with the differential
\begin{align*}
  v \cdot d\left(\omega \frac{u^k}{k!}\right) &= v \cdot \left( i_{E}(\omega) \frac{u^k}{k!} + \left(\frac{dq}{2} \wedge \omega\right) \frac{u^{k-1}}{(k-1)!} \right) \\
  &= i_v(i_E(\omega)) \frac{u^k}{k!} + (-1)^{|\omega|-1} i_E(\omega) \wedge (i_v \Hess(q)) \frac{u^{k-1}}{(k-1)!} \\
  &+ i_v(\frac{dq}{2} \wedge\omega) \frac{u^{k-1}}{(k-1)!} + (-1)^{|\omega|+1} (dq/2 \wedge \omega \wedge (i_v \Hess(q))) \frac{u^{k-2}}{(k-2)!}
\intertext{while}
  d\left(v \cdot (\omega \frac{u^k}{k!}) \right) &= d\left(i_v(\omega) \frac{u^k}{k!} + (-1)^{|\omega|} \omega \wedge (i_v \Hess(q)) \frac{u^{k-1}}{(k-1)!} \right) \\
  &= i_E(i_v(\omega)) \frac{u^k}{k!} + (-1)^{|\omega|} i_E(\omega \wedge (i_v \Hess(q))) \frac{u^{k-1}}{(k-1)!} \\
  &+ \frac{dq}{2} \wedge i_v(\omega) \frac{u^{k-1}}{(k-1)!} + (-1)^{|\omega|} \frac{dq}{2} \wedge \omega \wedge (i_v \Hess(q)) \frac{u^{k-2}}{(k-2)!}
\end{align*}

\medskip

  {\noindent}{\bf Step 5: Checking equivalence.}\\ We now verify that the algebra map $\phi$ of Step 4 induces an equivalence on underlying complexes, by using the description of the underlying complex given in Step 3.   
  
The first observation is that $\bt$ goes to $\bt$: More precisely, the isomorphism of Step 3 is in fact an isomorphism of $\O_X\ps{\bt}$-modules; in Step 3, $\bt^k$ was dual to $u^k/k!$, which is compatible with $\bt$ acting by $d/du$.  Thus, $\phi$ is a map of locally free $\O_X\ps{\bt}$-modules of the same rank, and it suffices to work locally and match up generators.  This is straightforward.
\end{proof}

\begin{remark}
  Sheaves of Clifford algebras and quadratic bundles also appear in Kuznetsov's homological projective duality (\cite{Kuznetsov}).  The relationship of those results to the previous Theorem can be loosely summarized as a relative version of the  \emph{LG/CY correspondence} (\cite{Orlov-LGCY}):  Let $\PP(\tQ_0) \subset \PP(\sQ^\dual)$ denote the bundle of projective quadrics associated to $(\sQ, Q)$.  Then, the LG/CY correspondence asserts (assume $\dim \sQ \geq 2$) the existence of a fully-faithful functor $\DSing^\gr\! \tQ_0  \hookrightarrow \DCoh \PP(\tQ_0)$.\footnote{In general, the picture is a little delicate due to the interaction of ``$\GG_m$-weight gradings'' with duality: The direction of the functor depends on some numerology.  Roughly, it is an equivalence if the projective zero-locus is relative Calabi-Yau, fully-faithful in the direction indicated if it is Fano, and fully-faithful the other way if it is of general type.} 
\end{remark}

\begin{na}
  Kuznetsov's $\mathfrak{B}$ is our $\PreCliff_{\O_X}\!(\sQ,Q)$ viewed as a weight graded, non differential by formality, algebra.  Kuznetsov's $\mathfrak{B}_0$ (the even part of $\mathfrak{B}$) is $\left(\Cliff_{\O_X}\!(\sQ,Q)\right)_0$ in our notation (the subscript denotes taking weight zero part), where this is regarded as an ungraded, non differential by formality, algebra.  Kuznetsov's constructs ($n \geq 2$) a semiorthogonal decomposition of $\DCoh \PP(\tQ_0)$ whose first term is $\DCoh \mathfrak{B}_0\mod(\QC(X))$, which may be identified with $\DSing^{\gr} \tQ_0$.
\end{na}

\part{Appendices}
\appendix

\section{Descent for \texorpdfstring{$\QCsh$}{QC^!} and \texorpdfstring{$\MF^\kinfty$}{MF}}\label{app:descent}
\subsection{Preliminaries}
We need the following standard local-to-global tool:
\begin{lemma}\label{lem:rhom} Suppose $\pi\colon U \to X$ is a flat map between \eqref{cond:star} derived stacks.
  \begin{enumerate} \item For $\F,\G \in \QC(X)$  there is a natural map
  \[ \pi^* \HHom^{\otimes X}_{\QC(X)}(\F, \G) \longrightarrow \HHom^{\otimes U}_{\QC(U)}(\pi^* \F, \pi^* \G) \in \QC(U) \]
  which is an equivalence provided that either
  \begin{itemize} 
    \item $\F \in \Perf(X)$ and $\G$ is arbitrary; or,
    \item $\F$ pseudo-coherent, and $\G$ is (locally) bounded above.
  \end{itemize}
\item For $\F,\G \in \Ind\DCoh(X)$ there is a natural map
  \[ \pi^* \HHom^{\otimes X}_{\Ind\DCoh(X)}(\F, \G) \longrightarrow \HHom^{\otimes U}_{\Ind\DCoh(U)}(\pi^* \F, \pi^* \G) \in \QC(U) \] which is an equivalence provided that $\F \in \DCoh(X)$ and $\G$ is arbitrary.
  \end{enumerate}
\end{lemma}
\begin{proof}  \mbox{}
  \begin{enumerate} \item
      The map is adjoint to a morphism
  \[ \HHom^{\otimes X}_{\QC(X)}(\F,\G) \to \pi_*\HHom^{\otimes U}_{\QC(U)}(\pi^* \F, \pi^* \G) = \HHom^{\otimes X}_{\QC(U)}(\pi^* \F,\pi^* \G) \] characterized by the mapping property
  \[ \Map_{\QC(X)}(T \otimes \F, \G) \stackrel{\pi^*}\longrightarrow \Map_{\QC(U)}(\pi^* T \otimes \pi^* \F, \pi^* \G) \]

  If $\F$ is perfect, then
  \[ \HHom^{\otimes X}_{\QC(X)}(\F, \G) = \F^\dual \otimes \G \]
  so that the claim is immediate from $\pi^*$ being symmetric-monoidal.

  Next suppose that $\F$ is pseudo-coherent and $\G$ (locally) bounded above: We first reduce to the case of $U$ and $X$ affine.\footnote{This reduction, and thus (i), does not actually require the hypothesis \eqref{cond:star}.}  The affine case implies that \[ [p \colon \Spec A \to X] \mapsto \RHom^{\otimes A}(p^* \F, p^* \G) \]  is a Cartesian section, i.e., lies in $\lim_{\Aff^{\flat}/X} \QC(A) = \QC(X)$ where the last equality is by faithfully flat descent.  Since tensor product commutes with pullback, one readily checks that it satisfies the universal property characterizing $\RHom^{\otimes X}(\F, \G)$.  By faithfully flat descent, and since colimits of sheaves are preserved under fiber products, it suffices to prove that the map is an equivalence after further pullback along each fppf map $q\colon\Spec B \to U$ admitting a lift $\pi' \colon \Spec B \to \Spec A$ of $\pi$; that is, we must check that the natural map
  \[ q^* \pi^* \RHom^{\otimes X}(\F, \G) \longrightarrow q^* \RHom^{\otimes U}(\pi^* \F, \pi^* \G) \] is an equivalence.  But by the above (applied once to $X$, once to $Y$) we naturally identify both sides with $\RHom^{\otimes B}(\res{\F}{B}, \res{\G}{B})$.

We may thus suppose $X = \Spec A$ and $U = \Spec B$.  Since $X$ is affine, it is in particular quasi-compact so that (shifting if necessary) we may suppose that $\F$ is connective and that $\G$ is bounded above.  Furthermore, since $X$ is affine we may write $\F \isom \left| P_\bullet\right|$ as the geometric realization of a diagram of finite free connective $A$-modules, $P_k \isom A^{\oplus n_k}$. In this case, we may identify
  \[ \pi^* \HHom^{\otimes X}_{\QC(X)}(\F, \G) = B \otimes_A \Tot\left\{P_\bullet^\dual \otimes_A \G\right\} \]
  \[ \HHom^{\otimes U}_{\QC(U)}(\pi^* \F, \pi^* \G) = \Tot\left\{ B \otimes_A (P_\bullet^\dual \otimes_A \G)\right\} \]
  and it remains to verify that tensor in fact commuted with the $\Tot$ by computing homotopy groups using a Bousfield-Kan spectral sequence.  Since $B$ is flat over $A$, we see that $\pi_0 B \otimes_{\pi_0 A} -$ of the Bousfield-Kan spectral sequence for the first $\Tot$ identifies with the B-K spectral sequence for the second $\Tot$, so that it suffices to prove that both are convergent.  Noting that $P_k^\dual \otimes_A \G = (A^{\oplus n_k})^\dual \otimes_A \G \isom \G^{\oplus n_k}$ has homotopy groups in the same degrees as $\G$, and the same after the flat base extension $B \otimes_A -$, we readily conclude that they are convergent: If $N$ is such that $\tau_{\geq N} \G = 0$, then the only terms that contribute to $\pi_\ell \Tot$ are ($\pi_0 B \otimes_{\pi_0 A} -$ of) $\pi_{q+\ell}(P_q^\dual \otimes_A \G) \isom \pi_{q+\ell}(\G)^{\oplus n_q}$ for $1 \leq q < N$.

\item The map is constructed analogously to that in (i) above.  Note that $\pi^*$ always preserves pseudo-coherent, and here it preserves (local) boundedness since $\pi$ is flat.  So, if $\F \in \DCoh(X)$ then $\pi^* \F \in \DCoh(X)$.  Suppose now that $\F \in \DCoh(X)$, and $\G = ``\dlim_\beta" \G_\beta \in \Ind \DCoh(X)$.  We claim that there is a natural equivalence
  \[ \HHom^{\otimes X}_{\Ind\DCoh(X)}(\F, \G) = \dlim_\beta \HHom^{\otimes X}_{\QC(X)}(\F, \G_\beta) \] (and similarly on $U$). Indeed for $T \in \Perf(X)$ we have $T \otimes \F \in \DCoh(X)$, so that
  \begin{align*}
    \Map_{\QC(X)}(T, \HHom^{\otimes}_{\Ind\DCoh(X)}(\F, \G)) &= \Map_{\Ind\DCoh(X)}(T \otimes \F, \G) \\
    &= \dlim_\beta \Map_{\DCoh(X)}(T \otimes \F, \G_\beta) \\
    &= \dlim_\beta \Map_{\QC(X)}(T, \HHom^{\otimes X}_{\QC(X)}(\F,\G_\beta)) \\
    &= \Map_{\QC(X)}(T, \dlim_\beta \HHom^{\otimes X}_{\QC(X)}(\F, \G_\beta)
  \end{align*} since $T$ is compact.  This reduces us to the case $\F, \G \in \DCoh(X)$, which follows from the second point of the $\QC$ case.\qedhere
  \end{enumerate}
\end{proof}

\subsection{Descent for \texorpdfstring{$\QCsh$}{QC^!}}
\begin{defn} Suppose $X$ is a derived scheme (or stack).  Let $X_{et}$ (resp.,  $X_{sm}$) denote the (small) site of morphisms $f \colon U \to X$ such that $f$ is representable, bounded, and \'etale (resp., smooth); covers are defined as usual (i.e., surjectivity on geometric points). Note that any morphism between objects of $X_{et}$ (resp., $X_{sm}$) is \'etale (resp., of finite Tor dimension).   By Nisnevich descent, we mean descent for the Grothendieck topology generated by Nisnevich distinguished squares.
\end{defn}

\begin{na} Recall that all morphisms in $X_{sm}$ are of finite Tor-dimension.  So, it makes sense to consider $\Ind \DCoh$ as a pre-sheaf on $X_{sm}$ via \emph{star pullbacks}:
  \[ U \mapsto \Ind \DCoh(U), \qquad f\colon U' \to U \mapsto f^*\colon\Ind\DCoh(U) \to \Ind\DCoh(U') \]
  We will do this \emph{only for the next proposition}, elsewhere we will use shriek pullbacks.
\end{na}
\begin{prop}\label{prop:qcsh-sheaf-inj} Suppose that $X$ is an \eqref{cond:star} derived stack, and that $\pi = \pi_\bullet \colon U_\bullet \to X$ is a smooth hypercover.  Then, the natural functor
  \[ \pi^* \colon \Ind\DCoh(X) \longrightarrow \Tot\left\{ \Ind\DCoh(U_\bullet), \pi_\bullet^*\right\} \] to the descent category is fully faithful. Consequently, the natural functor
  \[ \pi^! \colon\Ind\DCoh(X) \longrightarrow \Tot\left\{\Ind\DCoh(U_\bullet), \pi_\bullet^!\right\} \] is also fully faithful.
\end{prop}
\begin{proof} We must show that the natural map
  \[ \Map_{\Ind\DCoh X}\left(``\dlim_\alpha" \F_\alpha, ``\dlim_\beta" \G_\beta\right) \longrightarrow \Tot_\bullet \Map_{\Ind\DCoh U_\bullet}\left(``\dlim_\alpha" \pi_\bullet^* \F_\alpha, ``\dlim_\beta" \pi_\bullet^* \G_\beta\right) \] is an equivalence.  Since homotopy limits commute, we may reduce to the case of no $\alpha$, i.e., just mapping out of $\F \in \DCoh(X)$.
  
Suppose $\F \in \DCoh(X)$, and $``\dlim_\beta" \G_\beta \in \Ind \DCoh(X)$.  Set $U_{-1} = X$, and
\[ \HHom_n \eqdef \HHom^{\otimes U_n}_{\Ind\DCoh U_n}\left(\res{\F}{U_n}, ``\dlim_\beta" \res{\G_\beta}{U_n}\right) \] for $n \geq -1$.  By \autoref{lem:rhom}, the natural map 
\[ \res{\left[\HHom^{\otimes X}_{\QCsh(X)}\left(\F, ``\dlim_\beta" \G_\beta\right)\right]}{U_n} \longrightarrow \HHom^{\otimes U_n}_{\QCsh(U_n)}\left(\res{\F}{U_n}, ``\dlim_\beta" \res{\G_\beta}{U_n} \right) \] 
is an equivalence.  Now, fppf descent for $\QC$ implies
\begin{align*} \Map_{\Ind\DCoh X}\left(\F, ``\dlim_\beta" \G_\beta\right) &= \Map_{\QC(X)}(\O_X, \HHom_{-1}) \\
  &= \Tot_\bullet \Map_{\QC(U_n)}\left(\res{\O_X}{U_n}, \res{\HHom_{-1}}{U_n}\right) \\
  &= \Tot_\bullet \Map_{\QC(U_n)}(\O_{U_n}, \HHom_n) \\
  &= \Tot_\bullet \Map_{\Ind\DCoh U_n}\left(\res{\F}{U_n}, ``\dlim_\beta" \res{\G_\beta}{U_n}\right)
\end{align*}
as desired.

Finally, it remains to prove the ``Consequently'':  Note that the totalization of a cosimplicial object coincides with the homotopy limit over its underlying \emph{semi-cosimplicial} object.  All the morphisms in the semi-simplicial object underlying $U_\bullet$ are smooth, so that there is an equivalence $f^!(-) = f^*(-) \otimes_{\O} \omega_f$ and the relative dualizing complex $\omega_f\isom \det \LL_f$ is invertible.  Let $\omega_{\pi_\bullet} = \pi^!(\O_X)$, regarded as an invertible object in $\Tot\left\{\QC(U_\bullet)^\otimes,f^*\right\}$. Then, $\pi^!(-) \isom \pi^*(-) \otimes_{\O_{U_\bullet}} \omega_{\pi_\bullet}$ is fully-faithful since both $\pi^*$ and tensoring by an invertible object are fully-faithful.
\end{proof}

\begin{lemma}\label{lem:qcsh-sheaf-red} Suppose $f\colon  X' \to X$ is a map of \eqref{cond:star} derived stacks.  Then,
  \begin{enumerate}
    \item Suppose $f$ is surjective (on field valued points) and proper with $(f_*, f^!)$ an adjoint pair after any base-change (e.g., finite).\footnote{More generally this condition is satisfied if $f$ is a relative proper algebraic space, or (in char. $0$) a relative proper DM stack.}  If $\Ind \DCoh(-)$ is a sheaf on $X'_{sm}$, then it is a sheaf on $X_{sm}$. 
    \item Suppose that $f$ satisfies the conditions of (ii).  Then, there is a natural equivalence
      \[ \Ind\DCoh(X) = (q^! q_*)\mod\left(\Ind\DCoh(X')\right) \]
  \end{enumerate}
\end{lemma}
\begin{proof} \mbox{}
  \begin{enumerate}
    \item First note the property of $f$ being proper and surjective is stable under base-change.  Consequently, it suffices to show if $\pi  =\pi_\bullet\colon U_\bullet \to X$ is a smooth hypercover of $X$ itself, then the functor $\pi^*\colon \Ind\DCoh(X) \to \Tot\left\{\Ind\DCoh U_\bullet\right\}$ to the descent category is an equivalence. By the previous Proposition, it suffices to show that $\pi^*$ is essentially surjective.  Since the totalization may be computed in $\Pr^L$, viewing $p^*$ as left adjoint to $p_*$, we note that $\pi^*$ admits a right adjoint $\pi_*$ which is explicitly given by $\pi_*(\F_\bullet) = \Tot\left\{ (\pi_\bullet)_* \F_\bullet \right\}$.  It suffices to show that the counit $\pi^* \pi_* \to \id$ is an equivalence.  Since the left-adjoint $\pi^*$ is fully faithful, the unit map $\id \to \pi_* \pi^*$ is an equivalence; so, it suffices to show that $\pi_*$ is conservative: Indeed, consider the factorization of $\id_{\pi_* \F}$ as
  \[ \pi_* \F \stackrel{\sim}\longrightarrow \pi_* \pi^* \pi_* \F \stackrel{\pi_*(\mathrm{counit})}\longrightarrow \pi_* \F \]
Since the categories involved are stable, and the functors exacts, it suffices to show the following: If $\pi_*(\F_\bullet) = 0$, then $\F_\bullet = 0$; since $\F_n$, $n>0$ is a pull-back of $\F_0$, it suffices to show that $\F_0 = 0$ under these hypotheses.  So, we must prove that the functor
\[ (\pi_0)^* \pi_* = (\pi_0)^* \Tot\left\{ (\pi_n)_* \F_n \right\} \] is conservative.

Let $U'_n = U_n \times_X X'$, $\pi'_\bullet$ the base-changed structure maps, and $f_n\colon U'_n \to U_n$ the first projection.  It suffices to show that $(f_0)^! \circ (\pi_0)^* \circ \pi_*$ is conservative.  By various standard compatibilities (which one, e.g., first checks on $\QC$ then extends to $\QCsh$ by $t$-bounded-above arguments):
\begin{align*}
  (f_0)^! (\pi_0)^* \Tot\left\{(\pi_n)_* \F_n \right\} &= (\pi'_0)^* f^! \Tot\left\{ (\pi_n)_* \F_n\right\} \\
  \intertext{Since $f^!$ is a right adjoint, it commutes with arbitrary limits}
  &= (\pi'_0)^* \Tot\left\{ f^! (\pi_n)_* \F_n \right\} \\
  &= (\pi'_0)^* \Tot\left\{ (\pi'_n)_*  (f_n)^! \F_n \right\}  = (\pi'_0)^* (\pi'_\bullet)_*((f_\bullet)^! \F_\bullet)\\
\end{align*}
If this vanishes, then (by the hypothesis on $X'$) we find that $f_0^! \F_0 = 0$.  So, it suffices to show that $(f_0)^!\colon \Ind\DCoh U_0 \to \Ind\DCoh U'_0$ is conservative.  Note that $f_0$ is, being a base-change of $f$, also finite and surjective so that it suffices to show the following: If $f\colon  X' \to X$ is a proper (in the sense of the footnote), surjective, map of \eqref{cond:star} derived stacks, then $f^!$ is conservative.  Given $\F \in \Ind\DCoh(X)$ such that $f^! \F = 0$, it suffices to show that $\Map_{\Ind\DCoh(X)}(\K, \F) = 0$ for all $\K \in \DCoh(X)$. By \autoref{prop:qcsh-sheaf-inj}, this is smooth local on $X$, so we may suppose $X = \Spec A$ is an affine derived scheme.  Considering the diagrams
\[\vcenter{\xymatrix{\pi_0 X' \ar[d] \ar[r] & \pi_0 X \ar[d] \\ X' \ar[r] & X }}\quad \text{and} \qquad \vcenter{\xymatrix{\Ind\DCoh(\pi_0 X')& \Ind\DCoh(\pi_0 X)\ar[l] \\\ar[u] \Ind\DCoh(X') & \ar[u]\ar[l]\Ind\DCoh(X) }} \] and noting that $\Ind\DCoh(X)\to\Ind\DCoh(\pi_0 X)$ is conservative by \autoref{lem:coh-red}, we may reduce to the case of $X$ and $X' = \Spec \pi_0 A$ discrete and in particular ordinary separated stacks.  Since $\pi_0 A$ is Noetherian, Chow's Lemma for stacks \cite{Olsson-proper} shows that $X'$ receives a proper surjection from a projective $\pi_0 A$-scheme.  Thus, it suffices to prove the claim in case $X'$ is a projective $\pi_0 A$-scheme; let $\O(1)$ be a relatively ample line bundle.  Let $T$ denote the smallest thick subcategory of $\DCoh(X)$ containing $f_* \DCoh(X')$.  Since $f^!$ is right adjoint to $f_*$, it suffices to show that $T = \DCoh(X)$: Indeed, $\ker f^!$ is right-orthogonal to $T$.

Since $X$ is quasi-compact, a $t$-structure argument shows that it suffices to show that the intersection of $T$ with the heart $T^{\heart} = T \cap \DCoh(X)^\heart$ is all of $\Coh(X) = \DCoh(X)^\heart$.  By the usual form of devissage, noting that $T^{\heart}$ is closed under direct summands, it suffices to show that for all $x \in X$ there is some $\G(x) \in  \DCoh(X')$ such that $p_* \G(x) \in T^{\heart}$ and has a non-zero fiber over $x$.  Since $f$ is surjective, we may take $x' \in X'$ lying over it, and set $\G(x) = \O_{\ol{x'}} \otimes \O(N)$ for $N \vinograd 0$ large enough so that $p_* \G(X) \in T^{\heart}$: That such an $N$ exists is Serre's Theorem. Note that $\G(x) \otimes \O_{X,x}$ is a non-zero, coherent, $\O_{X,x}$-module whence the fiber at $x$ is non-zero by Nakayama's Lemma.
\item By the proof of (ii), we have seen that $f^!$ is conservative.  Since it preserves all colimits, Lurie's Barr-Beck Theorem applies to prove the first equality.   \qedhere
  \end{enumerate}
\end{proof}

We are now ready to prove the main result of this subsection:
\begin{theorem}\label{thm:qcsh-sheaf} Suppose $X$ is a \eqref{cond:star} derived stack.  Let $\Ind \DCoh(-)$ denote the pre-sheaf
\[ U \mapsto \Ind \DCoh(U),\qquad [f \colon U \to U'] \mapsto f^! \]   Then,
\begin{enumerate} 
  \item $\Ind \DCoh(-)$ has Nisnevich descent, and finite \'etale descent.
  \item $\Ind \DCoh(-)$ has representable \'etale descent.
  \item Suppose furthermore $X$ is a derived DM stack, then $\Ind \DCoh(-)$ has smooth descent.
  \item Suppose furthermore $X$ is a derived DM stack almost of finite-presentation over $k$.  Then, $\Ind \DCoh(-)$ has smooth descent and $\Ind \DCoh(X)$ coincides with $\QCsh(X)$ (as defined in \autoref{sec:more-general}).
  \end{enumerate}
\end{theorem}
\begin{proof} \mbox{}\begin{enumerate}
  \item
    By \autoref{prop:qcsh-sheaf-inj}, the pullback map to the descent category is fully faithful for any smooth hypercover so that it suffices to check essential surjectivity. 
    
  {\noindent}{\bf Step 1: Finite \'etale covers.}\\ Suppose $p\colon X' = U_0 \to X$ is an \'etale cover which is a \emph{finite} morphism, and $\pi_\bullet\colon U_\bullet \to X$ the Cech nerve.  Let $U'_\bullet = U_\bullet \times_X X'$, with $\pi'_\bullet\colon U'_\bullet \to X'$ the base-changed maps, and $p'_\bullet = U'_\bullet \to U_\bullet$ the projections.  
Suppose $\{\F_\bullet\}$ is such that $\Tot\left\{ (\pi_n)_* \F_n\right\} = 0$; we must show that $\F_0 = 0$.

Note that $\pi_0$ is both finite and \'etale, so that $(\pi_0)^* = (\pi_0)^!$ preserves all limits since it is right-adjoint to $(\pi_0)_*$.  Consequently, there are natural equivalences
\begin{align*}
 0 = (\pi_0)^* \Tot\left\{ (\pi_n)_* \F_n \right\} &= \Tot\left\{ (\pi_0)^* (\pi_n)_* \F_n \right\} \\
  &= \Tot\left\{ (\pi'_n)_* (p_n)^* \F_n \right\} \\
  &= (\pi')_* (\pi')^* \F_0
\end{align*}
Fully-faithfullness of $(\pi')^*$ implies that the unit $\F_0 \to (\pi')_* (\pi')^* \F_0$ is an equivalence, so that $\F_0 = 0$.

\medskip

{\noindent}{\bf Step 2: Distinguished Nisnevich squares.} 
Suppose given a distinguished Nisnevich square
  \[ \xymatrix{ U' \ar[d]_{p'} \ar[r]^{j'} & X' \ar[d]^p \\ U \ar[r]_j & X } \] with $j$ an open immersion, $p$ \'etale.  Let $Z = X \setminus U$ and $Z' = X' \setminus U'$.  We must prove that
  \[ \pi^* \colon \Ind\DCoh X \longrightarrow \Ind\DCoh U \times_{\Ind\DCoh U'} \Ind\DCoh X' \] is an equivalence.  We know that $\pi^*$ is fully faithful so that it suffices to prove essential surjectivity.  Given the adjunction $(\pi^*, \pi_*)$, it suffices to show that the counit $\pi^* \pi_* \to \id$ is an equivalence; since $\pi^*$ is fully faithful, the unit $\id \to \pi_* \pi^*$ is an equivalence and considering the following factorization of the identity on $\pi_* \F$
  \[ \pi_* \F \stackrel{\sim}\longrightarrow \pi_* \pi^* \pi_* \F \stackrel{\pi_*(\mathrm(counit)}\longrightarrow \pi_* \F \]
  reduces us to showing that $\pi_*$ is conservative.  Since all categories involved are stable and $\pi_*$ is exact, it suffices to prove that $\ker \pi_* = 0$.
Suppose 
  \[ \F_\star = \left(\F_U, \F_{U'}, \F_{X'}\right) \in \Ind\DCoh U \longrightarrow \Ind\DCoh U \times_{\Ind\DCoh U'} \Ind\DCoh X'  \] and recall that 
  \[ \pi_* (\F_\star) = j_* \F_U  \times_{j_* p'_* \F_{U'}} p_* \F_{X'} = j_* \F_U \times_{p_* (j')_* \F_{U'}} p_* \F_{X'} \in \Ind\DCoh X\]  It suffices to construct equivalences $j^* \pi_* \F_\star \isom \F_U$ and $p^* \pi_* \F_\star = \F_{X'}$, for then $\pi_* \F_\star = 0$ implies $\F_U = 0$ and $\F_{X'} = 0$ (and so $\F_{U'} = 0$).
  
Note that the counit $j^* j_* \to \id$ is an equivalence and that there is a natural equivalence $j^* p_* = (p')_* (j')^*$: Both are true on $\QC$ and all functors involved are $t$-bounded-above.    Consequently,
  \[ j^* \pi_* \F_\star = \F_U \times_{(p')_* \F_{U'}} (p')_* \F_{U'} = \F_U. \]
It remains to provide an equivalence $p^* \pi_* \F_\star = \F_{X'}$.  First note that
\begin{align*}p^* \pi_* \F_{\star} &=  p^*\left(j_*(\F_U) \times_{p_* (j')_* \F_{U'}} p_* \F_{X'}\right)  \\
  &= p^* j_*(\F_U) \times_{p^* p_* (j')_* \F_{U'}} p^* p_* \F_{X'} \\
  &= (j')_* \F_{U'} \times_{(j')_* (p')^* (p')_* \F_{U'}} p^* p_* \F_{X'} 
\end{align*}
and from the last term we obtain a natural map $\phi\colon p^* \pi_* \F_\star \to \F_{X'}$ using the structure maps that the counit $p^* p_* \to \id$.  Let $i\colon Z_{\red} \to X$ and $i'\colon Z'_{\red} \to X'$.  By \autoref{lem:open-closed-conserv} below, it suffices to show that $(j')^* \phi$ and $(i')^! \phi$ are equivalences.  The former is straightforward (both sides naturally identity with $\F_{U'}$), as is the latter (since $\res{p}{Z'_\red}\colon Z'_\red \to Z_\red$ is an isomorphism by the definition of distinguished Nisnevich).
\item By \autoref{lem:qcsh-sheaf-red} we may reduce to the case of $X$ discrete.  Then, note that descent for distinguished Nisnevich squares and finite \'etale covers implies representable \'etale descent (\cite[Theorem~D, Remark~5.4]{Rydh-devissage}).
\item A derived DM stack with affine diagonal admits a representable \'etale cover by a scheme, so that every \'etale cover  admits a representable refinement representable-\'etale locally.  Since the cotangent complex of a derived DM stack is connective, any smooth cover admits an \'etale refinement.
\item By (iii) applied to $\Spec A$ for $A \in \DRngfp_k$, $\QCsh(-)$ is a smooth sheaf on $X$.  Note that (iii) applies to $X$, so that $\Ind \DCoh(-)$ is a smooth sheaf on $X$.  Since $X$ was assumed DM and almost of finite-presentation, $X$ is \'etale locally of the form $\Spec A$ for $a \in \DRngfp_k$ so that the two sheaves are locally isomorphic.\qedhere
  \end{enumerate}
\end{proof}

\begin{lemma}\label{lem:open-closed-conserv} Suppose $X$ is a \eqref{cond:star} derived stack, $j\colon U \subset X$ a quasi-compact open, and $i\colon Z_{\red}\to X$ the reduced-induced structure on the closed complement.  Suppose $\phi\colon \F \to \G$ is a morphism in $\Ind \DCoh(X)$.  Then, $\phi$ is an equivalence if and only if $j^* \phi$ and $i^! \phi$ are both equivalences.
\end{lemma}
\begin{proof} Taking $\cone(\phi)$, it suffices to show that $\F \in \Ind\DCoh(X)$ is zero iff $j^* \F=0$ and $i^! \F = 0$.  One direction is clear, so suppose $j^* \F = 0$ and $i^! \F = 0$.  We must show that $\Map_{\Ind\DCoh(X)}(\K, \F) = 0$ for all $\K \in \DCoh(X)$.  By \autoref{lem:rhom}, it suffices to show that $\RHom_{\Ind\DCoh(X)}^{\otimes X}(\K, \F) = 0 \in \QC(X)$  and the question is fppf local on $X$ by fppf descent for $\QC(X)$.  In particular, we may assume that $X$ is affine so that we are in the situation where we have sketched a proof of \autoref{lem:dcohz-supt}.  It thus suffices to note that $i^!\colon \Ind \DCoh_Z(X) \to \Ind \DCoh(Z)$ is conservative, since its left-adjoint $i_*$ hits a generating set by \autoref{lem:coh-red}.
\end{proof}

\begin{remark} In fact, more is true than \autoref{thm:qcsh-sheaf}.  Let $X_{h}$ denote the (derived) Grothendieck topology on representable, bounded, almost finitely-presented $X$-stacks generated by distinguished Nisnevich squares and proper (with $(f_*, f^!)$ adjunction) surjective maps.\footnote{The naming is suggested by the fact that on an ordinary Noetherian scheme, the ordinary Grothendieck topology generated by the Nisnevich squares and proper surjections is precisely the ordinary $h$-topology.} The Theorem together with \autoref{prop:qcsh-sheaf-proper} below imply that $\Ind \DCoh(-)$ has $h$-descent in this funny sense: Since $X_h$ has covering morphisms which are not flat, the corresponding $\infty$-topos looks substantially different from the ordinary $h$-topos of $\pi_0 X$ even if $X$ is a discrete affine scheme, e.g., the map $X_{\red} \to X$ is no longer a monomorphism, so that the natural map $\F(X) \to \F(X_{\red})$ need not be an equivalence.
\end{remark}

\begin{prop}\label{prop:qcsh-sheaf-proper} Suppose that $X$ is a \eqref{cond:star} derived stack.  Then, $\Ind\DCoh(-)$ has \emph{proper} descent on $X$: i.e., Suppose  $q\colon X' \to X$ is proper and surjective, and let $\pi = \pi_\bullet\colon \left\{ X'_\bullet = (X')^{\bullet/X}\right\} \to X$ be the Cech nerve of $q$.  Then, the functor
  \[ \pi^! \colon \Ind\DCoh(X) \to \Tot\left\{\Ind\DCoh(X'_\bullet), f^!\right\} \] is an equivalence of categories.
\end{prop}
\begin{proof}
Note that all structure maps in $X'_\bullet$ are proper since $X' \to X$ is proper, and in particular separated; so, the totalization may be regarded as being computed in either $\Pr^R$ or $\Pr^L$.  Consequently $\pi^!$ admits a left-adjoint $\pi_*$.  Consequently $\pi^!$ admits both preserves colimits and admits a left-adjoint $\pi_*$.  We can check that $\pi_*$ is computed by the geometric realization
\[ (\pi_*\left(\F_\bullet\right) = \left| (\pi_n)_* (\F_n) \right| \]
It now suffices to check that the unit and counit maps $\pi_* \pi^! \to \id$ and $\id \to \pi^! \pi_*$ are equivalences.  Since $q^!$ is conservative by \autoref{lem:qcsh-sheaf-red}, so is $\pi^!$.  Thus it is enough to check that the unit map is an equivalence (c.f., \autoref{lem:mono}(iv)).

For the unit: Since $q^!$ is conservative by \autoref{lem:qcsh-sheaf-red}  it suffices to check this after applying $q^!$, so that we are interested in verifying that map
\[ q^! \pi_* \pi^! \F = \left| q^! (\pi_n)_* (\pi_n)^! \F \right\| \longrightarrow q^! \F \] is an equivalence.
Let $p_n \colon X'_{n+1} \to X'$ be the first projection (i.e., this is the base change of $\pi_n$ along $q$), and $q_n \colon X'_{n+1} \to X_n$ the last projection (i.e., the induced map on simplicial objects).  Base-change gives $q^! (\pi_n)_* (\pi_n)^! = (p_n)_* (q_n)^! (\pi_n)^! \F = (p_n)_* (\pi_{n+1})^! \F$, so that our augmented simplicial diagram is in fact \emph{split} and consequently a colimit diagram.
\end{proof}

\subsection{Descent for $\MF^\kinfty$} 
\begin{prop}\label{prop:mf-sheaf}  Suppose $X$ is a \eqref{cond:starf} derived stack, and $f\colon  X \to \AA^1$.
  \begin{enumerate} \item The assignments
      \[ U \mapsto \PreMF(U, \res{f}{U}) \qquad \text{and} \qquad U \mapsto \PreMF^\kinfty(U, \res{f}{U}) \]
      determine sheaves of $k\ps{\bt}$-linear $\infty$-categories on $X_{et}$.
    \item The assignment
      \[ \qquad U \mapsto \MF^\kinfty(U, \res{f}{U}) \]
      determines a sheaf of $k\pl{\bt}$-linear $\infty$-categories on $X_{et}$.
 \end{enumerate}
\end{prop}
\begin{proof}\mbox{}\begin{enumerate}
  \item Note that any \'etale cover of $X_{et}$ restricts to an \'etale cover of $X_0$, and that a $k\ps{\bt}$-linear presheaf is a sheaf if and only if it is a sheaf forgetting the extra linear structure.  So, it suffices to show that $\DCoh$ and $\QCsh$ are sheaves on $(X_0)_{et}$; the former follows from the analogous theorem for $\QC$ and the local definition of $\DCoh$ (since $X_0$ is coherent), and the latter follows from the analogous theorem for $\QCsh$ (\autoref{thm:qcsh-sheaf}).
  \item It suffices to note that $- \ohotimes_{k\ps{\bt}} k\pl{\bt}\colon \dgcatbig_{k\ps{\bt}} \to \dgcatbig_{k\pl{\bt}}$ commutes with homotopy-limits, since $k\pl{\bt}\mod \in \dgcatbig_{k\ps{\bt}}$ is dualizable (c.f., \autoref{lem:dual-mod}).\qedhere
\end{enumerate}
\end{proof}

\section{Integral transforms for (Ind) Coherent Complexes}\label{app:coh-fmk}
We give here an exposition of the Tensor Product and Functor Theorems for $\QCsh$ of derived schemes.  As this is essentially a mild generalization of \cite{Lunts}, we will be brief.

\subsection{Fully faithful}
\begin{prop}\label{prop:coh-fmk-ff} Suppose $S$ is a regular \eqref{cond:star} derived stack, and $X$ and $Y$ \eqref{cond:star} derived stacks over $S$.  Then, exterior product over $S$ determines a well-defined and fully faithful functor
  \[ \boxtimes_S\colon \DCoh(X) \otimes_S \DCoh(Y) \longrightarrow \DCoh(X \times_S Y) \]
\end{prop}
\begin{proof} We first check that it is well-defined.  Since $\DCoh$ (with star pullback) is an fppf sheaf, the question is local on $X$ and $Y$ so that we may suppose $X = \Spec R$, $Y = \Spec R'$.  Since $S$ is assumed to have affine diagonal, $X$ and $Y$ are also affine over $S$ so that the pushforward is $t$-exact, etc.  Exterior product always preserves pseudo-coherence, and since $S$ is regular we conclude that it preserves being locally bounded since we may check so after pushforward to $S$.  

  Next we check that it is fully-faithful, i.e., that the exterior product
    \[ \RHom^{\otimes S}_{\QC(X)}(\F_X, \G_X) \otimes_{\O_S} \RHom^{\otimes S}_{\QC(Y)}(\F_Y, \G_Y) \longrightarrow \RHom^{\otimes S}_{\QC(X \times_S Y)}\left(\F_X \boxtimes_S \F_Y, \G_X \boxtimes_S \G_Y\right) \] is an equivalence in $\QC(S)$ for all $\F_X, \G_X \in \DCoh(X)$, and $\F_Y,\G_Y \in \DCoh(Y)$.  By \autoref{lem:rhom}, the claim is fppf local on $X, Y$ and $S$, so that we may assume they are all affine: $S = \Spec A$, $X = \Spec R$, $Y = \Spec R'$.
    
The claim is clear when $\F_X = \O_X$ and $\F_Y = \O_Y$, and so more generally whenever $\F_X$ and $\F_Y$ are perfect.    Shifting as necessary, we may suppose that $\G_X$, $\G_Y$, and $\G_X \boxtimes_S \G_Y$ are all co-connective (i.e., $\pi_i = 0$ for $i \geq 0$), while $\F_X$, $\F_Y$ (and hence $\F_X \boxtimes_S \F_Y$) are connective.  Then, $\F_X$ (resp., $\F_Y$) can be written as the geometric realizations of a diagram of finite free $R$-modules $P_\bullet$ (resp., $R'$-modules $P'_\bullet$); since exterior product preserves colimits, $\F_X \boxtimes_S \F_Y$ will then be the realization of the bisimplicial object $P_\bullet \boxtimes_S P'_\bullet$.  Under our connectivity assumptions, it is straightforward to check that the Bousfield-Kan spectral sequence for $\RHom_X(\F_X, \G_X) =\Tot\RHom_X(P_\bullet,\G_X)$ is convergent and that the $\RHom$ complex is co-connective:
\[E^1_{p,q} = \pi_p \RHom_X(P_q, \G_X) = (\pi_p \G_X)^{\oplus n_p} \Rightarrow \pi_{p-q} \RHom_X(\F_X,\G_X) \] where $n_p$ is the rank of the finite free $R$-module $P_p$.  Similarly, the $\RHom$ complexes on $Y$ and $X \times Y$ are also co-connective.

Next, note that that $- \otimes_S -$ commutes with totalizations of co-connective objects in each variables.  This follows by another Bousfield-Kan spectral sequence, since $S$ is regular so that it has bounded flat dimension.

Putting the above together, we may conclude
  \begin{align*} \RHom_X(\F_X, \G_X) \otimes_S \RHom_Y(\F_Y, \G_Y) &= \left(\Tot_a \RHom_X(\P_a, \G_X)\right) \otimes_S \left(\Tot_b \RHom_Y(\P'_b, \G_Y)\right) \\
    &= \Tot_{a,b} \left(\RHom_X(\P_a, \G_X) \otimes_S \RHom_Y(\P'_b, \G_Y)\right) \\
    &= \Tot_{a,b} \RHom_{X \times_S Y}(\P_a \boxtimes_S \P'_b, \G_X\boxtimes_S \G_Y)\\
    &= \RHom_{X \times_S Y}(\F_X \boxtimes_S \F_Y, \G_X\boxtimes_S \G_Y)\qedhere
  \end{align*}
\end{proof}

\subsection{Shriek preliminaries}
\begin{lemma}\label{lem:dualiz} Suppose $X$ is a \eqref{cond:starf} derived stack over $S = \Spec k$, $\F, \G \in \DCoh(X)$.  Then, there are natural equivalences
  \begin{enumerate}
    \item $\DD(\F \boxtimes \G) = \DD(\F) \boxtimes \DD(\G)$.
    \item $\omega_X \shotimes \G = \G$ 
  \end{enumerate}
\end{lemma}
\begin{proof}\mbox{}
  Note that $(p_2)^! \G = \omega_X \boxtimes \G$.  In particular, $\omega_{X^2}= (p_2)^! \omega_X = \omega_X \boxtimes \omega_X$.  Part (i) now follows from \autoref{prop:coh-fmk-ff} and the formula $\DD(-)=\HHom^{\otimes_X}(-,\omega)$.  Part (ii) follows by noting that
      \[ \omega_X \shotimes \G = \Delta^!(\omega_X \boxtimes \G) = \Delta^! (p_2)^! \G = \G \qedhere \]
\end{proof}

\begin{lemma}\label{lem:trace} Suppose $X$ is a \eqref{cond:starf} derived stack over $S = \Spec k$, and  $\F, \G \in \DCoh(X)$.  Then, there is a natural equivalence in $\QC(X)$
  \[ \F \shotimes \G \isom \HHom^{\otimes_X}_X(\DD\F, \G) \]
\end{lemma}
\begin{proof} Since $\Delta$ is finite, we have a relative adjunction $(\Delta_*,\Delta^!)$ and we may rewrite
  \begin{align*}
    \F \shotimes \G &= \Delta^!(\F \boxtimes \G) \\
    &= (p_1)_* \Delta_* \HHom^{\otimes}_X\left( \O_X, \Delta^!(\F \boxtimes \G)\right) \\
    &= (p_1)_* \HHom^{\otimes}_{X^2}\left( \Delta_* \O_X, \F \boxtimes \G\right) \\
    \intertext{and since $\Delta_* \O_X \in \DCoh(X)$ we may apply coherent duality to rewrite this as}
    &= (p_1)_* \HHom^{\otimes}_{X^2}\left( \DD(\F \boxtimes \G), \DD \Delta_* \O_X\right) \\
    \intertext{Applying \autoref{lem:dualiz}(i)}
    &= (p_1)_* \HHom^{\otimes}_{X^2}\left( \DD(\F) \boxtimes \DD(\G), \DD \Delta_* \O_X\right) \\
    &= (p_1)_* \HHom^{\otimes}_{X^2}\left( \DD(\F) \boxtimes \O_X, \HHom^{\otimes}_{X^2}\left( \O_X \boxtimes \DD(\G), \DD \Delta_* \O_X\right)\right) \\
    \intertext{Undoing the above operations on the inner-$\HHom^{\otimes}$:}
    &= (p_1)_* \HHom^{\otimes}_{X^2}\left( \DD(\F) \boxtimes \O_X, \HHom^{\otimes}_{X^2}\left( \Delta_* \O_X, \DD(\O_X \boxtimes \DD(\G))\right)\right) \\
    &= (p_1)_* \HHom^{\otimes}_{X^2}\left( \DD(\F) \boxtimes \O_X, \Delta_* \left(\omega_X \shotimes \G\right)\right) 
    \intertext{Applying the relative $(\Delta^*,\Delta_*)$ adjunction}
    & = (p_1)_* \Delta_* \HHom^{\otimes}_{X}\left(\DD(\F) \otimes \O_X, \omega_X \shotimes \G\right)\\
    \intertext{Finally we complete by \autoref{lem:dualiz}(ii)}
    &= \HHom^{\otimes}_{X}\left(\DD(\F), \G\right)\qedhere
  \end{align*}
\end{proof}

\begin{remark}\autoref{lem:trace} admits the following reformulation: Define
  \[ \ev\colon \QCsh(X) \otimes \QCsh(X) \to k\mod \qquad \F \otimes \G \mapsto \RGamma(\F \shotimes \G) = \RHom^{\otimes k}_{X^2}(\Delta_* \O_X, \F \boxtimes \G) \] Then, the functor $\DD(-)\colon \DCoh(X) \to \QCsh(X)^\dual = \Fun^{ex}(\DCoh(X), k\mod)$ is characterized by
  \[ \RHom^{\otimes k}_{X}(\DD(\F), -) = \ev(\F \otimes -) = \RGamma(\F \shotimes -) = \RHom^{\otimes k}_{X^2}(\Delta_* \O_X, \F \boxtimes -) \]
  Grothendieck duality implies that this is part of a duality datum giving $\QCsh(X) \isom \QCsh(X)^\dual$.
\end{remark}

\begin{theorem}\label{thm:coh-fmk} Suppose $S = \Spec k$ is a perfect field; that $X, Y$ are almost finitely-presented \eqref{cond:starf} stacks over $S$; and that $Z_X \subset X$, $Z_Y \subset Y$ are closed subsets.  Then, there are equivalences of categories
  \[ \xymatrix{
  \Fun^L_k(\QCsh_{Z_X}(X), \QCsh_{Z_Y}(Y)) & \ar[l]^-{\sim}_-{\Phish} \QCsh_{Z_X \times_S Z_Y}(X \times_S Y) \\
  \QCsh_{Z_X}(X) \ohotimes_k \QCsh_{Z_Y}(Y) \ar[ur]^-{\sim}_-{\boxtimes} \ar[u]_-{\sim}^-{\Psish}   } \]
  where 
  \begin{itemize}
    \item $\boxtimes$ denotes external tensor product over $S$, and restricts to an equivalence on compact objects
  \[ \boxtimes \colon \DCoh_{Z_X}(X) \otimes_k \DCoh_{Z_Y}(Y) \stackrel{\sim}\longrightarrow \DCoh_{Z_X \times_S Z_Y}(X \times_S Y) \] 
\item  $\Phish(\K) = (p_2)_*\left(p_1^!(-) \shotimes \K \right)$ is the $!$-Fourier-Mukai functor with kernel $\K$.  
\item $\Psish(\F \otimes \G) = \Hom_{X/S}(\DD(\F), -) \otimes_k \G$ for $\F, \G$ compact objects.
  \end{itemize}

  Restricting to the case $X=Y$:
  \begin{itemize}
    \item $\id_{\QCsh_Z(X)} = \Phish(\omega_{\Delta,Z})$, where $\omega_{\Delta,Z} = \Delta_* \RGamma_Z(\omega_X)$ and $\omega_X = \DD(\O_X)$ is the dualizing complex and $\RGamma_Z(-)\colon \QCsh(X) \to \QCsh_Z(X)$ is (the Ind-coherent version of) local cohomology along $Z$.
    \item More generally, $\Phish(\Delta_* \F) = \F \shotimes -$.
    \item $\ev(\Phish(\K)) = \Hom_{\QCsh(X^2)}(\Delta_* \O_X, \K)$ (no support condition!).
  \end{itemize}
\end{theorem}
\begin{proof} 
  The Grothendieck Duality anti-equivalence respects supports and so restricts to $\DD(-)\colon \DCoh_{Z_X}(X)^\op \isom \DCoh_{Z_X}(X)$.  This implies that $\QCsh_{Z_X}(X)$ is self-dual over $\QC(S)$ via $\DD(-)$, so that $\Psish$ is an equivalence (it does not even matter that the target category is of geometric origin).  We will now verify commutativity of the diagram, the indicated formulas, and only finally that the relevant maps are equivalences.

  \smallskip

{\noindent}{\bf Diagram commutes: }\\
Let us prove that the diagram commutes up to natural equivalence.  Since each of $\boxtimes$, $\Psish$, and $\Phish$ is colimit preserving it suffices to give a natural equivalence $\Psish_{\F \otimes \G} = \Phish_{\F \boxtimes \G}$ for $\F \in \DCoh(X)$, $\G \in \DCoh(Y)$.  Since both functors are colimit preserving, we may check this for $T \in \DCoh(X)$:
  \begin{align*}
    \Phish_{\F \boxtimes \G}(T) &= (p_2)_*\left( (p_1)^! T \shotimes (\F \boxtimes \G)\right) \\
&= (p_2)_*\left((T \boxtimes \omega_Y) \shotimes (\F \boxtimes \G)\right) \\
&= (p_2)_* \RHom^{\otimes}_{X \times Y}\left( \DD(\F) \boxtimes \DD(\G), T \boxtimes \omega_Y\right) \\
&= (p_2)_*\left(\RHom^{\otimes}_X(\DD(\F), T) \boxtimes \RHom^{\otimes}_Y(\DD(\G), \omega_Y) \right) \\
&= (p_2)_*\left(\RHom^{\otimes}_X(\DD(\F), T) \boxtimes \RHom^{\otimes}_Y(\O_Y,\G) \right) \\
&= \RGamma\left(\RHom^{\otimes}_X(\DD(\F),T)\right) \otimes_k \G \\
&= \Map_X(\DD(\F), T) \otimes_k \G \\
&= \Psish_{\F \otimes \G}(T)
  \end{align*}
  Here have have implicitly used \autoref{lem:dualiz}, \autoref{lem:trace}, and coherent duality.

\smallskip

{\noindent}{\bf Formulaire:}
We first prove that $\Phish_{\Delta_* \F}(-) = - \shotimes \F$.  Extending by colimits, it suffices to note that for $T,\F \in \DCoh(X)$
\begin{align*}
  \Phish_{\Delta_* \F}(T) &= (p_2)_*\left( (p_1)^! T \shotimes \Delta_* \F\right) \\
  &= (p_2)_* \RHom^{\otimes}_{X^2}\left( \DD(T) \boxtimes \O_Y, \Delta_* \F \right) \\
  &= (p_2)_* \Delta_* \RHom^{\otimes}_X\left( \Delta^*(\DD(T) \boxtimes \O_Y), \F\right) \\
  &= \RHom^{\otimes}_X(\DD(T), \F) \\
  & = T \shotimes \F
\end{align*}
By \autoref{lem:dualiz}(ii), it follows that $\Phish_{\Delta_* \omega_X} = \id_{\QCsh(X)}$.  More generally, setting $\omega_{\Delta,Z} = \Delta_* \RGamma_Z(\omega_X)$, we see that
\[ \Phish_{\omega_{\Delta,Z}}(T) = T \shotimes \RGamma_Z(\omega_X) = \RGamma_Z(T) \shotimes \omega_X = \RGamma_Z(T). \]  Since $\RGamma_Z$ is the identify functor on $\QCsh_Z(X)$, we obtain $\Phish_{\omega_{\Delta,Z}}=\id_{\QCsh_Z(X)}$ in case of supports.

To check the formula for the trace, it suffices (since both sides preserve colimits in both variables) to check it in case $\K = \F \boxtimes \G$ with $\F, \G \in \DCoh(X)$. Applying \autoref{lem:trace} we see that
\[ \ev\left(\Phish_{\K}\right) = \ev\left(\Psish_{\F \otimes \G}\right) = \Map_X\left(\DD(\F), \G\right) = \RGamma\left( \F \shotimes \G \right) = \Map_{\QCsh(X^2)}\left(\Delta_* \O_X, \F \boxtimes \G \right) \]

  \smallskip

{\noindent}{\bf Equivalences: }\\
Since the diagram commutes and $\Psish$ is an equivalence, it suffices to show that $\boxtimes$ is an equivalence.  By \autoref{prop:coh-fmk-ff} it preserves compact objects and is fully faithful.  It suffices to show that it is essentially surjective on compact objects.  In \autoref{prop:fmk-coh-surj} below, we we handle the case without support conditions.  Let us show how this implies the general case:
\[ \xymatrix{
\DCoh(Z_X) \otimes \DCoh(Z_Y) \ar[d] \ar[r]^{\sim} &\DCoh(Z_X \times Z_Y) \ar[d] \\
\DCoh_{Z_X}(X) \otimes \DCoh_{Z_Y}(Y)  \ar[r] & \DCoh_{Z_X \times Z_Y}(X \times Y)  }\]
We have seen that the bottom horizontal arrow is fully faithful, so since both categories are stable and idempotent complete it suffices to show that it has dense image.  We have seen that the the top horizontal arrow is an equivalence.  The right vertical arrow has dense image by \autoref{lem:coh-red}.  Consequently, the bottom horizontal arrow has dense image as desired.
\end{proof}

\subsection{Devissage}
\begin{lemma}\label{lem:u-z} Suppose $X$ is a quasi-compact quasi-separated scheme and $U \subset X$ is a quasi-compact open, with closed complement $Z = X - U$.  Suppose that $\boxtimes$ is an equivalence for the pairs $(U,Y)$ and $(Z,Y)$.  Then, it is an equivalence for $(X, Y)$.
\end{lemma}
\begin{proof}
  Observe that $\DCoh(Z \times_S Y) \to \DCoh_{Z \times_k Y}(X \times_S Y)$ has dense image by \autoref{lem:coh-red} and filtering by powers of the ideal sheaf of $Z$.  Considering the diagram
  \[ \xymatrix{\DCoh(Z) \otimes_k \DCoh(Y) \ar[d] \ar[r] & \DCoh(Z \times_S Y) \ar[d] \\
  \DCoh_Z(X) \otimes_k \DCoh(Y) \ar[r] & \DCoh_{Z \times_k Y}(X \times_S Y) } \] we see that the right vertical arrow has dense image; since the top horizontal arrow does by assumption, so does the bottom horizontal arrow.

  Consider the diagram
  \[ \xymatrix{
  \DCoh_Z(X) \otimes_k \DCoh(Y) \ar@{->>}[d] \ar[r] & \DCoh(X) \otimes_k \DCoh(Y) \ar[d] \ar[r] & \DCoh(U) \otimes_k \DCoh(Y) \ar[d]^{\sim} \\
\DCoh_{Z \times_S Y}(X \times_S Y) \ar[r] & \DCoh(X \times_S Y) \ar[r] & \DCoh(U \times_S Y) } \]
We claim both rows are Verdier-Drinfeld sequences.  For the bottom row, this is the usual localization sequence of a closed subset for $\DCoh$.  For the top row, reduce to the usual localization sequence by \autoref{lem:small-tensor}.
 Set $ \A = \langle \im \boxtimes_{X,Y}\rangle \subset \DCoh(X \times_S Y)$. We will show that $\A = \DCoh(X \times_S Y)$, using the following categorical version of the ``$5$-lemma'':

Examining the left-most arrow, we see that $\A$ contains $\DCoh_{Z \times_S Y}(X \times_S Y)$.  Letting $\ol{A}$ denote its image in the Verdier quotient $\DCoh(U \times_S Y)$, it suffices to show that $\ol{\A}$ is dense in $\DCoh(U \times_S Y)$. Since the right-most vertical arrow is an equivalence, this follows from observing that $\DCoh(X) \otimes_k \DCoh(Y) \to \DCoh(U) \otimes_k \DCoh(Y)$ has dense image.
\end{proof}

\begin{prop}\label{prop:fmk-coh-surj} Suppose $k$ is a perfect field, $S = \Spec k$, and that $X$, $Y$ are almost finitely-presented \eqref{cond:starf} derived stack over $S$.  Then, the exterior product induces equivalences
  \[ \boxtimes \colon \DCoh(X) \otimes_{k} \DCoh(Y) \stackrel{\sim}{\longrightarrow} \DCoh(X \times Y) \]
  \[ \boxtimes \colon \Ind\DCoh(X) \ohotimes_{k} \Ind\DCoh(Y) \stackrel{\sim}{\longrightarrow} \Ind\DCoh(X \times Y) \]
  This remains true with support conditions.
\end{prop}
\begin{proof}  We have seen how to reduce the case with support conditions to that without in \autoref{thm:coh-fmk}.  Also, note that it suffices to prove either the small or the Ind-completed version.

Suppose $U_\bullet \to X$ is an \'etale cover, so that $\Ind\DCoh X = \Tot\left\{\Ind\DCoh U_\bullet\right\}$ by \autoref{thm:qcsh-sheaf}; since $U_\bullet \times Y \to  X \times Y$ is again an \'etale cover, we also have $\Ind \DCoh (X \times Y) = \Tot\left\{\Ind\DCoh U_\bullet \times U\right\}$. Since $\Ind\DCoh(Y)$ is dualizable over $k\mod$ by \autoref{lem:dual-mod}, $- \ohotimes_k \Ind\DCoh(Y)$ preserves arbitrary limits.  Consequently, we have a diagram of equivalences
  \begin{align*} \Ind\DCoh(X) \ohotimes_k \Ind\DCoh(Y) &\stackrel{\pi^* \otimes \id}\longrightarrow \Tot\left\{\Ind\DCoh(U_\bullet)\right\} \ohotimes_k \Ind\DCoh(Y) \\
    & \stackrel{\sim}\longrightarrow \Tot\left\{\Ind\DCoh(U_\bullet \times Y)\right\} \stackrel{(\pi,\id)^*}\longleftarrow \Ind\DCoh(X \times Y) \end{align*}
Exterior product commutes with finite Tor-dimension pullbacks, so we conclude that our claim is local on $X$.  Similarly, it is local on $Y$.  Consequently, we may reduce to the case of $X$ and $Y$ affine derived schemes.

We will now prove the small, idempotent complete, variant.  Since $\boxtimes$ is fully faithful by \autoref{prop:coh-fmk-ff}, it suffices to prove that it is essentially surjective.  Since both the tensor product and $\DCoh(X \times_S Y)$ are stable and idempotent complete, it suffices to show that the image of $\boxtimes$ is \emph{dense} in the sense that its thick-closure is the whole category.

  \smallskip

  {\noindent}{\bf Step 1. Case of $X$, $Y$ regular (discrete) schemes:}\\ The analogous statement is well-known (see, e.g., T\"oen or \cite{BFN}) with $\DCoh$ replaced by $\Perf$ throughout.  Since $X$, $Y$ are regular we have $\DCoh(X) = \Perf(X)$, $\DCoh(Y) = \Perf(Y)$. It remains to observe that $X \times_S Y$ is again regular, since $X, Y$ are finite-type over a perfect field $k$. Consequently, $\DCoh(X \times_S Y) = \Perf(X \times_S Y)$, and we're done.

  \smallskip

  {\noindent}{\bf Step 2. Reduction to the case $X$, $Y$ reduced (discrete) schemes:}\\
  Consider the natural map $i\colon (\pi_0 X)_\red \to X$.  Under our finiteness hypotheses, it is proper and consequently we obtain a functor $i_*\colon \DCoh( (\pi_0 X)_\red) \to \DCoh(X)$.  The standard filtration argument shows that every object of $\DCoh(X)$ admits a filtration with associated graded in the image of $i_*$, from which it follows that $i_*\colon \DCoh( (\pi_0 X)_\red) \to \DCoh(X)$ has dense image.
  
  Consider the diagram
  \[ \xymatrix{ \DCoh(X) \otimes_{k} \DCoh(Y) \ar[r]^{\boxtimes} & \DCoh(X \times_S Y) \\ \DCoh( (\pi_0 X)_\red) \otimes_k \DCoh( (\pi_0 Y)_\red) \ar[r]_-{\boxtimes_{\red}} \ar[u] & \DCoh( (\pi_0 X)_\red \times_S (\pi_0 Y)_\red) \ar[u] } \]  If $\boxtimes_{\red}$ has dense image then so does $\boxtimes$, since the right vertical arrow has dense image by the above (the map $[\pi_0(X \times_S Y)]_\red \to X \times_S Y$ factors through $(\pi_0 X)_\red \times_S (\pi_0 Y)_\red$, and is in fact an equivalence under our hypotheses).

  \smallskip

  {\noindent}{\bf Step 3. Reduction to the case $X$, $Y$ integral (discrete) schemes: }\\ By Step 2, we may assume $X$, $Y$ are reduced schemes.  Since they are finite-type over a field, they have finitely-many irreducible components $X_1, \ldots, X_n$, $Y_1, \ldots, Y_m$.  Using \autoref{lem:u-z}, we may induct on the number of irreducible components.

  \smallskip

  {\noindent}{\bf Step 4. Completing the proof:  }\\ By the above, we may suppose $X$, $Y$ are integral schemes.  By Noetherian induction, we may suppose the claim is known for all pairs $(X',Y')$ such that $\dim X' \leq \dim X$, $\dim Y' \leq \dim Y$ with at least one of these inequalities is strict.  Since $X$, $Y$ are integral and of finite-type over a perfect field, they are generically regular.  Let $U \subset X$, $V \subset Y$ be dense open regular subsets, and $Z_X = X - U$, $Z_Y = Y - V$.  Using \autoref{lem:u-z}, we see that the claim holds for $(X,Y)$ if it holds for $(U,V)$, $(Z_X, V)$, $(U, Z_Y)$, and $(Z_X, Z_Y)$: The first of these follows by Step 1, while the rest follow by the inductive hypothesis.
\end{proof}

\begin{remark} 
After reducing to the case of a reduced discrete scheme, one can also conclude quite quickly using de Jong's alterations and \autoref{lem:qcsh-sheaf-red}(ii): Using de Jong's alterations one may produce proper surjective maps $p\colon \sq{X} \to X$, $q\colon \sq{Y} \to Y$ with $\sq{X}$ and $\sq{Y}$ regular.  Then, \autoref{lem:qcsh-sheaf-red}(ii) identifies $\QCsh(X) = (p^! p_*)\mod\QCsh(\sq{X})$ and similarly for $\QCsh(Y)$ (using $q$) and for $\QCsh(X \times Y)$ (using $p \times q$).  Since $\sq{X}$ and $\sq{Y}$ are regular, as is their product, we know that $\QCsh(\sq{X}) \ohotimes_k \QCsh(\sq{Y}) = \QCsh(\sq{X} \times \sq{Y})$.  Finally, it suffices to identify $(p \times q)^! (p \times q)_*$ with the algebraic tensor-product monad.
\end{remark}

\subsection{Extensions}
\begin{prop}\label{prop:ext-smth-fmk} Suppose $S$ is regular \eqref{cond:star} stack.
\begin{enumerate} 
\item Suppose $Y \to S$ is a smooth relative scheme.  Then,
\[ \boxtimes \colon \DCoh(X) \otimes_{S} \DCoh(Y) \longrightarrow \DCoh(X \times_S Y) \]
is an equivalence for all excellent (i.e., $\pi_0 X$ is an excellent ordinary scheme) derived stacks $X$ over $S$.  If $S$ is excellent (in the sense that all schemes of finite-type over it are excellent), then this holds for any almost finitely-presented \eqref{cond:star} derived stack over $S$.
\item Suppose $S$ is regular and excellent; that $X, Y$ are \eqref{cond:star} derived DM stacks over $S$; and that $Z_X \subset X$, $Z_Y \subset Y$ are closed subsets.  Suppose furthermore that $Z_Y$, with its reduced induced scheme structure, is smooth over $S$.  Then,
  \[ \boxtimes\colon \DCoh_{Z_X}(X) \otimes_S \DCoh_{Z_Y}(Y) \longrightarrow \DCoh_{Z_X \times_S Z_Y}(X \times_S Y) \] is an equivalence.
\end{enumerate}
\end{prop}
\begin{proof} \mbox{}
\begin{enumerate}
\item The second sentence follows from the first. As in \autoref{prop:fmk-coh-surj}, the question is local on $X$ so that we may suppose $X$ is affine.  If $X$ is regular, than so is $X \times_S Y$ (being smooth over $X$) and we are done by the analogous statement for $\Perf$.  Otherwise, we may proceed by Noetherian induction on $X$, as in the proof of \autoref{prop:fmk-coh-surj}. Since all derived schemes occurring in the Noetherian induction will be almost finitely-presented over $X$, they will all be Noetherian and excellent.  As there, we reduce to the case of $X$ discrete and reduced, and apply \autoref{lem:u-z} to reduce to the case of $X$ integral.  By excellence, there is an open dense subset on which $X$ is regular and applying \autoref{lem:u-z} the Noetherian induction continues.
\item We will reduce to the case of $Y$ smooth over $S$ and without support conditions i.e., (i): As before, this is local on $X$ and $Y$, so that we may suppose they are affine.  Let $Z_X$, $Z_Y$ denote the reduced induced scheme structures on the closed subsets, and consider the diagram
  \[ \xymatrix{
  \DCoh(Z_X) \otimes_S \DCoh(Z_Y) \ar[d]  \ar[r]&\DCoh(Z_X \times_S Z_Y) \ar[d]\\
  \DCoh_{Z_X}(X) \otimes_S \DCoh_{Z_Y}(Y) \ar[r]  &\DCoh_{Z_X \times_S Z_Y}(X \times Y)} \]
  The horizontal maps are fully faithful by \autoref{prop:coh-fmk-ff}, so it suffices to prove that the bottom horizontal map has dense image; but, the right-hand vertical arrow has dense image.  This reduces us to showing that the top horizontal map is an equivalence.\qedhere
\end{enumerate}
\end{proof}

\subsection{Hochschild-type invariants of coherent complexes}\label{subsect:coh-hh}
\begin{corollary}\label{cor:hh} Suppose $X$ is a finite-type \eqref{cond:starf} derived stack over a perfect field $k$.  Then, Grothendieck duality induces
  \begin{enumerate}
    \item An isomorphism
      \[ \bHH^\bullet(\DCoh(X)) \stackrel{\sim}{\longrightarrow} \bHH^\bullet(\Perf(X)) \] of Hochschild cochain complexes.
    \item A ``Poincar\'e duality''
      \[ \bHH_\bullet(\DCoh(X)) \stackrel{\sim}{\longrightarrow} \Hom_{\QC(X)}(\Delta^* \Delta_* \O_X, \omega_X) = \RGamma\left[X, \DD\left(\ul{\bHH}_\bullet(\Perf(X))\right)\right]\]
    \item Suppose $Z \subset X$ a closed subset.  Then,
      \[ \bHH_\bullet(\DCoh_Z(X)) \stackrel{\sim}\longrightarrow \RGamma_Z \left(\ul{\bHH}_\bullet(\DCoh(X))\right)\]
  \end{enumerate}
\end{corollary} 
\begin{proof}\mbox{}
  \begin{enumerate} 
    \item Recall that $\id_{\QC(X)} = \Phi_{\O_\Delta}$ and $\id_{\QCsh(X)} = \Phish_{\omega_\Delta}$ (\autoref{thm:coh-fmk}).  So,
\begin{align*}
  \bHH^\bullet(\DCoh X) &= \Hom_{\Fun^L(\QCsh(X), \QCsh(X))}(\id,\id) \\
&= \Hom_{\QCsh(X^2)}(\omega_\Delta, \omega_\Delta)  \\
&= \Hom_{\DCoh(X^2)}(\DD \O_\Delta, \DD \O_\Delta)\\
&= \Hom_{\DCoh(X^2)}(\O_\Delta, \O_\Delta)\\
&= \Hom_{\QC(X^2)}(\O_\Delta, \O_\Delta)\\
&= \Hom_{\Fun^L(\QC(X), \QC(X))}(\id, \id) \\
&= \bHH^\bullet(\Perf X)
	\end{align*}
      \item
Recall that $\ul{\bHH}_\bullet(\Perf(X)) = \Delta^* \Delta_* \O_X$ is the sheafified Hochschild homology of $\Perf X$.  Then, \autoref{thm:coh-fmk} implies
\begin{align*}
  \bHH_\bullet(\DCoh X) &= \ev(\id_{\QCsh(X)}) \\
  &= \Hom_{\QCsh(X^2)}(\Delta_* \O_X, \Delta_* \omega_X) \\
  &= \Hom_{\DCoh(X^2)}(\Delta_* \O_X, \Delta_* \omega_X) \\
  &= \Hom_{\QC(X^2)}(\Delta_* \O_X, \Delta_* \omega_X) \\
  &= \Hom_{\QC(X)}(\Delta^* \Delta_* \O_X, \omega_X) \\
  &= \RGamma\left(\DD(\Delta^* \Delta_* \O_X)\right)
\end{align*}
\item Recall that
  \[ \ul{\bHH}_\bullet(\DCoh X) = (p_1)_* \HHom_{\QCsh(X^2)}\left( \O_X, \Delta^! \Delta_* \omega_X \right) =  \Delta^! \Delta_* \omega_X \] and that
  \begin{align*} \bHH_\bullet(\DCoh_Z X) &= \Map_{\QCsh(X^2)}\left(\Delta_* \O_X,\Delta_* \ul{\RGamma}_Z \omega_X \right) \\
    &= \Map_{\QCsh(X)}\left(\O_X, \Delta^! \Delta_* \ul{\RGamma}_Z\omega_X\right) \\
    &= \RGamma\left(\Delta^! \Delta_* \ul{\RGamma}_{Z^2} \omega_X\right) \\
    \intertext{Note that $\Delta^! \circ \ul{\RGamma}_{Z^2} \isom \RGamma_Z \circ \Delta^!$ (the left adjoints coincide), and that the natural map $\Delta_*\circ\ul{\RGamma}_Z \stackrel{\sim}\to \ul{\RGamma}_{Z^2}\circ \Delta_*$ is an equivalence (e.g., using the Cech-nerve description of \autoref{sec:more-general}.  So, we conclude}
    &= \RGamma\left(\Delta^! \ul{\RGamma}_{Z^2} \Delta_* \omega_X\right) \\
    &= \RGamma_Z \left(\Delta^! \Delta_* \omega_X\right)  \qedhere
  \end{align*}
\end{enumerate}
\end{proof}

\begin{remark}\label{rem:HH_*-dual} In particular, item (ii) implies that
  \[ U \mapsto \bHH_\bullet(\DCoh(U)) \] forms a sheaf of quasi-coherent complexes, which we'll denote $\ul{\bHH}_\bullet(\DCoh X)$.  (This can also be seen directly.)  Then, (ii) may be reformulated as the (more evidently a duality) assertion that
  \[ \ul{\bHH}_\bullet(\DCoh(X)) = \DD\left( \ul{\bHH}_\bullet(\Perf(X))\right) \]   If $X$ is proper, this implies a (vector space) duality on global sections.

  Note that this really is using duality: In the case that $X$ is smooth over a characteristic zero field, and identifying $\ul{\bHH}_\bullet(\Perf(X)) = \Omega_X^\bullet$ via HKR, this is a reflection of the sheaf perfect-pairing $\wedge \colon \Omega^\bullet_X \otimes \Omega^\bullet_X \to \omega_X$ (where $\Omega^\bullet_X = \oplus_i \Omega_X[i]$).
\end{remark}

\begin{remark} Meanwhile, item (i) seems somewhat bizarre.  It does, however, lead to the following observation: 
  
  Suppose $X$ is lci over a perfect field, and that (for simplicity) $X$ is affine.  Then, it is (?) known that thick subcategories of $\DCoh(X)$ may be classified by $\GG_m$-equivariant specialization-closed subsets of $\Spec \pi_* \bHH_\bullet(\Perf X)$.  Using the above, we may interpret this latter space as intrinsic to $\DCoh(X)$.
\end{remark}

\begin{remark}\label{rem:4-corners} \autoref{cor:hh} may be flushed out to the following picture:
  \[ \xymatrix{
  & \bHH^\bullet \DCoh(X) \ar@{-}[dl]_{(i)}^{\sim} \ar@{--}[dr]^{(ii)} \\ \bHH^\bullet \Perf(X) \ar@{..}[dr]_{(iv)}  & & \bHH_\bullet \DCoh(X) \ar@{--}[dl]^{(iii)} \\ & \bHH_\bullet \Perf(X) } \]
  \begin{enumerate}
    \item Are \emph{isomorphic} by the Corollary.
    \item Differ by a \emph{shift} provided $X$ is Calabi-Yau in the very weak sense that $\omega_X \isom \O_X[-d]$ for some $d$.  (For this, $X$ need not be smooth.  For instance, any Gorenstein local ring is Calabi-Yau in this sense.)  Indeed, \autoref{thm:coh-fmk} allows us to identify
      \[ \bHH^\bullet \DCoh(X) = \Hom_{\QCsh(X^2)}(\omega_\Delta, \omega_\Delta) = \Hom_{\QC(X)}(\Delta^* \Delta_* \O_X, \O_X) \]
      \[ \bHH_\bullet \DCoh(X) = \Hom_{\QCsh(X^2)}(\O_\Delta, \omega_\Delta) = \Hom_{\QC(X)}(\Delta^* \Delta_* \O_X, \omega_X) \]
    \item Are \emph{linearly dual} provided that $X$ is proper.  A sheafified (``local'') version of this duality holds always, by \autoref{rem:HH_*-dual}. 
    \item Are \emph{dual up to a shift} provided $X$ is proper and Calabi-Yau. (This is very well-known, at least when $X$ is also regular.)
  \end{enumerate}
\end{remark}

\bibliography{note}

\end{document}

%% file: macros.tex
\usepackage{mathrsfs}
\usepackage[mathcal]{eucal}
\usepackage{microtype}

\usepackage{fullpage}

\usepackage{textcmds}  
\usepackage{amsfonts, amssymb, amsmath, amsthm, amsopn, bm, latexsym, bbm}
\usepackage[dvipsnames,usenames]{color}
\usepackage{graphicx, fancyhdr, fancybox, ifthen, enumitem}
\usepackage[ps,matrix,curve,frame,arrow,rotate,line]{xy}

\usepackage{xr-hyper}

\usepackage{aliascnt} 
\usepackage[colorlinks=true,linkcolor=blue,citecolor=blue,urlcolor=blue,citebordercolor={0 0 1},urlbordercolor={0 0 1},linkbordercolor={0 0 1}]{hyperref} 
\usepackage[shortalphabetic]{amsrefs} 

\let\oldfootnotemark\footnotemark
\let\oldfootnotetext\footnotetext
\let\oldfootnote\footnote
\renewcommand\footnote[1]{\addtocounter{footnote}{1}\hypertarget{fnbackref.\arabic{footnote}}{}\addtocounter{footnote}{-1}\oldfootnote{#1\fnbackref}}
\renewcommand\footnotemark{\addtocounter{footnote}{1}\hypertarget{fnbackref.\arabic{footnote}}{}\addtocounter{footnote}{-1}\oldfootnotemark}
\renewcommand\footnotetext[1]{\oldfootnotetext{#1\fnbackref}}
\newcommand{\fnbackref}{\hyperlink{fnbackref.\arabic{footnote}}{\footnotesize$\uparrow$}}

\definecolor{mydefnblue}{rgb}{.2,.2,.7}
\newcommand{\demph}[1]{\textcolor{mydefnblue}{\it #1}}

\makeatletter
\newcommand{\sss}{\@startsection{subsubsection}{2}{0pt}{-3ex
plus -1ex minus -0.2ex}{-2mm plus -0pt minus
-2pt}{\normalfont\bfseries}} 
\makeatother

\definecolor{mynotecol}{rgb}{.4,.1,.1}
\let\oldmarginpar\marginpar
\renewcommand\marginpar[1]{\mbox{}\oldmarginpar[\raggedleft\hspace{0pt}\textcolor{mynotecol}{\small #1}]%
{\raggedright\hspace{0pt}\textcolor{mynotecol}{\small #1}}}

\newcommand{\refnewtheoremn}[4]{%
\newaliascnt{#1}{#2}
\newtheorem{#1}[#1]{#3}
\aliascntresetthe{#1}
\expandafter\providecommand\csname #1autorefname\endcsname{#4}}

\newcommand{\refnewtheorem}[3]{\refnewtheoremn{#1}{#2}{#3}{#3}}

\theoremstyle{plain}
\newtheorem{theorem}{Theorem}[subsection]

\refnewtheorem{conjecture}{theorem}{Conjecture}
\refnewtheoremn{prop}{theorem}{Proposition}{Prop.}
\refnewtheorem{lemma}{theorem}{Lemma}
\refnewtheoremn{corollary}{theorem}{Corollary}{Cor.}
\refnewtheorem{claim}{theorem}{Claim}

\theoremstyle{definition}
\refnewtheoremn{defn}{theorem}{Definition}{Def.}
\refnewtheorem{constr}{theorem}{Construction}
\refnewtheorem{example}{theorem}{Example}
\refnewtheorem{computation}{theorem}{Computation}
\refnewtheorem{notation}{theorem}{Notation}

\refnewtheorem{remark}{theorem}{Remark}
\refnewtheorem{spec}{theorem}{Speculation}

\refnewtheoremn{na}{theorem}{}{\!\!}

	{\begin{Sbox}\begin{minipage}{\linewidth}\begin{conjecture}}%
	{\end{conjecture}\end{minipage}\end{Sbox}\par\vspace{2pt}\noindent\fbox{\vbox{\vspace{.25pt}\TheSbox\vspace{.25pt}}}}

	{\begin{Sbox}\begin{minipage}{\linewidth}\begin{theorem}}%
	{\end{theorem}\end{minipage}\end{Sbox}\par\vspace{2pt}\noindent\fbox{\vbox{\vspace{.25pt}\TheSbox\vspace{.25pt}}}}

	{\begin{Sbox}\begin{minipage}{\linewidth}\begin{defn}}%
	{\end{defn}\end{minipage}\end{Sbox}\par\vspace{2pt}\noindent\fbox{\vbox{\vspace{.25pt}\TheSbox\vspace{.25pt}}}}

\newtheorem{rprob}{Problem}

	{\begin{Sbox}\begin{minipage}{\linewidth}\begin{rprob}}%
	{\end{rprob}\end{minipage}\end{Sbox}\par\vspace{2pt}\noindent\fbox{\vbox{\vspace{.25pt}\TheSbox\vspace{.25pt}}}}

\newcommand{\heart}{\heartsuit}

\newcommand{\eqdef}{\overset{\text{\tiny def}}{=}}
\newcommand{\ilim}{\mathop{\varprojlim}\limits}
\newcommand{\dlim}{\mathop{\varinjlim}\limits}

\newcommand{\isom}{\simeq}           

\newcommand{\quot}{/\mkern-6.5mu/}
\newcommand{\ps}[1]{[\![#1]\!]}
\newcommand{\pl}[1]{(\!(#1)\!)}
\newcommand{\sh}[1]{#1^{!}}
\newcommand{\dual}{{\mkern-1.5mu{}^{\vee}}}

\newcommand{\Lotimes}{\stackrel{_L}{\otimes}}
\newcommand{\shotimes}{\overset{!}{\otimes}}
\newcommand{\ohotimes}{\oh{\otimes}}

\newcommand{\XZ}{\oh{X_{\mkern-2muZ}}}

\newcommand{\ol}[1]{\overline{#1}}
\newcommand{\ul}[1]{\underline{#1}}
\newcommand{\oh}[1]{\widehat{#1}}
\newcommand{\sq}[1]{\widetilde{#1}}

\newcommand{\res}[2]{\left. #1 \right|_{#2}}

\renewcommand{\AA}{\mathbb{A}}
\newcommand{\CC}{\mathbb{C}}
\newcommand{\DD}{\mathbb{D}}
\newcommand{\GG}{\mathbb{G}}

\newcommand{\LL}{\mathbb{L}}

\newcommand{\PP}{\mathbb{P}}

\newcommand{\ZZ}{\mathbb{Z}}

\let\vinograd\gg 
\renewcommand{\gg}{\mathfrak{g}}
\newcommand{\mm}{\mathfrak{m}}

\newcommand{\bF}{\mathbf{F}}

\newcommand{\F}{\mathscr{F}}
\newcommand{\G}{\mathscr{G}}

\newcommand{\K}{\mathscr{K}}
\renewcommand{\L}{\mathscr{L}}
\newcommand{\N}{\mathscr{N}}
\renewcommand{\P}{\mathscr{P}}
\newcommand{\sQ}{\mathscr{Q}}

\newcommand{\X}{\mathscr{X}}
\newcommand{\Y}{\mathscr{Y}}
\newcommand{\Z}{\mathscr{Z}}

\newcommand{\A}{\mathcal{A}}

\newcommand{\C}{\mathcal{C}}
\newcommand{\D}{\mathcal{D}}

\newcommand{\I}{\mathcal{I}}
\newcommand{\M}{\mathcal{M}}

\renewcommand{\O}{\mathcal{O}}
\newcommand{\tL}{\mathcal{L}}
\newcommand{\tQ}{\mathcal{Q}}

\newcommand{\bHH}{\mathbf{HH}}

\newcommand{\bHC}{\mathbf{HC}}

\newcommand{\pt}{\ensuremath{\text{pt}}}

\newcommand{\op}{{\mkern-1mu\scriptstyle{\mathrm{op}}}}
\newcommand{\Tate}{{\mkern-2.5mu\scriptstyle{\mathrm{Tate}}}}
\newcommand{\tot}{{\mkern-1.5mu\scriptstyle{\mathrm{tot}}}}
\newcommand{\red}{{\mkern-1.5mu\scriptstyle{\mathrm{red}}}}

\newcommand{\ev}{{\mkern-1.5mu\scriptstyle{\mathrm{ev}}}}

\newcommand{\gr}{{\mkern-1.25mu\scriptstyle{\mathrm{gr}}}}
\newcommand{\kinfty}{{\mkern-1mu\scriptstyle{\infty}}}

\renewcommand{\mod}{\text{\rm{-mod}}}  
\newcommand{\comod}{\text{\rm{-comod}}}
\newcommand{\dom}{\text{\rm{mod-}}}
\newcommand{\rmod}{\dom}

\newcommand{\Aff}{\mathrm{Aff}}

\newcommand{\SO}{\operatorname{SO}}

\DeclareMathOperator{\MF}{MF}
\DeclareMathOperator{\PreMF}{PreMF}
\DeclareMathOperator{\CircMF}{CircMF}
\DeclareMathOperator{\Cpx}{Cpx}

\DeclareMathOperator{\Coh}{Coh}
\DeclareMathOperator{\DCoh}{DCoh}
\DeclareMathOperator{\PsCoh}{PsCoh}
\DeclareMathOperator{\DSing}{DSing}
\DeclareMathOperator{\Perf}{Perf}
\DeclareMathOperator{\QC}{QC}
\newcommand{\QCsh}{\sh{\QC}}

\DeclareMathOperator{\Fuk}{Fuk}

\DeclareMathOperator{\Cliff}{Cliff}
\DeclareMathOperator{\PreCliff}{PreCliff}

\newcommand{\Phish}{\sh{\Phi}}
\newcommand{\Psish}{\sh{\Psi}}

\DeclareMathOperator{\tr}{tr}
\DeclareMathOperator{\im}{im}
\DeclareMathOperator{\cone}{cone}

\DeclareMathOperator{\Aut}{Aut}

\DeclareMathOperator{\End}{End}
\DeclareMathOperator{\Ext}{Ext}
\DeclareMathOperator{\Hom}{Hom}
\DeclareMathOperator{\Map}{Map}

\DeclareMathOperator{\RGamma}{R\Gamma}

\DeclareMathOperator{\Sym}{Sym}

\DeclareMathOperator{\RHom}{RHom}
\newcommand{\HHom}{\mathcal{RH}om} 
\DeclareMathOperator{\Kos}{Kos}

\DeclareMathOperator{\Spec}{Spec}
\DeclareMathOperator{\Spf}{Spf}

\DeclareMathOperator{\Proj}{Proj}

\DeclareMathOperator{\cval}{cval}
\DeclareMathOperator{\crit}{crit}

\DeclareMathOperator{\cosk}{cosk}

\DeclareMathOperator{\Fun}{Fun}

\DeclareMathOperator{\id}{id}

\DeclareMathOperator{\Ind}{Ind}

\DeclareMathOperator{\fib}{fib}

\DeclareMathOperator{\Tot}{Tot}
\newcommand{\hocolim}{\operatorname{hocolim}\limits}
\newcommand{\holim}{\operatorname{holim}\limits}

\newcommand{\dgcat}{\mathbf{dgcat}}

\newcommand{\Cat}{\mathbf{Cat}}
\newcommand{\sSet}{\mathbf{sSet}}

\newcommand{\sAb}{\mathbf{sAb}}
\newcommand{\Mon}{\mathbf{Mon}}
\newcommand{\CMon}{\mathbf{CMon}}
\newcommand{\sGp}{\mathbf{sGp}}
\newcommand{\Sp}{\mathbf{Sp}}

\newcommand{\Kan}{\mathbf{Kan}}

\newcommand{\CAlg}{\mathbf{CAlg}}

\newcommand{\Alg}{\mathbf{Alg}}

\newcommand{\adjunct}[2]{\xymatrix@1{ #1 \ar@<.7ex>[r] & \ar@<.7ex>[l] #2 }}

\newcommand{\ssetr}[2]{\xymatrix@1{ #1 \ar[r] & \ar@<.7ex>[l] #2 \ar@<-.7ex>[l] \ar@<.7ex>[r] \ar@<-.7ex>[r] & \cdots \ar@<1.4ex>[l] \ar@<-1.4ex>[l] \ar[l] }}

\newcommand{\ssetlar}[4]{\xymatrix@1{  #1 \ar@<.7ex>[r]^-{#2} \ar@<-.7ex>[r]_-{#3} & #4}}

\newcommand{\ssetl}[2]{\xymatrix@1{ \cdots \ar@<1.4ex>[r] \ar@<-1.4ex>[r] \ar[r] & #2 \ar@<.7ex>[r] \ar@<-.7ex>[r] \ar@<.7ex>[l] \ar@<-.7ex>[l] & #1 \ar[l]}}

\newcommand{\ssetrr}[3]{\xymatrix@1{ #1 \ar[r] & #2 \ar@<.7ex>[l] \ar@<-.7ex>[l] \ar@<.7ex>[r] \ar@<-.7ex>[r] & #3 \ar@<1.4ex>[l] \ar@<-1.4ex>[l] \ar[l]  \ar@<1.4ex>[r] \ar@<-1.4ex>[r] \ar[r] & \cdots\ar@<.7ex>[l] \ar@<-.7ex>[l] \ar@<2.1ex>[l] \ar@<-2.1ex>[l]}}

\newcommand{\ssetll}[3]{\xymatrix@1{ \cdots\ar@<.7ex>[r] \ar@<-.7ex>[r] \ar@<2.1ex>[r] \ar@<-2.1ex>[r] & #3 \ar@<1.4ex>[r] \ar@<-1.4ex>[r] \ar[r]  \ar@<1.4ex>[l] \ar@<-1.4ex>[l] \ar[l] & #2 \ar@<.7ex>[r] \ar@<-.7ex>[r] \ar@<.7ex>[l] \ar@<-.7ex>[l] & #1 \ar[l]}}